\documentclass[11pt]{amsart}

\usepackage{amsmath,
amssymb,
amscd,
amsfonts,
bm,
color,
enumitem,
epsfig,
float,
graphicx,
hyperref,
latexsym,
mathrsfs,
mathtools,
psfrag,
transparent,
url,
wasysym,
wrapfig,
tikz,
tikz-cd,
xr
}
\usepackage{fullpage}
\usepackage[all]{xy}
\usepackage[T1]{fontenc}

\newtheorem{theorem}{Theorem}[section]
\newtheorem{corollary}[theorem]{Corollary}

\newtheorem{conjecture}[theorem]{Conjecture}

\newtheorem{lemma}[theorem]{Lemma}

\newtheorem{rmk}[theorem]{Remark}    

\theoremstyle{definition}
\newtheorem{definition}[theorem]{Definition}
\newtheorem{example}[theorem]{Example}
\theoremstyle{remark}

\theoremstyle{plain}
\newcommand{\thistheoremname}{}
\newtheorem{genericthm}[theorem]{\thistheoremname}

\newtheorem*{genericthm*}{\thistheoremname}
\newenvironment{namedthm*}[1]
  {\renewcommand{\thistheoremname}{#1}%
   \begin{genericthm*}}
  {\end{genericthm*}}
  
\newcommand{\ep}{\varepsilon}
\newcommand{\eps}{\epsilon}
\newcommand{\beps}{{\bm{\epsilon}}}

\def\om{\omega}

\def\P{\Phi}

\newcommand\cC{\mathcal{C}}

\newcommand\cE{\mathcal{E}}
\newcommand\cF{\mathcal{F}}
\newcommand\cG{\mathcal{G}}
\newcommand\cH{\mathcal{H}}

\newcommand\cK{\mathcal{K}}
\newcommand\cL{\mathcal{L}}
\newcommand\cM{\mathcal{M}}

\newcommand\cR{\mathcal{R}}

\newcommand\cU{\mathcal{U}}
\newcommand\cV{\mathcal{V}}

\newcommand{\bE}{\mathbb{E}}

\newcommand{\bN}{\mathbb{N}}

\newcommand{\bR}{\mathbb{R}}

\newcommand{\fC}{\mathfrak{C}}
\newcommand{\fc}{\mathfrak{c}}

\newcommand{\fK}{\mathfrak{K}}
\newcommand{\fk}{\mathfrak{k}}
\newcommand{\fl}{\mathfrak{l}}
\newcommand{\fv}{\mathfrak{v}}

\newcommand{\fg}{\mathfrak{g}}

\def\rD{{\rm D}}
\def\rT{{\rm T}}
\def\rd{{\rm d}}

\newcommand{\on}{\operatorname}

\renewcommand\Im{\on{Im}}

\newcommand\coker{\on{coker}}

\newcommand{\im}{\on{im}}

\newcommand{\ind}{{\on{ind}}\,}

\def\pd{\partial}

\def\pd{\partial}

\newcommand{\ol}{\overline}
\newcommand{\ul}{\underline}
\newcommand{\ti}{\widetilde}

\newcommand{\wh}{\widehat}

\newenvironment{itemlist}
   { \begin{list} {$\bullet$}
         { \setlength{\topsep}{.5ex}  \setlength{\itemsep}{.5ex} \setlength{\leftmargin}{2.5ex} } }
   { \end{list} }
   
   \newenvironment{iitemlist}
   { \begin{list} {$\bullet$}
         { \setlength{\topsep}{.5ex}  \setlength{\itemsep}{.5ex} \setlength{\leftmargin}{7ex} } }
   { \end{list} }
   
   \newenvironment{iiitemlist}
   { \begin{list} {$\bullet$}
         { \setlength{\topsep}{.5ex}  \setlength{\itemsep}{.5ex} \setlength{\leftmargin}{7.5ex} } }
   { \end{list} }

\newcounter{qcounter}

\makeatletter
\@namedef{subjclassname@2020}{%
  \textup{2020} Mathematics Subject Classification}
\makeatother


\begin{document}
\sloppy

\title{Adiabatic Fredholm Theory}
\author{Nathaniel Bottman}
\author{Katrin Wehrheim}

\maketitle

\begin{abstract}
We develop a robust functional analytic framework for adiabatic limits. This framework consist of a notion of adiabatic Fredholm family, several possible regularity properties, and an explicit construction that provides finite dimensional reductions that fit into all common regularization theories. We show that thhese finite dimensional reductions inherit global continuity and differentiability properties from the adiabatic Fredholm family. 
Moreover, we indicate how to construct adiabatic Fredholm families that describe the adiabatic limits for the nondegenerate Atiyah-Floer conjecture 
and strip-shrinking in quilted Floer theory. 
\end{abstract}

{\small
\setcounter{tocdepth}{2}
\tableofcontents
}

\section{Introduction}
This paper develops an analytic framework for adiabatic limits in geometric analysis -- towards systematizing the use of PDE degenerations to relate invariants arising from moduli spaces of solutions of nonlinear elliptic PDEs of quite different types.
Consider the following examples:

\begin{example} \label{ex:adiabatic}
\begin{itemlist}
\item[{\bf [AF]}]
The (nondegenerate) Atiyah-Floer conjecture \cite{dostoglou-salamon, salamon_ICM} relates an instanton Floer homology (closely related to Donaldson invariants) to a Lagrangian Floer homology by degenerating the anti-self-duality equation on a bundle over $\bR^2\times\Sigma$ to a Cauchy-Riemann equation for maps $\bR^2\to\cR(\Sigma)$ to a representation space of a Riemann surface $\Sigma$. 
\item[{\bf [GW]}]
The symplectic vortex invariants for a Hamiltonian group action $G\times M\to M$ on a symplectic manifold $M$ are identified
\cite{gaio-salamon, cieliebak-gaio-siera-salamon} with Gromov-Witten invariants of the symplectic quotient $M /\!\!/ G$ -- under a number of assumptions, in particular monotonicity -- 
by degenerating the symplectic vortex equations for maps $\Sigma\to M$ coupled with a $G$-connection on a bundle over $\Sigma$ to a Cauchy-Riemann equation for maps $\Sigma\to M /\!\!/ G$. 
\item[{\bf [SW]}]
Relationships between Seiberg-Witten invariants and Heegaard-Floer theory \cite{kutluhan-lee-taubes} can be interpreted \cite[\S 10]{cieliebak-gaio-siera-salamon} as a degeneration of the Seiberg-Witten equations on a bundle over $\bR^2\times\Sigma$ to a Cauchy-Riemann equation for maps $\bR^2\to{\rm Sym}^g(\Sigma)$ to the symmetric product of a Riemann surface $\Sigma$.
\item[{\bf [QF]}]
Quilted Floer theory \cite{wehrheim_woodward_geometric_composition} was shown -- under transversality and monotonicity assumptions -- to be invariant under geometric composition of Lagrangian correspondences, by a degeneration of the Cauchy-Riemann equation for tuples of strips of varying width (whose boundaries are coupled via Lagrangian conditions) -- allowing widths to go to zero, and replacing the shrunk strip by composition of the associated Lagrangian correspondences. 
\end{itemlist}
\end{example}

Examples [AF] and [QF] will be described in more detail in \S~\ref{examples}.
Each of these degenerations can be described locally by a smooth family of nonlinear Fredholm maps (or elliptic differential operators) $\cF_\eps:\ol\Gamma\to\ol\Omega$ with fixed domain and target for $\eps>0$ that are related to an energy functional $\cE_\eps:\ol\Gamma\to[0,\infty)$ with the property that solutions have a fixed energy $\cF_\eps(\gamma)=0 \;\Rightarrow\; \cE_\eps(\gamma)=E_0$ that is $\eps$-independent. In the limit $\eps\to 0$, however, the energy functional diverges on most of $\ol\Gamma$ -- with the exception of a subspace $\Gamma_0:=\{\gamma\in\ol\Gamma \,|\, \sup_{\eps> 0}E_\eps(\gamma) <\infty \} \subset\ol\Gamma$. Corresponding, the differential operators $\cF_\eps$ diverge on most of $\ol\Gamma$, but their restrictions to $\Gamma_0$ have a well-defined limit $\cF_0:=\lim_{\eps\to 0}\cF_\eps|_{\Gamma_0}$. 
The limit map (or operator) $\cF_0:\Gamma_0\to\ol\Omega$ has the same energy of solutions $\cF_0(\gamma)=0 \;\Rightarrow\; \lim_{\eps\to 0}\cE_\eps(\gamma)=E_0$, although it is no longer Fredholm (or elliptic). However, there are projections $\Gamma_0 \twoheadrightarrow \Gamma^{\rm red}_0$ and 
$\ol\Omega \twoheadrightarrow \Omega^{\rm red}_0$ such that $\cF_0$ descends to a an operator $\cF^{\rm red}_0:\overline\Gamma^{\rm red}_0\to\ol\Omega^{\rm red}_0$ that -- after completion -- is Fredholm (or elliptic) with the same index as $\cF_\eps$.

\begin{rmk}[The notion ``adiabatic'']
The name ``adiabatic limit'' or ``adiabatic degeneration'' that has been used for these examples was probably motivated by some or all of the notions
\begin{itemlist}
\item[$\cdot$]
adiabatic change in mechanics -- a slow deformation of the Hamiltonian,
\item[$\cdot$]
adiabatic invariants in dynamical systems -- stay approximately constant under slow change,
\item[$\cdot$]
adiabatic processes in thermodynamics -- that don't increase the entropy, 
\item[$\cdot$]
adiabatic change in quantum mechanics -- where energy states cannot transition. 
\end{itemlist}
In the context of geometric analysis, we suggest using these terms when both an energy and a Fredholm index are preserved -- akin to an adiabatic invariant -- which prevents major changes of geometric characteristics in the limit. We would add the common characteristic of this preservation taking place despite of a serious analytic degeneration -- e.g.\ quantities diverging, Fredholmness being lost, and being regained only after restriction to a smaller ``state space''. 

We will give a technical definition of ``adiabatic Fredholm family'' in Definition~\ref{def:fredholm}, which is highly specific to our purposes and results from a significant reformulation of the classical examples. It does not include an energy function, as this is more relevant to compactness arguments than the construction of local charts. Instead, the key property of an ``adiabatic Fredholm limit" is that Fredholmness holds relative to a family of $\eps$-dependent norms, which are bounded below but not equivalent to the $\eps=0$ norm. 
Note that we do not intend to restrict the use of the term ``adiabatic'' to this specific meaning -- rather use it as an adjective to indicate the generalization of a classical notion that serves to better analyze some adiabatic limits.
\end{rmk}

The (Floer-type / Gromov-Witten / ...) invariants arising from the PDEs in Example~\ref{ex:adiabatic} are constructed by giving the moduli spaces of PDE solutions (modulo symmetries and after compactification) the structure of a space that can be counted (e.g.\ a compact oriented $0$-manifold), integrated over, or associated with some type of fundamental class. For a moduli space $\ol\cM_\eps$ arising from a PDE with fixed $\eps\geq 0$, various types of such ``regularization structures" have been constructed -- all arising from local Fredholm descriptions $\cF_\eps:\ol\Gamma\to\ol\Omega$ as described above.
Now any ``regularization theory" (a method for assigning counts or (virtual/relative/...) fundamental classes to moduli spaces) crucially relies on the fact that Fredholm descriptions have transverse stabilizations. These are used differently by the various theories, but a key fact in all theories is that the stabilizations exist and induce ``finite dimensional reductions" in the following sense.

\begin{rmk} \label{rmk:finite dimensional} \rm
Let $\ol\cM$ be a topological space. 
Then a {\bf local Fredholm description} of $\ol\cM$ consists of a $\cC^1$ Fredholm\footnote{
A nonlinear map is called Fredholm if its linearizations $\rD\cF(\gamma_0)$ are Fredholm at each zero $\gamma_0\in\cF^{-1}(0)$.} 
map $\cF:\ol\Gamma \to \ol\Omega$ between Banach spaces and a continuous injection $\cF^{-1}(0)\hookrightarrow\ol\cM$ that is a homeomorphism onto its image. (More precisely, $\cF$ is defined on an open subset $\cV_\Gamma\subset\Gamma$.)  

Since the cokernel of a Fredholm map is finite dimensional, any such local Fredholm description can be extended to a {\bf stabilized Fredholm map} $\ti\cF: \fC\times\ol\Gamma \to \ol\Omega$ that is transverse to $0$ (after restriction to a smaller domain). Here $\fC$ is a finite dimensional vector space that is either isomorphic to $\coker \rD\cF(\gamma_0)$ or can be viewed to contain it as a subspace, and transversality guarantees that $\ti\cF^{-1}(0)$ is a finite dimensional manifold equipped with a $\cC^1$ map $\pi:\ti\cF^{-1}(0)\to\fC$ whose zero set is identified with $\cF^{-1}(0)$. 
This fact can be viewed as ``local obstruction bundle'', source of ``transverse perturbations" (arising from the regular values of $\pi$), or a ``finite dimensional reduction''.

More generally, a {\bf finite dimensional reduction} of a local Fredholm description is a $\cC^1$ map $f:\fK \to \fC$ between finite dimensional vector spaces and a continuous injection $\phi: f^{-1}(0) \hookrightarrow \cF^{-1}(0)$ that is a homeomorphism onto its image.
 A review of the classical construction of finite dimensional reductions -- in a less well-known formulation that we developed for the adiabatic generalization -- can be found in Lemmas~\ref{lem:inverses classical}, \ref{lem:contraction classical}, \ref{lem:solution classical}, \ref{lem:finite dimensional reduction classical}.

Thus any local Fredholm description, via stabilization and finite dimensional reduction, induces a {\bf local finite dimensional description}  of $\ol\cM$ that consists of a $\cC^1$ map $f:\bR^k\to\bR^n$ and a continuous injection $f^{-1}(0) \hookrightarrow \ol\cM$ that is a homeomorphism onto its image. 
\end{rmk}

Once finite dimensional descriptions of a moduli space are constructed, they induce ``local regularization structures" (e.g.\ local Euler classes or local perturbations), and the content of a regularization theory is to patch these into a global regularization structure for the moduli space. Then invariants are defined by counting or integrating ``regularized moduli spaces" (and packaging the results in an algebraic structure than can be shown to be independent of choices.)

Now the idea for identifying invariants for $\eps=1$ and $\eps=0$ is that the union of moduli spaces $\ol\cM_{[0,1]}:=\bigcup_{0\leq\eps\leq 1}\ol\cM_\eps$ could be equipped with a regularization structure that plays the role of a cobordism relating the invariants arising from $\ol\cM_0$ and $\ol\cM_1$.
The crucial analytic step in such a program is to stabilize the local Fredholm descriptions 
$\cF_0^{{\rm red}\;-1}(0) \cup \bigcup_{0<\eps\leq 1}\cF_\eps^{-1}(0) \hookrightarrow \ol\cM_{[0,1]}$ 
and perform finite dimensional reduction to obtain a local finite dimensional description of $\ol\cM_{[0,1]}$. 
To date, such description for an adiabatic limit has been achieved only in the special case of $\cF_0^{{\rm red}\;-1}(0)$ starting out as a $0$-manifold arising from $\cF_0$ being transverse to $0$.\footnote{Transversality in this case of Fredholm index $0$ means that at each solution $\cF_0(\gamma)=0$ the linearization $\rD\cF_0(\gamma)$ is an isomorphism.}
In that case no stabilization is necessary and a bijection $\ol\cM_0\simeq \ol\cM_\eps$ for sufficiently small $\eps>0$ is established in two steps by the

\smallskip
\noindent
{\bf classical adiabatic method} (for transverse problems of Fredholm index $0$):
\begin{itemlist}
\item
Given any solution $\gamma^{\rm red}_0\in\cF_0^{{\rm red}\;-1}(0)$ one constructs a family of solutions $\gamma_\eps\in\cF_{\eps}^{-1}(0)$ for $\eps>0$ by utilizing a lift of $\gamma^{\rm red}_0$ to $\gamma_0\in\Gamma_0$ to obtain ``near-solutions'': $\cF_\eps(\gamma_0)\to 0$ as $\eps\to 0$. Then for sufficiently small $\eps$ a Newton-Picard iteration finds exact solutions $\gamma_0+\xi_\eps\in\cF_{\eps}^{-1}(0)$. 
(See e.g.\ \cite[Thm.5.1]{dostoglou-salamon}, \cite[Thm.3.1.1]{wehrheim_woodward_geometric_composition}.)
\item 
A compactness theorem (e.g.\ \cite[Thm.9.2]{dostoglou-salamon}, \cite[Thm.3.3.1]{wehrheim_woodward_geometric_composition},  \cite{bottman_wehrheim}) shows that solutions $\gamma_\eps\in\cF_{\eps}^{-1}(0)$ for $\eps\to 0$ will converge to a lift $\gamma_0\in\Gamma_0$ of some solution $\gamma^{\rm red}_0\in\cF_0^{{\rm red}\;-1}(0)$.
\end{itemlist}
This can also be viewed as constructing a $1$-cobordism $\bigcup_{\gamma^{\rm red}_0\in\cF_0^{{\rm red}\;-1}(0)}  \{\gamma^{\rm red}_0\} \cup \bigcup_{0<\eps\leq\eps_0} \{\gamma_0+\xi_\eps\}$ with boundary $\cF_0^{{\rm red}\;-1}(0) \sqcup \cF_{\eps_0}^{-1}(0)$. However, the local finite dimensional descriptions of $\ol\cM_{[0,1]}$ constructed here are the trivial functions $f:[0,\eps_0]\to\bR^0=\{0\}$ with a nontrivial injection $[0,\eps_0]\to\ol\cM_{[0,1]}$, which maps $0\mapsto \gamma^{\rm red}_0$ and $0<\eps\mapsto \gamma_0+\xi_\eps$. This injection is constructed by working only with solutions of the PDEs -- thus provides no hints on how to lift a stabilization or finite dimensional reduction of $\cF^{\rm red}_0:\overline\Gamma^{\rm red}_0\to\ol\Omega^{\rm red}_0$ to a construction for  $\cF_\eps:\overline\Gamma\to\ol\Omega$.

\begin{rmk}[The broken dream of using sc-retracts] \label{rmk:broken}
Ever since polyfold theory \cite{hofer-wysocki-zehnder_polyfold} entered the market of ideas in 2004, we have expected to eventually obtain finite dimensional descriptions for adiabatic limits $\cF_\eps \to \cF_0 \twoheadrightarrow \cF^{\rm red}_0$ by casting them as a single sc-Fredholm section over a polyfold of the form $\overline\Gamma^{\rm red}_0 \cup \bigcup_{0<\eps\leq1} \ol\Gamma$.  After all, the very first examples of ``sc-retracts'' can be used to describe this kind of base space as the image of a family of projections $(\pi_\eps: \bE \to \bE )_{\eps\in[0,1]}$  with $\im\pi_0 \simeq \ol\Gamma^{\rm red}_0$ and $\im\pi_\eps \simeq \ol\Gamma$ for $\eps>0$. 
However, the ambient ``sc-space'' for this family is, essentially, $\bE=L^2(\bR)\times\ol\Gamma^{\rm red}_0\times L^2(\bR)\times\ol\Gamma$, and -- despite significant efforts -- we have been unable to find an extension of the adiabatic family of PDEs to this or any other ambient space of an sc-retract describing the adiabatic limit. 

We have found that a crucial difference between gluing constructions -- for which polyfold theory offers a robust analytic framework -- and the Newton-Picard iteration used for adiabatic limits is the following: Gluing constructions use the Newton-Picard iteration to find a PDE solution near a pre-glued map $\oplus_R(u_a, u_b)$, which interpolates between two PDE solutions $u_a,u_b$ on a neck of length $R$. 
Now this pregluing map $(R,u_a,u_b)\mapsto \oplus_R(u_a, u_b)$ is well-defined for any pair of maps $u_a,u_b$, and is accompanied by a similar map $(R,u_a,u_b)\mapsto \ominus_R(u_a, u_b)$ which keeps track of the information lost by $\oplus_R$.\footnote{In fact $(\oplus_R, \ominus_R)$ is an isomorphism; see \cite[\S 2.3]{fabert-fish-golovko-wehrheim} for a survey of this core idea of polyfold theory.} 
Together, these serve to reinterpret the pregluing construction as a chart for the ambient space $\Gamma_\infty \cup \bigcup_{R\geq 1} \Gamma_R$ of maps from domains with neck length $1\leq R \leq \infty$. The total space for this sc-retract is $[1,\infty]\times\Gamma_\infty$ -- thus supports the same type of elliptic PDE as the gluing problem. 

For adiabatic limits the analogue of the pregluing is the lift $(\eps,\gamma^{\rm red}_0) \mapsto \gamma_0$. However, this is a well-defined map only on a dense subspace of $\ol\Gamma^{\rm red}_0$ which contains all smooth maps, in particular the PDE solutions. Even if these lifts did cover $\Gamma_0$, we would then need a direct sum $\Gamma=\Gamma_0\oplus\Gamma_{\ominus}$ to mimic the construction of the sc-retract from $(\oplus,\ominus)$. Such splittings, however, have been elusive for the adiabatic limits {\rm [AF]}, {\rm [QF]} that we studied in detail. 
\end{rmk}

The gauge-theoretic Examples~\ref{ex:adiabatic}~[AF], [SW] are naturally monotone and can be regularized with geometric methods, so a satisfying level of generality is achievable with the classical adiabatic method. However, the inherently symplectic Examples~[GW], [QF] are severely limited by the monotonicity assumption, which is required both to avoid bubbling (which in both cases is understood to algebraically obstruct the desired result) and to ensure that transversality of $\cF^{\rm red}_0$ can be achieved by a geometric method -- i.e.\ one that is compatible with the classical method. 
For Example [QF] the algebraic impact of bubbling has now been understood and cast into the first author's proposal of the symplectic $(A_\infty,2)$-category \cite{abouzaid-bottman}. And with the compactness theorem long established \cite{bottman_wehrheim}, a local finite dimensional description for the moduli spaces near $\eps=0$ for Example~\ref{ex:adiabatic}~[QF] became the only foundational piece missing for a systematic description of the functorial properties of Fukaya categories. The present paper fills this gap. 

Such a general description of the moduli spaces of quilted Floer trajectories (solutions of several Cauchy-Riemann equations, coupled by Lagrangian seam conditions; see \S\ref{examples} for details) in the case of one strip-width going to zero requires extending the adiabatic ``strip-shrinking'' analysis in \cite{wehrheim_woodward_geometric_composition} to cases where the moduli spaces are not a priori cut out transversely, and of any expected dimension -- i.e.\ Fredholm index. 
Moreover, the proposed symplectic $(A_\infty,2)$-category also requires extending this adiabatic analysis to multiple strips shrinking to width zero -- at any ratio of speeds. 
And for the rapidly growing number of proposed applications of pseudoholomorphic quilts -- which utilize more general domains, symplectic fibrations, shrinking of annuli, and possibly even more sophisticated degenerations -- each use case requires its own version of this adiabatic analysis. 
Finally, any non-monotone use case requires coupling the requisite adiabatic analysis with the gluing analysis for the four types of bubbles exhibited in \cite{bottman_wehrheim} -- in particular the figure-eight bubbles, of which any number can appear simultaneously in a top boundary stratum. 
To serve all those use cases in a non-partisan manner -- i.e.\ without restricting the user to a particular regularization theory --  the goal of this paper is to develop a robust functional analytic framework for adiabatic limits that 
\begin{itemlist}
\item[(a)]
applies at least to Example~\ref{ex:adiabatic}~[QF], ideally to more of the known examples;
\item[(b)]
provides finite dimensional reductions that fit into all common regularization theories;  
\item[(c)]
is compatible with all common approaches to gluing. 
\end{itemlist}
We propose such a framework with the notion of an {\bf adiabatic Fredholm family} in Definition~\ref{def:fredholm}.
Note that it would be relatively easy to create a definition that satisfies just one of the goals (a) or (b). To avoid this common issue of new theoretical frameworks, we used goal (a) as a guardrail in developing this notion, while striving to meet goal (b) by proving Theorem~\ref{thm:charts} below -- i.e.\ in working towards finite dimensional reductions, we allowed ourselves to add conditions to Definition~\ref{def:fredholm} only when these were satisfied in Example~\ref{ex:adiabatic}~[QF]. As such, the actual definition has become quite technical, so here is a simplified version that combines Definitions~\ref{def:fredholm}, \ref{def:adiabatic C-l} for $\cC^{\ell=1}$-regularity.  

\smallskip
\noindent
{\bf Simplified Definition:}
A {\bf $\cC^1$-regular adiabatic Fredholm family} consists of  
\begin{itemize}
\item
$\bigl(\cF_\eps : \Gamma \to \Omega\bigr)_{\eps\in\Delta}$  a family of maps between real vector spaces indexed by a topological space $\Delta$ with a distinguished point $0=0_\Delta\in\Delta$ such that $\cF_0(0)=0$;
\item
 families of norms $\|\cdot\|^\Gamma_\eps$ and $\|\cdot\|^\Omega_\eps$ on $\Gamma$ and $\Omega$, respectively, indexed by $\eps\in\Delta$.
\end{itemize}
These are required to have the following properties.
\begin{itemlist}
\item[\bf{[Lower Bound on Norms]}]  
$\|\gamma \|^\Gamma_0 \leq \|\gamma \|^\Gamma_\eps$ and 
$\|\omega \|^\Omega_0 \leq \|\omega \|^\Omega_\eps$
for all $\gamma\in\Gamma$, $\omega\in\Omega$, and $\eps\in\Delta$.
\item[\bf{[Uniform $\mathcal{C}^1$ Regularity]}]
$\cF_\eps : (\Gamma,\|\cdot\|^\Gamma_\eps) \to (\Omega,\|\cdot\|^\Omega_\eps)$ is uniformly $\cC^1$ for each $\eps\in\Delta$ and satisfies 
$$
\bigl\| \rD \cF_\eps (\gamma^\fl)  \bigr\|^{\cL(\ol\Gamma_\eps,\ol\Omega_\eps )}  \leq C^1_\cF
\qquad\text{and}\qquad
\bigl\| \rD\cF_\eps(\gamma^\fl) -  \rD\cF_\eps (\gamma^\fk) \bigr\|^{\cL(\ol\Gamma_\eps,\ol\Omega_\eps )} 
\leq c^1_\cF(\|\gamma^\fl-\gamma^\fk\|^{\Gamma}_\eps) 
\qquad \forall\; \gamma^\fl, \gamma^\fk\in\Gamma
$$
with a constant $C^1_\cF\geq 1$ and a monotone continuous function $c^1_\cF : [0,\infty) \to [0,\infty)$ with $c^1_\cF(0)= 0$.
\item[\bf{[Fredholm Property \& Constant Index]}]
The linearizations at $0=0_\Gamma$ extend to Fredholm operators $\overline{\rD\cF_\eps(0)}: \overline\Gamma_\eps\to \overline\Omega_\eps$ between the completions $\overline\Gamma_\eps:=\overline\Gamma^{\|\cdot\|^\Gamma_\eps}$ and $\overline\Omega_\eps:=\overline\Omega^{\|\cdot\|^\Omega_\eps}$ for each $\eps\in\Delta$. 

Their Fredholm index is independent of $\eps\in\Delta$.
 \item[\bf{[$\mathbf{\epsilon=0}$ Fredholm Estimate \& Uniform Cokernel Bound]}] 
There is a projection $\pi_\fK:\Gamma\to\fK:=\ker \overline{\rD\cF_0(0)}$ and an inclusion $\fC:=\coker\overline{\rD\cF_0(0)}\subset\Omega$ with 
$\sup_{\eps\in\Delta} \sup_{\|\fc\|^\Omega_0\leq 1} \| \fc  \|^\Omega_\eps <\infty$ such that
$$
\| \gamma \|^\Gamma_0 + \| \fc  \|^\Omega_0  \leq C_0 \bigl( \|\pi_\fK(\gamma) \|^\fK + \| \rD\cF_0(0) \gamma - \fc \|^\Omega_0  \bigr)\qquad \forall\; (\gamma,\fc)\in\Gamma\times \fC .
$$ 
\item[\bf{[Uniform Fredholm-ish Estimate]}] 
$\| \gamma \|^\Gamma_\eps \leq C_1 \bigl( \| \rD\cF_\eps(0) \gamma \|^\Omega_\eps + \| \gamma \|^\Gamma_0 \bigr)$ for all $\gamma\in\Gamma$ and $\eps\in\Delta$.

\item[\bf{[Pointwise $\cC^1$-continuity w.r.t.\ $\mathbf{\Delta}$]}] 
$\bigl\|  \rD\cF_\eps (0) -   \rD\cF_0 (0)  \bigr\|^{\cL(\ol\Gamma_\ep,\ol\Omega_0)} \to 0$ as $\eps\to \eps_0$
and, given any $\eps_0\in\Delta$ and solutions $\gamma_0,\gamma_1\in \Gamma$ of $\cF_{\eps_0}(\gamma_0)=\fc_0\in\fC$, $\rD\cF_{\eps_0}(\gamma_0)\gamma_1=\fc_1\in\fC$, we have 
$$
\bigl\| \cF_\eps (\gamma_0)  -  \cF_{\eps_0} (\gamma_0) \bigr\|^\Omega_\eps 
+ \bigl\| \rD\cF_\eps (\gamma_0)\gamma_1  -  \rD\cF_{\eps_0} (\gamma_0) \gamma_1\bigr\|^\Omega_\eps 
\underset{\eps\to \eps_0}{\longrightarrow} 0 .
$$
\item[\bf{[Regularizing]}] 
The nonlinear and linearized operators are regularizing in the sense that
\begin{align*}
\gamma\in \overline\Gamma_\eps, \; \overline\cF_\eps(\gamma) \in \Omega \quad \Longrightarrow \quad \gamma\in \Gamma ,\\
\gamma_0\in \Gamma, \; \gamma\in\overline\Gamma_\eps, \; \rD\overline\cF_\eps(\gamma_0)\gamma \in \Omega \quad \Longrightarrow \quad \gamma\in\Gamma .
\end{align*}
\end{itemlist}

For users wishing to avoid noncomplete normed vector spaces, the alternative formulation in terms of an extended adiabatic Fredholm family can be found in Lemma~\ref{lem:fredholm}. However, the regularizing property will be crucial for constructing finite dimensional reductions that are continuous in $\eps\in\Delta$. 
Defined formally in Definition~\ref{def:regularizing}, the regularizing property encodes elliptic regularity of the PDEs -- similar to \cite[Def.3.1.16]{hofer-wysocki-zehnder_polyfold}.\footnote{The regularizing property is unfortunately named similarly, but should not be confused with regularization -- the process of associating well-defined counts or fundamental classes to moduli spaces, as e.g.\ surveyed in \cite[\S 3]{fabert-fish-golovko-wehrheim}.}

\begin{rmk}[Applicability to examples of adiabatic limits]
In the application to Examples~\ref{ex:adiabatic}, 
\begin{itemlist}
\item
$\Gamma,\Omega$ are spaces of smooth connections or maps -- equipped with $\eps$-dependent Sobolev norms, 
\item
$\Delta=[0,\eps_0]$, though our formulation allows for $\Delta$ to be any chart of an underlying moduli space of domains -- including one with generalized corners, 
\item
$\cF_\eps$ represent the $\eps$-dependent PDEs in a local slice.\footnote{For gauge theoretic symmetries one usually builds gauge fixing into the PDE. For quotients by reparametrization one usually builds local slice conditions into the domain.}
\end{itemlist}
The original formulations, including the above abstract summary, are not of this form since the PDEs contain negative exponents of $\eps$ which diverge for $\eps\to 0$. They can, however, be brought into this form by multiplying with a suitable positive exponent of $\eps$ such that $\lim_{\eps\to 0} \cF_\eps =:\cF_0$ exists -- and combines the $\eps=0$ PDE with the condition for solutions to be contained in $\Gamma_0$.  
Here our breakthrough discovery was the fact that $\cF_0$ is Fredholm on a completion of the same space as $\cF_\eps$ -- just after completion with a non-equivalent norm -- in fact, a Sobolev norm that is significantly weaker than the one used in the classical adiabatic method for 
$\cF^{\rm red}_0:\overline\Gamma^{\rm red}_0\to\ol\Omega^{\rm red}_0$, which works on a ``smaller'' Banach space with a stronger norm than the new Fredholm operator $\cF_0:\overline\Gamma_0\to \overline\Omega_0$. 
\end{rmk}

\begin{rmk}[Compatibility with common approaches to gluing]  \label{rmk:gluing}
Regarding goal (c), the descriptions of the Examples~\ref{ex:adiabatic} in terms of adiabatic Fredholm families utilize the classical Fredholm descriptions for fixed $\eps>0$, which are compatible with both classical and polyfold gluing methods as the pregluing construction $\oplus_R(u_a,u_b)$ (sketched in Remark~\ref{rmk:broken}) can be applied to smooth maps $u_a,u_b$ and yields smooth maps. It can similarly be applied to transfer representations of the cokernel from $R=\infty$ to finite gluing parameters $R$.
The main challenge is in ensuring that the new Fredholm description for $\eps=0$ is also compatible with the gluing analysis -- in fact, is the limit of $\eps\searrow 0$. 

As a result, we will be dealing with Fredholm problems parametrized by $\eps\in\Delta$ and gluing parameters $R$ -- with the notable exception of gluing figure eight bubbles into shrinking strips, where the strip width $\eps>0$ determines the gluing parameter $R$ that will match seams. This at least heuristically aligns with the construction of $(A_\infty,2)$-associahedra in \cite{bottman_realizations}. 
\end{rmk}

\begin{rmk}[Relationship to the classical adiabatic method] 
The analytic ingredients of the Newton-Picard iteration in the classical adiabatic method are represented in the above definition as follows: 
The ``near solutions" are given by {\rm [partial $\cC^0$-continuity w.r.t.\ $\mathbf{\Delta}$]} applied to $\gamma_0=0$. The existence of bounded right inverses is classically proven by estimates for the formal adjoint operators $\rD\cF_\eps(0)^*$ that are obtained by combining ``quadratic estimates" -- equivalent to {\rm [uniform $\mathcal{C}^1$ regularity]} with a linear function $c^1_\cF$ -- with {\rm [uniform Fredholm-ish estimates]} and the fact that the linearized operator for $\eps=0$ is an isomorphism $\ol\Gamma^{\rm red}_0 \overset{\sim}{\to} \Omega^{\rm red}_0$, which is represented by estimates for the formal adjoint operator $\ol{\rD\cF_0(0)}^*$ -- a part of the classical adiabatic analysis that wasn't well positioned for generalizations. 

Here our breakthrough insight was that the isomorphism property of $\ol{\rD\cF_0(0)}$ in the classical case 
could be understood as a special case of an isomorphism property of a Fredholm stabilization coupled with projection to the kernel, 
\begin{equation}\label{eq:stabilization}
\ol P_0 \,:\quad \ol\Gamma_0 \times\fC \;\to\; \fK\times\ol\Omega_0, \qquad (\gamma,\fc) \;\mapsto\; \bigl( \pi_\fK (\gamma) , \ol{\rD\cF_0(0)} \gamma - \fc \bigr) . 
\end{equation}
For a general adiabatic Fredholm family, this isomorphism is represented by the {\rm [$\epsilon=0$ Fredholm Estimate]}
$\| \gamma \|^\Gamma_0 + \| \fc  \|^\Omega_0  \leq C_0 \bigl( \|\pi_\fK(\gamma) \|^\fK + \| \rD\cF_0(0) \gamma - \fc \|^\Omega_0  \bigr)$, which implies injectivity of $\ol P_0$, so that surjectivity follows from $\ol P_0$ having Fredholm index $0$.  

This turned out to be an analytically more robust formulation than the classical notion of a Fredholm splitting, in which we have direct sum decompositions of the domain $\ol\Gamma_0 = \fK \oplus W$ and the target $\ol\Omega_0 = \fC \oplus \im\ol{\rD\cF_0(0)}$ so that the restricted operator $\ol{\rD\cF_0(0)}|_W$ is an isomorphism to its image. 
\end{rmk}

A general adiabatic Fredholm family just needs to be defined on a convex subset $\cV_\Gamma\subset\Gamma$ containing $0=0_\Gamma$ such that $\cV_\Gamma\subset(\Gamma,\|\cdot\|^\Gamma_0)$ is open. 
Now the main result of this paper is that the new notion of $\cC^1$-regular adiabatic Fredholm family (a) is satisfied by Examples~\ref{ex:adiabatic} as shown in \S\ref{examples}, and (b) induces local finite dimensional reductions as follows:

\begin{theorem} \label{thm:charts}
Suppose an open subset $\cU\subset\overline\cM$ of a topological space (such as a compactified moduli space) is described as the zero set of a $\cC^\ell$-regular adiabatic Fredholm family $\bigl( (\cF_\eps:\cV_\Gamma\to \Omega )_{\eps\in\Delta} , \ldots  \bigr)$ as in Definitions~\ref{def:fredholm}, \ref{def:adiabatic C-l} for $\ell\geq 0$. 
That is, there is a homeomorphism
$$
\psi \,:\; \textstyle\bigcup_{\eps\in\Delta} \{\eps\}\times \cF_\eps^{-1}(0)  \;\overset{\sim}{\to}\; \cU . 
$$
Then this induces a finite dimensional reduction that describes $\ol\cM$ locally as the zero set of a map between finite dimensional spaces, 
$$
f:\Delta_\sigma\times\cV_\fK\to\fC, \; (\eps,\fk)\mapsto f_\eps(\fk)
\qquad\text{and}\qquad 
\psi_f: f^{-1}(0) \to \ol\cM  . 
$$ 
More precisely, 
$\fK=\ker\rD\cF_0(0)\subset\Gamma$ is the kernel and $\fC \simeq\coker \rD\cF_0(0)\subset\Omega$ is the cokernel of the linearization of $\cF_{\eps=0}$ at $0\in\cV_\Gamma\subset\Gamma$. 
Then the finite dimensional reduction is defined on open subsets $\Delta_\sigma\subset\Delta$ and $\cV_\fK\subset\fK$. It describes $\overline\cM$ locally by composition $\psi_f=\psi\circ\phi$ with a homeomorphism for some $\delta_\sigma>0$
$$
\phi \,:\; f^{-1}(0) \;\overset{\sim}{\to}\;
\bigl( \textstyle\bigcup_{\eps\in\Delta_\sigma} \{\eps\}\times \overline\cF_\eps^{-1}(0)  \bigr) \cap
\bigl\{ (\eps,\gamma) \,\big|\, \|\gamma\|^\Gamma_\eps < \delta_\sigma \bigr\}
\;\subset\; \Delta \times \bigl( \Gamma , \|\cdot \|_0 \bigr) . 
$$
Moreover, this finite dimensional reduction is $\cC^\ell$ in the sense that the differentials of order $0\leq k \leq \ell$ (from $k$-fold differentiation in $\fK$; see Remark~\ref{rmk:multilinear}) form continuous maps 
$$
\Delta_\sigma\times\cV_\fK \to \cL^k(\fK^k,\fC) , \quad 
(\eps,\fk_0)\mapsto \rD^k f_\eps(\fk_0) . 
$$ 
\end{theorem}

\begin{rmk}[Prior Results]
The prior adiabatic limit proofs in Examples~\ref{ex:adiabatic} work with $\fK=\{0\}$, $\fC=\{0\}$, and hence the trivial maps $f:\Delta_\sigma=[0,\eps_\sigma) \to \{0\}$. Their Newton-Picard iteration induces maps $\phi:f^{-1}(0)=[0,\eps_\sigma) \to \bigl\{ (\eps,\gamma) \,\big|\, \cF_\eps(\gamma)=0,  \|\gamma-\gamma_0\|^\Gamma_\eps < \delta_\sigma \bigr\}$ for each $\gamma^{\rm red}_0\in\cF^{\rm red\;-1}(0)$, for which continuity and injectivity for fixed $\eps$ holds vacuously since the domain is a single point. Continuity for varying $\eps\in[0,\eps_\sigma)$ was not even a reasonable question for lack of an ambient topological space.
\end{rmk}

\begin{rmk}[Fit with common regularization theories]
Regarding goal (b) -- finite dimensional reductions that fit into all common regularization theories --  
note that the finite dimensional descriptions $f:\Delta_\sigma\times\cV_\fK\to\fC$ of moduli spaces resulting from an adiabatic Fredholm family via Theorem~\ref{thm:charts} come in the form of Remark~\ref{rmk:finite dimensional} up to restricting to an open subset $\cV_\fK\subset\fK$ (a standard modification) and introducing an extra factor $\Delta_\sigma$ -- in applications a chart of an underlying finite dimensional moduli space of domains. 
While we obtain no differentiability in the directions of $\Delta_\sigma$, the derivatives in directions of $\fK$ are continuous with respect to $\Delta_\sigma$. 
Indeed, for a $\cC^1$-regular adiabatic Fredholm family, the assertion of Theorem~\ref{thm:charts} is the continuity of both the map  
$\Delta_\sigma\times \cV_\fK \to \fC$, $(\eps,\fk)\mapsto f_\eps(\fk)$ 
and the differential in the direction of $\fK$, in the sense that 
$\Delta_\sigma\times\cV_\fK \to \cL(\fK,\fC)$, $(\eps,\fk_0)\mapsto\rD f_\eps(\fk_0)$ is a continuous map to the space of bounded linear operators $\fK\to\fC$. 

Since the work of Siebert \cite[\S 1]{Siebert} this differentiability in all but finitely many directions has been widely understood as the optimal requirement for local regularization structures. 
In fact, the relinquishing of differentiability in finitely many directions is also a core technique in proofs of the polyfold theoretic sc-Fredholm property; see \cite[Prop.4.23]{hofer-wysocki-zehnder_GW} for the Gromov-Witten case and \cite[Def.4.3]{wehrheim_Fredholm} for a general formulation. 
A regularization theory just needs local regularization structures with sufficiently large spaces of perturbations or stabilizing factors to cover the cokernels of the differentials in the differentiable directions. (Thus we do need some differentiable directions.) 
Moreover, a regularization theory needs to compatibly patch these local regularization structures -- which, however, is formulated quite differently in different theories, so remains the job of the user. 
\end{rmk}

A more general version of Theorem~\ref{thm:charts} is stated in Corollary~\ref{cor:family finite dimensional reduction}. 
The proof takes up \S\ref{proofs} and proceeds as follows: \S\ref{inverses} lifts the stabilization isomorphism \eqref{eq:stabilization} to $\eps\ne 0$. \S\ref{contraction} utilizes these isomorphisms to rewrite the equations 
\begin{equation}\label{eq:equivalent}
\cF_\eps(\gamma)=0 
\qquad\Longleftrightarrow\qquad
\bigl( A_\eps(\fk,w) , w - B_\eps(\fk,w) \bigr) = (0,0),
\end{equation} 
where each $B_\eps(\fk,\cdot)$ is a contraction mapping -- analogous to the notion of an sc-basic germ in \cite[Def.3.1.7]{hofer-wysocki-zehnder_polyfold}.  \S\ref{solution} constructs solution maps $\sigma_\eps:\cV_\fK \to W=\fC\times \Gamma$ that solve the fixed point equations $w=B_\eps(\fk,w)$ and shows how various regularity assumptions on $\cF_\eps$ transfer to the contractions $B_\eps$ and then the solution maps $\sigma_\eps$. This is where the technical heart of our work is in the interplay of Definition~\ref{def:adiabatic C-l} with Theorem~\ref{thm:family solution} and Example~\ref{ex:adiabatic}: While it is relatively easy to formulate a definition of ``adiabatic regularity" that implies regularity of the solution maps, or to capture the regularity properties of the examples in a notion of ``adiabatic regularity", the crucial contribution of this paper is in finding a notion of adiabatic $\cC^\ell$ regularity that (a) is satisfied in Examples~\ref{ex:adiabatic} and (b) implies via Theorem~\ref{thm:family solution} the adiabatic $\cC^\ell$ regularity of the solution maps. 
Here, in addition, the notion of adiabatic $\cC^0$ regularity was designed to ensure global continuity of the solution map $(\eps,\fk) \to \fC \times \Gamma$ with respect to the weakest norm $\|\cdot \|^\Gamma_0$. 
These results are analogous to the properties of the solution germs in \cite[Thm.3.3.3]{hofer-wysocki-zehnder_polyfold}, thus make the case that adiabatic Fredholm families fit into the polyfold regularization theory. 

To fit adiabatic Fredholm families into regularization theories via finite dimensional charts or obstruction bundles we deduce in \S\ref{reduction} that \eqref{eq:equivalent} is equivalent to a  finite dimensional equation
$$
\cF_\eps(\gamma)=0 
\qquad\Longleftrightarrow\qquad
A_\eps(\fk,\sigma_\eps(\fk)) = 0 .
$$
Here the finite dimensional reduction map $f: \cV_\fK \to \fC, \fk\mapsto A_\eps(\fk,\sigma_\eps(\fk))$ is defined on an open subset $\cV_\fK\subset \fK$ of the kernel $\fK=\ker  \overline{\rD\cF_0(0)}$, maps to the cokernel $\fC=\coker  \overline{\rD\cF_0(0)}$, and inherits the regularity of the adiabatic Fredholm family. 

To make this analytic work both accessible and educational, we formulate a classical analogue of 
Theorem~\ref{thm:charts} in Theorem~\ref{thm:charts classical} -- the fact that every $\cC^1$ Fredholm map has a $\cC^1$ finite dimensional reduction -- and begin each section with a review of our version of the construction step in classical Fredholm theory before going into the new technical work of making the construction uniformly and compatibly for all $\eps\in\Delta$.

%
%

\bigskip

{\it We are most grateful to MPIM Bonn for physically hosting -- and thus making possible -- the breakthrough phase of our work. 
The second author also wishes to acknowledge the annoying accuracy of Dietmar Salamon's oracle that ``it's just estimates". }

\section{Adiabatic Fredholm Families}

 This section develops the notions of an adiabatic Fredholm family and its possible regularity properites. 
We begin with the least restrictive notion -- which will suffice for the existence of invertible Fredholm stabilizations \S\ref{inverses}, yielding contractions \S\ref{contraction}, solution maps \S\ref{solution}, and finite dimensional reductions \S\ref{reduction}, but guaranteeing only continuity and differentiability for fixed $\eps\in\Delta$ -- no regularity for varying $\eps$. The latter will require further properties of the adiabatic Fredholm family that are developed in \S\ref{regularity}. 

\vfill
\pagebreak

\begin{definition} \label{def:fredholm}
An {\bf adiabatic Fredholm family} consists of the following data. 
\begin{itemize}
\item
$\Gamma$ and $\Omega$ are real vector spaces.\footnote{These are usually infinite dimensional and we do not specify a norm or topology at this point. In our applications, these are spaces of smooth functions, which we will later  complete w.r.t.\ $\eps$-dependent Sobolev norms.}
\item
$\cV_\Gamma\subset\Gamma$ is a convex subset containing $0=0_\Gamma\in\Gamma$.\footnote{In our applications, this is a small open
 ball in a $\cC^0$-norm that is weaker than any of the $\eps$-dependent norms.} 
\item
$\Delta$ is a topological space with a distinguished point $0=0_\Delta\in\Delta$.\footnote{This is usually a local chart for a finite dimensional topological manifold. In our applications, it will be $\bR^m\times [0,1)^n$ or a generalized corner of an underlying moduli space of domains.} 
\item
$\cF_\eps : \cV_\Gamma \to \Omega$ is a family of maps indexed by $\eps\in\Delta$. 
\item
$\|\cdot\|^\Gamma_\eps$ and $\|\cdot\|^\Omega_\eps$ are families of norms on $\Gamma$ and $\Omega$, respectively, indexed by $\eps\in\Delta$.
\item $\fK\subset\Gamma$ and $\fC\subset\Omega$ are finite dimensional subspaces
equipped with the norms $\|\fk\|^\fK:=\|\fk\|^\Gamma_0$ for $\fk\in\fK$ and $\|\fc\|^\fC:=\|\fc\|^\Omega_0$ for $\fc\in\fC$.
\item
$\pi_\fK:\Gamma\to\fK$ is a linear projection.
\item
$C_0, C_1, C_\fC \in (0,\infty)$ are constants.
\item $c : [0,\infty) \to [0,\infty)$ is a monotone continuous function with $c(0)= 0$, and \\
$c_\Delta : \Delta \to [0,\infty)$ is a continuous function with $c_\Delta(0)=0$.
\end{itemize}
These are required to have the following properties.
\begin{itemlist}
\item[\bf{[Openness of Domain]}]  
$\cV_\Gamma\subset(\Gamma,\|\cdot\|^\Gamma_0)$ is open.
\item[\bf{[Lower Bound on Norms]}]  
$\|\gamma \|^\Gamma_0 \leq \|\gamma \|^\Gamma_\eps$ and 
$\|\omega \|^\Omega_0 \leq \|\omega \|^\Omega_\eps$
for all $\gamma\in\Gamma$, $\omega\in\Omega$, and $\eps\in\Delta$.
\item[\bf{[Fibrewise $\mathcal{C}^1$ Regularity]}]
$\cF_\eps : (\cV_\Gamma,\|\cdot\|^\Gamma_\eps) \to (\Omega,\|\cdot\|^\Omega_\eps)$ is uniformly $\cC^1$ for each $\eps\in\Delta$.
\item[\bf{[Fredholm]}]
The linearization $\rD\cF_\eps (0) : \Gamma \to \Omega$ at $0=0_\Gamma$ for each fixed $\eps\in\Delta$ extends to a Fredholm operator $\overline{\rD\cF_\eps(0)}: \overline\Gamma_\eps\to \overline\Omega_\eps$ between the completions $\overline\Gamma_\eps:=\overline\Gamma^{\|\cdot\|^\Gamma_\eps}$ and $\overline\Omega_\eps:=\overline\Omega^{\|\cdot\|^\Omega_\eps}$.
\item[\bf{[Index]}]
The Fredholm index of the linearizations $\overline{\rD\cF_\eps(0)}$ is independent of $\eps\in\Delta$. 
Moreover, the kernel $\ker \overline{\rD\cF_0(0)}\subset \overline\Gamma_0$ is contained in the dense subset $\Gamma$ and equals to the given subspace $\fK=\ker \overline{\rD\cF_0(0)}\subset\Gamma$.
The cokernel $\coker\overline{\rD\cF_0(0)}=\overline\Omega_0 / \Im \overline{\rD\cF_0(0)}$ is isomorphic to the given subspace $\coker\overline{\rD\cF_0(0)} \simeq \fC\subset\Omega$. 
As a result we have
$\ind \overline{\rD\cF_\eps(0)}=\dim\fC - \dim\fK$ for each $\eps\in\Delta$.
\item[\bf{[$\mathbf{\epsilon=0}$ Fredholm Estimate]}] 
The projection $\pi_\fK:\Gamma\to\fK$ and inclusion $\fC\subset\Omega$ provide a Fredholm stabilization of $\overline{\rD\cF_0(0)}$ in the sense that
$$
\| \gamma \|^\Gamma_0 + \| \fc  \|^\fC  \leq C_0 \bigl( \|\pi_\fK(\gamma) \|^\fK + \| \rD\cF_0(0) \gamma - \fc \|^\Omega_0  \bigr)\qquad\text{for all}\; (\gamma,\fc)\in\Gamma\times \fC .
$$ 
\item[\bf{[Uniform Fredholm-ish Estimate]}] 
$\| \gamma \|^\Gamma_\eps \leq C_1 \bigl( \| \rD\cF_\eps(0) \gamma \|^\Omega_\eps + \| \gamma \|^\Gamma_0 \bigr)$ for all $\gamma\in\Gamma$ and $\eps\in\Delta$. 
\item[\bf{[Uniform Cokernel Bound]}] 
$ \| \fc  \|^\Omega_\eps  \leq C_\fC  \| \fc  \|^\Omega_0 =C_\fC  \| \fc  \|^\fC$ for all $\fc\in\fC$ and $\eps\in\Delta$.
\item[\bf{[Quadratic-ish Estimate]}] 
$
\bigl\|  \rD\cF_\eps (\gamma_0)\gamma -   \rD\cF_\eps (0)\gamma  \bigr\|^\Omega_\eps \leq c(\|\gamma_0\|^\Gamma_\eps)  \| \gamma \|^\Gamma_\eps$ for all $\gamma_0\in\cV_\Gamma$, $\gamma\in\Gamma$, and $\eps\in\Delta$.
\item[\bf{[Continuity of Derivatives at $0$]}] 
$
\bigl\|  \rD\cF_\eps (0)\gamma -   \rD\cF_0 (0)\gamma  \bigr\|^\Omega_0 \leq c_\Delta(\eps)  \| \gamma \|^\Gamma_\eps$ for all $\gamma\in\Gamma$ and $\eps\in\Delta$.
\item[\bf{[Near-Solution]}] $\cF_0(0)=0$, that is $\gamma=0$ solves the equation $\cF_\eps(\gamma)=0\in\Omega$ for $\eps=0$.
Moreover, the map $\Delta\to \Omega$, $\eps\mapsto \cF_\eps(0)$ is continuous at $\eps=0$ in the sense that $\| \cF_\eps(0) \|^\Omega_\eps = \| \cF_\eps(0) - \cF_0(0) \|^\Omega_\eps \to 0$ as $\eps\to 0$. 
\end{itemlist}
\end{definition}

An application-minded reader may now want to skip to \S\ref{examples} to get an idea of what such an adiabatic Fredholm family looks like in practice.

\subsection{Extension to Completions and Tangent Map Notation}

Users content with Definition~\ref{def:fredholm} can skip most this section, as the extension to Banach space completions is mostly a technical step in the proof. 
Users wishing to avoid noncomplete normed vector spaces can work directly with the following notion of an extended adiabatic Fredholm family -- a family of Fredholm maps between Banach spaces that have common dense subspaces. 
Users wishing to make sense of adiabatic $\cC^\ell$ regularity for $\ell\geq 2$ will need the tangent map notation introduced in the later part of this section.

\begin{lemma} \label{lem:fredholm}
An adiabatic Fredholm family as in Definition~\ref{def:fredholm} induces the following data of an {\bf extended adiabatic Fredholm family}. 
\begin{itemize}
\item
$\overline\Gamma_\eps:=\overline\Gamma^{\|\cdot\|^\Gamma_\eps}$ and $\overline\Omega_\eps:=\overline\Omega^{\|\cdot\|^\Omega_\eps}$ are Banach spaces obtained by completing the fixed spaces $\Gamma$ and $\Omega$ w.r.t.\ norms that vary with $\eps\in\Delta$. 
\item
$\overline\cV_{\overline\Gamma,\eps}:=\overline{\cV_\Gamma}\subset\overline\Gamma_\eps$ 
and
$\cV_{\overline\Gamma,\eps}:={\rm int}(\overline{\cV_\Gamma})\subset\overline\Gamma_\eps$ are convex subsets containing $0=0_\Gamma\in\Gamma\subset\overline\Gamma_\eps$. These are obtained by 
taking the closure of a fixed convex subset $0\in\cV_\Gamma\subset\Gamma\subset\overline\Gamma_\eps$, and by thaking the interior of this closure.
\item
$\overline\cF_\eps : \cV_{\overline\Gamma,\eps} \to \overline\Omega_\eps$ are maps obtained by continuous extension of the maps $\cF_\eps:\Gamma\to\Omega$.
\item
$\|\cdot\|^\Gamma_\eps$ and $\|\cdot\|^\Omega_\eps$ are norms on $\overline\Gamma_\eps$ and $\overline\Omega_\eps$, obtained by continuous extension of the norms specified on $\Gamma$ and $\Omega$.
\item 
$\fK\subset\overline\Gamma_\eps$ and $\fC\subset\overline\Omega_\eps$ are inclusions of the fixed finite dimensional subspaces $\fK\subset\Gamma$ and $\fC\subset\Omega$,  
equipped with the norms $\|\fk\|^\fK=\|\fk\|^\Gamma_0$ for $\fk\in\fK$ and $\|\fc\|^\fC=\|\fc\|^\Omega_0$ for $\fc\in\fC$.
\item
$\overline\pi_\fK:\overline\Gamma_\eps\to\fK$ is a linear projection obtained by continuous extension of $\pi_\fK:(\Gamma,\|\cdot\|^\Gamma_\eps)\to(\fK,\|\cdot\|^\fK)$. 
\item
$C_0, C_1, C_\fC \in (0,\infty)$ are the same constants as in Definition~\ref{def:fredholm}.
\item 
$c : [0,\infty) \to [0,\infty)$ is a monotone continuous function with $c(0)= 0$, and \\
$c_\Delta : \Delta \to [0,\infty)$ is a continuous function with $c_\Delta(0)=0$ -- the same as in Definition~\ref{def:fredholm}.
\end{itemize}
These inherit the following properties.
\begin{itemlist}
\item[\bf{[Lower Bound on Norms]}]  
$\|\gamma \|^\Gamma_0 \leq \|\gamma \|^\Gamma_\eps$ and 
$\|\omega \|^\Omega_0 \leq \|\omega \|^\Omega_\eps$
for all $\gamma\in\overline\Gamma_\eps$, $\omega\in\overline\Omega_\eps$, and $\eps\in\Delta$.
\item[\bf{[Openness of Domain]}]  
$\cV_{\overline\Gamma,\eps}\subset\overline\Gamma_\eps$ is open for every $\eps\in\Delta$.
\item[\bf{[Fibrewise $\mathcal{C}^1$ Regularity]}]
$\overline\cF_\eps: \cV_{\overline\Gamma,\eps} \to \overline\Omega_\eps$ is uniformly $\cC^1$ for each $\eps\in\Delta$.
\item[\bf{[Fredholm]}]
$\overline\cF_\eps : \cV_{\overline\Gamma,\eps} \to \overline\Omega_\eps$ linearizes at $0=0_\Gamma\in\overline\Gamma_\eps$ to the Fredholm operator $\rD\overline\cF_\eps(0)=\overline{\rD\cF_\eps(0)} : \overline\Gamma_\eps\to \overline\Omega_\eps$ for each $\eps\in\Delta$.
\item[\bf{[Index]}]
The Fredholm index of these linearizations is $\ind \rD\overline\cF_\eps(0)=\dim\fC - \dim\fK$ for each $\eps\in\Delta$, where $\fK=\ker\rD\overline\cF_0(0)$ and $\fC\simeq\coker\rD\overline\cF_0(0)$. 
\item[\bf{[$\mathbf{\epsilon=0}$ Fredholm Estimate]}] 
The projection $\overline\pi_\fK:\overline\Gamma_\eps\to\fK$ and inclusion $\fC\subset\overline\Omega_\eps$ provide a Fredholm stabilization of $\rD\overline\cF_0(0)$ in the sense that
$$
\| \gamma \|^\Gamma_0 + \| \fc  \|^\fC  \leq C_0 \bigl( \|\pi_\fK(\gamma) \|^\fK + \| \rD\overline\cF_0(0) \gamma - \fc \|^\Omega_0  \bigr)\qquad\text{for all}\; (\gamma,\fc)\in\overline\Gamma_\eps\times \fC .
$$ 
\item[\bf{[Uniform Fredholm-ish Estimate]}] 
$\| \gamma \|^\Gamma_\eps \leq C_1 \bigl( \| \rD\overline\cF_\eps(0) \gamma \|^\Omega_\eps + \| \gamma \|^\Gamma_0 \bigr)$ for all $\gamma\in\overline\Gamma_\eps$ and $\eps\in\Delta$. 
\item[\bf{[Uniform Cokernel Bound]}] 
$ \| \fc  \|^\Omega_\eps  \leq C_\fC  \| \fc  \|^\Omega_0=C_\fC  \| \fc  \|^\fC$ for all $\fc\in\fC$ and $\eps\in\Delta$.
\item[\bf{[Quadratic-ish Estimate]}] 
$
\bigl\|  \rD\overline\cF_\eps (\gamma_0)\gamma -   \rD\overline\cF_\eps (0)\gamma  \bigr\|^\Omega_\eps \leq c(\|\gamma_0\|^\Gamma_\eps)  \| \gamma \|^\Gamma_\eps$ for all $\gamma_0\in\cV_{\overline\Gamma,\eps}$, $\gamma\in\overline\Gamma_\eps$, and $\eps\in\Delta$.
\item[\bf{[Continuity of Derivatives at $0$]}] 
$
\bigl\|  \rD\overline\cF_\eps (0)\gamma -   \rD\overline\cF_0 (0)\gamma  \bigr\|^\Omega_0 \leq c_\Delta(\eps)  \| \gamma \|^\Gamma_\eps$ for all $\gamma\in\overline\Gamma_\eps$ and $\eps\in\Delta$.
\item[\bf{[Near-Solution]}] $\overline\cF_0(0)=0$, that is $\gamma=0$ solves the equation $\overline\cF_\eps(\gamma)=0\in\Omega$ for $\eps=0$.
Moreover, the map $\Delta\to\overline\Omega_\eps$, $\eps\mapsto\overline\cF_\eps(0)$ is continuous at $\eps=0$ in the sense that $\| \overline\cF_\eps(0) \|^\Omega_\eps = \| \overline\cF_\eps(0) - \overline\cF_0(0) \|^\Omega_\eps \to 0$ as $\eps\to 0$. 
\end{itemlist}
\end{lemma}

\vfill
\pagebreak

\begin{proof}
%
The crucial fact for this Lemma is the fact that the a uniformly $\cC^\ell$ map between normed vector spaces extends continuously to a $\cC^\ell$ map between the completed vector spaces -- as shown in Lemma~\ref{lem:extension}.  
In particular, the $\ell$-th differential $\rD^\ell \overline\cF_\eps=\overline{\rD^\ell\cF_\eps}: \cV_{\ol\Gamma,\eps} \to \cL^\ell(\overline\Gamma_\eps^\ell;\overline\Omega_\eps)$ is given by the continuous extension of 
$\cV_\Gamma \to \cL^\ell(\overline\Gamma_\eps^\ell;\overline\Omega_\eps), \gamma \mapsto \overline{\rD^\ell\cF_\eps(\gamma)}$, where each $\overline{\rD^\ell\cF_\eps(\gamma)}: \overline\Gamma_\eps^\ell\to\overline\Omega_\eps$ is the continuous extension of $\rD\cF_\eps^\ell(\gamma): \Gamma_\eps^\ell\to\Omega_\eps$. 
Here $\cV_{\overline\Gamma, \eps}:={\rm int}(\overline{\cV_\Gamma})$ is the interior of the closure $\overline{\cV_\Gamma}\subset\overline\Gamma_\eps$ of $\cV_\Gamma\subset\overline\Gamma_\eps$. 

The rest of this proof is a useful exercise in getting familiar with the conditions for adiabatic Fredholm families -- and how they each interact with the norms, completions, and continuous extensions. 
Note in particular that the linear projection $\pi_\fK:(\Gamma,\|\cdot\|^\Gamma_\eps)\to(\fK,\|\cdot\|^\fK)$ is automatically bounded since $\fK$ is finite dimensional.
\end{proof}

Here and throughout we are using the notions of Fr\'echet differential and Fr\'echet differentiability classes -- reviewed in the following remark. 

\begin{rmk} \label{rmk:multilinear} \rm
Throughout, we consider normed vector spaces $(V,\|\cdot\|^V)$ and $(W,\|\cdot\|^W)$ and a continuous map $f:\cV_V\to W$ defined on an open subset $\cV_V\subset V$. 

\begin{itemlist}
\item
Such a map is called {\bf differentiable at $\mathbf{v_0\in \mathcal{V}_V}$} if there exists a bounded linear map $\rD f(v_0):V\to W$ such that $\| f(v_0+v) - f(v_0) - \rD f(v_0) v \|^W / \|v\|^V \to 0$ as $ \|v\|^V\to 0$, or equivalently 
$$
\forall\varepsilon>0 \; \exists\delta>0 : \quad \bigl( v\in V , \|v\|^V \leq \delta  \quad \Rightarrow \quad \| f(v_0+v) - f(v_0) - \rD f(v_0) v \|^W \leq \varepsilon \|v\|^V \bigr) .
$$ 
If $f:\cV_V\to W$ is differentiable at $v_0\in \cV_V$, then this condition determines the linear map $\rD f(v_0):V\to W$ uniquely, and it is called the {\bf differential of $\mathbf f$ at $\mathbf{v_0}$}. 

\item
Further, $f:\cV_V\to W$ is called {\bf of class $\mathbf{\mathcal\cC}^1$} if it is differentiable at all $v_0\in \cV_V$ and the differential $\rD f: \cV_V \to \cL(V,W), v_0 \mapsto \rD f(v_0)$ is a map of class $\mathcal\cC^0$, that is, a continuous map to the space of bounded linear operators. 

\item
Now we can then iteratively define the notion of a {\bf map of class $\mathbf{\mathcal\cC}^\ell$} for all integers $\ell\geq 2$, by calling $f:\cV_V\to W$ {\bf of class $\mathbf{\mathcal\cC}^\ell$} if $\rD f: \cV_V \to \cL(V,W)$ is of class $\mathcal\cC^{\ell-1}$. Equivalently, $f$ is of class $\mathcal\cC^{\ell-1}$ and the $(\ell-1)$-fold differential $\rD\ldots\rD f: \cV_V \to \cL(V,\ldots,V,\cL(V,W)\ldots)$ is of class $\mathcal\cC^{1}$.

\item
This definition views $\cL(V,W)$ as a vector space with norm $\| \Phi \|^{\cL(V,W)} = \sup_{\|v\|^V\leq 1 } \| \Phi(v) \|^W$, and results in higher derivatives $\rD \rD f (v_0) \in \cL(V,\cL(V,W))$, $\rD \rD \rD f (v_0) \in \cL(V,\cL(V,\cL(V,W)))$, etc.. These spaces are ``not easy to handle'', so as in \cite[\S 4.3--4]{coleman} we will add to the notion of class $\mathcal\cC^\ell$ an identification of the $\ell$-fold differential $\rD \ldots \rD f (v_0) \in \cL(V, \ldots ,\cL(V,W) \ldots)$ with the {\bf $\ell$-th differential} $\rD^\ell f(v_0)\in \cL^\ell(V^\ell;W)$, viewed as an $\ell$-linear map $V^\ell\to W$. Then we can say that $f:\cV_V\to W$ is {\bf of class $\mathbf{\mathcal\cC}^\ell$} if it is of class $\mathcal\cC^{\ell-1}$ and $\rD^{\ell-1} f: \cV_V \to \cL^{\ell-1}(V^{\ell-1};W)$ is of class $\mathcal\cC^1$.

\item
Here $\cL^\ell(V^\ell;W)$ denotes the space of maps $\Phi:V^\ell\to W$ that are $\ell$-linear, i.e.\ linear in each of the $\ell$ entries, $\Phi(r_1 v_1, \ldots, r_\ell v_\ell) = r_1  \ldots r_\ell  \Phi(v_1, \ldots, v_\ell)$ for all $v_1,\ldots,v_\ell\in V$ and $r_1,\ldots,r_\ell\in\bR$, and continuous, i.e.\ have a finite norm
\begin{equation}\label{eq:multilinear operator norm}
\| \Phi \|^{\cL^\ell(V^\ell;W)} = \sup_{\|v_1\|^V,\ldots, \|v_\ell\|^V\leq 1 } \| \Phi(v_1,\ldots,v_\ell) \|^W  .
\end{equation}

\vfill
\pagebreak

\item
Note that $\cL^1(V;W)=\cL(V,W)$, while for $\ell\geq 2$ the space $\cL^\ell(V^\ell;W)$ of maps that are linear in each factor is quite different from the space $\cL(V^\ell,W)$ of linear maps. 
Rather, one identifies $\cL(V, \ldots ,\cL(V,W) \ldots)$ with $\ell$ factors of $V$ with $\cL^{\ell}(V^\ell;W)$ by iteratively identifying $\Psi\in \cL(V,\cL^{\ell-1}(V^{\ell-1}=V\times\ldots\times V,W))$ with $\Psi'\in\cL^{\ell}(V^\ell;W)$ given by $\Psi'(v_1,\ldots,v_{\ell-1},v_\ell)=\Psi(v_\ell)(v_1,\ldots,v_{\ell-1})$. Here the last equation also serves to define an inverse map $\cL^{\ell}(V^\ell;W)\to \cL(V,\cL^{\ell-1}(V^{\ell-1};W))$, and this bijection is an isometric isomorphism since
\begin{align*}
\| \Psi' \|^{\cL^\ell(V^\ell;W)} 
&= \sup_{\|v_1\|^V,\ldots, \|v_\ell\|^V\leq 1 } \| \Psi'(v_1,\ldots,v_\ell) \|^W \\
&= \sup_{\|v_\ell\|^V\leq 1 }   \sup_{\|v_1\|^V,\ldots, \|v_{\ell-1}\|^V\leq 1 }  \| \Psi(v_\ell) (v_1,\ldots,v_{\ell-1}) \|^W \\
&= \sup_{\|v_\ell\|^V\leq 1 }   \| \Psi(v_\ell) \|^{\cL^{\ell-1}(V^{\ell-1};W)} 
\;=\; \| \Psi \|^{\cL(V,\cL^{\ell-1}(V^{\ell-1};W))}.
\end{align*}

\item
In summary, a continuous map $f:\cV_V\to W$ is $\cC^k$ if and only if for each $1\leq\ell\leq k$ the $\ell$-th differential exists 
$$
\rD^\ell f : \cV_V \;\to\; \cL^\ell(V^\ell;W), \quad v_0 \mapsto \rD^\ell f (v_0) 
$$
and is continuous w.r.t.\ the above norm on $\cL^\ell(V^\ell;W)$, that is
$$
 \sup_{\|v_1\|^V,\ldots, \|v_\ell\|^V\leq 1 } \| \rD^\ell f(v) (v_1,\ldots,v_\ell) - \rD^\ell f(v_0) (v_1,\ldots,v_\ell)  \|^W  
 \;\underset{v\to v_0}\longrightarrow\; 0 . 
$$ 

\item
Finally, note that -- assuming $\rD^\ell f$ exists and is differentiable -- we can compute $\rD^{\ell+1} f$ from directional derivatives
$$
 \rD^{\ell+1} f(v_0) (v_1,\ldots,v_\ell,v_{\ell+1}) = \tfrac\rd{\rd t}|_{t=0}  \rD^\ell f(v_0+t v_{\ell+1}) (v_1,\ldots,v_\ell) . 
$$
\end{itemlist}
\end{rmk}

Thus we are prepared for the technical proof of this section -- the surprisingly subtle fact that continuous extension to Banach completions preserves Fr\'echet differentiability classes.

\begin{lemma} \label{lem:extension}
Let $(\Gamma,\|\cdot\|^\Gamma)$ and $(\Omega,\|\cdot\|^\Omega)$ be normed vector spaces, and suppose that $\cF : \cV_\Gamma \to \Omega$ is a uniformly continuous map from an open subset $\cV_\Gamma\subset\Gamma$. 
Then there is a unique continuous extension $\overline\cF : \overline{\cV_\Gamma} \to \overline\Omega$ of $\overline\cF|_{\cV_\Gamma}=\cF$, where $\overline\Gamma:=\overline\Gamma^{\|\cdot\|^\Gamma}$ and $\overline\Omega:=\overline\Omega^{\|\cdot\|^\Omega}$ are the completions and $\overline{\cV_\Gamma}\subset\overline\Gamma$ is the closure of $\cV_\Gamma\subset\overline\Gamma$. In fact, $\overline\cF$ is uniformly continuous. 

If, moreover, $\cF : \cV_\Gamma \to \Omega$ is uniformly $\cC^\ell$ for some $\ell\geq 1$, then its continuous extension $\overline\cF|_{\cV_{\overline\Gamma}}: \cV_{\overline\Gamma} \to \overline\Omega$ is uniformly $\cC^\ell$ as well, where $\cV_{\overline\Gamma}:={\rm int}(\overline{\cV_\Gamma})$ is the interior of $\overline{\cV_\Gamma}\subset\overline\Gamma$. In particular, $\rD^\ell \overline\cF=\overline{\rD^\ell\cF}: \cV_{\ol\Gamma} \to \cL^\ell(\overline\Gamma^\ell;\overline\Omega)$ is given by the continuous extension of 
$\cV_\Gamma \to \cL^\ell(\overline\Gamma^\ell;\overline\Omega), \gamma \mapsto \overline{\rD^\ell\cF(\gamma)}$, where each $ \overline{\rD^\ell\cF(\gamma)}: \overline\Gamma^\ell\to\overline\Omega$ is the continuous extension of $ \rD\cF^\ell(\gamma): \Gamma^\ell\to\Omega$. 
\end{lemma}

\begin{proof}
Recall that the completion $\overline\Gamma=\{ [(\gamma_i)_{i\in\bN}] \,|\, (\gamma_i)_{i\in\bN}\subset\Gamma \;\text{Cauchy sequence}\}$ is constructed from Cauchy sequences by identifying sequences $(\gamma_i)\sim (\gamma'_i)$ if $\|\gamma_i-\gamma'_i\|^\Gamma\to 0$. 
Uniform continuity guarantees that $\cF$ maps Cauchy sequences to Cauchy sequences and preserves the equivalence relation as well. So $\overline\cF:[(\gamma_i)]\mapsto [(\cF(\gamma_i))]$ is well defined for any Cauchy sequence $(\gamma_i)_{i\in\bN}\subset\cV_\Gamma$. 
This map satisfies $\overline\cF|_{\cV_\Gamma}=\cF$ since $\cV_\Gamma\subset\overline\Gamma$ is contained via $\gamma \mapsto (\gamma)_{i\in\bN}$. The extension to the closure $\overline{\cV_\Gamma}$ is then uniquely determined, since for any convergent sequence $\cV_\Gamma\ni \gamma_i \to \gamma_\infty\in\overline{\cV_\Gamma}$ continuity requires $\overline\cF(\gamma_\infty)=\lim_{i\to\infty}\cF(\gamma_i)$. This agrees with the construction of $\overline\cF$ since $[(\cF(\gamma_i))]\in\overline\Omega$ is the limit of $(\cF(\gamma_i))_{i\in\bN}$ w.r.t.\ the norm on the completion, 
$$
\| (\omega_i)_{i\in\bN} \|^{\overline\Omega} = \lim_{i\to\infty} \| \omega_i \|^\Omega . 
$$
Indeed, $\| \cF(\gamma_j) - [(\cF(\gamma_i))_{i\in\bN}] \|^{\overline\Omega} = \lim_{i\to\infty} \| \cF(\gamma_j) - \cF(\gamma_i) \|^\Omega \to 0$ as $j\to\infty$ follows from the Cauchy property of $(\cF(\gamma_i))$. 
In fact, the uniform continuity of $\cF$, stated as
$$
\forall\varepsilon>0 \; \exists\delta>0 : \quad \bigl(  \gamma',\gamma\in\Gamma,  \|\gamma'-\gamma\|^\Gamma \leq\delta \;\Rightarrow\;  \| \cF(\gamma') - \cF(\gamma)  \|^\Omega \leq \eps  \bigr) , 
$$
directly transfers to $\overline\cF$ with the same constants: 
\begin{align*}
[(\gamma'_i)], [(\gamma_i)] \in\overline\Gamma,  \| [(\gamma'_i)] - [(\gamma_i)] \|^{\overline\Gamma} \leq\delta 
&\;\Rightarrow\; 
\exists I\in\bN: \forall i\geq I \;  \| \gamma'_i - \gamma_i \|^\Gamma \leq\delta \\
&\;\Rightarrow\; 
\exists I\in\bN: \forall i\geq I \;  \| \cF(\gamma'_i) - \cF(\gamma_i) \|^\Omega \leq \varepsilon \\
&\;\Rightarrow\; 
\| [(\cF(\gamma'_i))] - [(\cF(\gamma_i))] \|^{\overline\Omega} = \lim_{i\to\infty} \| \cF(\gamma'_i) - \cF(\gamma_i) \|^\Omega \leq \varepsilon .
\end{align*}
Next, we assume that $\cF$ is uniformly $\cC^1$, that is it is uniformly continuous, differentiable, and $\rD\cF:\cV_\Gamma\to\cL(\Gamma,\Omega)$ is uniformly continuous. To prove the Lemma for $\ell=1$ we then need to prove that $\overline\cF|_{\cV_{\overline\Gamma}}$ is differentiable with uniformly continuous differential $\rD\overline\cF:\cV_{\overline\Gamma}\to\cL(\overline\Gamma,\overline\Omega)$. We begin by constructing a candidate for this uniformly continuous family of linear operators. 

Since each differential $\rD\cF(\gamma):\Gamma\to\Omega$ at $\gamma\in\Gamma$ is linear and continuous (hence uniformly continuous), it extends to a bounded linear operator between the completions $\overline{\rD\cF(\gamma)}:\overline\Gamma\to\overline\Omega$ with the same norm $\|\overline{\rD\cF(\gamma)}\|^{\cL(\overline\Gamma,\overline\Omega)}=\| \rD\cF(\gamma)\|^{\cL(\Gamma,\Omega)}$. 
Then uniform continuity of the differential $\rD\cF:\cV_\Gamma\to\cL(\Gamma,\Omega)$ can be phrased as
$$
\forall\varepsilon>0 \; \exists\delta>0 : \quad \bigl(  \gamma',\gamma\in\cV_\Gamma,  \|\gamma'-\gamma\|^\Gamma \leq\delta \;\Rightarrow\; \forall\eta\in\Gamma :  \|\rD\cF(\gamma')\eta - \rD\cF(\gamma)\eta \|^\Omega \leq \varepsilon \|\eta\|^\Gamma \bigr). 
$$
By going to the completion in the conclusion this implies that $\gamma\mapsto \overline{\rD\cF(\gamma)}$ is a uniformly continuous map $\cV_\Gamma \to \cL(\overline\Gamma,\overline\Omega)$, 
$$
\forall\varepsilon>0 \; \exists\delta>0 : \quad \bigl(  \gamma',\gamma\in\cV_\Gamma,  \|\gamma'-\gamma\|^\Gamma \leq\delta \;\Rightarrow\; \forall\eta\in\overline\Gamma :  \| \overline{\rD\cF(\gamma')}\eta - \overline{\rD\cF(\gamma)}\eta \|^{\overline\Omega} \leq \varepsilon \|\eta\|^{\overline\Gamma} \bigr). 
$$
Then this uniformly continuous map extends to a map on the completion $\overline{\rD\cF}: \cV_{\overline\Gamma}\to \cL(\overline\Gamma,\overline\Omega), \gamma_\infty=\lim_{i\to\infty}\gamma_i \mapsto \overline{\rD\cF}(\gamma_\infty):=\lim_{i\to\infty} \overline{\rD\cF(\gamma_i)}$ with again the same uniform continuity property
$$
\forall\varepsilon>0 \; \exists\delta>0 : \quad \bigl(  \gamma',\gamma\in\cV_{\overline\Gamma},  \|\gamma'-\gamma\|^{\overline\Gamma} \leq\delta \;\Rightarrow\; \forall\eta\in\overline\Gamma :  \| \overline{\rD\cF}(\gamma')\eta - \overline{\rD\cF}(\gamma)\eta \|^{\overline\Omega} \leq \varepsilon \|\eta\|^{\overline\Gamma} \bigr). 
$$
Moreover, the operator norm of $\overline{\rD\cF}$ at points $\gamma_\infty\in\cV_{\overline\Gamma}\smallsetminus\cV_\Gamma$ is the limit of operator norms of $\rD\cF$, 
$$
\|\overline{\rD\cF}(\gamma_\infty)\|^{\cL(\overline\Gamma,\overline\Omega)}=\lim_{i\to\infty}\|\overline{\rD\cF(\gamma_i)}\|^{\cL(\overline\Gamma,\overline\Omega)}=\lim_{i\to\infty}\| \rD\cF(\gamma_i)\|^{\cL(\Gamma,\Omega)}.
$$ 
It remains to show that this uniformly continuous family of bounded linear operators is in fact the differential of the completed nonlinear map $\overline\cF$. For that purpose we estimate for any $\gamma_\infty=\lim_{i\to\infty}\gamma_i \in\cV_{\overline\Gamma}$ and $\eta_\infty=\lim_{j\to\infty}\eta_j \in\cV_{\overline\Gamma}$ 
\begin{align*}
\bigl\| \overline\cF(\gamma_\infty+\eta_\infty) - \overline\cF(\gamma_\infty) - \overline{\rD\cF}(\gamma_\infty) \eta_\infty \bigr\|^{\overline\Omega} 
&=
\lim_{i\to\infty} \bigl\| \overline\cF(\gamma_i+\eta_\infty) -  \overline\cF(\gamma_i) - \overline{\rD\cF(\gamma_i)} \eta_\infty \bigr\|^{\overline\Omega}
\\
&=
\lim_{i\to\infty} \lim_{j\to\infty}  \bigl\| \cF(\gamma_i+\eta_j) - \cF(\gamma_i) - \rD\cF(\gamma_i) \eta_j \bigr\|^\Omega
\\
&=
\lim_{i\to\infty} \lim_{j\to\infty}  \bigl\| \textstyle\int_0^1  \rD\cF(\gamma_i+\lambda \eta_j) \eta_j \rd\lambda - \rD\cF(\gamma_i) \eta_j \bigr\|^\Omega
\\
&\leq
\lim_{i\to\infty} \lim_{j\to\infty} \bigl( \textstyle\int_0^1   \bigl\| \rD\cF(\gamma_i+\lambda \eta_j) - \rD\cF(\gamma_i) \bigr\|^{\cL(\Gamma,\Omega)} \rd\lambda  \; \| \eta_j  \|^\Gamma  \bigr)
\\
&=
\lim_{i\to\infty} \lim_{j\to\infty} \bigl( \textstyle\int_0^1   \bigl\| \overline{\rD\cF(\gamma_i+\lambda \eta_j)} - \overline{\rD\cF(\gamma_i)} \bigr\|^{\cL(\overline\Gamma,\overline\Omega)} \rd\lambda  \; \| \eta_j  \|^\Gamma  \bigr)
\\
&=
\lim_{i\to\infty}  \textstyle\int_0^1   \bigl\| \overline{\rD\cF}(\gamma_i+\lambda \eta_\infty) - \overline{\rD\cF(\gamma_i)} \bigr\|^{\cL(\overline\Gamma,\overline\Omega)} \rd\lambda  \; \| \eta_\infty  \|^{\overline\Gamma}
\\
&=
  \textstyle\int_0^1   \bigl\| \overline{\rD\cF}(\gamma_\infty+\lambda \eta_\infty) - \overline{\rD\cF}(\gamma_\infty) \bigr\|^{\cL(\overline\Gamma,\overline\Omega)} \rd\lambda \; \| \eta_\infty  \|^{\overline\Gamma} .
\end{align*}
Now for $\eta_\infty\ne 0$ this estimate can be rewritten as
\begin{align*}
\bigl\| \overline\cF(\gamma_\infty+\eta_\infty) - \overline\cF(\gamma_\infty) - \overline{\rD\cF}(\gamma_\infty) \eta_\infty \bigr\|^{\overline\Omega} \; / \; \| \eta_\infty  \|^{\overline\Gamma} 
&\leq
\textstyle\int_0^1   \bigl\| \overline{\rD\cF}(\gamma_\infty+\lambda \eta_\infty) - \overline{\rD\cF}(\gamma_\infty) \bigr\|^{\cL(\overline\Gamma,\overline\Omega)} \rd\lambda , 
\end{align*}
where we can argue with the continuity of $\overline{\rD\cF}$ at $\gamma_\infty$ that the right hand side converges to zero as $\|\eta_\infty\|^{\overline\Gamma}\to 0$. This proves that $\overline{\rD\cF}(\gamma_\infty)$ is indeed the differential of $\overline\cF$ at $\gamma_\infty$. And since $\overline{\rD\cF}:\cV_{\overline\Gamma}\to\cL(\overline\Gamma,\overline\Omega)$ was constructed to be uniformly continuous, this proves that $\overline\cF|_{\cV_{\overline\Gamma}}$ is uniformly $\cC^1$. This proves the Lemma in case $\ell=1$ with the additional fact that $\rD \overline\cF=\overline{\rD\cF}:\cV_{\overline\Gamma} \to \cL(\overline\Gamma,\overline\Omega)$ is the continuous extension of $\cV_\Gamma \to \cL(\overline\Gamma,\overline\Omega), \gamma \mapsto \overline{\rD\cF(\gamma)}$, where each $ \overline{\rD\cF(\gamma)}: \overline\Gamma\to\overline\Omega$ is the continuous extension of $ \rD\cF(\gamma): \Gamma\to\Omega$. 

We will now extend this statement to $\ell\geq 2$ by induction, assuming it is already established that for any uniformly $\cC^{\ell-1}$-map $\cF:\cV_\Gamma\to\Omega$ between open subsets of normed vector spaces the continuous extension $\overline\cF :\cV_{\overline\Gamma}\to\overline\Omega$ is uniformly $\cC^{\ell-1}$ as well, with $\rD^{\ell-1} \overline\cF=\overline{\rD^{\ell-1}\cF}$ given by the continuous extension of 
$\cV_\Gamma \to \cL^{\ell-1}(\overline\Gamma^{\ell-1};\overline\Omega), \gamma \mapsto \overline{\rD^{\ell-1}\cF(\gamma)}$, where each $ \overline{\rD^{\ell-1}\cF(\gamma)}: \overline\Gamma^{\ell-1}\to\overline\Omega$ is the continuous extension of $ \rD\cF^{\ell-1}(\gamma): \Gamma^{\ell-1}\to\Omega$. 

For the induction step, we assume in addition that $\cF : (\cV_\Gamma,\|\cdot\|^\Gamma) \to (\Omega,\|\cdot\|^\Omega)$ is uniformly $\cC^\ell$, and aim to prove that $\overline\cF$ is uniformly $\cC^\ell$ as well, with $\rD^\ell\overline\cF=\overline{\rD^\ell\cF}$. 
From the induction assumption we already know that $\overline\cF$ is uniformly $\cC^{\ell-1}$ with $\rD^{\ell-1}\overline\cF=\overline{\rD^{\ell-1}\cF}$. So it remains to prove that $\rD^{\ell-1}\overline\cF$ is uniformly $\cC^1$ with $\rD\rD^{\ell-1}\overline\cF=\overline{\rD^\ell\cF}$. 
The assumption that $\overline\cF$ is uniformly $\cC^\ell$ implies in particular that $\rD^{\ell-1}\cF : (\cV_\Gamma,\|\cdot\|^\Gamma) \to (\cL^{\ell-1}(\Gamma^{\ell-1};\Omega),\|\cdot\|^{\cL^{\ell-1}(\Gamma^{\ell-1};\Omega)})$ is uniformly $\cC^1$, where the norm on multilinear operators $\Phi:\Gamma^{\ell-1}\to\Omega$ is as in Remark~\ref{rmk:multilinear} 
$$
\| \Phi \|^{\cL^{\ell-1}(\Gamma^{\ell-1};\Omega)} = \sup \bigl\{ \| \Phi(\gamma_1,\ldots,\gamma_{\ell-1}) \|^\Omega \,\big|\, \|\gamma_1\|^\Gamma,\ldots,\|\gamma_{\ell-1}\|^\Gamma \leq 1  \bigr\} . 
$$
The uniform $\cC^1$-regularity means that there is a map $\cV_\Gamma\to\cL(\Gamma, \cL^{\ell-1}(\Gamma^{\ell-1};\Omega)), \gamma \mapsto \rD\rD^{\ell-1}\cF(\gamma)$ that is uniformly continuous and satisfies
$$
 \bigl\|  \rD^{\ell-1}\cF(\gamma+\eta) -  \rD^{\ell-1}\cF(\gamma) -  \rD\rD^{\ell-1}\cF(\gamma)\eta \bigr\|^{\cL^{\ell-1}(\Gamma^{\ell-1};\Omega))}  /  \|\eta\|^\Gamma 
 \;\underset{\|\eta\|^\Gamma \to 0}{\longrightarrow}\; 0 .
$$
Here we follow Remark~\ref{rmk:multilinear} to identify $\cL(\Gamma, \cL^{\ell-1}(\Gamma^{\ell-1};\Omega))\simeq\cL^\ell(\Gamma^\ell;\Omega)$ and call the resulting operator $\rD^\ell\cF(\gamma):=\rD\rD^{\ell-1}\cF(\gamma) \in \cL^\ell(\Gamma^\ell;\Omega)$ the $\ell$-th derivative of $\cF$ at $\gamma\in\cV_\Gamma$. Since each $\rD^\ell\cF(\gamma):\Gamma^\ell\to\Omega$ is uniformly continuous, it extends to a uniformly continuous map between the completions $\overline{\rD^\ell\cF(\gamma)}:\overline{\Gamma^\ell}=\overline\Gamma^\ell\to \overline\Omega$ with the same multilinearity property (by continuity) and the same norm $\|\overline{\rD^\ell\cF(\gamma)}\|^{\cL^\ell(\overline\Gamma^\ell;\overline\Omega)}=\| \rD^\ell\cF(\gamma)\|^{\cL^\ell(\Gamma^\ell;\Omega)}$. Next, as in the $\ell=1$ case above, the fact that the operators $\rD^\ell\cF(\gamma)\in\cL^\ell(\Gamma^\ell;\Omega)$ vary uniformly continuously with $\gamma\in \cV_\Gamma$ implies uniform continuity of the continuous extension $\overline{\rD^\ell\cF}: \cV_{\overline\Gamma} \to \cL^\ell(\overline\Gamma^\ell; \overline\Omega), \gamma\mapsto \overline{\rD^\ell\cF(\gamma)}$. Finally, we claim that this uniformly continuous family of operators -- viewed as a map $\overline{\rD^\ell\cF}: \cV_{\overline\Gamma} \to \cL(\overline\Gamma, \cL^{\ell-1}(\overline\Gamma^{\ell-1}; \overline\Omega))$ -- is the differential of $\rD^{\ell-1}\overline\cF$. Once established, this means that $\rD^{\ell-1}\overline\cF$ is uniformly $\cC^1$ and thus $\overline\cF|_{\cV_{\overline\Gamma}}$ is uniformly $\cC^\ell$, as claimed. 

\vfill
\pagebreak

To show that $\overline{\rD^\ell\cF}$ is in fact the differential of $\rD^{\ell-1}\overline\cF$ we estimate for any $\gamma_\infty=\lim_{i\to\infty}\gamma_i \in\cV_{\overline\Gamma}$, $\eta_\infty=\lim_{j\to\infty}\eta_j \in\overline\Gamma$ such that $\gamma_\infty+\eta_\infty\in\cV_{\overline\Gamma}$, and $\ul\kappa_\infty=(\kappa_{n,\infty})_{n=1,\ldots,\ell-1}=\lim_{j\to\infty}\ul\kappa_j =(\kappa_{n,j})_{n=1,\ldots,\ell-1} \in\overline\Gamma^{\ell-1}$ 
\begin{align*}
& \bigl\| \bigl( \rD^{\ell-1}\overline\cF(\gamma_\infty+\eta_\infty) - \rD^{\ell-1}\overline\cF(\gamma_\infty) - \overline{\rD^\ell\cF}(\gamma_\infty) \eta_\infty \bigr)  \ul\kappa_\infty \bigr\|^{\overline\Omega} 
\\
&= \bigl\| \overline{\rD^{\ell-1}\cF}(\gamma_\infty+\eta_\infty)  \ul\kappa_\infty - \overline{\rD^{\ell-1}\cF}(\gamma_\infty)  \ul\kappa_\infty - \overline{\rD^\ell\cF}(\gamma_\infty) (\eta_\infty, \ul\kappa_\infty) \bigr\|^{\overline\Omega} 
\\
&=
\lim_{i\to\infty} \bigl\| \overline{\rD^{\ell-1}\cF}(\gamma_i+\eta_\infty) \ul\kappa_\infty - \overline{\rD^{\ell-1}\cF(\gamma_i)} \ul\kappa_\infty - \overline{\rD^\ell\cF(\gamma_i)} (\eta_\infty, \ul\kappa_\infty) \bigr\|^{\overline\Omega}
\\
&=
\lim_{i\to\infty} \lim_{j\to\infty}  \bigl\| \rD^{\ell-1}\cF(\gamma_i+\eta_j)\ul\kappa_j - \rD^{\ell-1}\cF(\gamma_i)\ul\kappa_j - \rD^\ell\cF(\gamma_i) (\eta_j,\ul\kappa_j) \bigr\|^{\overline\Omega}
\\
&=
\lim_{i\to\infty} \lim_{j\to\infty}  \bigl\| \rD^{\ell-1}\cF(\gamma_i+\eta_j)\ul\kappa_j - \rD^{\ell-1}\cF(\gamma_i)\ul\kappa_j - \rD^\ell\cF(\gamma_i) (\eta_j,\ul\kappa_j) \bigr\|^{\Omega}
\\
&=
\lim_{i\to\infty} \lim_{j\to\infty}  \bigl\| \textstyle\int_0^1  \rD^\ell\cF(\gamma_i+\lambda \eta_j) (\eta_j,\ul\kappa_j) \rd\lambda - \rD^\ell\cF(\gamma_i) (\eta_j,\ul\kappa_j) \bigr\|^{\Omega}
\\
&\leq
\lim_{i\to\infty} \lim_{j\to\infty} \bigl( \textstyle\int_0^1   \bigl\| \rD^\ell\cF(\gamma_i+\lambda \eta_j) - \rD^\ell\cF(\gamma_i) \bigr\|^{\cL^\ell(\Gamma^\ell;\Omega)} \rd\lambda  \; \| \eta_j \|^\Gamma \prod_{n=1}^{\ell-1} \|\kappa_{n,j}  \|^\Gamma  \bigr)
\\
&=
\lim_{i\to\infty} \lim_{j\to\infty} \bigl( \textstyle\int_0^1   \bigl\| \overline{\rD^\ell\cF(\gamma_i+\lambda \eta_j)} - \overline{\rD^\ell\cF(\gamma_i)} \bigr\|^{\cL^\ell(\overline\Gamma^\ell;\overline\Omega)} \rd\lambda  \;  \| \eta_j \|^{\overline\Gamma} \prod_{n=1}^{\ell-1} \|\kappa_{n,j}  \|^{\overline\Gamma}  \bigr)
\\
&=
\lim_{i\to\infty}  \textstyle\int_0^1   \bigl\| \overline{\rD^\ell\cF}(\gamma_i+\lambda \eta_\infty) - \overline{\rD^\ell\cF}(\gamma_i) \bigr\|^{\cL^\ell(\overline\Gamma^\ell;\overline\Omega)} \rd\lambda  \; \| \eta_\infty \|^{\overline\Gamma} \prod_{n=1}^{\ell-1} \|\kappa_{n,\infty}  \|^{\overline\Gamma} 
\\
&=
  \textstyle\int_0^1   \bigl\| \overline{\rD^\ell\cF}(\gamma_\infty+\lambda \eta_\infty) - \overline{\rD^\ell\cF}(\gamma_\infty) \bigr\|^{\cL^\ell(\overline\Gamma^\ell;\overline\Omega)} \rd\lambda \;  \| \eta_\infty \|^{\overline\Gamma} \prod_{n=1}^{\ell-1} \|\kappa_{n,\infty}  \|^{\overline\Gamma} .
\end{align*}
In the operator norm \eqref{eq:multilinear operator norm} on $(\ell-1)$-linear maps, this implies
\begin{align*}
& \bigl\|  \rD^{\ell-1}\overline\cF(\gamma_\infty+\eta_\infty) - \rD^{\ell-1}\overline\cF(\gamma_\infty) - \overline{\rD^\ell\cF}(\gamma_\infty) \eta_\infty  \bigr\|^{\cL^{\ell-1}(\overline\Gamma^{\ell-1};\overline\Omega)} \\
&=  \sup_{\| \kappa_{1,\infty} \|^{\overline\Gamma}, \ldots, \| \kappa_{\ell-1,\infty} \|^{\overline\Gamma} \leq 1} \bigl\| \bigl( \rD^{\ell-1}\overline\cF(\gamma_\infty+\eta_\infty) - \rD^{\ell-1}\overline\cF(\gamma_\infty) - \overline{\rD^\ell\cF}(\gamma_\infty) \eta_\infty \bigr)  \ul\kappa_\infty \bigr\|^{\overline\Omega} 
\\
&\leq 
  \textstyle\int_0^1   \bigl\| \overline{\rD^\ell\cF}(\gamma_\infty+\lambda \eta_\infty) - \overline{\rD^\ell\cF}(\gamma_\infty) \bigr\|^{\cL^\ell(\overline\Gamma^\ell;\overline\Omega)} \rd\lambda \;  \| \eta_\infty \|^{\overline\Gamma} , 
\end{align*}
which we can rewrite for $\eta_\infty\ne 0$ as
\begin{align*}
& \bigl\| \rD^{\ell-1}\overline\cF(\gamma_\infty+\eta_\infty) - \rD^{\ell-1}\overline\cF(\gamma_\infty) - \overline{\rD^\ell\cF}(\gamma_\infty) \eta_\infty \bigr\|^{\cL^{\ell-1}(\overline\Gamma^{\ell-1};\overline\Omega)} \; / \; \| \eta_\infty  \|^{\overline\Gamma} \\
&\leq
\textstyle\int_0^1   \bigl\| \overline{\rD^\ell\cF}(\gamma_\infty+\lambda \eta_\infty) - \overline{\rD^\ell\cF}(\gamma_\infty) \bigr\|^{\cL^\ell(\overline\Gamma^\ell;\overline\Omega)} \rd\lambda . 
\end{align*}
Here we can argue with the continuity of $\overline{\rD^\ell\cF}$ at $\gamma_\infty$ that the right hand side converges to zero as $\|\eta_\infty\|^{\overline\Gamma}\to 0$. 
This proves that $\overline{\rD^\ell\cF}(\gamma_\infty)$ is indeed the differential of $\rD^{\ell-1}\overline\cF$ at $\gamma_\infty$, and thus $\overline\cF|_{\cV_{\overline\Gamma}}$ is uniformly $\cC^\ell$ with $\rD^{\ell-1}\overline\cF=\overline{\rD^\ell\cF}$, as claimed. 
\end{proof}

Finally, the notions of adiabatic $\cC^\ell$ regularity with respect to $\Delta$ for $\ell\geq 2$ will be defined in terms of tangent maps, rather than differentials, which we define in analogy to \cite[Def.1.1.14-15]{hofer-wysocki-zehnder_polyfold}.\footnote{Note, however, that we are not implementing sc-calculus in this paper, so our definitions agree with \cite{hofer-wysocki-zehnder_polyfold} only in the case of finite dimensional normed spaces.}

\vfill
\pagebreak

\begin{definition}\label{def:tangent map notation}
For any open subset $\cV_V\subset V$ of a topological vector space we define the {\bf tangent space} $\rT\cV_V := \cV_V \times V$ as a subset of $V^2=V\times V=\rT V$. 
For $\ell\geq 2$ we define {\bf higher tangent spaces} iteratively by $\rT^{\ell+1}\cV_V := \rT [\rT^\ell \cV_V]$. Equivalently, we have 
\begin{align*}
\rT^{\ell+1}\cV_V &:=\; \rT [\rT^\ell \cV_V] \;=\; \rT^\ell \cV_V \times \text{ambient vector space of $\rT^\ell \cV_V$} 
\\
&\phantom{:}=\; \rT^\ell \cV_V \times \rT^\ell V  \;=\; \cV_V \times V \times \ldots \times V   \;=\; \cV_V \times V^{N_\ell} \\
&\;\subset\; \rT^\ell V \times \rT^\ell V \;=\; \rT^{\ell+1} V \;=\; V^{N_\ell+1}
\end{align*}
with $N_\ell:=2^\ell-1$. 
We moreover denote $\rT^0 \cV_V:=\cV_V$ for efficiency of notation. 

When $V$ is equipped with a norm $\|\cdot\|^V$, then we equip higher tangent spaces $\rT^\ell V$ with the norm 
$$
\| (v_0,\ldots,v_{N_\ell}) \|^{\rT^\ell V} :=  \|v_0\|^V + \ldots + \|v_{N_\ell}\|^V. 
$$
The entries of a higher tangent vector $\ul v=(v_0,\ldots,v_{N_\ell})$ are referred to as {\bf base point} $v_0\in \cV_V$ and {\bf vector entries} $v_1,\ldots,v_{N_\ell}\in V$.\footnote{This differentiation arises from the different roles of base point and vector entries in the following notion of higher tangent maps. 
The vector entries $v_{i\geq 1}$ are always unbounded and can be multiplied by any scalar -- with linearity properties in the higher differentials making up the higher tangent maps. The base point $v_0$ is usually restricted to a bounded domain and -- even if it was scaleable -- has linearity properties only when studying a linear map.}
We denote the {\bf higher tangent fibers} -- the subsets of higher tangent vectors with fixed base point -- by  
$$
\rT_{v}^\ell \cV_V \,:=\; \bigl\{ \ul v=(v_0,\ldots,v_{N_\ell}) \in \rT^\ell\cV_V  \,\big|\, v_0 = v \bigr\} 
\;\simeq\; \bigl\{ (v_1,\ldots,v_{N_\ell}) \in V^{N_\ell} \bigr\} 
$$
and equip them with the {\bf fiber norm}\footnote{
The choice of this norm simplifies estimates of higher differentials in the space of multilinear maps. 
It is, of course, equivalent to the restriction of $\| \cdot \|^{\rT^\ell V}$ to the fiber -- but by constants depending on $\ell$. 
}
\begin{equation}\label{eq:fiber norm}
\| (v_0,\ldots,v_{N_\ell}) \|^{\rT^\ell_\bullet V} \,:=\;  \max\bigl\{ \|v_1\|^V , \ldots ,  \|v_{N_\ell}\|^V \bigr\}  . 
\end{equation}

Now consider normed vector spaces  $(V,\|\cdot\|^V)$ and $(W,\|\cdot\|^W)$ and a differentiable map $f:\cV_V\to W$ defined on an open subset $\cV_V\subset V$. 
We compile the map and its differential $\rD f : \cV_V \to \cL(V,W)$ (defined in Remark~\ref{rmk:multilinear}) into the {\bf tangent map}
$$
\rT f \,:\; \rT\cV_V = \cV_V \times V  \;\to\; \rT W  = W \times W , \qquad (v_0,v) \;\mapsto\; \bigl( f(v_0) , \rD f (v_0) v \bigr) . 
$$
Then we define higher tangent maps iteratively by $\rT^{\ell+1} f := \rT [\rT^\ell f]$ for $\ell\geq 2$, that is $\rT^{\ell+1} f (\ul v, \ul v') = (\rT^\ell f(\ul v) , \rD [\rT^\ell f](\ul v) \ul v')$ for $(\ul v, \ul v')\in\rT^{\ell+1}\cV_V = \rT^{\ell}\cV_V\times \rT^{\ell} V$. We denote $\rT^0 f:=f$ for efficient notation.
\end{definition}

\begin{rmk}\label{rmk:tangent map} \rm
\begin{itemlist}
\item
The advantage of the tangent map notation is that it keeps track of the base point of the tangent vectors, so that the notationally tricky chain rule for differentials $\rD(f\circ g)(p) = \rD f(g(p)) \circ \rD g (p)$ simplifies to the chain rule for tangent maps as in \cite[Thm.1.3.1.]{hofer-wysocki-zehnder_polyfold}: 

If $f:\cV_V \to W$ and  $g:\cV_U \to V$ are differentiable with $g(\cV_U)\subset\cV_V$, then  
\begin{equation}\label{eq:chain}
\rT(f\circ g) = \rT f\circ \rT g  .
\end{equation}
\item
The higher tangent maps are made up of -- but algebraically quite different from -- higher differentials. 
This difference is already notable for a linear map $F:V\to W$, whose differential is $\rD F (v_0) = F$ for any base point $v_0\in V$, and thus all higher differentials vanish $\rD^{\ell\geq 2}F =0$. In contrast, the tangent map is of the form $\rT F(v_0,v_1)=\bigl(F(v_0), F(v_1)\bigr)$, and since higher tangent maps arise from differentiation in all variables -- rather than just in the base point $v_0$ -- they are given by the linear map in each component. That is, for all $\ell\geq 1$ we have
\begin{align*}
\rT^\ell F \,:\; \rT^\ell V = V \times \ldots  \times V  &\;\to\; \rT^\ell W  = W \times \ldots \times W , \\\qquad (v_0,v_1, \ldots,v_{N_\ell})  &\;\mapsto\; \bigl( F(v_0) , F(v_1), \ldots ,  F(v_{N_\ell}) \bigr) .
\end{align*}
Indeed, arguing by induction, the $i$-th component of $\rD[\rT^{\ell} F ](\ul v) \ul v'$ is $ \tfrac\rd{\rd t}|_{t=0}\bigl(  F(v_i+t v'_i)  \bigr) = F (v'_i)$. 

\item
For a general -- sufficiently differentiable -- map $f:\cV_V\to W$ we can express higher tangent maps $\rT^\ell f$ in terms of higher differentials as follows: For $\ell=2$ we have
\begin{align*}
\rT^2 f \,:\; \rT^2\cV_V = \cV_V \times V  \times V  \times V  &\;\to\; \rT^2 W  = W \times W  \times W  \times W , \\\qquad (v_0,v_1,v_2,v_3)  &\;\mapsto\; \bigl( f(v_0) , \rD f (v_0) v_1 ,  \rD f (v_0) v_2 ,   \rD^2 f (v_0) (v_1,v_2) + \rD f (v_0) v_3 \bigr) ,
\end{align*}
where we can recover $\rD^2 f (v_0) (v_1,v_2)$ as the last entry of $\rT^2 f (v_0,v_1,v_2,0)$. 

Then $\rT^3 f \,:\; \rT^3\cV_V = \rT^2\cV_V \times \rT^2 V   \;\to\; \rT^3 W  = \rT^2 W \times \rT^2 W$ is  
$(\ul v, \ul v') \mapsto \bigl( \rT^2 f (\ul v) , \rD\rT^2 f (\ul v) \ul v' \bigr)$ where  
$\rD\rT^2 f (v_0,v_1,v_2,v_3) (v_4=v'_0, v_5=v'_1, v_6=v'_2, v_7=v'_3) \;=\; (w'_0, w'_1,w'_2,w'_3)$ consists of
\begin{align*}
 w'_0 &= \tfrac\rd{\rd t}|_{t=0} \bigl(  f(v_0+t v'_0) \bigr)
 \;=\; \rD f(v_0) v'_0 , \\
 w'_1 &=  \tfrac\rd{\rd t}|_{t=0}  \bigl(  \rD f (v_0+t v'_0) (v_1+t v'_1) \bigr)
  \;=\;\rD^2 f (v_0) ( v_1, v_0' ) +  \rD f (v_0) v'_1 , \\
 w'_2 &=  \tfrac\rd{\rd t}|_{t=0} \bigl(  \rD f (v_0+t v'_0) (v_2+t v'_2) \bigr)
 \;=\; \rD^2 f (v_0) ( v_2 , v_0' ) +  \rD f (v_0) v'_2 ,  \\
 w'_3 &=  
 \tfrac\rd{\rd t}|_{t=0}  \bigl(  \rD^2 f (v_0+t v'_0) (v_1+t v'_1 , v_2+t v'_2) + \rD f (v_0+t v'_0) (v_3+t v'_3) \bigr) \bigr) \\
&= \rD^3 f (v_0) (v_1, v_2, v'_0 ) + \rD^2 f (v_0) (v'_1, v_2)  + \rD^2 f (v_0) (v_1, v'_2) 
  + \rD^2 f (v_0) (v_3, v'_0)   + \rD f (v_0) v'_3 \\
&= \rD^3 f (v_0) (v_1, v_2 , v_4 ) + \rD^2 f (v_0) (v_5, v_2)  + \rD^2 f (v_0) (v_1,v_6) 
  + \rD^2 f (v_0) (v_3, v_4)   + \rD f (v_0) v_7 .
 \end{align*}
Inductively, we obtain
\begin{align*}
\rT^\ell f \,:\; \rT^\ell\cV_V = \cV_V \times \ldots  \times V  &\;\to\; \rT^\ell W  = W \times \cdots \times W , \\\qquad (v_0,\ldots,v_{N_\ell})  &\;\mapsto\; \bigl( f(v_0) , \cdots ,  \rD^\ell f (v_0) (v_1,\ldots, v_{2^{\ell-1}}) +  \cdots + \rD f (v_0) v_{2^\ell-1} \bigr), 
\end{align*}
where all other terms $\cdots$ are sums of differentials $\rD^k f(v_0) (v_*, \ldots, v_*)$ of order ${1\leq k\leq \ell-1}$ with $(v_*, \ldots, v_*)$ some permutation of a subset of $(v_1,\ldots,v_{2^\ell-1})$. 

\item
In particular, we can recover the $\ell$-th differential $\rD^\ell f (v_0) (v_1,\ldots, v_{2^{\ell-1}})$ as the last entry of $\rT^\ell f (v_0,v_1,v_2,0,v_4,0,\ldots,v_{2^{\ell-1}},0,\ldots,0)$.
Recall here that the $k$-th differential $\rD^k f(v_0)\in \cL^k(V^\ell;W)$ is a $k$-linear map $V^k\to W$ as defined in Remark~\ref{rmk:tangent map}

However, the different summands in each component of $\rT^\ell f$ have different numbers of arguments, thus $\rT^\ell f$ has no evident linearity properties (unless $f$ itself is linear). 
\end{itemlist}
\end{rmk}

The simplification of the chain rule in the tangent map notation is counterbalanced by algebraic complications in the relationship between higher differentials and higher tangent maps -- which we need to analyze in order to transfer analytic bounds between them. 
As seen in the above remark, a higher differential $\rD^\ell f$ can be read off from the last component of the corresponding higher tangent map $\rT^\ell f$, evaluated on specific higher tangent vectors. Thus we can transfer uniform bounds at a base point $v_0\in\cV_W$ by
\begin{align}\label{eq:Df by Tf bound}
\| \rD^\ell f(v_0) \|^{\cL^\ell(V,W)} 
&= 
\sup\bigl\{  \| \rD^\ell f(v_0) (v_1,\ldots,v_\ell) \|^W  \,\big|\, \|v_1\|^V,\ldots,\|v_\ell\|^V\leq 1 \bigr\} 
\\
& \leq 
\sup\bigl\{ \| \rT^\ell f(\ul v) \|^{\rT^\ell_\bullet W}   \,\big|\, \| \ul v \|^{\rT_\bullet V} \leq 1 \bigr\}  .
\nonumber 
\end{align}
Transfering continuity estimates already becomes more tricky, as we need to compare derivatives in the same directions. We can estimate this for $u,v\in\cV_W$ by
\begin{align}\label{eq:Df by Tf continuity}
& \| \rD^\ell f(u) -  \rD^\ell f(v) \|^{\cL^\ell(V,W)} 
 \\
&= \sup\bigl\{  \| \rD^\ell f(u) (x_1,\ldots,x_\ell) - \rD^\ell f(v) (x_1,\ldots,x_\ell) \|^W  \,\big|\, \|x_1\|^V,\ldots,\|x_\ell\|^V\leq 1 \bigr\} 
\nonumber\\
& \leq 
\sup\bigl\{ \| \rT^\ell f( (u,0,\ldots,0) + \ul x ) - \rT^\ell f((v,0,\ldots,0) + \ul x) \|^{\rT^\ell_\bullet W}   \,\big|\, 
\ul x \in \rT_0 V , \| \ul x \|^{\rT_\bullet V} \leq 1 \bigr\} . 
\nonumber
\end{align}
Bounding higher tangent maps in terms of higher differentials is more complicated.
The following estimate is adapted to our application needs. It also inductively bounds $\rT^{\ell+1} f=(\rT^{\ell} f,\rD\rT^{\ell}f)$.  

\begin{lemma} \label{lem:DT}
There exist constants $C^{\ell}_{\rT}$ for all $\ell\geq 0$ such that, given any 
$\ell$-fold differentiable map $f:\cV_V\to W$ between normed vector spaces 
$(V,\|\cdot\|^V)$ and $(W,\|\cdot\|^W)$ defined on an open subset $\cV_V\subset V$, 
we have for all $\ul v =(v_0,\ldots)\in\rT^{\ell}\cV_V$ 
\begin{align*}
\bigl\| \rD\rT^{\ell}f (\ul v) \bigr\|^{\cL(\rT^{\ell}V,\rT^{\ell}W )}  
&\leq 
C^{\ell}_{\rT} \textstyle\sum_{k=1}^{\ell+1} 
 \bigl\| \rD^k f (v_0)  \bigr\|^{\cL^{k}(V^{k},W)}  (\| \ul v \|^{\rT_\bullet^\ell V} \})^k  ,
\end{align*}
and  for all $\ul u =(u_0,\ldots), \ul v =(v_0,\ldots)\in\rT^{\ell}\cV_V$
\begin{align*}
& \bigl\| \rD\rT^{\ell} f (\ul u)  -   \rD\rT^{\ell}f (\ul v) \bigr\|^{\cL(\rT^{\ell}V,\rT^{\ell}W )}   \\
&\qquad\qquad \qquad \qquad \leq 
C^{\ell}_{\rT} \textstyle\sum_{k=1}^{\ell+1} 
\Bigl( 
 \bigl\| \rD^{k}f (u_0) -   \rD^{k}f (v_0) \bigr\|^{\cL^{k}(V^{k},W)} \max\{ \| \ul u \|^{\rT_\bullet^\ell V},  \| \ul v \|^{\rT_\bullet^\ell V} \}^k
 \\
 &\qquad\qquad \qquad \qquad\qquad \qquad \qquad  
 + \bigl\| \rD^k f (v_0)  \bigr\|^{\cL^{k}(V^{k},W)}  \| \ul u - \ul v \|^{\rT_\bullet^\ell V}   \max\{ \| \ul u \|^{\rT_\bullet^\ell V},  \| \ul v \|^{\rT_\bullet^\ell V} \}^{k-1}  
 \Bigr) .
\end{align*}
\end{lemma}

\begin{proof}
For $\ul v=(v^0,v^1,\ldots,v^{N_{\ell}}) \in\rT^{\ell}\cV_V$ we have from Remark~\ref{rmk:tangent map}
\begin{align*}
\bigl\| \rD\rT^{\ell}f (\ul v) \bigr\|^{\cL(\rT^{\ell}V,\rT^{\ell}W )}  
 &= 
\sup_{\| \ul x\|^{\rT^{\ell}V}\leq 1}
\bigl\|  \rD\rT^{\ell}f (\ul v) \ul x \bigr\|^{\rT^{\ell}W}  \\
&= 
\sup_{\| \ul x\| \leq 1} \textstyle\sum_{i=0}^{N_\ell}
\Bigl\|   \textstyle\sum_{*}  \rD^{k_*}f (v_0) (v_* \ldots v_* ,x_* ,v_* \ldots v_*)  \Bigr\|^W 
\\
&\leq 
\sup_{\| \ul x\| \leq 1} \textstyle\sum_{i=0}^{N_\ell}  \textstyle\sum_{*}
\bigl\| \rD^{k_*}f (v_0) \bigr\|^{\cL^{k_*}(V^{k_*},W)}   \| v_* \|^V \ldots  \| v_* \|^V   \| x_* \|^V  \| v_* \|^V\ldots  \| v_* \|^V  
\\ 
&\leq 
\textstyle \sum_{i=0}^{N_\ell}  \textstyle\sum_{*}
\bigl\| \rD^{k_*}f (v_0) \bigr\|^{\cL^{k_*}(V^{k_*},W)}   \| v_* \|^V \ldots  \| v_* \|^V   \| v_* \|^V\ldots  \| v_* \|^V  
\\ 
&\leq 
C^{\ell} \textstyle\sum_{k=1}^{\ell+1} 
 \bigl\| \rD^k f (v_0)  \bigr\|^{\cL^{k}(V^{k},W)}  (\| \ul v \|^{\rT_\bullet^\ell V} \})^k  ,
\end{align*}
where $C^{\ell}$ is a universal constant determined by the combinatorics of applying the chain rule to express higher tangent maps in terms of higher differentials, which we will replace below by $C^\ell\leq C^\ell_{\rT}$. 

Similarly, for $\ul u=(u^0,u^1,\ldots,u^{N_{\ell}}) , \ul v=(v^0,v^1,\ldots,v^{N_{\ell}}) \in\rT^{\ell}\cV_V$ we have
\begin{align*}
& \bigl\| \rD\rT^{\ell} f (\ul u)  -   \rD\rT^{\ell}f (\ul v) \bigr\|^{\cL(\rT^{\ell}V,\rT^{\ell}W )}  \\
&= 
\sup_{\| \ul x\|^{\rT^{\ell}V}\leq 1}
\bigl\| \rD\rT^{\ell} f (\ul u) \ul x -   \rD\rT^{\ell}f (\ul v) \ul x \bigr\|^{\rT^{\ell}W}  \\
&= 
\sup_{\| \ul x\| \leq 1} \textstyle \sum_{i=0}^{N_\ell}
\Bigl\|   \textstyle\sum_{*} \rD^{k_*}f (u_0) (u_* \ldots u_* , x_* ,u_* \ldots u_*) -   \rD^{k_*}f (v_0) (v_* \ldots v_* ,x_* ,v_* \ldots v_*)  \Bigr\|^W \\
&\leq 
\sup_{\| \ul x\| \leq 1}  \textstyle \sum_{i=0}^{N_\ell}  \textstyle\sum_{*}
\Bigl( 
 \bigl\| \rD^{k_*}f (u_0) (u_* \ldots u_* , x_* ,u_*\ldots u_*) -   \rD^{k_*}f (v_0) (u_* \ldots u_* , x_* , u_* \ldots u_*)  \bigr\|^W\\
&\qquad \qquad \qquad \qquad \qquad \qquad 
+  \bigl\| \rD^{k_*}f (v_0) (u_* - v_*, u_* \ldots u_* , x_* , u_*\ldots u_*) \bigr\|^W  + \ldots \\
&\qquad \qquad \qquad \qquad \qquad \qquad 
 +  \bigl\| \rD^{k_*}f (v_0) (v_* \ldots v_*, u_* - v_*, x_* , u_*\ldots u_*)  \bigr\|^W   \\
&\qquad \qquad \qquad \qquad \qquad \qquad 
+  \bigl\| \rD^{k_*}f (v_0) (v_* \ldots v_* , x_* , u_* - v_*, u_* \ldots u_*)  \bigr\|^W  + \ldots  \\
&\qquad \qquad \qquad \qquad \qquad \qquad 
+  \bigl\| \rD^{k_*}f (v_0) (v_* \ldots v_*, x_* , v_* \ldots v_*, u_*- v_*)  \bigr\|^W 
\Bigr)  \\
\end{align*}
\begin{align*}
&\leq 
\sup_{\| \ul x\| \leq 1} \textstyle \sum_{i=0}^{N_\ell}  \textstyle\sum_{*}
\Bigl( 
 \bigl\| \rD^{k_*}f (u_0) -   \rD^{k_*}f (v_0) \bigr\|^{\cL^{k_*}(V^{k_*},W)} \| u_* \|^V \ldots  \| u_* \|^V  \| x_* \|^V  \| u_* \|^V \ldots  \| u_* \|^V  \\
&\qquad \qquad \qquad \qquad \qquad 
+  \bigl\| \rD^{k_*}f (v_0)  \bigr\|^{\cL^{k_*}(V^{k_*},W)}  \| u_* - v_* \|^V   \| u_* \|^V \ldots  \| u_* \|^V  \| x_* \|^V \| u_* \|^V\ldots  \| u_* \|^V  + \ldots \\
&\qquad \qquad \qquad \qquad \qquad  
 +  \bigl\| \rD^{k_*}f (v_0) \bigr\|^{\cL^{k_*}(V^{k_*},W)}   \| v_* \|^V \ldots  \| v_* \|^V  \| u_* - v_* \|^V   \| x_* \|^V \| u_* \|^V\ldots  \| u_* \|^V   \\
&\qquad \qquad \qquad \qquad \qquad  
 +  \bigl\| \rD^{k_*}f (v_0) \bigr\|^{\cL^{k_*}(V^{k_*},W) }  \| v_* \|^V \ldots  \| v_* \|^V   \| x_* \|^V  \| u_* - v_* \|^V  \| u_* \|^V\ldots  \| u_* \|^V  + \ldots \\
&\qquad \qquad \qquad \qquad \qquad 
 +  \bigl\| \rD^{k_*}f (v_0) \bigr\|^{\cL^{k_*}(V^{k_*},W)}   \| v_* \|^V \ldots  \| v_* \|^V   \| x_* \|^V  \| v_* \|^V\ldots  \| v_* \|^V   \| u_* - v_* \|^V
  \Bigr)  \\ 
&\leq 
\textstyle\sum_{i=0}^{N_\ell}  \textstyle\sum_{*}
\Bigl( 
 \bigl\| \rD^{k_*}f (u_0) -   \rD^{k_*}f (v_0) \bigr\|^{\cL^{k_*}(V^{k_*},W)} \| u_* \|^V \ldots  \| u_* \|^V  \| u_* \|^V \ldots  \| u_* \|^V   \\
&\qquad \qquad \qquad 
+  \textstyle \sum_{j=1}^{k_*} \bigl\| \rD^{k_*}f (v_0)  \bigr\|^{\cL^{k_*}(V^{k_*},W)}  \| u_j - v_j \|^V \max\{  \| u_* \|^V,  \| v_* \|^V\} \ldots \max\{  \| u_* \|^V,  \| v_* \|^V\} \Bigr)  \\ 
&\leq 
C^{\ell} \textstyle\sum_{k=1}^{\ell+1} 
 \bigl\| \rD^{k}f (u_0) -   \rD^{k}f (v_0) \bigr\|^{\cL^{k}(V^{k},W)} ( \| \ul u \|^{\rT_\bullet^\ell V} )^k
 \\
 &\qquad
+ C^{\ell} \textstyle\sum_{k=1}^{\ell+1} 
k \, \bigl\| \rD^k f (v_0)  \bigr\|^{\cL^{k}(V^{k},W)}  \| \ul u - \ul v \|^{\rT_\bullet^\ell V}   \max\bigl\{ \| \ul u \|^{\rT_\bullet^\ell V},  \| \ul v \|^{\rT_\bullet^\ell V} \bigr\}^{k-1}  
 \Bigr) 
 \\
 &\leq 
C^{\ell}_{\rT} \textstyle\sum_{k=1}^{\ell+1} 
\Bigl( 
 \bigl\| \rD^{k}f (u_0) -   \rD^{k}f (v_0) \bigr\|^{\cL^{k}(V^{k},W)} \max\bigl\{ \| \ul u \|^{\rT_\bullet^\ell V},  \| \ul v \|^{\rT_\bullet^\ell V} \bigr\}^k
 \\
 &\qquad \qquad \qquad 
 + \bigl\| \rD^k f (v_0)  \bigr\|^{\cL^{k}(V^{k},W)}  \| \ul u - \ul v \|^{\rT_\bullet^\ell V}   \max\bigl\{ \| \ul u \|^{\rT_\bullet^\ell V},  \| \ul v \|^{\rT_\bullet^\ell V} \bigr\}^{k-1}  
 \Bigr) 
\end{align*}
with $C^{\ell}_{\rT}:=(\ell+1) C^\ell$.
\end{proof}

\subsection{Regularity Properties} \label{regularity}

This section defines various properties of adiabatic Fredholm families that will be needed to guarantee that their finite dimensional reductions are continuous and/or differentiable in the sense defined in Corollary~\ref{cor:family finite dimensional reduction}. 
The following notion is necessary for all such regularity proofs, including continuity. 
It encodes elliptic regularity of the PDEs -- more specifically the fact that any solution of the inhomogeneous PDE with a smooth right hand side is itself smooth. This is also a corollary of the regularizing notion in \cite[Def.3.1.16]{hofer-wysocki-zehnder_polyfold}.

\begin{definition} \label{def:regularizing}
An adiabatic Fredholm family $\bigl( (\cF_\eps:\cV_\Gamma\to \Omega )_{\eps\in\Delta} , \ldots  \bigr)$ as in Definition~\ref{def:fredholm} is called {\bf regularizing} if for all $\eps\in\Delta$ any solution $\gamma\in \overline\Gamma_\eps$ for the extended adiabatic Fredholm family of a nonlinear equation $\overline\cF_\eps(\gamma)=\omega$ or a linear equation $\rD\overline\cF_\eps(\gamma_0)\gamma=\omega$ at $\gamma_0\in\Gamma$, whose right hand side lies in the $\eps$-independent dense target subspace $\omega\in\Omega\subset\overline\Omega_\eps$, is guaranteed to lie in the $\eps$-independent dense domain subspace $\gamma\in\Gamma\subset\overline\Gamma_\eps$. 
More precisely, we have the two implications
\begin{align}
\gamma\in \cV_{\overline\Gamma,\eps}, \; \overline\cF_\eps(\gamma) \in \Omega \quad \Longrightarrow \quad \gamma\in \cV_\Gamma ,
\label{eq:reg 0}
\\
\gamma_0\in\cV_\Gamma, \; \gamma\in\overline\Gamma_\eps, \; \rD\overline\cF_\eps(\gamma_0)\gamma \in \Omega \quad \Longrightarrow \quad \gamma\in\Gamma .
\label{eq:reg 1}
\end{align}
\end{definition}

\begin{rmk} \label{rmk:regularizing}
Using the tangent map notation from Definition~\ref{def:tangent map notation}, the regularizing property can equivalent be phrased as 
\begin{equation}
\ul\gamma\in \rT\cV_{\overline\Gamma,\eps}, \; \rT\overline\cF_\eps(\ul\gamma) \in \rT\Omega \quad \Longrightarrow \quad \ul\gamma\in \rT\cV_\Gamma .
\label{eq:reg T}
\end{equation}
Indeed, our definitions yield $\rT\Omega = \Omega\times \Omega$ as well as
$\rT\cV_{\Gamma}=  \cV_{\Gamma} \times \Gamma$
and
$\rT\cV_{\overline\Gamma,\eps}=  \cV_{\overline\Gamma,\eps} \times  \overline\Gamma,\eps$. 
Now for $\ul\gamma=(\gamma_0,\gamma_1)\in  \cV_{\overline\Gamma,\eps} \times \overline\Gamma,\eps$ we have
$\rT\overline\cF_\eps(\ul\gamma)=\bigl( \overline\cF_\eps(\gamma_0) , \rD \overline\cF_\eps(\gamma_0)\gamma_1 \bigr)$. 
So if \eqref{eq:reg T} holds, then applying it to $(\gamma_0,0)$ yields \eqref{eq:reg 0}, whereas applying it to 
$(\gamma_0,\gamma_1)\in  \cV_{\Gamma} \times \overline\Gamma,\eps$ yields \eqref{eq:reg 1}. 
Conversely, \eqref{eq:reg 0} and \eqref{eq:reg 1} imply \eqref{eq:reg T}
since the first component gives $\overline\cF_\eps(\gamma_0)\in\Omega$, which by  \eqref{eq:reg 0} implies $\gamma_0\in\cV_\Gamma$. Then the second component is $ \rD \overline\cF_\eps(\gamma_0)\gamma_1\in\Omega$ and \eqref{eq:reg 1} applies to yield $\gamma_1\in\Gamma$. 

Thus it makes sense to define the regularizing property for higher tangent maps $\rT^\ell\overline\cF_\eps$ as well -- which is done in \eqref{def:regularizing T} as a part of the adiabatic $\cC^\ell$ regularity. 
\end{rmk}

The next notion encodes differentiability for fixed $\eps\in\Delta$. 

\begin{definition} \label{def:fibrewise C-l}
Given $\ell\geq 2$, an adiabatic Fredholm family as in Definition~\ref{def:fredholm} is called {\bf fibrewise $\cC^\ell$-regular} if it satisfies
\begin{itemlist}
\item[\bf{[Fibrewise $\mathcal{C}^\ell$ Regularity]}]
$\cF_\eps : (\cV_\Gamma,\|\cdot\|^\Gamma_\eps) \to (\Omega,\|\cdot\|^\Omega_\eps)$ is uniformly $\cC^\ell$ for each $\eps\in\Delta$. 
\end{itemlist}
\end{definition}

Next, we will use the language of tangent maps from Definition~\ref{def:tangent map notation} to formulate the notions of (higher) tangent maps of an adiabatic Fredholm family being continuous with respect to $\Delta$. Formulating these is complicated by the fact that we have not imposed strong enough conditions to guarantee the existence of a topology on $\Delta\times\Gamma$ that restricts on each ``fiber'' $\{\eps\}\times\Gamma$ to the topology induced by the norm $\|\cdot\|^\Gamma_\eps$. Instead, we will work with two types of continuity: Pointwise continuity in $\Delta$ for a fixed $\gamma_0\in\Gamma$ and uniform continuity in $\Gamma$ whose uniformity is independent of $\eps\in\Delta$. These are matters where it is not a priori clear that there are notions that are both weak enough to be satisfied in examples of adiabatic limits and strong enough to provide the desired regularity of finite dimensional reductions.

The following notion of pointwise continuity generalizes the [Near Solution] and [Continuity of Derivatives at $0$] properties of Definition~\ref{def:fredholm} -- equivalent to $\| \cF_\eps(0) - \cF_0(0) \|^\Omega_\eps \underset{\eps\to 0}\to 0$ and
$\bigl\|  \rD\cF_\eps (0) -   \rD\cF_0 (0)  \bigr\|^{\cL(\ol\Gamma_\eps,\ol\Omega_0)}  \underset{\eps\to 0}\to 0$ -- 
to limits $\Delta\ni \eps\to\eps_0\ne 0$, base points points $\gamma_0\ne 0$, and higher tangent maps. However, this is actually not true in the Examples~\ref{ex:adiabatic} unless $\gamma_0$ is sufficiently regular and lies in the subset $\Gamma_0\subset\Gamma$ from the classical adiabatic formulation in the introduction. We can enforce this by restricting the pointwise continuity requirement to $\gamma_0$ that solve $\cF_{\eps_0}(\gamma_0)\in\fC$ -- for appropriate representations of the cokernel $\fC\subset\Omega$ that naturally exist in the examples. 

\begin{definition} \label{def:pointwise regular}
An adiabatic Fredholm family $\bigl( (\cF_\eps:\cV_\Gamma\to \Omega )_{\eps\in\Delta} , \ldots  \bigr)$ of Definition~\ref{def:fredholm} is called 
\begin{itemize}
\item
{\bf pointwise continuous in $\mathbf{\Delta}$ at solutions modulo $\mathbf{\mathfrak C}$} (short: {\bf continuous in $\mathbf{\Delta}$ rel.\ $\mathbf{\mathfrak C}$}) if, given any $\eps_0\in\Delta$ and a solution $\gamma_0\in\cV_\Gamma$ of $\cF_{\eps_0}(\gamma_0)\in \fC$, we have 
$\bigl\|  \cF_\eps (\gamma_0) -   \cF_{\eps_0} (\gamma_0)  \bigr\|^\Omega_\eps \rightarrow 0$ as $\eps\to \eps_0$;
\item
{\bf pointwise $\cC^\ell$-continuous in $\mathbf{\Delta}$ at solutions modulo $\mathbf{\mathfrak C}$}  (short: {\bf $\cC^\ell$-continuous in $\mathbf{\Delta}$ rel.\ $\mathbf{\mathfrak C}$}) for a given $\ell\geq 0$ if 
it is fibrewise $\cC^\ell$ as in Definition~\ref{def:fibrewise C-l} and the family $(\rT^\ell\cF_\eps)_{\eps\in\Delta}$ is 
continuous in $\Delta$ rel.\ $\mathfrak C$, that is  
given any $\eps_0\in\Delta$ and a solution $\ul\gamma_0\in\rT^\ell\cV_\Gamma$ of $\rT^\ell\cF_{\eps_0}(\ul\gamma_0)\in \rT^\ell\fC$ we have 
$\bigl\|  \rT^\ell\cF_\eps (\ul\gamma_0)  -   \rT^\ell\cF_{\eps_0} (\ul\gamma_0) \bigr\|^{\rT^\ell\Omega}_\eps \rightarrow 0$ as $\eps\to \eps_0$.
\end{itemize}
\end{definition}

Recall here that fibrewise $\cC^1$ regularity is built into the notion of adiabatic Fredholm family. 

\begin{rmk} \label{rmk:pointwise C-1}\rm 
Note that pointwise continuity in $\Delta$ rel.\ $\mathfrak C$ is pointwise $\cC^0$-continuity  in $\Delta$ rel.\ $\mathfrak C$.

For $\ell=1$ recall that $\rT\cF_\eps(\gamma_0,\gamma_1)=(\cF_\eps(\gamma_0), \rD\cF_\eps(\gamma_0)\gamma_1)$ for $(\gamma_0,\gamma_1)\in\rT\cV_\Gamma=\cV_\Gamma\times\Gamma$. So an adiabatic Fredholm family is $\cC^1$ in $\Delta$ rel.\ $\mathfrak C$ if and only if it is continuous in $\Delta$ rel.\ $\mathfrak C$ and so is its differential: 
\begin{itemize}
\item
Given any $\eps_0\in\Delta$ and solutions $\gamma_0\in\cV_\Gamma$ of $\cF_{\eps_0}(\gamma_0)\in\fC$ and
$\gamma_1\in\Gamma$ of $\rD\cF_{\eps_0}(\gamma_0)\gamma_1\in \fC$, we have 
$\bigl\|  \rD\cF_\eps (\gamma_0)\gamma_1  -   \rD\cF_{\eps_0} (\gamma_0)\gamma_1 \bigr\|^\Omega_\eps \rightarrow 0$ as $\eps\to \eps_0$.
\end{itemize}
However, for $\ell\geq 2$ there is no evident characterization of $\cC^\ell$-continuity in $\Delta$ rel.\ $\mathfrak C$ in terms of the higher differentials -- since the components of $\rT^\ell\cF_\eps$ are sums of higher differentials.
%
\end{rmk}

Next, the following notion of uniform continuity generalizes the [Quadratic-ish Estimate] in Definition~\ref{def:fredholm} -- equivalent to $
\bigl\|  \rD\cF_\eps (\gamma_0) -   \rD\cF_\eps (0)  \bigr\|^{\cL(\ol\Gamma_\eps,\ol\Omega_\eps)} \leq c(\|\gamma_0 - 0 \|^\Gamma_\eps)$ -- 
to any pair of base points points $\gamma^\fl,\gamma^\fk\in\cV_\Gamma$ and higher tangent maps. 
Then we package pointwise and uniform continuity into notions of adiabatic regularity as follows:

\begin{definition} \label{def:adiabatic C-l}
Given $\ell\geq 0$, an adiabatic Fredholm family $\bigl( (\cF_\eps:\cV_\Gamma\to \Omega )_{\eps\in\Delta} , \ldots  \bigr)$ of Definition~\ref{def:fredholm}  is called {\bf adiabatic $\cC^\ell$-regular} -- or an {\bf adiabatic $\cC^\ell$-regular Fredholm family} --  if it satisfies the following: 
\begin{itemlist}
\item
The family is fibrewise $\cC^\ell$-regular as in Definition~\ref{def:fibrewise C-l}. 
\item
The family is regularizing as in Definition~\ref{def:regularizing}, and in case $\ell\geq 2$ so is the $\ell$-th tangent map in the sense that 
\begin{equation}\label{def:regularizing T}
\ul\gamma\in \rT^\ell\cV_{\overline\Gamma,\eps}, \; \rT^\ell\overline\cF_\eps(\ul\gamma) \in \rT^\ell\Omega \quad \Longrightarrow \quad \ul\gamma\in \rT^\ell\cV_\Gamma .
\end{equation}
\item
The family is pointwise $\cC^\ell$-continuous in $\Delta$ at solutions modulo $\mathbf{\mathfrak C}$ as in Definition~\ref{def:pointwise regular}.
\item[$\bullet$ \bf{[Uniform Continuity of $\rD^k\cF_\eps$ for $\mathbf{1\leq k\leq \ell \,}$]}]
In case $\ell=0$ there is no further condition. In case $\ell\geq 1$ we require monotone continuous functions $c^k_\cF : [0,\infty) \to [0,\infty)$ with $c^k_\cF(0)= 0$ 
for $1\leq k \leq \ell$ so that for all $\eps\in\Delta$ and $\gamma^\fl, \gamma^\fk\in\cV_\Gamma$ we have
$$
\bigl\| \rD^k\cF_\eps(\gamma^\fl) -  \rD^k\cF_\eps (\gamma^\fk) \bigr\|^{\cL^k(\ol\Gamma_\eps^k,\ol\Omega_\eps )} 
\leq c^k_\cF(\|\gamma^\fl-\gamma^\fk\|^{\Gamma}_\eps) . 
$$
\item[$\bullet$ \bf{[Uniform Bound on $\rD^k\cF_\eps(0)$ for $\mathbf{1\leq k\leq \ell \,}$]}]
In case $\ell=0$ there is no further condition. In case $\ell\geq 1$ we require constants $C^k_\cF\geq 1$ for $1\leq k \leq \ell$ so that for all $\eps\in\Delta$ we have
\begin{equation}\label{eq:DF bound}
\bigl\| \rD^k \cF_\eps (0)  \bigr\|^{\cL^k(\ol\Gamma_\eps^k,\ol\Omega_\eps )}  \leq C^k_\cF   \qquad\forall\eps\in\Delta, \gamma\in\cV_\Gamma .
\end{equation}
\end{itemlist}
An adiabatic Fredholm family is called {\bf adiabatic $\cC^\infty$-regular} -- or an {\bf adiabatic $\cC^\infty$-regular Fredholm family} -- if it is adiabatic $\cC^\ell$-regular for all $\ell\in\bN$.
\end{definition}

\begin{rmk} \label{rmk:uniform DF}
If an adiabatic Fredholm family satisfies {\rm [Uniform Continuity of $\rD^k \cF_\eps$ for $1\leq k\leq \ell \,$]} and {\rm [Uniform Bound on $\rD^k\cF_\eps(0)$ for $1\leq k\leq \ell \,$]} for some $\ell\geq 1$, then the uniform continuity transfers to the higher tangent maps of Definition~\ref{def:tangent map notation} and the extended adiabatic Fredholm family of Lemma~\ref{lem:fredholm}: 
{\rm
\begin{itemlist}
\item[\bf{[Uniform Continuity of $\rD\rT^{\ell-1}\ol\cF_\eps$]}] 
Given any $\delta>0$ there is a monotone continuous function 
$c^{\ell,\delta}_{\rT\cF}:[0,\infty)\to [0,\infty)$ with $c^{\ell,\delta}_{\rT\cF}(0)=0$ so that 
for all $\eps\in\Delta$, 
and $\ul\gamma^\fl, \ul\gamma^\fk\in\rT^{\ell-1}\cV_{\overline\Gamma,\eps}$ with 
$\|\gamma^\fl_0\|^\Gamma_\eps, \|\gamma^\fk_0\|^\Gamma_\eps \leq \delta$
we have
$$
\bigl\| \rD\rT^{\ell-1}\ol\cF_\eps (\ul\gamma^\fl)  -   \rD\rT^{\ell-1}\ol\cF_\eps (\ul\gamma^\fk) \bigr\|^{\cL(\rT^{\ell-1}\ol\Gamma_\eps,\rT^{\ell-1}\ol\Omega_\eps )} 
\leq 
c^{\ell,\delta}_{\rT\cF}(\|\ul\gamma^\fl - \ul\gamma^\fk\|^{\rT^{\ell-1}\Gamma}_\eps) \,  \max\bigl\{ 1, \| \ul\gamma^\fl \|^{\rT_\bullet^{\ell-1} \Gamma}_\eps,  \| \ul\gamma^\fk \|^{\rT_\bullet^{\ell-1} \Gamma}_\eps \bigr\}^\ell .
$$
\end{itemlist}


\smallskip
In fact, this estimate for the extended adiabatic Fredholm family follows by continuous extension from the analogous estimate for the adiabatic Fredholm family $\cF_\eps$ since Lemma~\ref{lem:extension} identifies $\rD\rT^{\ell-1}\ol\cF_\eps$ with the continuous extension of $\rD\rT^{\ell-1}\cF_\eps$. 

In case $\ell=1$ we have $\rD\rT^{\ell-1}\cF_\eps=\rD^\ell\cF_\eps$ so that uniform continuity holds with 
$c^1_{\rT\cF}:= c^1_\cF$ and the uniform bound assumptions on $\rD\cF_\eps$ and $\|\gamma\|^
\Gamma_\eps, \|\gamma^\fk\|^\Gamma_\eps$ are not needed. 
Uniform continuity of $\rD\rT^{\ell-1}\cF_\eps$ for $\ell\geq 2$ does require both assumptions. 
To check it we first combine [Uniform Continuity of $\rD^k\cF_\eps$ and Uniform Bound on $\rD^k\cF_\eps(0)$] for $1\leq k\leq \ell$
to bound $\rD^k\cF_\eps(\gamma)$ at $\gamma\neq 0$, 
\begin{align} 
\bigl\| \rD^k \cF_\eps (\gamma)  \bigr\|^{\cL^k(\ol\Gamma_\eps^k,\ol\Omega_\eps )} 
&\leq 
\bigl\| \rD^k \cF_\eps (0)  \bigr\|^{\cL^k(\ol\Gamma_\eps^k,\ol\Omega_\eps )} 
+ 
\bigl\| \rD^k\cF_\eps(\gamma) -  \rD^k\cF_\eps (0) \bigr\|^{\cL^k(\ol\Gamma_\eps^k,\ol\Omega_\eps )} 
\nonumber\\
&\leq C^k_\cF  + c^k_\cF(\|\gamma\|^{\Gamma}_\eps)  
\;\leq\; C^k_\cF  + c^k_\cF(\delta)   
\qquad\qquad
\forall
\; \gamma \in \cV_\Gamma,  \; \|\gamma\|^\Gamma_\eps \leq \delta. 
\label{eq:DF uniform at nonzero base}
\end{align}
Now we can combine the assumptions with Lemma~\ref{lem:DT} to estimate for any given $\delta>0$ and for all $\eps\in\Delta$ and 
$\ul\gamma^\fl=(\gamma^\fl_0\ldots) , \ul\gamma^\fk=(\gamma^\fk_0\ldots) \in\rT^{\ell-1}\cV_\Gamma$ 
with $\|\gamma^\fl_0\|^\Gamma_\eps, \|\gamma^\fk_0\|^\Gamma_\eps \leq \delta$
\begin{align*}
& \bigl\| \rD\rT^{\ell-1}\cF_\eps (\ul\gamma^\fl)  -   \rD\rT^{\ell-1}\cF_\eps (\ul\gamma^\fk) \bigr\|^{\cL(\rT^{\ell-1}\ol\Gamma_\eps,\rT^{\ell-1}\ol\Omega_\eps )}  \\
&\leq 
C^{\ell-1}_{\rT} \textstyle\sum_{k=1}^{\ell} 
\Bigl( 
 \bigl\| \rD^{k}\cF_\eps (\gamma_0^\fl) -   \rD^{k}\cF_\eps (\gamma_0^\fk) \bigr\|^{\cL^{k}(\ol\Gamma_\eps^{k},\ol\Omega_\eps)}  \max\bigl\{ \| \ul\gamma^\fl \|^{\rT_\bullet^{\ell-1} \Gamma}_\eps,  \| \ul\gamma^\fk \|^{\rT_\bullet^{\ell-1} \Gamma}_\eps \bigr\}^k
 \\
 &\qquad\qquad \qquad \qquad  
 + \bigl\| \rD^k \cF_\eps (\gamma_0^\fl)  \bigr\|^{\cL^k(\ol\Gamma_\eps^k,\ol\Omega_\eps )}  \|\ul\gamma^\fl  - \ul\gamma^\fk \|^{\rT_\bullet^{\ell-1} \Gamma}_\eps   \max\bigl\{ \| \ul\gamma^\fl \|^{\rT_\bullet^{\ell-1} \Gamma}_\eps,  \| \ul\gamma^\fk \|^{\rT_\bullet^{\ell-1} \Gamma}_\eps \bigr\}^{k-1}  
 \Bigr)
\\
&\leq 
C^{\ell-1}_{\rT} \textstyle\sum_{k=1}^{\ell} 
\Bigl( c^k_\cF(\|\gamma^\fl-\gamma^\fk\|^{\Gamma}_\eps)  \max\bigl\{ \| \ul\gamma^\fl \|^{\rT_\bullet^{\ell-1} \Gamma}_\eps,  \| \ul\gamma^\fk \|^{\rT_\bullet^{\ell-1} \Gamma}_\eps \bigr\}^k
 \\
 &\qquad\qquad \qquad \qquad  
 + \bigl( C^k_\cF  + c^k_\cF(\delta) \bigr)  \|\ul\gamma^\fl  - \ul\gamma^\fk \|^{\rT_\bullet^{\ell-1} \Gamma}_\eps   \max\bigl\{ \| \ul\gamma^\fl \|^{\rT_\bullet^{\ell-1} \Gamma}_\eps,  \| \ul\gamma^\fk \|^{\rT_\bullet^{\ell-1} \Gamma}_\eps \bigr\}^{k-1}  
 \Bigr)
\\
&\leq c^{\ell,\delta}_{\rT\cF}(\|\ul\gamma^\fl - \ul\gamma^\fk\|^{\rT^{\ell-1}\Gamma}_\eps) \,  \max\bigl\{ 1, \| \ul\gamma^\fl \|^{\rT_\bullet^{\ell-1} \Gamma}_\eps,  \| \ul\gamma^\fk \|^{\rT_\bullet^{\ell-1} \Gamma}_\eps \bigr\}^\ell 
\end{align*}  
with $c^{\ell,\delta}_{\rT\cF}(x):= C^{\ell-1}_{\rT} \textstyle\sum_{k=1}^{\ell} (  c^k_\cF(x) +  C^k_\cF \, x + c^k_\cF(\delta) \, x )$.
}
\end{rmk}

%
%
%

\vfill
\pagebreak

\section{Examples}
\label{examples}

This section sketches how Examples~\ref{ex:adiabatic} [QF] and [AF] can be described in terms of adiabatic Fredholm families.

\subsection{Quilted Floer Theory and Geometric Composition of Lagrangians}

Consider Lagrangian correspondences ${L_{01} \subset M_0^- \times M_1}$ and $L_{12} \subset M_1^- \times M_2$, where $M_\ell=(M_\ell,\om_{M_\ell})$ are symplectic manifolds and \(M_\ell^- \coloneqq (M_\ell, -\om_{M_\ell})\).
The {\bf geometric composition} of such Lagrangian correspondences is 
$L_{01} \circ L_{12} \coloneqq {\rm Pr}_{02}( L_{01} \times_{M_1} L_{12}) $, the image under the projection ${\rm Pr}_{02}\colon M_0^- \times M_1 \times M_1^- \times M_2 \to M_0^- \times M_2 $ of the fiber product
\begin{align*}
L_{01} \times_{M_1} L_{12} \coloneqq (L_{01} \times L_{12})
 \cap (M_0^- \times \Delta_1 \times M_2) .
\end{align*}
Here \(\Delta_1 \subset M_1 \times M_1^-\) denotes the diagonal.
If \(L_{01}\times L_{12}$ intersects $M_0^- \times \Delta_1 \times M_2\) transversely then \( {\rm Pr}_{02}\colon L_{01} \times_{M_1} L_{12} \to M_0^-\times M_2\) is a Lagrangian immersion, in which case we call \(L_{01}\circ L_{12}\) an {\bf immersed composition}.
In the case of {\bf embedded composition}, where the projection is injective and hence a Lagrangian embedding, some strict monotonicity and Maslov index assumptions allowed \cite{wehrheim_woodward_geometric_composition} to establish an isomorphism of quilted Floer cohomologies
\begin{equation} \label{eq:HFiso} 
HF(\ldots, L_{01},L_{12}, \ldots ) \cong HF(\ldots , L_{01} \circ L_{12}, \ldots) .
\end{equation}
This isomorphism can also be stated in terms of classical Lagrangian Floer homologies for Cartesian products. For example, it identifies 
$HF(L_0 \times L_{12},L_{01} \times L_2)\simeq HF(L_0 \times L_2,L_{01} \circ L_{12})$
by relating Floer trajectories in $M_0\times M_1\times M_2$ Floer trajectories in $M_0\times M_2$, as indicated in Figure~\ref{figure:layer}. 

\begin{figure}[ht]
\begin{picture}(0,0)%
\includegraphics{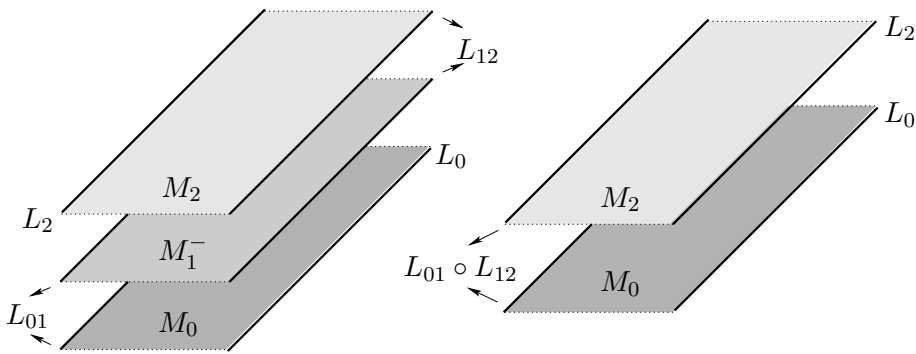}%
\end{picture}%
\setlength{\unitlength}{1865sp}%
\begingroup\makeatletter\ifx\SetFigFont\undefined%
\gdef\SetFigFont#1#2#3#4#5{%
  \reset@font\fontsize{#1}{#2pt}%
  \fontfamily{#3}\fontseries{#4}\fontshape{#5}%
  \selectfont}%
\fi\endgroup%
\begin{picture}(12252,4524)(136,-4213)
\put(226,-3880){\makebox(0,0)[lb]{{{$L_{01}$}}}}
\put(436,-2600){\makebox(0,0)[lb]{{{$L_2$}}}}
\put(2296,-2186){\makebox(0,0)[lb]{{{$M_2$}}}}
\put(2251,-3086){\makebox(0,0)[lb]{{{$M_1^-$}}}}
\put(2251,-3986){\makebox(0,0)[lb]{{{$M_0$}}}}
\put(6211,-350){\makebox(0,0)[lb]{{{$L_{12}$}}}}
\put(5931,-1750){\makebox(0,0)[lb]{{{$L_0$}}}}
\put(8126,-2321){\makebox(0,0)[lb]{{{$M_2$}}}}
\put(8126,-3446){\makebox(0,0)[lb]{{{$M_0$}}}}
\put(5511,-3250){\makebox(0,0)[lb]{{{$L_{01}\circ L_{12}$}}}}
\put(11900,-50){\makebox(0,0)[lb]{{{$L_{2}$}}}}
\put(11900,-1200){\makebox(0,0)[lb]{{{$L_0$}}}}
\end{picture}%
\caption{Tuples of pseudoholomorphic strips that are related by the isomorphism
$HF(L_0 \times L_{12},L_{01} \times L_2)\simeq HF(L_0 \times L_2,L_{01} \circ L_{12})$}
\label{figure:layer}
\end{figure}

The analytic core of the proof was an adiabatic limit called ``strip-shrinking'', in which a triple of pseudoholomorphic strips coupled by Lagrangian seam conditions degenerates to a pair of strips, via the width of the middle strip shrinking to zero. 
Here the monotonicity and embeddedness assumptions allowed for an implicit exclusion of all bubbling, and to apply the classical adiabatic method (as sketched in the introduction) to transversely cut out moduli spaces of dimension $0$. 

For general (compact or geometrically bounded) symplectic manifolds and general immersed composition\footnote{Unlike embeddedness, the transversality required for immersed composition can always be achieved by a small Hamiltonian perturbation of the Lagrangians.} of the Lagrangian correspondences, \cite{bottman_wehrheim} establishes a Gromov compactness theorem for strip shrinking -- including a full geometric understanding of all bubbling. The algebraic impact of bubbling was then cast into the first author's proposal of the symplectic $(A_\infty,2)$-category \cite{abouzaid-bottman}. So the foundational piece missing for a systematic description of the functorial properties of Fukaya categories is a local finite dimensional description for strip-shrinking moduli spaces, which is now provided by Theorem~\ref{thm:charts} -- coupled with gluing constructions as outlined in Remark~\ref{rmk:gluing} -- whenever the moduli space is described as the zero set of an adiabatic Fredholm family. 

In the following we develop this adiabatic Fredholm description for the archetypical sample case of three Floer strips coupled by Lagrangians $L_{01}, L_{12}$, with the middle strip of width $\eps>0$ being replaced in the $\eps=0$ limit by coupling the two remaining strips via the immersed Lagrangian $L_{01}\circ L_{12}$. 
For that purpose we fix the following data: 
\begin{itemlist}
\item
$(M_0,\omega_0, J_0), (M_1,\omega_1, J_1), (M_2,\omega_2, J_2)$ are symplectic manifolds equipped with compatible almost complex structures. 
\item
$L_0\subset M_0$, $L_{01} \subset M_0^- \times M_1$, $L_{12} \subset M_1^- \times M_2$, $L_2 \subset M_2$ are Lagrangian submanifolds satisfying 
\begin{itemize}
\item 
immersed composition $L_{01} \times L_{12} \; \pitchfork \; M_0^- \times \Delta_1 \times M_2$, 
\item
transverse intersection $L_0 \times L_{12} \; \pitchfork \; L_{01} \times L_2$. 
\end{itemize}
\item  $\ul x^\pm = (x_0^\pm,x_1^\pm, x_2^\pm)\in L_0 \times L_{12} \cap L_{01} \times L_2$ are intersection points of the tuples of Lagrangians.
\end{itemlist}
Then the {\bf strip-shrinking moduli space} $\cM_{[0,1]} := \bigcup_{\eps\in[0,1]} \cM_\eps$ is the union of the following moduli spaces for $\eps>0$ resp.\ $\eps=0$ 
\begin{align*}
\cM_\beps := \left\{
\left.
\left( 
\begin{aligned}
u_0:\bR\times[0,1]\to M_0 \\
u_1:\bR\times[0,\beps]\to M_1 \\
u_2:\bR\times[0,1]\to M_2 
\end{aligned}
\right)
\;\right|\;
\begin{aligned}
\partial_s u_0 + J_0(u_0) \partial_t u_0 = 0  ,\\
\partial_s u_1 + J_1(u_1) \partial_t u_1 = 0  ,\\
\partial_s u_2 + J_2(u_2) \partial_t u_2 = 0  , 
\end{aligned}
\;\;
\begin{aligned}
&u_0|_{t=0}\in L_0, \; u_2|_{t=1}\in L_2 ,\\
& (u_0|_{t=1},u_1|_{t=0})\in L_{01}, \\
& (u_1|_{t=\beps},u_2|_{t=0})\in L_{12} 
\end{aligned}
\right\} /  \; \bR , 
\end{align*}
\begin{align*}
\cM_{\bm{0}} := \left\{
\left.
\left( 
\begin{aligned}
u_0:\bR\times[0,1]\to M_0 \\
u_1:\bR\times{\bm{\{0\}}}\to M_1 \\
u_2:\bR\times[0,1]\to M_2 
\end{aligned}
\right)
\;\right|\;
\begin{aligned}
\partial_s u_0 + J_0(u_0) \partial_t u_0 = 0  ,\\
\\
\partial_s u_2 + J_2(u_2) \partial_t u_2 = 0  , 
\end{aligned}
\;\;
\begin{aligned}
&u_0|_{t=0}\in L_0, \; u_2|_{t=1}\in L_2 ,\\
& (u_0|_{t=1},u_1|_{t={\bm{0}}}) \in L_{01}, \\
& (u_1|_{t={\bm{0}}},u_2|_{t=0}) \in L_{12} 
\end{aligned}
\right\} /  \; \bR .
\end{align*}
Here and throughout we require 
$\lim_{s\to\pm\infty} \bigl(u_0(s,\cdot), u_1(s,\cdot), u_2(s,\cdot)\bigr) = (x_0^\pm,x_1^\pm, x_2^\pm)$ and quotient by simultaneous $\bR$-shifts $\bigl(u_0(\cdot,\cdot), u_1(\cdot,\cdot), u_2(\cdot,\cdot)\bigr)\sim \bigl(u_0(R+\cdot,\cdot), u_1(R+\cdot,\cdot), u_2(R+\cdot,\cdot)\bigr)$ for all $R\in\bR$. 

Note here that for embedded geometric composition $L_{01}\circ L_{12}\subset M_0 \times M_2^-$ the $\eps=0$ moduli space can be identified with a more traditional view of Floer strips in $M_0\times M_2^-$, 
\begin{align*}
\cM_{\bm{0}} \simeq \left\{
\left.
\left( 
\begin{aligned}
u_0:\bR\times[0,1]\to M_0 \\
u_2:\bR\times[0,1]\to M_2 
\end{aligned}
\right)
\;\right|\;
\begin{aligned}
\partial_s u_0 + J_0(u_0) \partial_t u_0 = 0  ,\\
\partial_s u_2 + J_2(u_2) \partial_t u_2 = 0  , 
\end{aligned}
\;\;
\begin{aligned}
&u_0|_{t=0}\in L_0, \; u_2|_{t=1}\in L_2 ,\\
& (u_0|_{t=1},u_2|_{t=0}) \in L_{01}\circ L_{12} 
\end{aligned}
\right\} /  \; \bR , 
\end{align*}
since $u_1:\bR \to M_1$ is determined by the unique lift of $(u_0|_{t=1},u_2|_{t=0}):\bR\to {\rm Pr}_{02}( L_{01}\times_{M_1} L_{12})$ to $(u_0|_{t=1},u_1,u_1,u_2|_{t=0}):\bR\to  L_{01}\times_{M_1} L_{12}$.
And when ${\rm Pr}_{02}: L_{01}\times_{M_1} L_{12} \to M_0\times M_1^-\times M_1\times M_2^-$ is an immersion, then the lift $u_1$ is the appropriate analytic data to keep track of. 

This is the perspective used for the adiabatic analysis in \cite{wehrheim_woodward_geometric_composition}. 
To make the connection with the general description of adiabatic limits in the introduction, note that the moduli spaces for $\eps>0$ can be also be identified with moduli spaces of tuples of maps $(u_0,u_1,u_2):\bR\times[0,1]\to M_0\times M_1 \times M_2$ defined on a common strip of width $1$, solving a deformation of the Cauchy-Riemann PDE. This identification arises by rescaling $u^{\rm new}_1(s,t)=u^{\rm old}_1(s,\eps \, t)$, so that $\partial_t u^{\rm new}_1 = \eps \, u^{\rm old}_1$ results in 
\begin{align*}
\cM_\beps \simeq \left\{
\left.
\left( 
\begin{aligned}
u_0:\bR\times[0,1]\to M_0 \\
u_1:\bR\times[0,1]\to M_1 \\
u_2:\bR\times[0,1]\to M_2 
\end{aligned}
\right)
\;\right|\;
\begin{aligned}
& \partial_s u_0 + J_0(u_0) \partial_t u_0 = 0  ,\\
& \partial_s u_1 + \beps^{-1} J_1(u_1) \partial_t u_1 = 0  ,\\
& \partial_s u_2 + J_2(u_2) \partial_t u_2 = 0  , 
\end{aligned}
\;\;
\begin{aligned}
&u_0|_{t=0}\in L_0, \; u_2|_{t=1}\in L_2 ,\\
& (u_0|_{t=1},u_1|_{t=0})\in L_{01}, \\
& (u_1|_{t=1},u_2|_{t=0})\in L_{12} 
\end{aligned}
\right\} /  \; \bR . 
\end{align*}
The breakthrough idea for describing this strip-shrinking moduli space by an adiabatic Fredholm family is to multiply the PDEs with $J_0$, $\eps J_1$, and $J_2$ to obtain a uniform description for all $0\leq \eps \leq 1$
\begin{align*}
\cM_\beps \simeq \left\{
\left.
\left( 
\begin{aligned}
u_0:\bR\times[0,1]\to M_0 \\
u_1:\bR\times[0,1]\to M_1 \\
u_2:\bR\times[0,1]\to M_2 
\end{aligned}
\right)
\;\right|\;
\begin{aligned}
& \partial_t u_0 - J_0(u_0) \partial_s u_0 = 0  ,\\
& \partial_t u_1 - \beps J_1(u_1) \partial_s u_1 = 0  ,\\
& \partial_t u_2 - J_2(u_2) \partial_s u_2 = 0  , 
\end{aligned}
\;\;
\begin{aligned}
&u_0|_{t=0}\in L_0, \; u_2|_{t=1}\in L_2 ,\\
& (u_0|_{t=1},u_1|_{t=0})\in L_{01}, \\
& (u_1|_{t=1},u_2|_{t=0})\in L_{12} 
\end{aligned}
\right\} /  \; \bR . 
\end{align*}
That is, the idea is to view the family of nonlinear differential operators 
$\partial_t - (J_0, \eps J_1, J_2)\partial_s$ as an adiabatic Fredholm family on a space of maps $\ul u : \bR\times[0,1]\to M_0\times M_1 \times M_2$ with Lagrangian boundary conditions, limits $\ul x^\pm$, and a suitable local slice condition for the action of $\bR$. 

In order to be able to import the estimates from \cite{wehrheim_woodward_geometric_composition} and \cite{bottman_wehrheim}) we need one more reformulation of the moduli spaces as tuples of maps to $M_{02}:=M_0^-\times M_2$ and $\hat M := M_{0211}:=M_0\times M_2^-\times M_1^-\times M_1$, with almost complex structures $J_{02}:=(-J_0,J_2)$ and  $\hat J:=(J_0,-J_2,-J_1,J_1)$, which 
sets up the analysis to effectively utilize the transversality of the intersection 
$L_{01} \times L_{12} \, \pitchfork \, M_0^- \times \Delta_1 \times M_2$.
This is described in detail in \cite[Figure 4]{wehrheim_woodward_geometric_composition}, resulting in
\begin{align*}
\cM_\beps \simeq \left\{
\left( 
\begin{aligned}
u_{02}:\bR\times[0,1]\to M_{02} \\
\hat u = (\hat u_{02}, \hat u_{11}) :\bR\times[0,1]\to \hat M
\end{aligned}
\right)
\;\left|\;
\begin{aligned}
& \partial_t u_{02} - J_{02}(u_{02}) \partial_s u_{02} = 0  ,\\
& \partial_t \hat u - \beps \, \hat J(\hat u) \partial_s \hat u = 0  , 
\end{aligned}
\;\;
\begin{aligned}
& u_{02}|_{t=1} \in L_0\times L_2 ,\\
& u_{02}|_{t=0} =\hat u_{02} |_{t=0} ,\\
& \hat u_{11} |_{t=0} \in \Delta_1, \\
& \hat u |_{t=1} \in L_{0211}
\end{aligned}
\right.
\right\} /  \; \bR . 
\end{align*}
Here $L_{0211}\subset\hat M=M_0\times M_2^-\times M_1^-\times M_1$ denotes the Lagrangian obtained by appropriate permutation of the factors in $L_{01} \times L_{12}\subset M_0^-\times M_1\times M_1^- \times M_2$. 
Now fix a solution $\ul u=\bigl(u_{02},\hat u = (\hat u_{02},\hat u_{11}) \bigr)$ for $\eps=0$ -- noting that its second component $\hat u=\hat u(s)$ is independent of $t\in[0,1]$. 
Then we will describe a neighbourhood of $[\ul u] \in \cM_{[0,1]}$ as the zero set
$$
\textstyle\bigcup_{\eps\in\Delta} \{\eps\}\times \cF_\eps^{-1}(0)  \;\; \overset{\sim}{\longrightarrow}\;\; \cU \;\subset\; \bigcup_{\eps\in[0,1]} \cM_\eps  \;=\;  \cM_{[0,1]}
$$
of an adiabatic Fredholm family, starting with the domain space\footnote{Here and throughout we write $\rT_u L$ as a short hand for $u^*\rT L$.} 
\begin{align*}
\Gamma:= \left\{
\left( 
\begin{aligned}
&\xi_{02} \\
&\hat\xi = (\hat\xi_{02},\hat\xi_{11})
\end{aligned}
\right)
\in
\begin{aligned}
\cC^\infty(\bR\times[0,1], \rT_{u_{02}} M_{02}) \\
\times \; \cC^\infty(\bR\times[0,1], \rT_{\hat u} \hat M)
\end{aligned}
\; \left|\; 
\begin{aligned}
& \xi_{02}|_{t=1} \in \rT_{u_{02}|_{t=1}}(L_0\times L_2) ,\\
& \xi_{02}|_{t=0} =\hat \xi_{02} |_{t=0} ,\\
& \hat \xi_{11} |_{t=0} \in \rT_{\hat u_{11}}\Delta_1, \\
& \hat \xi |_{t=1} \in\rT_{\hat u} L_{0211}  , 
\end{aligned}
\;\;
\begin{aligned}
& \text{[slice]} ,\\ 
& \text{[decay]} .
\end{aligned}
\right.
\right\}
\end{align*}
Here [slice] is a slicing condition such as $\xi_{02}(z_0) \in H_{02}$ at a point where $\partial_s u_{02}(z_0)\ne 0$ is nonzero, thus has a codimension $1$ complement  in  $\rT_{u_{02}(z_0)}= \bR \partial_s u_{02}(z_0) \oplus H_{02}$. 
And [decay] is an exponential decay condition such as 
$e^{\delta |R|} \| \xi_{02}|_{ [R,R+1]\times [0,1]} \|_{\cC^2} + \| \hat\xi |_{ [R,R+1]\times [0,1]} \|_{\cC^2} \to 0 $ as $R\to\pm\infty$ with some $\delta>0$ that is automatic for solutions of the PDEs (which requires the uniform exponential decay of \cite[Lemma 3.2.3]{wehrheim_woodward_geometric_composition} 
) and will ensure that we can equip $\Gamma$ with well-defined norms. 

We specify a convex $\cC^0$-open neighbourhood of $0=0_\Gamma\in\Gamma$ by restricting to the injectivity radius of exponential maps $e_{02}: \rT_{u_{02}} M_{02} \to M_{02}$ and $\hat e: \rT_{\hat u} \hat M \to \hat M$, 
$$
\cV_\Gamma \,:=\;  \bigl\{ (\xi_{02}, \hat\xi ) \in \Gamma \,\big|\,  \| \xi_{02} \|_{\cC^0} , \| \hat\xi \|_{\cC^0} < \text{injectivity radius} \bigr\} . 
$$
The target space has no boundary or slicing conditions, just the same exponential decay condition
\begin{align*}
\Omega:= \left\{
\left.\left( 
\begin{aligned}
&\eta_{02} \\
&\hat\eta 
\end{aligned}
\right)
\in
\begin{aligned}
\cC^\infty(\bR\times[0,1], \rT_{u_{02}} M_{02}) \\
\times \; \cC^\infty(\bR\times[0,1], \rT_{\hat u} \hat M)
\end{aligned}
\; \right|\; 
\begin{aligned}
& \text{[limit]}
\end{aligned}
\right\}.
\end{align*}
With that we can specify the maps $\cF_\eps : \cV_\Gamma \to \Omega$ for $\eps\in\Delta:=[0,1]$ by 
\begin{align*}
\cF_\eps (\xi_{02}, \hat\xi ) 
\,:=\; 
\bigl( \; \Phi_{02}(\xi_{02})^{-1}  ( \partial_t - J_{02} \partial_s ) \, e_{02}(\xi_{02}) 
\;,\; 
\hat\Phi(\hat\xi)^{-1}  ( \partial_t - \eps \hat J \partial_s ) \, \hat e(\hat\xi)
\; \bigr) , 
\end{align*}
where $\Phi(\xi)$ denotes the parallel transport
$\rT_{u}M\to \rT_{e (\xi)}M$ along the path $\tau\mapsto e(\tau\xi)$. 
Next, we specify norms on $\Gamma$ and $\Omega$ for $\eps\in(0,1]$ by pullback of the $H^2\cap W^{1,4}$ and $H^1\cap L^4$ norms that were introduced in \cite[\S3.1]{wehrheim_woodward_geometric_composition} to obtain uniform quadratic estimates, 
\begin{align*}
\|(\xi_{02}, \hat\xi) \|^\Gamma_{\eps>0} 
&\,:=\;  \| \xi_{02}\|_{H^2} 
+ \eps^{1/2} \bigl( \| \hat\xi\|_{L^2} +  \| \nabla_s\hat\xi\|_{L^2}  +  \| \nabla_s^2\hat\xi\|_{L^2}   \bigr) 
+ \eps^{1/4} \bigl( \| \hat\xi\|_{L^4} +  \| \nabla_s\hat\xi\|_{L^4}   \bigr) 
\\
&\qquad \qquad 
+ \eps^{-1/2} \bigl( \| \nabla_t\hat\xi\|_{L^2} +  \| \nabla_t\nabla_s\hat\xi\|_{L^2}  +  \| \nabla_s\nabla_t \hat\xi\|_{L^2}   \bigr)
+ \eps^{-3/2}  \| \nabla_t^2\hat\xi\|_{L^2}  + \eps^{-3/4}  \| \nabla_t\hat\xi\|_{L^4}  ,
\\
\|(\eta_{02}, \hat\eta) \|^\Omega_{\eps>0}  
&\,:=\;  \| \eta_{02}\|_{H^1} 
+ \eps^{-1/2} \bigl( \| \hat\eta\|_{L^2} +  \| \nabla_s\hat\eta\|_{L^2}   \bigr) 
+ \eps^{-3/2} \| \nabla_t\hat\eta\|_{L^2}  
+ \eps^{-3/4}  \| \hat\eta\|_{L^4}  .
\end{align*}
The completions w.r.t.\ these norms will be the same Banach spaces for all $\eps>0$ -- they are just equipped with $\eps$-dependent equivalent norms --  
\begin{align*}
\ol\Gamma_{\eps>0} 
&\;\subset\;
[H^2\cap W^{1,4}](\bR\times[0,1], \rT_{u_{02}} M_{02}) \; \times \; [H^2\cap W^{1,4}](\bR\times[0,1], \rT_{\hat u} \hat M) , \\
\ol\Omega_{\eps>0} 
&\;=\;
[H^1\cap L^4](\bR\times[0,1], \rT_{u_{02}} M_{02}) \; \times \; [H^1\cap L^4](\bR\times[0,1], \rT_{\hat u} \hat M) ,
\end{align*}
where $\ol\Gamma_{\eps>0}$ is the closed subspace specified by the Lagrangian boundary conditions and slice condition. 
For $\eps=0$ we developed norms to obtain Fredholm properties for $\cF_0$ guided by the following. 

\begin{rmk}[The auxiliary isomorphism for $\eps=0$]  \label{rmk:auxiso}
The second component of $\rD\cF_0(0)$ is $\hat\xi\mapsto \nabla_t\hat\xi$, which -- with the decoupled boundary conditions in $\Gamma$ -- induces isomorphisms for any $k\in\bN_0$,  $s \geq 0$ 
\begin{align*}
\ti \nabla_t \,:\; 
&\bigl\{ \hat\xi \in W^{k+1,1}([0,1],H^s(\bR, \rT_{\hat u} \hat M))  \,\big|\,  \hat \xi_{11} |_{t=0} \in \rT_{\hat u_{11}}\Delta_1, \hat \xi |_{t=1} \in\rT_{\hat u} L_{0211}  \bigr\} \\
&\qquad\qquad\qquad\qquad\qquad\qquad\qquad
\to \; W^{k,1}([0,1],H^s(\bR, \rT_{\hat u} \hat M)) \times H^s(\bR, \rT_{u_{02}|_{t=0}} L_{02}) \\
& \qquad\qquad\qquad\;
\hat\xi = (\hat\xi_{02}, \hat\xi_{11} ) 
\qquad \;\mapsto\; \qquad  \bigl( \; \nabla_t \hat \xi \;,\; \pi_{02} \, \hat\xi_{02}|_{t=0}  \; \bigr). 
\end{align*}
\rm
Indeed, its inverse is $(\hat \eta , \lambda) \mapsto \hat\xi(t) = {\rm Pr}_{\rT L_{02}}^{-1}\lambda + \int_0^t \hat\eta(x) \rd x - (\pi_{02}^\perp \times \pi_{11}) \bigl(\int_0^1 \hat\eta(x) \rd x\bigr)$. This uses the isomorphism \eqref{eq:immersed tangent spaces} and the projections $\pi_{02}: \rT_{u_{02}|_{t=0}} M_{02} \to \rT_{u_{02}|_{t=0}}  L_{02}$ 
and $\pi_{02}^\perp \times \pi_{11} : \rT_{\hat u} \hat M \to  (\rT_{u_{02}|_{t=0}}  L_{02})^\perp \times \rT_{\hat u_{11}}\Delta_1$ to a complement of $\rT_{\hat u} (M_{02}\times \Delta_1) \cap \rT_{\hat u}  L_{0211} = \rT_{\hat u} \ti L_{02}\subset \rT_{\hat u} (M_{02}\times \Delta_1)$ -- all of which are
constructed in the following remark for immersed geometric compositions.
\end{rmk}

\begin{rmk}[Notation and Splittings for the Lagrangian Immersion]  \label{rmk:immersion} \rm
We denote by $\ti L_{02} := M_{02}^-\times\Delta_1 \cap L_{0211}\subset M_{0211}=M_0\times M_2^-\times M_1^-\times M_1$ the fibre product $L_{01} \times_{M_1} L_{12} \subset M_0^- \times M_1\times M_1^- \times M_2$ after permuting components.
Then the Lagrangian immersion is given by projection to the first factors in $M_{0211}=M_{02}^-\times M_{11}$,  
$$
{\rm Pr}_{L_{02}} \,:\; 
\ti L_{02}  \;\to\; L_{02} \,:=\; L_{01} \circ L_{12}  \;\subset\; M_{02}  . 
$$
Next, note that $\hat u :\bR\to \ti L_{02}$ is a lift of ${\rm Pr}_{L_{02}}\circ\hat u = u_{02}|_{t=0}  :\bR\to L_{02}$. So if ${\rm Pr}_{L_{02}}$ is an embedding, then the linearizations ${\rm Pr}_{\rT L_{02}} := \rT {\rm Pr}_{L_{02}}: \rT \ti L_{02} \to \rT L_{02}$ induce isomorphisms 
\begin{equation}\label{eq:immersed tangent spaces}
{\rm Pr}_{\rT L_{02}} \,:\; 
\rT_{\hat u} \ti L_{02}  \;\overset{\sim}{\to}\; \rT_{u_{02}|_{t=0}} L_{02}  . 
\end{equation}
When ${\rm Pr}_{L_{02}}$ is an immersion, then ${\rm Pr}_{\rT L_{02}}(\rT_{\hat u(s)} \ti L_{02})\subset \rT_{u_{02}(s,0)} M_{02}$ are still Lagrangian subspaces for all $s\in\bR$ (see  \cite[Lemma 2.0.5]{wehrheim2010quilted}), so that it makes sense to define the immersed tangent spaces
$\rT_{u_{02}(s,0)} L_{02}:= {\rm Pr}_{\rT L_{02}}(\rT_{\hat u(s)} \ti L_{02})$ -- which makes \eqref{eq:immersed tangent spaces} an isomorphism by definition. 
This allows us to introduce an orthogonal splitting for the  immersed Lagrangian tangent spaces 
$$
(\pi_{02}\times \pi_{02}^\perp) \, :\;  
\rT_{u_{02}|_{t=0}} M_{02} 
\;\overset{\sim}{\to}\;  
\rT_{u_{02}|_{t=0}} L_{02} \times  (\rT_{u_{02}|_{t=0}} L_{02})^\perp  . 
$$
Now another consequence of immersed composition \cite[(18)]{wehrheim_woodward_geometric_composition} is the direct sum
\begin{align*}
\rT_{\hat u} \hat M
&\;=\;  
C_{\hat u}  \;\oplus\; \rT_{\hat u} \ti L_{02}
 \;\oplus\; (\rT_{u_{02}|_{t=0}}  L_{02})^\perp \times \{0\} 
 \;\oplus\; \{0\}  \times \rT_{\hat u_{11}}\Delta_1 
\\
\text{with }\qquad 
\rT_{\hat u}  L_{0211} & \;=\;  C_{\hat u}  \;\oplus\; \rT_{\hat u} \ti L_{02} , \\
\rT_{\hat u} (M_{02}\times \Delta_1) & \;=\;  \rT_{\hat u} \ti L_{02}   \;\oplus\; (\rT_{u_{02}|_{t=0}}  L_{02})^\perp \times \{0\} 
 \;\oplus\;   \{0\}  \times \rT_{\hat u_{11}}\Delta_1   . 
\end{align*}
This yields an explicit projection $\pi_{02}^\perp \times \pi_{11} : \rT_{\hat u} \hat M \to  (\rT_{u_{02}|_{t=0}}  L_{02})^\perp \times \rT_{\hat u_{11}}\Delta_1$ to a complement of $\rT_{\hat u} (M_{02}\times \Delta_1) \cap \rT_{\hat u}  L_{0211} = \rT_{\hat u} \ti L_{02}\subset \rT_{\hat u} (M_{02}\times \Delta_1)$. 
Here we fix a global projection to the diagonal such as
$\pi_{11}: \rT M_{11}\to (\rT\Delta_1)^\perp , (\xi_1,\xi'_1) \mapsto \frac 12 (\xi_1+\xi'_1 , \xi_1+\xi'_1 )$ with $M_{11}:=M_1^-\times M_1$. 
\end{rmk}

Based on Remark~\ref{rmk:auxiso} we will show Fredholm properties for $\cF_0$ in Lemma~\ref{lem:fred 0} -- after completion with respect to the norms
\begin{align*}
\|(\xi_{02}, \hat\xi) \|^\Gamma_{0} 
&\,:=\;  \| \xi_{02}\|_{H^{3/2}} 
+ \| \hat\xi\|_{L^1([0,1],H^1(\bR))} +  \| \nabla_t\hat\xi\|_{L^1([0,1],H^1(\bR))}  +  \| \nabla_t^2\hat\xi\|_{L^1([0,1],L^2(\bR))}  ,
\\
\|(\eta_{02}, \hat\eta) \|^\Omega_{0}  
&\,:=\;  \| \eta_{02}\|_{H^{1/2}} 
+ \| \hat\eta\|_{L^1([0,1],H^1(\bR))} +  \| \nabla_t\hat\eta\|_{L^1([0,1],L^2(\bR))}   .
\end{align*}
These norms are chosen to complete $\Gamma$ and $\Omega$ to the Banach spaces
\begin{align*}
\ol\Gamma_{0} 
&\;\subset\;
H^{3/2}(\bR\times[0,1], \rT_{u_{02}} M_{02}) \; \times \; 
\bigl( W^{1,1}([0,1],H^1(\bR, \rT_{\hat u} \hat M))\cap W^{2,1}([0,1],L^2(\bR, \rT_{\hat u} \hat M)) \bigr)  , \\
\ol\Omega_{0} 
&\;=\;
H^{1/2}(\bR\times[0,1], \rT_{u_{02}} M_{02}) \; \times \; 
\bigl( L^1([0,1],H^1(\bR, \rT_{\hat u} \hat M))\cap W^{1,1}([0,1],L^2(\bR, \rT_{\hat u} \hat M)) \bigr)  ,
\end{align*}
where $\ol\Gamma_{0}$ is the closed subspace specified by the Lagrangian boundary conditions and slice condition.
Now we need to understand the kernel and cokernel of the linearization $\rD\cF_0(0)$ at $\cF_0(0)=0$ 
to prove its Fredholm property and -- towards completing the data of an adiabatic Fredholm family -- construct an appropriate projection $\pi_\fK:\Gamma\to\fK:=\ker \rD\cF_0(0)$ and representation $\coker\rD\cF_0(0) \simeq \fC \subset \Omega$ by a subspace of smooth sections. 

\begin{lemma} \label{lem:fred 0}
The continuous extension of $\ol{\rD\cF_0(0)}:\ol\Gamma_0\to\ol\Omega_0$ to the completions of 
$(\Gamma,\|\cdot\|^\Gamma_0)$ and $(\Omega,\|\cdot\|^\Omega_0)$ is a Fredholm operator with 
kernel 
\begin{align*}
\ker \ol{\rD\cF_0(0)}
&\;=\;\ker \rD\cF_0(0)  \;\simeq\; \ker \rD_{u_{02}}^{\rm Lag}
\\
&\;=\; \bigl\{ (\fk_{02}, \hat\fk) \,\big|\, \fk_{02} \in \ker \rD_{u_{02}}^{\rm Lag}, \hat\fk= {\rm Pr}_{\rT L_{02}}^{-1} \fk_{02}|_{t=0} \bigr\}  \;=:\,  \fK \;\subset\;\Gamma 
\end{align*}
and cokernel $\coker \ol{\rD\cF_0(0)} \simeq \coker \ker \rD_{u_{02}}^{\rm Lag} \simeq (\im \rD_{u_{02}}^{\rm Lag})^{\perp_{L^2}}$ 
represented by
$$
\fC \,:=\; \bigl\{ (\fc_{02}, 0) \,\big|\, \fc_{02} \in (\im \rD_{u_{02}}^{\rm Lag})^{\perp_{L^2}} \bigr\}  \;\subset\;\Omega  . 
$$
Here $\rD_{u_{02}}^{\rm Lag}$ is the linearized operator of the original $\eps=0$ PDE -- the Cauchy-Riemann equation for strips in $M_{02}$ with Lagrangian boundary conditions, 
\begin{align*}
& \rD_{u_{02}}^{\rm Lag}\,:\; \bigl\{ \xi_{02} \in H^1(\bR\times[0,1],\rT_{u_{02}} M_{02}) \,\big|\,   
 \xi_{02}|_{t=0} \in \rT_{u_{02}|_{t=0}}L_{02}
 ,  \xi_{02}|_{t=1} \in \rT_{u_{02}|_{t=1}}(L_0\times L_2) 
\bigr\} 
\\
&\qquad\qquad\qquad\qquad\qquad\qquad\qquad\qquad\; 
\;\to\; \qquad L^2(\bR\times[0,1],\rT_{u_{02}} M_{02}) 
 \\
& 
\qquad\qquad\qquad\qquad\qquad\qquad\quad
\xi_{02}\qquad \;\mapsto\;  \nabla_t \xi_{02} - J_{02}(u_{02}) \nabla_s \xi_{02} - \nabla_{\xi_{02}} J_{02} \partial_s u_{02} . 
\end{align*}
In particular, this identifies the Fredholm index of $\ol{\rD\cF_0(0)}$ with the index of $\rD_{u_{02}}^{\rm Lag}$.
Moreover, for any projection $\pi_{\fK_{02}}: H^1(\ldots) \to\fK_{02}:= \ker \rD_{u_{02}}^{\rm Lag}$ and choice of norm on $\fK_{02}$ we obtain a stabilized Fredholm estimate with some constant $C_0$
for all $(\xi_{02}, \hat\xi)\in\ol\Gamma_0, (\fc_{02},0)\in\fC$,  
$$
\| (\xi_{02}, \hat\xi) \|^\Gamma_0 + \| (\fc_{02},0) \|^\Omega_0 \leq C_0 \bigl( \|\pi_{\fK_{02}}(\xi_{02}) \|^{\fK_{02}} + \| \ol{\rD\cF_0(0)} (\xi_{02}, \hat\xi) - (\fc_{02},0) \|^\Omega_0  \bigr). 
$$
\end{lemma}

The proof is deferred to the end of this section. 
Comparing with the abstract formulation in Definition~\ref{def:fredholm}, the corresponding projection to the kernel
$\pi_\fK:\Gamma\to\fK$ can then be obtained by $\pi_\fK(\xi_{02}, \hat\xi):=\bigl( \fk_{02}=\pi_{\fK_{02}}(\xi_{02}), {\rm Pr}_{\rT L_{02}}^{-1} \fk_{02}|_{t=0} \bigr)$. These particular constructions of projection to the kernel and cokernel representation only interacting with the first component in $\Gamma$ are crucial to obtain the $\eps$-dependent properties of an adiabatic Fredholm family in the following announcement of a theorem -- whose full proof is in the process of being written up in the general context of multi-strip-shrinking in any quilt.

\begin{theorem} \label{thm:QF}
Given a choice of projection $\pi_\fK:\Gamma\to\fK$ and cokernel representation $\fC\subset\Omega$ 
as in Lemma~\ref{lem:fred 0}, there exist constants $C_0, C_1, C_\fC \in (0,\infty)$ and continuous functions
$c : [0,\infty) \to [0,\infty)$, $c_\Delta : \Delta \to [0,\infty)$
that supplement the above data\footnote{
As defined, the norm bounds $\|\cdot\|^\bullet_0\leq \|\cdot\|^\bullet_\eps$ will hold up to a uniform constant. Alternatively, we can meet these bounds precisely by multipling the norms $\|\cdot\|^\Gamma_0$ and $\|\cdot\|^\Omega_0$ with fixed constants.} 
$\bigl( (\cF_\eps:\cV_\Gamma\to \Omega )_{\eps\in[0,1]} ,\|\cdot\|^\Gamma_\eps, \|\cdot\|^\Omega_\eps \bigr)$ 
to form a regularizing $\cC^1$-regular adiabatic Fredholm family as in Definition~\ref{def:fredholm}, \ref{def:adiabatic C-l}. 
\end{theorem}

\begin{proof}[Sketch of Proof of Theorem~\ref{thm:QF}]
We begin by going through the conditions Definition~\ref{def:fredholm}.

\begin{itemlist}
\item[\bf{Openness of Domain:}]  
$\cV_\Gamma \,:=\;  \bigl\{ (\xi_{02}, \hat\xi ) \in \Gamma \,\big|\,  \| \xi_{02} \|_{\cC^0} , \| \hat\xi \|_{\cC^0} < \text{injectivity radius} \bigr\}$ is open in the relative topology of $\Gamma\subset \ol\Gamma_{0}$ 
by the Sobolev embedding $H^{3/2}\subset\cC^0$ on domains of dimension $2$ and the Sobolev embeddings $W^{1,1}\subset\cC^0$ and $H^1\subset\cC^0$ on domains of dimension $1$. 

\item[\bf{Lower Bound on $\Gamma$ Norms: }]  
$\|(\xi_{02}, \hat\xi )\|^\Gamma_0 \leq \|(\xi_{02}, \hat\xi ) \|^\Gamma_\eps$ 
holds up to a constant for all $(\xi_{02}, \hat\xi )\in\Gamma$ and $\eps\in [0,1]$ by the Sobolev embedding
$H^2\subset H^{3/2}$ on domains of dimension $2$ in the first component, and in the second component by the characterization
$$
H^2(\bR\times[0,1]) \;=\; H^2([0,1],L^2(\bR))  \cap H^1([0,1],H^1(\bR))  \cap L^2([0,1],H^2(\bR)) 
$$
and the inclusions $H^2=W^{2,2}\subset W^{2,1}$ and $H^1=W^{1,2}\subset W^{1,1}$ on compact domains. However, to obtain an $\eps$-independent constant, note that for $0<\eps\leq 1$ we can use $\eps^{-1}\geq 1$ to obtain
\begin{align*}
\|(\xi_{02}, \hat\xi) \|^\Gamma_{\eps} 
&\geq  \| \xi_{02}\|_{H^2} 
+  \| \nabla_t\hat\xi\|_{L^2} +  \| \nabla_t\nabla_s\hat\xi\|_{L^2}  +  \| \nabla_s\nabla_t \hat\xi\|_{L^2} 
+  \| \nabla_t^2\hat\xi\|_{L^2}  
\\
&\geq \| \xi_{02}\|_{H^{3/2}} 
 +  \| \nabla_t\hat\xi\|_{L^2([0,1],H^1(\bR))}  +  \| \nabla_t^2\hat\xi\|_{L^2([0,1],L^2(\bR))} 
\\
&\geq \| \xi_{02}\|_{H^{3/2}} 
 +  \| \nabla_t\hat\xi\|_{L^1([0,1],H^1(\bR))}  +  \| \nabla_t^2\hat\xi\|_{L^1([0,1],L^2(\bR))} 
 \\
&= \|(\xi_{02}, \hat\xi) \|^\Gamma_{0}  - \| \hat\xi\|_{L^1([0,1],H^1(\bR))} , 
\end{align*}
so to obtain a uniform constant in $\|(\xi_{02}, \hat\xi )\|^\Gamma_0 \leq C \|(\xi_{02}, \hat\xi ) \|^\Gamma_\eps$ 
(which can then be used to rescale the $\eps=0$ norm to make the bound hold without a constant)
it remains to bound 
$\| \hat\xi\|_{L^1([0,1],H^1(\bR))}$ by  $\| \xi_{02}\|_{H^{3/2}} +  \| \nabla_t\hat\xi\|_{L^1([0,1],H^1(\bR))}$. 
This follows from the boundary condition $\hat\xi_{02}|_{t=0}=\xi_{02}|_{t=0}$ and the bounded inverse of the isomorphism in Remark~\ref{rmk:auxiso}, 
\begin{align*}
\| \hat\xi\|_{L^1([0,1],H^1(\bR))}
&\leq C \bigl( \|\nabla_t \hat \xi \|_{L^1([0,1],H^1(\bR))}   +  \| \pi_{02} \, \hat\xi_{02}|_{t=0} \|_{H^1(\bR)}  \bigr)
\\
&\leq C \bigl( \|\nabla_t \hat \xi \|_{L^1([0,1],H^1(\bR))}   +  \| \pi_{02} \, \xi_{02}|_{t=0} \|_{H^1(\bR)}  \bigr)
\\
&\leq C' \bigl( \|\nabla_t \hat \xi \|_{L^1([0,1],H^1(\bR))}   +  \| \xi_{02} \|_{H^{3/2}(\bR\times[0,1])}  \bigr). 
\end{align*}

\item[\bf{Lower Bound on $\Omega$ Norms:}]  
For $\eps\in (0,1]$ we can use $\eps^{-1}\geq 1$ to obtain
\begin{align*}
\|(\eta_{02}, \hat\eta) \|^\Omega_{\eps}  
&\geq  \| \eta_{02}\|_{H^1} 
+  \| \hat\eta\|_{L^2} +  \| \nabla_s\hat\eta\|_{L^2}  
+  \| \nabla_t\hat\eta\|_{L^2}  
\\
&\geq  \| \eta_{02}\|_{H^1}  +   \tfrac 12 \| \hat\eta\|_{L^2([0,1],H^1(\bR))} +  \tfrac 12 \| \hat\eta\|_{H^1([0,1],L^2(\bR))} 
\\
&\geq  \| \eta_{02}\|_{H^{1/2}} 
+ \tfrac 12  \| \hat\eta\|_{L^1([0,1],H^1(\bR))} +  \tfrac 12 \| \nabla_t\hat\eta\|_{L^1([0,1],L^2(\bR))} 
\;\geq\; \tfrac 12 \|(\eta_{02}, \hat\eta) \|^\Omega_0  . 
\end{align*}
So $\|(\eta_{02}, \hat\eta) \|^\Omega_0 \leq \|(\eta_{02}, \hat\eta)\|^\Omega_\eps$
holds up to a uniform constant which can then be used to rescale the $\eps=0$ norm to meet the exact requirement. 

\item[\bf{Fibrewise $\mathcal{C}^1$ Regularity:}]
The differentials of $\cF_\eps : (\cV_\Gamma,\|\cdot\|^\Gamma_\eps) \to (\Omega,\|\cdot\|^\Omega_\eps)$ are uniformly $\cC^0$ by the Uniform Continuity of $\rD\cF_\eps$ proven below -- so in fact this uniform continuity in $\Gamma$ is uniform in $\eps\in[0,1]$. 
So it remains to check that $\cF_\eps : (\cV_\Gamma,\|\cdot\|^\Gamma_\eps) \to (\Omega,\|\cdot\|^\Omega_\eps)$ is uniformly $\cC^0$ for a fixed $\eps\in[0,1]$ -- which follows (even uniformly in $\eps$) from the Uniform Bounds on $\rD\cF_\eps$ proven below. 

\item[\bf{Fredholm:}]
The continuous extensions of the linearization $\overline{\rD\cF_\eps(0)}: \overline\Gamma_\eps\to \overline\Omega_\eps$ are Fredholm operators for $\eps=0$ by Lemma~\ref{lem:fred 0} and for $\eps>0$ by \cite{wehrheim_woodward_geometric_composition}

\item[\bf{Index:}]
\cite[Lemma 2.1.3]{wehrheim_woodward_geometric_composition} identifies the Fredholm index of $\overline{\rD\cF_1(0)}$ with that of $\rD_{u_{02}}^{\rm Lag}$, which Lemma~\ref{lem:fred 0} identifies with the index of $\overline{\rD\cF_0(0)}$. Moreover, the Fredholm indices for $0<\eps\leq1$ are constant since this is a continuous family of Fredholm operators between fixed Banach spaces.

Furthermore, Lemma~\ref{lem:fred 0} shows that the kernel $\fK=\ker \overline{\rD\cF_0(0)}$ is contained in the dense subset of smooth sections $\Gamma$, and that the cokernel $\coker\overline{\rD\cF_0(0)}\simeq \fC$ can be represented by a subset $\fC\subset\Omega$ of the smooth sections.

\item[\bf{$\mathbf{\epsilon=0}$ Fredholm Estimate:}] This is proven in Lemma~\ref{lem:fred 0}. 

\item[\bf{Uniform Fredholm-ish Estimate:}]  
\cite[Lemma 3.2.1]{wehrheim_woodward_geometric_composition} shows that for all $(\xi_{02}, \hat\xi)\in\Gamma$ and $\eps\in(0,1]$ we have
$$
c_1\| (\xi_{02}, \hat\xi) \|^\Gamma_\eps 
\leq  \| \rD\cF_\eps(0) (\xi_{02}, \hat\xi) \|^\Omega_\eps   
+  \| \xi_{02}\|_{L^2} + \eps^{1/2}  \| \hat\xi\|_{L^2} 
+ \| \hat\xi|_{t=1}\|_{H^1} + \| \xi_{02}|_{t=1}\|_{H^1} . 
$$
Then we bound $\| \xi_{02}\|_{L^2}+ \| \xi_{02}|_{t=1}\|_{H^1} \leq \| \xi_{02}\|_{H^{3/2}} \leq  \| (\xi_{02}, \hat\xi) \|^\Gamma_0$ (up to a uniform constant) by the Sobolev trace theorem. For the second component the estimates for $W^{1,1}([0,1])\subset\cC^0([0,1])\subset L^2([0,1])$ give (up to a uniform constant)
\begin{align*}
\eps^{1/2}  \| \hat\xi\|_{L^2} 
+ \| \hat\xi|_{t=1}\|_{H^1} 
&\leq 
  \| \hat\xi\|_{W^{1,1}([0,1],L^2(\bR))}
+ \| \hat\xi\|_{\cC^0([0,1],H^1(\bR))} 
\\
&\leq
  \| \hat\xi\|_{W^{1,1}([0,1],H^1(\bR))} \\
&= \| \hat\xi\|_{L^1([0,1],H^1(\bR))} +  \| \nabla_t\hat\xi\|_{L^1([0,1],H^1(\bR))} \;\leq\; \| (\xi_{02}, \hat\xi) \|^\Gamma_0 . 
\end{align*}
So this combines to a uniform constant $C_1$ in 
$\| (\xi_{02}, \hat\xi) \|^\Gamma_\eps \leq C_1 \bigl( \| \rD\cF_\eps(0) (\xi_{02}, \hat\xi) \|^\Omega_\eps + \| (\xi_{02}, \hat\xi) \|^\Gamma_0 \bigr)$.

\item[\bf{Uniform Cokernel Bound:}] 
With the specific choice of cokernel representation from Lemma~\ref{lem:fred 0} we have for $\eps\in (0,1]$
and for all $\fc=(\fc_{02},0)\in\fC$
$$ 
\| (\fc_{02},0)  \|^\Omega_\eps  \;=\;  \| \fc_{02}\|_{H^1}    \;\leq\; C_\fC \| \fc_{02}\|_{H^{1/2}}  
\;=\; C_\fC  \| (\fc_{02},0)  \|^\Omega_0 
$$ 
with a uniform constant $C_\fC$ since all norms on a finite dimensional space such as $\fC_{02}$ are equivalent. 

\item[\bf{Quadratic-ish Estimate:}] 
\cite[Lemma 3.1.5]{wehrheim_woodward_geometric_composition} proves
$$
\bigl\|  \rD\cF_\eps (\gamma_{02}, \hat\gamma)(\xi_{02}, \hat\xi) -   \rD\cF_\eps (0,0)(\xi_{02}, \hat\xi)  \bigr\|^\Omega_\eps \leq c(\|(\gamma_{02}, \hat\gamma)\|^\Gamma_\eps)  \| (\xi_{02}, \hat\xi) \|^\Gamma_\eps
$$ 
for all $(\gamma_{02}, \hat\gamma)\in\cV_\Gamma$, $(\xi_{02}, \hat\xi)\in\Gamma$, and $\eps\in[0,1]$ with a linear function $c(x)=C_2 x$. 

\item[\bf{Continuity of Derivatives at $\mathbf 0$:}] 
For $\xi=(\xi_{02}, \hat\xi)\in\Gamma$ and $0<\eps\leq 1$ we have -- since the first component is independent of $\eps$ -- 
\begin{align*}
&\bigl\|  \rD\cF_\eps (0)\xi -   \rD\cF_0 (0)\xi  \bigr\|^\Omega_0 \\
&=
\bigl\|  \bigl(\; 0 \,,\, 
\nabla_t \hat\xi - \eps \, \hat J(\hat u) \nabla_s \hat \xi -  \eps \, \nabla_{\hat \xi} \hat J \, \partial_s \hat u
\;-\;
\nabla_t \hat\xi 
\;\bigr) \bigr\|^\Omega_0  \\
&=
\eps \, \bigl(  \bigl\| 
\hat J(\hat u) \nabla_s \hat \xi -  \eps \, \nabla_{\hat \xi} \hat J \, \partial_s \hat u
\bigr\|_{L^1([0,1],H^1(\bR))} 
\;+\;  \bigl\| \nabla_t \bigl( 
\hat J(\hat u) \nabla_s \hat \xi -  \eps \, \nabla_{\hat \xi} \hat J \, \partial_s \hat u \bigr)
\bigr\|_{L^1([0,1],L^2(\bR))}  \bigr)
\\
&\leq
\eps \, C_u \bigl(  \|  \hat \xi \|_{L^2} + \| \nabla_s \hat \xi \|_{L^2}  + \|  \nabla_s^2 \hat \xi \|_{L^2}  + \|  \nabla_t \hat \xi \|_{L^2}  + \|  \nabla_t \nabla_s \hat \xi \|_{L^2}    \bigr)
\\
&\leq
\eps \, C_u \eps^{-1/2} \bigl( \eps^{1/2} \|  \hat \xi \|_{L^2} + \eps^{1/2}\| \nabla_s \hat \xi \|_{L^2}  + \eps^{1/2}\|  \nabla_s^2 \hat \xi \|_{L^2}  + \eps^{-1/2}\|  \nabla_t \hat \xi \|_{L^2}  + \eps^{-1/2}\|  \nabla_t \nabla_s \hat \xi \|_{L^2}    \bigr)
\\
&\leq c_\Delta(\eps) \|(\xi_{02}, \hat\xi) \|^\Gamma_\eps 
\end{align*} 
with $c_\Delta(x)= x^{1/2} \, C_u$ with a constant $C_u$ that is determined by $\hat J$ and $\hat u$. 

\item[\bf{Near-Solution:}] 
The family was constructed near a solution $(u_{02},\hat u)$ of the $\eps=0$ problem, which corresponds to $\cF_0(0)=0$, and \cite[Lemma 3.1.5]{wehrheim_woodward_geometric_composition} proves
$\| \cF_\eps(0) \|^\Omega_\eps \leq C_1 \eps^{1/4} \to 0$ as $\eps\to 0$.
\end{itemlist}

Then it remains to spell out and verify the conditions of Definition~\ref{def:adiabatic C-l}. 

\begin{itemlist}
\item[\bf{Regularizing:}] 
The meaning of this property is, first, that solutions of a nonlinear equation with smooth right hand side are smooth
$$
(\xi_{02}, \hat\xi)\in \cV_{\overline\Gamma,\eps}, \quad \overline\cF_\eps(\xi_{02}, \hat\xi) \in \cC^\infty \quad \Longrightarrow \quad (\xi_{02}, \hat\xi)\in \cC^\infty .
$$
Then the second requirement is that solutions of a linearized equation -- at a smooth base point -- with smooth right hand side are smooth
$$
(\xi_{02}, \hat\xi)\in\cC^\infty, \quad (\xi'_{02}, \hat\xi') \in\overline\Gamma_\eps, \quad \rD\overline\cF_\eps(\xi_{02}, \hat\xi) (\xi'_{02}, \hat\xi') \in \cC^\infty \quad \Longrightarrow \quad (\xi'_{02}, \hat\xi')\in\cC^\infty .
$$
For $\eps>0$ both statements follow from standard elliptic regularity. 
For $\eps=0$ the elliptic regularity for the Cauchy-Riemann operators on $M_{02}$ needs to be combined with the regularity that follows from the isomorphisms of Remark~\ref{rmk:auxiso} for arbitrarily large $k\in\bN_0$ and $s \geq 0$:  
Suppose we have already established $\xi_{02}\in H^s$ for some $s\geq 3/2$. 
Then $\hat\xi \in W^{1,1}([0,1],H^1(\bR, \rT_{\hat u} \hat M))$ satisfies $\nabla_t \hat \xi =\hat\eta \in\cC^\infty$ and the boundary conditions
$\hat \xi_{11} |_{t=0} \in \rT_{\hat u_{11}}\Delta_1$, $\hat \xi |_{t=1} \in\rT_{\hat u} L_{0211}$, 
and $\pi_{02} \, \hat\xi_{02}|_{t=0} = \pi_{02} \, \xi_{02}|_{t=0}\in H^{s-1/2}(\bR)$, and thus 
\begin{align*}
\hat\xi(\cdot,t) = 
{\rm Pr}_{\rT L_{02}}^{-1}\pi_{02} \, \xi_{02}|_{t=0}  
\;+\; \textstyle\int_0^t \hat \eta   \;-\; (\pi_{02}^\perp \times \pi_{11}) \bigl(\textstyle\int_0^1 \hat\eta\bigr)
\in \cC^\infty([0,1], H^{s-1/2}(\bR)). 
\end{align*} 

\item[\bf{Pointwise Continuity in $\mathbf{\Delta}$ at Solutions Modulo $\mathbf{\mathfrak C}$:}]
Given any $\eps_0\in[0,1]$ and a solution $(\xi_{02}, \hat\xi)\in\cV_\Gamma$ of $\cF_{\eps_0}(\xi_{02}, \hat\xi)=(\fc_{02},0)\in \fC$, we have 
\begin{align*}
\bigl\|  \cF_\eps (\xi_{02}, \hat\xi) -   \cF_{\eps_0} (\xi_{02}, \hat\xi)  \bigr\|^\Omega_\eps 
&= \bigl\| \bigl(\; 0 \,,\, 
 \hat\Phi(\hat\xi)^{-1} ( \partial_t - \eps \hat J \partial_s ) \, \hat e(\hat\xi)  
 - \hat\Phi(\hat\xi)^{-1}  ( \partial_t - \eps_0 \hat J \partial_s ) \, \hat e(\hat\xi)
\;\bigr)  \bigr\|^\Omega_\eps 
\\
&=  |\eps - \eps_0| \bigl\|  \bigl(\; 0 \,,\,  \hat\Phi(\hat\xi)^{-1} ( \hat J \partial_s ) \, \hat e(\hat\xi) \;\bigr) \bigr\|^\Omega_\eps .
\end{align*}
To show that this converges to $0$ as $\eps\to \eps_0$ we need to show that the $\eps$-norm of $\hat\eta:= \hat\Phi(\hat\xi)^{-1} ( \hat J \partial_s ) \, \hat e(\hat\xi)$ doesn't blow up faster than $|\eps-\eps_0|^{-1}$. 
For $\eps >0$ this norm is  
\begin{align*}
& \eps^{-1/2} \bigl( \| \hat\eta\|_{L^2} +  \| \nabla_s\hat\eta\|_{L^2}   \bigr) 
+ \eps^{-3/2} \| \nabla_t\hat\eta\|_{L^2}  
+ \eps^{-3/4}  \| \hat\eta\|_{L^4} 
\\
&
\qquad\qquad\qquad\qquad\qquad\qquad=
\eps^{-1/2} \bigl( \| \partial_s \hat e(\hat\xi) \|_{L^2} + \bigl\| \nabla_s \bigl( \hat\Phi(\hat\xi)^{-1} ( \hat J \partial_s ) \, \hat e(\hat\xi) \bigr) \bigr\|_{L^2}   \bigr)  \\
&\qquad\qquad\qquad\qquad\qquad\qquad\qquad
+ \eps^{-3/2} \bigl\| \nabla_t \bigl(\hat\Phi(\hat\xi)^{-1} ( \hat J \partial_s ) \, \hat e(\hat\xi) \bigr) \bigr\|_{L^2}  
+ \eps^{-3/4}  \| \partial_s \hat e(\hat\xi)  \|_{L^4} . 
\end{align*}
This is bounded in case $\eps\to\eps_0 >0$. 
However, for $\eps\to\eps_0=0$ we obtain 
\begin{align*}
\bigl\|  \cF_\eps (\xi_{02}, \hat\xi) -   \cF_0(\xi_{02}, \hat\xi)  \bigr\|^\Omega_\eps 
&\;=\; 
\eps^{1/2} \bigl( \| \partial_s \hat e(\hat\xi) \|_{L^2} +  \| \nabla_s \bigl( \hat\Phi(\hat\xi)^{-1} ( \hat J \partial_s ) \, \hat e(\hat\xi) \bigr) \|_{L^2}   \bigr) \\
&
\qquad
+ \eps^{-1/2} \| \nabla_t \bigl(\hat\Phi(\hat\xi)^{-1} ( \hat J \partial_s ) \, \hat e(\hat\xi) \bigr) \|_{L^2}  
+ \eps^{1/4}  \| \partial_s \hat e(\hat\xi)  \|_{L^4}, 
\end{align*}
which will converge to $0$ only if $\nabla_t \bigl(\hat\Phi(\hat\xi)^{-1} ( \hat J \partial_s ) \, \hat e(\hat\xi)= 0$, that is if we can ensure that all components are independent of $t\in[0,1]$. This is true for $\hat u$ since it arises from an $\eps_0=0$ solution. It holds for parallel transport $\Phi$, exponential map $\hat e$, and almost complex structure $\hat J$ by construction.\footnote{This shows that we need to set up the moduli spaces and their description with a $t$-independent almost complex structure on the shrinking strip, and need to work with $t$-independent connections and exponential map on this strip.}
For the section $\hat\xi$ this is where we crucially use the restriction of the pointwise convergence requirement to ``solutions modulo $\fC$'' and the fact that we represented the cokernel in the form $(\fc_{02},0)$. So the second component of the equation $\cF_0(\xi_{02}, \hat\xi)=(\fc_{02},0)\in \fC$ becomes $0=\hat\Phi(\hat\xi)^{-1} ( \partial_t \hat e(\hat\xi) ) =\nabla_t\xi$, 
which guarantees the desired pointwise convergence.

\item[\bf{Pointwise $\cC^1$-Continuity in $\mathbf{\Delta}$ at Solutions Modulo $\mathbf{\mathfrak C}$:}] 
By the previous item and Remark~\ref{rmk:pointwise C-1} it remains to consider the differentials at 
solutions $(\gamma_{02}, \hat\gamma)\in\cV_\Gamma$ of $\cF_{\eps_0}(\gamma_{02}, \hat\gamma)\in\fC$, in particular in case $\eps_0=0$ this ensures that the section $\hat\gamma$ is independent of $t\in[0,1]$. 
Moreover, the only dependence on $\eps$ in the differentials comes from linearizing 
$\hat\xi \mapsto \eps  \hat\Phi(\hat\xi)^{-1}  ( \hat J \partial_s ) \, \hat e(\hat\xi) =: \eps \,\cH(\hat\xi)$, where we denote $\cH(\hat\xi):=  \hat\Phi(\hat\xi)^{-1}  ( \hat J \partial_s ) \, \hat e(\hat\xi)$. 
Now given any $\eps_0\in[0,1]$ and a solution
$(\xi_{02}, \hat\xi)\in\Gamma$ of $\rD\cF_{\eps_0}(\gamma_{02}, \hat\gamma) (\xi_{02}, \hat\xi)\in \fC$ we have
\begin{align*}
& \bigl\|  \rD\cF_\eps (\gamma_{02}, \hat\gamma)(\xi_{02}, \hat\xi)  -   \rD\cF_{\eps_0} (\gamma_{02}, \hat\gamma)(\xi_{02}, \hat\xi) \bigr\|^\Omega_\eps 
\\
&
=
\bigl\| \bigl( \; 0 \; , \; (\eps-\eps_0) \rD\cH(\hat\gamma) \hat\xi  \; \bigr) \bigr\|^\Omega_\eps 
\\
&
=
| \eps-\eps_0|\bigl( 
\eps^{-1/2} \bigl( \| \rD\cH(\hat\gamma) \hat\xi\|_{L^2} +  \bigl\| \nabla_s\bigl( \rD\cH(\hat\gamma) \hat\xi \bigr\|_{L^2}   \bigr) 
+ \eps^{-3/2} \bigl\|\nabla_t\bigl(\rD\cH(\hat\gamma) \hat\xi \bigr)\|_{L^2}  
+ \eps^{-3/4}  \|\rD\cH(\hat\gamma) \hat\xi\|_{L^4} 
\bigr)
\end{align*}
which as before converges to $0$ when $\eps\to \eps_0>0$. For $\eps\to \eps_0=0$ the convergence requires 
$\nabla_t\bigl(\rD\cH(\hat\gamma) \hat\xi =0$, which is guaranteed by working at a solution modulo $\fC$. 
Indeed, 
$\rD\cF_0(\gamma_{02}, \hat\gamma) (\xi_{02}, \hat\xi) = (\fc_{02},0)$ implies 
$\rD\cK(\hat\gamma) \hat\xi =0$ for the linearization of $\cK(\hat\xi):= \hat\Phi(\hat\xi)^{-1}  \partial_t \, \hat e(\hat\xi)$, which amounts to the equation $\nabla_t\hat\xi=0$. 

\item[\bf{Uniform Continuity of $\rD\cF_\eps$:}]  
Generalizing the computations for \cite[Lemma 3.1.5]{wehrheim_woodward_geometric_composition} to a second nonzero base point $(\gamma^\fk_{02}, \hat\gamma^\fk)\ne(0,0)$ proves
$$
\bigl\|  \rD\cF_\eps (\gamma^\fl_{02}, \hat\gamma^\fl)(\xi_{02}, \hat\xi) -   \rD\cF_\eps (\gamma^\fk_{02}, \hat\gamma^\fk)(\xi_{02}, \hat\xi)  \bigr\|^\Omega_\eps \leq c(\|(\gamma^\fl_{02}, \hat\gamma^\fl)-(\gamma^\fk_{02}, \hat\gamma^\fk)\|^\Gamma_\eps)  \| (\xi_{02}, \hat\xi) \|^\Gamma_\eps
$$ 
for all $(\gamma^\fl_{02}, \hat\gamma^\fl), (\gamma^\fk_{02}, \hat\gamma^\fk)\in\cV_\Gamma$, $(\xi_{02}, \hat\xi)\in\Gamma$, and $\eps\in[0,1]$ with a linear function $c^1_\cF(x)=C_2 x$. 

\item[\bf{Uniform Bound on $\rD\cF_\eps(0)$:}] 
For $0<\eps\leq 1$  we have
\begin{align*}
& \bigl\| \rD \cF_\eps (0)  \bigr\|^{\cL(\ol\Gamma_\eps,\ol\Omega_\eps )}  
\\
&=
\sup_{\|\xi\|^\Gamma_\eps\leq 1}
\bigl\| \bigl( \; 
\nabla_t \xi_{02} -J_{02}(u_{02}) \nabla_s\xi_{02} -  \nabla_{\xi_{02}} J_{02} \, \partial_s u_{02}
\;,\;
\nabla_t \hat\xi - \eps \, \hat J(\hat u) \nabla_s \hat \xi -  \eps \, \nabla_{\hat \xi} \hat J \, \partial_s \hat u 
\; \bigr) \bigr\|^\Omega_\eps .
\end{align*}
This is bounded by the sum of the two components: From the first component we obtain
\begin{align*}
\sup_{\|\xi_{02}\|_{H^2}\leq 1}
\bigl\|
\nabla_t \xi_{02} -J_{02}(u_{02}) \nabla_s\xi_{02} -  \nabla_{\xi_{02}} J_{02} \, \partial_s u_{02}
\bigr\|_{H^1}
\end{align*}
which is independent of $\eps$. 
From the second component we obtain the supremum over
\begin{align*}
1\geq \|\hat\xi \|_\eps
&\,:=\;  
 \eps^{1/2} \bigl( \| \hat\xi\|_{L^2} +  \| \nabla_s\hat\xi\|_{L^2}  +  \| \nabla_s^2\hat\xi\|_{L^2}   \bigr) 
+ \eps^{1/4} \bigl( \| \hat\xi\|_{L^4} +  \| \nabla_s\hat\xi\|_{L^4}   \bigr) 
\\
&\qquad \qquad 
+ \eps^{-1/2} \bigl( \| \nabla_t\hat\xi\|_{L^2} +  \| \nabla_t\nabla_s\hat\xi\|_{L^2}  +  \| \nabla_s\nabla_t \hat\xi\|_{L^2}   \bigr)
+ \eps^{-3/2}  \| \nabla_t^2\hat\xi\|_{L^2}  + \eps^{-3/4}  \| \nabla_t\hat\xi\|_{L^4} 
\end{align*}
of the terms arising from the second component of the $\|\cdot\|^\Omega_\eps$ norm, 
\begin{align*}
&\sup_{\|\hat\xi\|_\eps\leq 1} \Bigl( 
 \eps^{-1/2} \bigl\| \nabla_t \hat\xi - \eps \, \hat J(\hat u) \nabla_s \hat \xi -  \eps \, \nabla_{\hat \xi} \hat J \, \partial_s \hat u \bigr\|_{L^2}
 +
  \eps^{-1/2} \bigl\| \nabla_s\bigl( \nabla_t \hat\xi - \eps \, \hat J(\hat u) \nabla_s \hat \xi -  \eps \, \nabla_{\hat \xi} \hat J \, \partial_s \hat u \bigr) \bigr\|_{L^2} \\
&\qquad\quad
 +
  \eps^{-3/2} \bigl\| \nabla_t\bigl( \nabla_t \hat\xi - \eps \, \hat J(\hat u) \nabla_s \hat \xi -  \eps \, \nabla_{\hat \xi} \hat J \, \partial_s \hat u \bigr) \bigr\|_{L^2} 
 +
\eps^{-3/4} \bigl\| \nabla_t \hat\xi - \eps \, \hat J(\hat u) \nabla_s \hat \xi -  \eps \, \nabla_{\hat \xi} \hat J \, \partial_s \hat u \bigr\|_{L^4} \Bigr) \\
&\leq
\sup_{\|\hat\xi\|_\eps\leq 1} \Bigl( 
 \eps^{-1/2} \bigl( \bigl\| \nabla_t \hat\xi \bigr\|_{L^2} + \eps \, \bigl\| \nabla_s \hat \xi \bigr\|_{L^2} +  \eps \, \bigl\| \nabla \hat J \bigr\|_{L^\infty} \| \partial_s \hat u \|_{L^\infty}  \bigl\| \hat \xi  \bigr\|_{L^2} \bigr) \\
&\qquad\qquad
 +
  \eps^{-1/2} \bigl( \bigl\| \nabla_s\nabla_t \hat\xi \bigr\|_{L^2}  
+ \eps \bigl\| \nabla_s^2\hat \xi \bigr\|_{L^2}
+ \eps \bigl\| \nabla_s\bigl( \hat J(\hat u) \bigr) \bigr\|_{L^\infty}\bigl\| \nabla_s \hat \xi   \bigr\|_{L^2} \\
&\qquad\qquad\qquad\qquad
+ \eps  \bigl\| \nabla \hat J \bigr\|_{L^\infty} \| \partial_s \hat u \|_{L^\infty}  \bigl\|  \nabla_s \hat \xi  \bigr\|_{L^2} 
+\eps  \bigl\| \nabla_s\bigl( \nabla \hat J \, \partial_s \hat u \bigr) \bigr\|_{L^\infty } \bigl\|  \hat\xi \bigr\|_{L^2}
  \bigr) 
  \\
&\qquad\qquad
 +
  \eps^{-3/2} \bigl( \bigl\| \nabla_t^2 \hat\xi \bigr\|_{L^2}  
+ \eps \bigl\| \nabla_t\nabla_s\hat \xi \bigr\|_{L^2}
+ \eps \bigl\| \nabla_t\bigl( \hat J(\hat u) \bigr) \bigr\|_{L^\infty}\bigl\| \nabla_s \hat \xi   \bigr\|_{L^2} \\
&\qquad\qquad\qquad\qquad
+ \eps  \bigl\| \nabla \hat J \bigr\|_{L^\infty} \| \partial_s \hat u \|_{L^\infty}  \bigl\|  \nabla_t \hat \xi  \bigr\|_{L^2} 
+\eps  \bigl\| \nabla_t\bigl( \nabla \hat J \, \partial_s \hat u \bigr) \bigr\|_{L^\infty } \bigl\|  \hat\xi \bigr\|_{L^2}
  \bigr)  \\
&\qquad\qquad 
 +
\eps^{-3/4} \bigl( \bigl\| \nabla_t \hat\xi \bigr\|_{L^4} + \eps \, \bigl\| \nabla_s \hat \xi \bigr\|_{L^4} +  \eps \, \bigl\| \nabla \hat J \bigr\|_{L^\infty} \| \partial_s \hat u \|_{L^\infty}  \bigl\| \hat \xi  \bigr\|_{L^4} \bigr) \Bigr) 
\end{align*}
\begin{align*}
&\leq
\sup_{\|\hat\xi\|_\eps\leq 1} \Bigl( 
 \eps^{-1/2} \, \bigl\| \nabla_t \hat\xi \bigr\|_{L^2} 
 +  \eps^{1/2} \, \bigl\| \nabla_s \hat \xi \bigr\|_{L^2} 
 +   \eps^{1/2} \, \bigl\| \nabla \hat J \bigr\|_{L^\infty} \| \partial_s \hat u \|_{L^\infty}  \bigl\| \hat \xi  \bigr\|_{L^2} \\
&\qquad\qquad
 +  \eps^{-1/2} \bigl\| \nabla_s\nabla_t \hat\xi \bigr\|_{L^2}  
+  \eps^{1/2}\bigl\| \nabla_s^2\hat \xi \bigr\|_{L^2}
+  \eps^{1/2} \bigl\| \nabla_s\bigl( \hat J(\hat u) \bigr) \bigr\|_{L^\infty}\bigl\| \nabla_s \hat \xi   \bigr\|_{L^2} \\
&\qquad\qquad\qquad\qquad
+  \eps^{1/2}  \bigl\| \nabla \hat J \bigr\|_{L^\infty} \| \partial_s \hat u \|_{L^\infty}  \bigl\|  \nabla_s \hat \xi  \bigr\|_{L^2} 
+ \eps^{1/2}  \bigl\| \nabla_s\bigl( \nabla \hat J \, \partial_s \hat u \bigr) \bigr\|_{L^\infty } \bigl\|  \hat\xi \bigr\|_{L^2}
  \\
&\qquad\qquad
 +
  \eps^{-3/2} \bigl\| \nabla_t^2 \hat\xi \bigr\|_{L^2}  
+ \eps^{-1/2}  \bigl\| \nabla_t\nabla_s\hat \xi \bigr\|_{L^2}
+  \eps^{-1/2}  \bigl\| \nabla_t\bigl( \hat J(\hat u) \bigr) \bigr\|_{L^\infty}\bigl\| \nabla_s \hat \xi   \bigr\|_{L^2} \\
&\qquad\qquad\qquad\qquad
+  \eps^{-1/2}   \bigl\| \nabla \hat J \bigr\|_{L^\infty} \| \partial_s \hat u \|_{L^\infty}  \bigl\|  \nabla_t \hat \xi  \bigr\|_{L^2} 
+ \eps^{-1/2}  \bigl\| \nabla_t\bigl( \nabla \hat J \, \partial_s \hat u \bigr) \bigr\|_{L^\infty } \bigl\|  \hat\xi \bigr\|_{L^2}
 \\
&\qquad\qquad 
 +
\eps^{-3/4} \, \bigl\| \nabla_t \hat\xi \bigr\|_{L^4} 
 + \eps^{1/4} \, \bigl\| \nabla_s \hat \xi \bigr\|_{L^4} 
 +  \eps^{1/4} \, \bigl\| \nabla \hat J \bigr\|_{L^\infty} \| \partial_s \hat u \|_{L^\infty}  \bigl\| \hat \xi  \bigr\|_{L^4}  \Bigr) \\
&\leq
\sup_{\|\hat\xi\|_\eps\leq 1} \|\hat\xi\|_\eps \Bigl( 
1 + \bigl\| \nabla \hat J \bigr\|_{L^\infty} \| \partial_s \hat u \|_{L^\infty} 
+ \bigl\| \nabla_s\bigl( \hat J(\hat u) \bigr) \bigr\|_{L^\infty}\bigl\| 
+  \eps^{-1} \bigl\| \nabla_t\bigl( \hat J(\hat u) \bigr) \bigr\|_{L^\infty} 
+   \eps^{-1} \bigl\| \nabla_t\bigl( \nabla \hat J \, \partial_s \hat u \bigr) \bigr\|_{L^\infty } \Bigr) 
\\
&\leq 
1 + \bigl\| \nabla \hat J \bigr\|_{L^\infty} \| \partial_s \hat u \|_{L^\infty} 
+ \bigl\| \nabla_s\bigl( \hat J(\hat u) \bigr) \bigr\|_{L^\infty}\bigl\| \bigr) =: C^1_\cF ,  
\end{align*}
where we again crucially use the fact that on the shrinking strip the almost complex structure $\hat J$ and base point $\hat u$ are constant in $t\in[0,1]$ so that we can cancel the terms with negative exponents of $\eps$ due to 
$\nabla_t\bigl( \hat J(\hat u) \bigr) =0$ and $ \nabla_t\bigl( \nabla \hat J \, \partial_s \hat u\bigr) = 0$. 
\end{itemlist}
This finishes the sketch of proof of Theorem~\ref{thm:QF} up to the following proof of its key lemma. 
\end{proof}

\begin{proof}[Proof of Lemma~\ref{lem:fred 0}]
The linearization of $\cF_0$ or $\ol\cF_0$ at $0$ with respect to any Sobolev norm is 
$$
\rD\cF_0(0) \,:\;  (\xi_{02}, \hat\xi) 
\mapsto 
\bigl( \;  \rD_{u_{02}}\xi_{02} \; \;,\; \nabla_t \hat\xi  \;\bigr)  , 
$$
where 
$\rD_{u_{02}}\xi_{02}:= \nabla_t \xi_{02} - J_{02}(u_{02}) \nabla_s \xi_{02} - \nabla_{\xi_{02}} J_{02} \partial_s u_{02}$ is a linearized Cauchy-Riemann operator at $u_{02}$,  
and the boundary conditions encoded in the domain $\Gamma$ are
\begin{equation}\label{eq:bc Gamma}
 \xi_{02}|_{t=1} \in \rT_{u_{02}|_{t=1}}(L_0\times L_2) , \quad
 \xi_{02}|_{t=0} =\hat \xi_{02} |_{t=0} ,\quad
 \hat \xi_{11}|_{t=0} \in \rT_{\hat u_{11}}\Delta_1, \quad
 \hat \xi |_{t=1} \in\rT_{\hat u} L_{0211} .
\end{equation}
To understand the kernel, note that for $\nabla_t\hat\xi=0$ the boundary conditions reduce to
\begin{equation}\label{eq:bc 0}
 \xi_{02}|_{t=1} \in \rT_{u_{02}|_{t=1}}(L_0\times L_2) , \qquad
 \xi_{02}|_{t=0} \in {\rm Pr}_{02}\bigl( \rT_{\hat u}(M_{02}^-\times\Delta_1) \cap \rT_{\hat u} L_{0211}\bigr) = 
 \rT_{u_{02}|_{t=0}}L_{02}  
\end{equation}
with the notation $\rT_{u_{02}(s,0)} L_{02}:= {\rm Pr}_{\rT L_{02}}(\rT_{\hat u(s)} \ti L_{02})$ from 
Remark~\ref{rmk:immersion}. 
Then the remaining boundary condition $\xi_{02}|_{t=0} =\hat \xi_{02} |_{t=0}$ determines $\hat \xi_{02} |_{t=0}\in  \rT_{u_{02}|_{t=0}}L_{02}$, which -- via the isomorphisms in Remark~\ref{rmk:auxiso} and \eqref{eq:immersed tangent spaces} uniquely determines $\hat\xi=\hat\xi(s)$ as $\hat \xi ={\rm Pr}_{\rT L_{02}}^{-1}(\xi_{02}|_{t=0})$. 

Thus the kernel of $\rD\cF_0(0)$ is identified with solutions of $\rD_{u_{02}}\xi_{02}=0$ with boundary conditions \eqref{eq:bc 0} -- that is, by the kernel of the linearized operator for two strips, which is a Fredholm operator 
$$
\rD_{u_{02}}^{\rm Lag}\,:\; \bigl\{ \xi_{02} \in H^{s+1}(\bR\times[0,1],\rT_{u_{02}} M_{02}) \,\big|\,   
\eqref{eq:bc 0}
  \bigr\} \;\to\; H^s(\bR\times[0,1],\rT_{u_{02}} M_{02}) . 
$$
Here the choice of Sobolev regularity $s\geq 0$ is immaterial -- by elliptic regularity the kernel is always the same finite dimensional space of smooth sections.
Thus we have identified the kernel as
$$
\ker \ol{\rD\cF_0(0)} \;=\; \ker \rD\cF_0(0) \;=\; \bigl\{ (\fk_{02}, \hat\fk) \,\big|\, \fk_{02} \in \ker \rD_{u_{02}}^{\rm Lag}, \hat\fk= {\rm Pr}_{\rT L_{02}}^{-1} \circ \fk_{02}|_{t=0} \bigr\}  \;\subset\;\Gamma .
$$
Similarly, the cokernel of $\rD_{u_{02}}^{\rm Lag}$ is independent of $s$. More precisely, the $L^2$-orthogonal complement $\fC_{02}:=(\im \rD_{u_{02}}^{\rm Lag})^\perp\subset L^2(\bR\times[0,1],\rT_{u_{02}} M_{02})$ consists of smooth sections by elliptic regularity,  
so that we have a direct sum
$H^s(\bR\times[0,1],\rT_{u_{02}} M_{02}) \;=\; \im \rD_{u_{02}}^{\rm Lag} \oplus \fC_{02}$ for all $s\geq 0$. 
Yet another equivalent Fredholm operator for this boundary value problem is  
\begin{align*}
\ti \rD_{u_{02}}^{\rm Lag}\,:\; \{\xi_{02}\in H^{3/2} \,|\,  \xi_{02}|_{t=1} \in \rT(L_0\times L_2) \bigr\} &\;\to\; H^{1/2}(\bR\times[0,1],\rT_{u_{02}} M_{02}) \times H^{1} (\bR,\rT_{u_{02}|_{t=0}} L_{02}^\perp) , 
\\
\xi_{02}\quad &\;\mapsto\; \quad \bigl( \; \rD_{u_{02}}\xi_{02} \;,\; \pi_{02}^\perp \circ \xi_{02}|_{t=0} \;\bigr) .
\end{align*}
This encodes the $t=0$ boundary condition in the operator by restriction $H^{3/2}(\bR\times[0,1])\to H^1(\bR), \xi_{02}\mapsto \xi_{02}|_{t=0}$ and projection to the orthogonal complement of the immersed Lagrangian tangent spaces $\pi_{02}^\perp : \rT_{u_{02}|_{t=0}} M_{02} \to  (\rT_{u_{02}|_{t=0}} L_{02})^\perp := \bigl( {\rm Pr}_{\rT L_{02}}\rT_{\hat u} \ti L_{02} \bigr)^\perp$ as in Remark~\ref{rmk:immersion}. 
This operator has the same kernel as $\rD_{u_{02}}^{\rm Lag}$ and isomorphic cokernel, so that we obtain the direct sum 
\begin{equation}\label{eq:auxstab}
H^{1/2}(\bR\times[0,1],\rT_{u_{02}} M_{02}) \times H^{1} (\bR,\rT_{u_{02}|_{t=0}} L_{02}^\perp) 
\;=\; \im \ti \rD_{u_{02}}^{\rm Lag} \;\oplus\; \{ ( \fc_{02} , 0 ) \,|\, \fc_{02} \in\fC_{02} \} . 
\end{equation}
Now we can establish a representation of the cokernel of $\ol{\rD\cF_0(0)}$ by proving the direct sum
\begin{align*}
&H^{1/2}(\bR\times[0,1],\rT_{u_{02}} M_{02}) \times \bigl( L^1([0,1],H^1(\bR,\rT_{\hat u} \hat M)) \cap W^{1,1}([0,1],L^2(\bR,\rT_{\hat u} \hat M)) \bigr) 
\\
&\quad\;=\; \im \ol{\rD\cF_0(0)} \;\oplus\; \{ ( \fc_{02} , 0 ) \,|\, \fc_{02} \in\fC_{02} \} . 
\end{align*}
To do so we rewrite the sum claim for any given  $( \eta_{02} , \hat \eta )$ in the left hand side 
\begin{align*}
&( \eta_{02} , \hat \eta ) =  \ol{\rD\cF_0(0)}(\xi_{02},\hat\xi) + (\fc_{02} , 0 ) , \qquad (\xi_{02},\hat\xi)\in\ol\Gamma_0 \\
\Leftrightarrow \qquad &
\eta_{02} = \rD_{u_{02}}\xi_{02} + \fc_{02} , \quad \hat \eta=\nabla_t\hat\xi , 
\qquad \xi_{02}\in H^{3/2}, 
\hat\xi \in W^{1,1}([0,1],H^1) \cap W^{2,1}([0,1],L^2)  , 
\eqref{eq:bc Gamma} 
\\
\Leftrightarrow \qquad &
\rD_{u_{02}}\xi_{02} = \eta_{02} - \fc_{02} , \quad \nabla_t\hat\xi = \hat \eta , \qquad  \eta_{02}\in{\rm dom} \ti \rD_{u_{02}}^{\rm Lag} ,  \hat \eta \in {\rm dom} \ti \nabla_t    , \quad  \xi_{02}|_{t=0} =\hat \xi_{02} |_{t=0} 
\\
\Leftrightarrow \qquad &
\ti \rD_{u_{02}}^{\rm Lag} \xi_{02} = (\eta_{02} - \fc_{02},  \pi_{02}^\perp \hat\xi_{02}|_{t=0}  )  , \quad 
\ti\nabla_t\hat\xi = (\hat \eta , \pi_{02} \xi_{02}|_{t=0}  )  , \qquad  \eta_{02}\in{\rm dom} \ti \rD_{u_{02}}^{\rm Lag} ,  \hat \eta \in {\rm dom} \ti \nabla_t . 
\end{align*}
Here we split $\xi_{02}|_{t=0} =\hat \xi_{02} |_{t=0}  \;\Leftrightarrow\; \pi_{02} \xi_{02}|_{t=0} = \pi_{02}\hat \xi_{02} |_{t=0} \;\text{and}\; \pi_{02}^\perp\xi_{02}|_{t=0} = \pi_{02}^\perp \hat \xi_{02} |_{t=0}$ and recognized each of these conditions as boundary conditions encoded in the operators $\ti \rD_{u_{02}}^{\rm Lag}$ and $\ti\nabla_t$. 
Abstractly, the resulting equations are too coupled to have evident solutions. However, we can utilize the explicit inverse 
$(\hat\eta,\lambda)\mapsto {\rm Pr}_{\rT L_{02}}^{-1}\lambda + \int_0^\bullet \hat\eta - (\pi_{02}^\perp \times \pi_{11}) \int \hat\eta$
of $\ti\nabla_t$ from Remark~\ref{rmk:auxiso} to solve the equations: 
Applied to $( \eta_{02} , \hat \eta )=(0,0)$ the equivalence identifies the intersection of the subspaces 
\begin{align*}
\ol{\rD\cF_0(0)}(\xi_{02},\hat\xi) = (\fc_{02} , 0 )  
\qquad\Leftrightarrow \qquad &
\ti \rD_{u_{02}}^{\rm Lag}\xi_{02} = (- \fc_{02},  \pi_{02}^\perp \hat\xi_{02}|_{t=0}  )  , \quad 
\ti\nabla_t\hat\xi = (0 , \pi_{02} \xi_{02}|_{t=0}  )
\\
\Leftrightarrow \qquad &
\ti \rD_{u_{02}}^{\rm Lag}\xi_{02} = (- \fc_{02},  \pi_{02}^\perp \hat\xi_{02}|_{t=0}  )  , \quad 
\hat\xi = {\rm Pr}_{\rT L_{02}}^{-1}  \pi_{02} \xi_{02}|_{t=0} 
\\
\Leftrightarrow \qquad &
\ti \rD_{u_{02}}^{\rm Lag}\xi_{02} = (- \fc_{02},  0 )  , \quad 
\hat\xi = {\rm Pr}_{\rT L_{02}}^{-1}  \pi_{02} \xi_{02}|_{t=0} 
\\
\Leftrightarrow \qquad &
\ti \xi_{02} = 0 , \quad 
\hat\xi = {\rm Pr}_{\rT L_{02}}^{-1}  \pi_{02} \, 0 = 0 . 
\end{align*}
And given 
$( \eta_{02} , \hat \eta )\in H^{1/2}(\bR\times[0,1],\rT_{u_{02}} M_{02}) \times \bigl( L^1([0,1],H^1(\bR,\rT_{\hat u} \hat M)) \cap W^{1,1}([0,1],L^2(\bR,\rT_{\hat u} \hat M)) \bigr)$ we can solve 
\begin{align*}
& ( \eta_{02} , \hat \eta ) =  \ol{\rD\cF_0(0)}(\xi_{02},\hat\xi) + (\fc_{02} , 0 )   \\
&\qquad \Leftrightarrow \quad 
\ti \rD_{u_{02}}^{\rm Lag}\xi_{02} = (\eta_{02} - \fc_{02},  \pi_{02}^\perp \hat\xi_{02}|_{t=0}  )  , \quad 
\ti\nabla_t\hat\xi = (\hat \eta , \pi_{02} \xi_{02}|_{t=0}  )  
\\
&\qquad \Leftrightarrow \quad 
\ti \rD_{u_{02}}^{\rm Lag}\xi_{02} =(\eta_{02} - \fc_{02},  \pi_{02}^\perp \hat\xi_{02}|_{t=0}  )   , \quad 
\hat\xi = {\rm Pr}_{\rT L_{02}}^{-1} \pi_{02} \xi_{02}|_{t=0}   + \textstyle\int_0^\bullet \hat \eta  - (\pi_{02}^\perp \times \pi_{11}) \int \hat\eta
\\
&\qquad\Leftrightarrow \quad 
\ti \rD_{u_{02}}^{\rm Lag}\xi_{02} = (\eta_{02} - \fc_{02},  \pi_{02}^\perp \textstyle \int \hat\eta_{11} )   , \quad 
\hat\xi = {\rm Pr}_{\rT L_{02}}^{-1} \pi_{02} \xi_{02}|_{t=0}   + \textstyle\int_0^\bullet \hat \eta  
\\
\end{align*}
by first finding a (not necessarily unique) pair $(\xi_{02},\fc_{02})$ that solves the stabilized equation for $\ti \rD_{u_{02}}^{\rm Lag}$ and then computing $\hat\xi$ from the above formula. 
This proves the claimed direct sum and thus identifies the cokernel of $\ol{\rD\cF_0(0)}$. Since both kernel and cokernel are finite dimensional, this establishes $\ol{\rD\cF_0(0)}$ as a Fredholm operator -- with index equal to the index of $\rD_{u_{02}}^{\rm Lag}$. 

Finally, the stabilized Fredholm estimate follows from the fact that the linear operator
$$
(\xi_{02}, \hat\xi,\fc_{02}) \mapsto \bigl( \pi_{\fK_{02}}(\xi_{02})  , \ol{\rD\cF_0(0)} (\xi_{02}, \hat\xi) - (\fc_{02},0) \bigr) 
$$
is bounded, surjective, and has Fredholm index $0$ -- hence has a bounded inverse. 
\end{proof}

\subsection{Instanton and Symplectic Floer Theory}

There are several conjectures of Atiyah-Floer type, which all relate the instanton Floer homology of a $3$-manifold -- defined from moduli spaces of anti-self-dual instantons -- to a symplectic Floer homology -- defined from moduli spaces of pseudoholomorphic maps to a representation space arising from dimensional reduction of the anti-self-duality equation. 
 
In the nondegenerate Atiyah-Floer conjecture $HF^{\rm inst}(Y_h)\simeq HF(\phi_f)$  proved in \cite{dostoglou-salamon} the $3$-manifold is the mapping torus $Y_h$ of a diffeomorphism $h:\Sigma\to\Sigma$ of a closed Riemann surface. It is induced by a bundle automorphism $f:Q\to Q$ of a nontrivial $G=SO(3)$-bundle $Q\to\Sigma$, which induces a nontrivial bundle $Q_f\to Y_h$ and the differential for the instanton Floer homology counts anti-self-dual connections on $\bR\times Q_f$. 
The bundle automorphism also induces a symplectomorphism $\phi_f:\cR(Q)\to\cR(Q)$ of the moduli space $\cR(Q)$ of flat connections on $Q$ (which can be identified with a $G$-representation space), and the differential for the symplectic Floer homology counts pseudoholomorphic maps $\bR\times[0,1]\to\cR(Q)$ with the boundary values matching via $\phi_f$. 

In the (original) Atiyah-Floer conjecture $HF^{\rm inst}(Y)\simeq HF^{\rm symp}(L_{H_0},L_{H_1})$ the $3$-manifold is a homology $3$-sphere $Y$. Then the instanton Floer differential counts anti-self-dual connections on the trivial $G=SU(2)$ bundle over $\bR\times Y$. 
A choice of Heegard splitting $Y=H_0\cup_\Sigma H_1$ gives rise to a singular representation space $\cR(\Sigma)$ with two Lagrangians $L_{H_0},L_{H_1}\subset\cR(\Sigma)$ whose intersection points correspond to flat connections on $Y$. Then the -- only conjecturally defined -- Lagrangian Floer homology is thought to be counting pseudoholomorphic maps $\bR\times[0,1]\to\cR(\Sigma)$ with boundary values in $L_{H_0}$ and $L_{H_1}$. 

The adiabatic limit approach to proving these relationships was developed by Salamon \cite{salamon_ICM}.  
After several localizations -- in the moduli space, in a local slice (centered at a reference connection $A_0+\Phi_0\rd s + \Psi_0\rd t$), and locally on the domain -- it studies the family of PDEs for $0<\eps\leq 1$ 
\begin{equation}
\label{eqep}
\left\{
\begin{aligned}
\pd_s A -\rd_A\P + *\bigl(\pd_t A - * \rd_A\Psi\bigr) &= 0, \\
\pd_s\Psi - \pd_t\P + [\P,\Psi] + \ep^{-2} * F_A & = 0 , \\
\nabla^0_s(\Phi-\Phi_0) + \nabla^0_t(\Psi-\Psi_0)  + \ep^{-2} * \rd_{A_0} *(A-A_0) & = 0 
\end{aligned}\right.
\end{equation}
for a triple of maps ${A:[-1,1]\times[0,1]\to\Omega^1(\Sigma,\fg_Q)}$ 
and ${\P,\Psi:[-1,1]\times[0,1]\to\Omega^0(\Sigma,\fg_Q)}$ to spaces of $1$- resp.\ $0$-forms with values in the associated bundle $\fg_Q= Q \times_{\rm Ad} \fg$.  
Their energy 
$$
\cE(A,\Phi,\Psi) =
\int_{[-1,1]\times[0,1]\times\Sigma} |\pd_s A - \rd_A\Phi|^2 
+ \ep^{-2} |F_A|^2
$$
is bounded, independently of $\eps$, by a global monotonicity formula. 
Thus the expectation is to obtain in the $\eps\to 0$ limit 
 a triple of maps ${A:[-1,1]\times[0,1]\to\Omega^1(\Sigma,\fg_Q)}$ 
and ${\P,\Psi:[-1,1]\times[0,1]\to\Omega^0(\Sigma,\fg_Q)}$ satisfying 
\begin{equation}\left\{   \label{eq0}
\begin{aligned}
\pd_s A -\rd_A\P + *\bigl( \pd_t A - \rd_A\Psi \bigr) &= 0, \\
*F_A &=0, \\
\rd_{A_0} *(A-A_0) & = 0 . 
\end{aligned}\right.
\end{equation}
Such an $\eps=0$ solution can then be interpreted as a map $[A] :[-1,1]\times[0,1]\to\cR(Q)$ to the moduli space of flat connections on $Q$ that solves the Cauchy-Riemann equation with respect to the almost complex structure $J$ induced by the Hodge $*$ operator
\begin{equation*} \label{eq hol}
\pd_s [A] + J([A]) \pd_t [A] = 0 . 
\end{equation*}
Our proposal for casting this classical adiabatic limit as an adiabatic Fredholm family is to multiply \eqref{eqep} by $\eps^2$ and change variables to $\eps^{\rm new}=(\eps^{\rm old})^2$ to obtain a family of PDEs for $0\leq\eps\leq 1$ which naturally combines \eqref{eqep} and \eqref{eq0}, 
\begin{equation*}
\cF_\eps(\alpha,\varphi,\psi)\,:=\; \left(
\begin{aligned}
\pd_s A -\rd_A\P + *\bigl(\pd_t A - * \rd_A\Psi\bigr)  \\
\ep \bigl( \pd_s\Psi - \pd_t\P + [\P,\Psi] \bigr)+  * F_A \\
\ep \bigl( \nabla^0_s(\Phi-\Phi_0) + \nabla^0_t(\Psi-\Psi_0) \bigr) + *\rd_{A_0} *(A-A_0) 
\end{aligned}
\right)
\;=\; 0 ,
\end{equation*}
where $A=A_0+\alpha,\Phi=\Phi_0+\phi,\Psi=\Psi_0+\psi$. 
This formulates the adiabatic limit near a solution $A_0+\Phi_0\rd s + \Psi_0\rd t$ and locally on $Z:=[-1,1]\times[0,1]$ -- as a family of maps $(\cF_\eps)_{\eps\in[0,1]}$ with common domain and target space
$$
\Gamma \, :=\;\cC^\infty(Z, \Omega^1(\Sigma,\fg_Q))\times \cC^\infty(Z, \Omega^0(\Sigma,\fg_Q))\times \cC^\infty(Z, \Omega^0(\Sigma,\fg_Q)) \;=:\, \Omega . 
$$
The norms for $\eps>0$ from \cite{dostoglou-salamon} become
\begin{align*}
\|(\alpha,\phi,\psi)\|^\Gamma_\eps &\,:=\; \eps^{-1/2} \|\alpha\|_{L^p(Z, W^{1,p}(\Sigma))} +  \|\nabla_Z\alpha\|_{L^p} + \|(\varphi,\psi)\|_{L^p(Z, W^{1,p}(\Sigma))}  + \eps^{1/2} \|(\varphi,\psi)\|_{L^p} ,
\\
\|(\alpha,\phi,\psi)\|^\Omega_\eps &\,:=\; 
 \|\alpha\|_{L^p}  + \eps^{-1/2} \|(\varphi,\psi)\|_{L^p} . 
\end{align*}
Now the key step is to find norms $\|\cdot \|^\Gamma_0\leq \|(\alpha,\phi,\psi)\|^\Gamma_\eps$ and
$\|\cdot \|^\Omega_0\leq \|(\alpha,\phi,\psi)\|^\Omega_\eps$ which give the $\eps=0$ linearized operator the Fredholm property 
\begin{equation*}
\overline{\rD\cF_0(0)}: (\alpha,\varphi,\psi) \;\mapsto\; 
\left(
\begin{aligned}
\pd_s \alpha -\rd_{A_0}\Phi + *\bigl(\pd_t \alpha - * \rd_{A_0}\psi\bigr)  \\
 * \rd_{A_0}\alpha  \qquad\;\; \\
 *\rd_{A_0} * \alpha  \qquad
\end{aligned}
\right). 
\end{equation*}

We have a conjecture that we would be happy to share with folks who are interested in making righteous use of this adiabatic limt. 

\begin{conjecture} \label{conj:AF}
The above data
$\bigl( (\cF_\eps:\cV_\Gamma\to \Omega )_{\eps\in[0,1]} ,\|\cdot\|^\Gamma_\eps, \|\cdot\|^\Omega_\eps \bigr)$ 
can be supplemented 
to form a regularizing $\cC^1$-regular adiabatic Fredholm family as in Definition~\ref{def:fredholm}, \ref{def:adiabatic C-l}. 
\end{conjecture}

%
%
%
%
%
%


\vfill
\pagebreak

\section{Proofs}  \label{proofs}

This section proves Theorem~\ref{thm:charts}. 
To make the strategy of this proof more accessible, we will work through each step first in the case of a classical Fredholm section -- reproving the following classical fact, which underlies most constructions of invariants from moduli spaces of PDEs. 

\begin{theorem} \label{thm:charts classical}
Suppose an open subset $\cU\subset\overline\cM$ of a topological space (such as a compactified moduli space) is described as the zero set of a $\cC^\ell$-regular Fredholm map $\cF:\cV_{\ol\Gamma}\to \ol\Omega$ for $\ell\geq 1$. More precisely, we assume the following: 
\begin{itemize}
\item
$(\ol\Gamma,\|\cdot\|^\Gamma)$ and $(\ol\Omega,\|\cdot\|^\Omega)$ are Banach spaces.
\item
$\cV_{\ol\Gamma}\subset\ol\Gamma$ is a an open neighbourhood of $0=0_{\ol\Gamma}\in\ol\Gamma$.
\item
$\cF : \cV_{\ol\Gamma} \to \ol\Omega$ is a $\cC^\ell$-map with $\cF(0)=0$, whose linearization $\rD\cF (0) : \ol\Gamma \to \ol\Omega$ is Fredholm, that is, $\fK:=\ker\rD\cF(0)$ and $\fC:=\coker\rD\cF(0)=\ol\Omega / \im\rD\cF(0)$ are finite dimensional.  
\item
$\psi : \cF^{-1}(0) \to \cU$  is a homeomorphism.
\end{itemize}
Then any choice of projection $\pi_\fK:\ol\Gamma\to\fK$ and inclusion $\fC\subset\ol\Omega$ such that $\ol\Omega=\fC\oplus\im\rD\cF(0)$ provides a Fredholm stabilization of $\rD\cF(0)$ in the sense that for some constant $C_0$ we have
$$
\| \gamma \|^\Gamma + \| \fc  \|^\Omega  \leq C_0 \bigl( \|\pi_\fK(\gamma) \|^\Gamma + \| \rD\cF(0) \gamma - \fc \|^\Omega  \bigr)\qquad\text{for all}\; (\gamma,\fc)\in\ol\Gamma\times \fC .
$$
And this induces a finite dimensional reduction that describes $\ol\cM$ locally as the zero set of a map between finite dimensional spaces, 
$$
f: \cV_\fK\to\fC, 
\qquad\text{and}\qquad 
\psi_f: f^{-1}(0) \to \ol\cM  
$$ 
defined on an open subset $\cV_\fK\subset\fK$. It describes $\cU\subset\overline\cM$ locally by composition $\psi_f=\psi\circ\phi$ with a homeomorphism for some $\delta_\sigma>0$
$$
\phi \,:\; f^{-1}(0) \;\overset{\sim}{\to}\;
 \cF^{-1}(0)  \cap \bigl\{ \gamma\in\ol\Gamma \,\big|\, \|\gamma\|^\Gamma < \delta_\sigma \bigr\} . 
$$
Moreover, this finite dimensional reduction is $\cC^\ell$ in the sense that the differentials of order $0\leq k \leq \ell$ (see Remark~\ref{rmk:multilinear}) form continuous maps 
$\cV_\fK \to \cL^k(\fK^k,\fC)$, $\fk_0\mapsto \rD^k f(\fk_0)$.
\end{theorem}

The traditional approach (see e.g.\ \cite[Rmk.4.2]{wehrheim_Fredholm}) to this proof is based on the {\bf Fredholm splitting} given by the direct sums
\begin{equation}\label{eq:splitting classical}
\ol\Gamma=\fK\oplus\ker\pi_\fK 
\qquad\text{and}\qquad
\ol\Omega=\fC\oplus\im\rD\cF(0) 
\end{equation}
and the fact that the linearization restricts to an isomorphism $\rD\cF(0)|_{\ker\pi_\fK}: \ker\pi_\fK\to \im\rD\cF(0)$. One then uses its inverse $Q:=(\rD\cF(0)|_{\ker\pi_\fK})^{-1}$ to rewrite for $\gamma = \fk + w \in \fK\oplus\ker\pi_\fK=\ol\Gamma$ 
$$
\cF(\gamma)=0 \qquad\Longleftrightarrow\qquad  \pi_\fC (\cF(\fk + w) ) = 0  \quad\text{and} \quad Q \, ({\rm Id}_{\ol\Omega}- \pi_\fC )\, \cF(\fk+w) = 0  . 
$$
Here the second equation for sufficiently small $\fk\in\cV_\fK$ is a fixed point equation for a retraction on $\ker\pi_\fK$ by $\cC^1$ regularity of $\cF$. The unique fixed points then define a solution map $\sigma:\cV_\fK \to \ker\pi_\fK$ such that 
$Q \, ({\rm Id}_{\ol\Omega}- \pi_\fC )\, \cF(\fk+w) = 0 \;\Leftrightarrow\; w=\sigma(\fk)$ and hence we obtain the desired finite dimensional reduction $f:\cV_\fK\to\ker\pi_\fK$, $\phi:f^{-1}(0)\to \cF^{-1}(0)$ from 
$$
\cF(\gamma)=0 \qquad\Longleftrightarrow\qquad  \gamma=\phi(\fk):=\fk+\sigma(\fk) \quad\text{with}\quad f(\fk):=\pi_\fC (\cF(\fk + \sigma(\fk) ) = 0    . 
$$
This formulation, however, is not suitable for the adiabatic analysis, as it heavily depends on direct sums -- which are hard to formulate on dense subsets with $\eps$-dependent norms. To remedy this, we will in \S\ref{inverses} replace in the Fredholm splitting by a {\bf Fredholm stabilization isomorphism}
\begin{equation}\label{eq:stabilization classical}
 \fC \times \ol\Gamma \;\rightarrow\; \fK \times \ol\Omega, \qquad
(\fc, \gamma) \;\mapsto\; ( \pi_\fK(\gamma), \rD\cF(0)\gamma - \fc ).
\end{equation}

The remaining proof proceeds analogous to the above outline:
\begin{itemlist}
\item[$\cdot$]
\S\ref{contraction} uses the inverse of this isomorphism to rewrite $\cF(\gamma)=0$ as a finite dimensional equation and a fixed point problem for a family of contractions. 
\item[$\cdot$]
\S\ref{solution} establishes the solution maps for this fixed point problem, in particular proves their $\cC^\ell$ regularity.
\item[$\cdot$]
\S\ref{reduction} deduces the finite dimensional reduction given the solution map and the remaining finite dimensional equation. 
\end{itemlist}
A useful zero-th step in this approach is to identify the zero set of the nonlinear map $\cF^{-1}(0)\simeq\cG^{-1}(0)$ with the zero set of a stabilized map of the same Fredholm index. 

\begin{lemma}\label{lem:equivalent-fredholm classical} 
Consider any map $\cF:\cV \to \ol\Omega$ from a set $\cV$ to a vector space $\ol\Omega$.  
Given any auxiliary map $\pi_\fK: \cV\to\fK$ to another vector space $\fK$ and auxiliary subspace $\fC\subset\ol\Omega$, the zero set $\cF^{-1}(0)$ is naturally identified with the zero set of the map
\begin{align*}
\cG: \fK \times \fC \times \cV &\;\rightarrow\; \fC \times \fK \times \ol\Omega, \\
(\fk, \fc, \gamma) &\;\mapsto\; ( \fc , \pi_\fK(\gamma)-\fk, \cF(\gamma) - \fc ).
\end{align*}
\end{lemma}
\begin{proof}
The identification is given by the map $(\pi_\fK(\gamma),0, \gamma) \mapsto \gamma$ with inverse $\gamma \mapsto (\pi_\fK(\gamma),0, \gamma)$. 
\end{proof}

This is the idea that we use to ``externalize'' the finite dimensional factors $\fK,\fC$ when dealing with adiabatic Fredholm families. 

\begin{lemma} \label{lem:equivalent-fredholm}
Given an adiabatic Fredholm family $\bigl( (\cF_\eps:\cV_\Gamma\to \Omega )_{\eps\in\Delta} , \ldots  \bigr)$ as in Definition~\ref{def:fredholm}, its union of zero sets $\bigcup_{\eps\in\Delta} \{\eps\}\times \cF_\eps^{-1}(0) \subset \Delta\times \Gamma$ is naturally identified with the zero set of the map
\begin{align*}
\cG: \Delta \times \fK \times \fC \times \cV_\Gamma &\;\rightarrow\; \fC \times \fK \times \Omega, \\
(\eps, \fk, \fc, \gamma) &\;\mapsto\; ( \fc , \pi_\fK(\gamma)-\fk, \cF_\eps(\gamma) - \fc ).
\end{align*}
The same holds for the extended adiabatic Fredholm family, that is $\overline\cF_\eps^{-1}(0)\subset\cV_{\overline\Gamma,\eps}$ is naturally identified with $\overline\cG(\eps,\cdot)^{-1}(0)\subset\fK\times\fC\times\cV_{\overline\Gamma,\eps}$. More precisely, the union of zero sets $\bigcup_{\eps\in\Delta} \{\eps\}\times \overline\cF_\eps^{-1}(0) \simeq \overline\cG^{-1}(0_\Delta)$ is identified with the preimage of the "zero section" $0_\Delta := \bigcup_{\eps\in\Delta} \{\eps\}\times \{(0,0,0)\} \subset \bigcup_{\eps\in\Delta} \{\eps\}\times \fC\times\fC\times \cV_{\overline\Gamma,\eps}$ under the "section" 
\begin{align*}
\overline\cG: \textstyle\bigcup_{\eps\in\Delta} \{\eps\}\times \fK \times \fC \times \cV_{\overline\Gamma,\eps} &\rightarrow  \textstyle\bigcup_{\eps\in\Delta} \{\eps\}\times \fC \times \fK \times \overline\Omega_\eps, \\
(\eps, \fk, \fc, \gamma) &\mapsto ( \fc , \overline\pi_\fK(\gamma)-\fk, \overline\cF_\eps(\gamma) - \fc ).
\end{align*}
In both cases, the identification for fixed $\eps\in\Delta$ is given by the map $(\overline\pi_\fK(\gamma),0, \gamma) \mapsto \gamma$ with inverse $\gamma \mapsto (\overline\pi_\fK(\gamma),0, \gamma)$; in the first case this restricts to $\cV_\Gamma\subset\Gamma$ with $\overline\pi_\fK|_{\Gamma\subset\overline\Gamma_\eps}=\pi_\fK$. 
 
If, moreover, the adiabatic Fredholm family is regularizing in the sense of Definition~\ref{def:regularizing}, then the zero sets of the extended adiabatic Fredholm family agree with the original zero sets $\overline\cF_\eps^{-1}(0)=\cF_\eps^{-1}(0)\subset\cV_\Gamma$ for each $\eps\in\Delta$ and $\overline\cG^{-1}(0)=\cG^{-1}(0)\subset  \Delta\times\fK \times \fC \times \cV_\Gamma$. 
\end{lemma}
\begin{proof}
By construction, the zero set $\cG^{-1}(0) \subset \Delta \times \fK \times \fC \times \cV_\Gamma$ is the subset of $(\eps, \fk =\pi_\fK(\gamma), \fc=0, \gamma)$ with $\cF_\eps(\gamma)=0$, thus a bijection to $\bigcup_{\eps\in\Delta} \{\eps\}\times \cF_\eps^{-1}(0)$ is given by $(\eps, \pi_\fK(\gamma),0, \gamma) \mapsto (\eps,\gamma)$ with inverse $(\eps,\gamma) \mapsto (\eps, \pi_\fK(\gamma),0, \gamma)$.

For fixed $\eps\in\Delta$, these maps are uniformly continuous w.r.t.\ the $\eps$-dependent norm on $\cV_\Gamma$, hence extend to the claimed natural bijection between $\overline\cF_\eps^{-1}(0)$ and $\overline\cG(\eps,\cdot)^{-1}(0)$ as in Lemma~\ref{lem:extension}.
These are the fibres of the preimage of the "zero section" $\overline\cG^{-1}(0_\Delta)=\bigcup_{\eps\in\Delta} \{\eps\}\times\overline\cG(\eps,\cdot)^{-1}(0)$.

Now suppose that the adiabatic Fredholm family is regularizing, in particular that we have the implication $\overline\cF_\eps(\gamma) \in \Omega \; \Rightarrow \; \gamma\in\cV_\Gamma$. Then every $\gamma\in\overline\cF_\eps^{-1}(0)\subset\cV_{\overline\Gamma,\eps}$ must lie in $\cF_\eps^{-1}(0)$ since $\overline\cF_\eps(\gamma)=0$ lies in $\Omega$, so that the regularizing property implies $\gamma\in\cV_\Gamma$ with $\cF_\eps(\gamma)=0$.
Analogously, $(\fk,\fc,\gamma)\in \overline\cG(\eps,\cdot)^{-1}(0)\subset\fK\times\fC\times\cV_{\overline\Gamma,\eps}$ implies $\overline\cF_\eps(\gamma)=\fc\in\fC\subset\Omega$, and hence $\gamma\in\cV_\Gamma$ with $\cG(\eps,\fk,\fc,\gamma)=0$. 
This proves the inclusions
$\overline\cF_\eps^{-1}(0)\subset\cF_\eps^{-1}(0)\subset\cV_\Gamma$ and $\overline\cG(\eps,\cdot)^{-1}(0)\subset\cG(\eps,\cdot)^{-1}(0)\subset  \fK \times \fC \times \cV_\Gamma$, the converse inclusions hold by the construction of $\overline\cF_\eps$ resp.\ $\overline\cG(\eps,\cdot)$ as extensions of $\cF_\eps$ resp.\ $\cG(\eps,\cdot)$.
\end{proof}

\subsection{Uniform Inverses} \label{inverses}

This section constructs the analogue of the right inverse used in the Newton-Picard Iteration of classical adiabatic limit methods. Conceptually, our inverse operators will arise from the fact that any Fredholm splitting \eqref{eq:splitting classical} is equivalent to a Fredholm stabilization \eqref{eq:stabilization classical}. Moreover, we will quantify the continuity of the resulting inverse operators when varying the base point, which will play a crucial role in establishing regularity of the finite dimensional charts resulting from an adiabatic Fredholm family. 

Starting in the classical Fredholm setting of Theorem~\ref{thm:charts classical}, we start by noting that any $\cC^1$ Fredholm map satisfies an estimate that -- for our present purposes -- can play the role of the [Quadratic-ish Estimate] and [Uniform Continuity of $\rD\cF_\eps$] in Definitions~\ref{def:fredholm} and \ref{def:adiabatic C-l}.

\begin{lemma} \label{lem:inverses classical}
Any Fredholm map $\cF:\cV_{\ol\Gamma}\to \Omega$ as in Theorem~\ref{thm:charts classical} satisfies a {\rm [Quadratic-ish Estimate]} 
\begin{equation} \label{eq:quadratic classical}
 \bigl\| \rD\cF(\gamma_0') \gamma - \rD\cF(\gamma_0) \gamma  \bigr\|^\Omega  \leq 
\tilde c^1_\cF(\gamma'_0,\gamma_0)  \| \gamma\|^{\ol\Gamma}
\qquad
\forall \gamma'_0,  \gamma_0\in \cV_{\overline\Gamma}, \gamma\in\Gamma , 
\end{equation}
where $\tilde c^1_\cF:\cV_{\ol\Gamma}\times \cV_{\ol\Gamma}\to [0,\infty)$ is a continuous map that vanishes on the diagonal $\tilde c^1_\cF(\gamma_0,\gamma_0)=0$. If, moreover, $\cV_{\ol\Gamma}$ is convex and $\|\rD^2\cF(\gamma_0)\|^{\cL^2(\ol\Gamma^2,\ol\Omega)}\leq C^2_\cF$ is bounded for $\gamma_0\in\cV_{\ol\Gamma}$,
then $\tilde c^1_\cF(\gamma'_0,\gamma_0)\leq C^2_\cF \|\gamma'_0-\gamma_0\|^\Gamma$ makes \eqref{eq:quadratic classical} a truly quadratic estimate. 
Further, there is a constant $\delta_Q>0$ so that for any $\gamma_0\in  \cV_{\overline\Gamma}$ with $\|\gamma_0\|^\Gamma\leq\delta_Q$ we have a {\bf Fredholm stabilization isomorphism}
\begin{align}  \label{eq:P classical}
P(\gamma_0) \, :  \; \fC \times \overline\Gamma&\rightarrow  \fK \times \overline\Omega, \\
( \fc, \gamma) &\mapsto (\pi_\fK(\gamma), \rD\cF(\gamma_0) \gamma - \fc )  .
\nonumber
\end{align}
Its inverse operators $Q(\gamma_0):=P(\gamma_0)^{-1}:   \fK \times \overline\Omega\to \fC \times \overline\Gamma$ are uniformly bounded -- where we equip $\fC\subset\ol\Omega$ and $\fK\subset\ol\Gamma$ with the induced norms $\|\fk\|^\fK:=\|\fk\|^\Gamma$ resp.\ $\|\fc\|^\fC:=\|\fc\|^\Omega$ -- 
\begin{align}\label{eq:Q-estimate classical}
 \| Q(\gamma_0)  \|^{\cL(\fC\times\overline\Gamma,\fK\times\overline\Omega)}
&\leq
C_Q  \qquad\forall \gamma_0\in\cV_{\overline\Gamma} , \|\gamma_0\|^\Gamma\leq\delta_Q . 
\end{align}
Moreover, the inverse operators $Q(\gamma_0)$ vary continuously with $\gamma_0\in\cV_{\overline\Gamma}$. More precisely, for all 
$\gamma_0,\gamma'_0\in\cV_{\overline\Gamma}$ with $\|\gamma_0\|^\Gamma, \|\gamma'_0\|^\Gamma\leq\delta_Q$ we have with $c:\cV_{\ol\Gamma}\times \cV_{\ol\Gamma}\to [0,\infty)$ from \eqref{eq:quadratic classical}
\begin{equation}\label{eq:Q cont classical}
 \|  Q(\gamma'_0) -  Q(\gamma_0) \|^{\cL(\fK\times\ol\Omega,\fC\times\overline\Gamma)}
\leq
(C_Q)^2 \; \tilde c^1_\cF(\gamma'_0, \gamma_0) . 
\end{equation}
\end{lemma}

\begin{proof}
To begin, the $\cC^1$ regularity of $\cF$ implies a [Quadratic-ish Estimate]  for $\gamma'_0,  \gamma_0\in \cV_{\overline\Gamma}, \gamma\in\Gamma$, 
$$
\bigl\| \rD\cF(\gamma_0') \gamma - \rD\cF(\gamma_0) \gamma  \bigr\|^\Omega  
\;\leq\;
\bigl\| \rD\cF(\gamma_0') - \rD\cF(\gamma_0) \bigr\|^{\cL(\ol\Gamma,\ol\Omega)} \| \gamma\|^\Gamma
\;=\; 
 \tilde c^1_\cF(\gamma'_0,\gamma_0)  \| \gamma\|^\Gamma
$$
with $\tilde c^1_\cF:\cV_{\ol\Gamma}\times \cV_{\ol\Gamma}\to [0,\infty)$ 
given by $\tilde c^1_\cF(\gamma'_0,\gamma_0) := \bigl\| \rD\cF(\gamma_0') - \rD\cF(\gamma_0) \bigr\|^{\cL(\ol\Gamma,\ol\Omega)}$. 
This map evidently vanishes on the diagonal $\tilde c^1_\cF(\gamma_0,\gamma_0)=0$, and 
$\cC^0$ regularity of $\rD\cF:\cV_{\ol\Gamma}\to\cL(\ol\Gamma,\ol\Omega)$ exactly is the statement that the map $c$ is continuous. 
If, moreover, $\|\rD^2\cF(\gamma_0)\|^{\cL^2(\ol\Gamma^2,\ol\Omega)}\leq C^2_\cF$ is uniformly bounded for $\gamma_0\in\cV_{\ol\Gamma}$, then -- utilizing the convexity of $\cV_\Gamma$ -- we have as claimed 
\begin{align*}
\tilde c^1_\cF(\gamma'_0,\gamma_0)
&= \bigl\| \rD\cF(\gamma_0') - \rD\cF(\gamma_0) \bigr\|^{\cL(\ol\Gamma,\ol\Omega)} 
\;=\; \bigl\| \textstyle \int_0^1 \rD[\rD\cF](\gamma_0 + \lambda(\gamma_0'-\gamma_0)) \; (\gamma_0'-\gamma) \; \rd\lambda \bigr\|^{\cL(\ol\Gamma,\ol\Omega)}\\
&\leq  \textstyle \int_0^1 \bigl\| \rD^2\cF(\gamma_0 + \lambda(\gamma_0'-\gamma_0))\bigr\|^{\cL^2(\ol\Gamma^2,\ol\Omega)}   \rd\lambda  \; \bigl\| \gamma_0'-\gamma_0) \bigr\|^\Gamma 
\;\leq\; C^2_\cF \,\|\gamma'_0-\gamma_0\|^\Gamma . 
\end{align*}
Next, for $\gamma_0=0$ we can show injectivity of $P(0)$ as follows: 
Given $(\fc,\gamma)\in\ker P(0)$, the direct sum $\ol\Omega=\fC\oplus\im\rD\cF(0)$ and $\rD\cF(0) \gamma - \fc = 0$ imply $\fc=0$ and $\gamma\in\ker\rD\cF(0)$. 
Now the direct sum $\ol\Gamma=\fK\oplus\ker\pi_\fK$ together with $\pi_\fK(\gamma)=0$ implies $\gamma=0$. 

Moreover, $P(0)$ is a stabilization of the Fredholm map $\rD\cF(0)$ in the sense that it result by Cartesian product with finite dimensional factors in the domain and codomain, and addition of a linear operator $(\fc,\gamma)\mapsto (\pi_\fK(\gamma), -\fc)\in \fK\times\fC$ composed with the compact embedding $ \fK\times\fC\subset\fK\times\overline\Omega$. 
Here the embedding $\fC\hookrightarrow \overline\Omega$ is compact since $\fC$ is finite dimensional. 
Thus $P(0)$ is a Fredholm operator of index 
$$
\ind P(0) \;=\; \ind \bigl( (\fc,\gamma)\mapsto (0, \rD\cF(0)\gamma) \bigr) 
\;=\;   \ind \rD\cF(0)+ \dim\fC - \dim\fK \;=\;  0 . 
$$
Since $P(0)$ is injective, the Fredholm index $0$ guarantees that $P(0)$ is also surjective, and hence a bijection. Since $P(0)$ is a bounded operator, its inverse $Q(0):=P(0)^{-1}:\fK\times\ol\Omega \to\fC\times\ol\Gamma$ is continuous as well, that is we have a bound with a preliminary constant $C^{\scriptscriptstyle \rm prelim}_Q$, 
\begin{align} \label{eq:P-estimate classical}
\| Q(0) (\fk,\omega) \|^{\fC\times \Gamma}
&\leq
C^{\scriptscriptstyle \rm prelim}_Q \| (\fk,\omega) \|^{\fK\times\Omega}   \qquad\forall (\fk,\omega)\in\fK\times\ol\Omega .
\end{align}
This invertibility extends to $\gamma_0\ne 0$ whenever $P(\gamma_0)$ is sufficiently close to $P(0)$ in the space of bounded linear operators. This proximity is controlled by \eqref{eq:quadratic classical}: For all $\gamma'_0,\gamma_0\in\cV_{\ol\Gamma}$ we have  
\begin{align}
\bigl\| P(\gamma'_0) - P(\gamma_0) \bigr\|^{\cL(\fC \times \overline\Gamma ,\fK \times \overline\Omega)}
&= \sup_{\| \fc \|^\fC + \|\gamma\|^\Gamma\leq 1}     \bigl\| \bigl( \pi_\fK(\gamma) , \rD\cF(\gamma'_0) \gamma - \fc  \bigr) - \bigl( \pi_\fK(\gamma) , \rD\cF(\gamma_0) \gamma - \fc  \bigr) \bigr\|^{\fK \times \overline\Omega} 
 \nonumber\\
&= \sup_{\| \fc \|^\fC + \|\gamma\|^\Gamma\leq 1}     \bigl\| \rD\cF(\gamma'_0) \gamma - \rD\cF(\gamma_0) \gamma  \bigr\|^\Omega  
\;\leq\; 
 \tilde c^1_\cF(\gamma'_0,\gamma_0) . 
\label{eq:P continuity classical}
\end{align}
Set $\gamma'_0=0$ here, then continuity of $\tilde c^1_\cF:\cV_{\ol\Gamma}\times \cV_{\ol\Gamma}\to [0,\infty)$ ensures an $\delta_Q>0$ so that $\|\gamma_0\|^\Gamma\leq\delta_Q$ implies 
$\tilde c^1_\cF(0,\gamma_0) = \tilde c^1_\cF(0,\gamma_0) - \tilde c^1_\cF(0,0) \leq \frac 1{2C^{\scriptscriptstyle \rm prelim}_Q}$ and hence 
$C^{\scriptscriptstyle \rm prelim}_Q\bigl\| P(0) - P(\gamma_0) \bigr\|^{\cL(\fC \times \overline\Gamma ,\fK \times \overline\Omega)}\leq \frac 12$. 
Then a classical construction yields uniformly bounded inverses of  $P_1:=P(\gamma_0)$ from the inverses $Q_0:=Q(0)$ of $P_0:=P(0)$ whenever $\gamma_0\in \cV_{\overline\Gamma}$ satisfies $\|\gamma_0\|^\Gamma\leq\delta_Q$ as follows: First the composed operator
$T:=Q_0 ( P_0 - P_1 ) : \overline W\to \overline W$ on $\overline W:=\fC\times\overline\Gamma$ is small in the sense that
\begin{align} \label{eq: T small}
\| T \|^{\cL(\overline W,\overline W)} 
&\leq 
\bigl\| Q_0 \bigr\|^{\cL(\fK \times \overline\Omega,\fC \times \overline\Gamma)}  \bigl\|  P_0 -  P_1 \bigr\|^{\cL(\fC \times \overline\Gamma ,\fK \times \overline\Omega)}  \\
&\leq C^{\scriptscriptstyle \rm prelim}_Q \, \tilde c^1_\cF(0,\gamma_0) 
\;\leq\; \tfrac 12 
\qquad\qquad\qquad
\forall \gamma_0\in \cV_{\overline\Gamma}, \|\gamma_0\|^\Gamma\leq\delta_Q. 
 \nonumber 
\end{align}
With that we can express $P_1 = P_0 ( {\rm Id} - T )$
since $Q_0$ is the left and right inverse of $P_0$ and thus
$$
P_0 ( {\rm Id} - T )
 \;=\;  
P_0   -  P_0 Q_0 ( P_0 - P_1) 
 \;=\; 
P_0 -  ( P_0 - P_1)  
 \;=\; 
 P_1.
$$
In this decomposition, $({\rm Id} - T)^{-1}=\sum_{n=0}^\infty T^n$ exists with operator norm bounded by 
\begin{equation} \label{eq:inverse bound}
\textstyle
\bigl\|  ({\rm Id} - T)^{-1} \bigr\|^{\cL(\overline W,\overline W)}  
\;\leq\; \sum_{n=0}^\infty (\|T \|^{\cL(\overline W,\overline W)})^n
\;=\; \tfrac 1{1-\|T\|^{\cL(\overline W,\overline W)} } 
\;\leq\; 2 .
\end{equation}
Now $Q_1 := ({\rm Id} - T)^{-1} Q_0$ is the left and right inverse of $P_1$ since 
\begin{align*}
Q_1 P_1 
 &=\;  
({\rm Id} - T)^{-1}  Q_0 P_0 ( {\rm Id} - T )
 \;=\;  
({\rm Id} - T)^{-1} \bigl(  {\rm Id} - T  \bigr) 
 \;=\;  {\rm Id} \\
\text{and}\qquad\qquad & \\
P_1 Q_1 
 &=\;  
 P_0 ( {\rm Id} - T )  ({\rm Id} - T)^{-1} Q_0
 \;=\;  
P_0 Q_0 
 \;=\;  {\rm Id} . 
\end{align*}
Going back to unabbreviated notation, this shows that $P(\gamma_0) =P_1$ is invertible for all $\gamma_0\in\cV_{\ol\Gamma}$, $\|\gamma_0\|^\Gamma\leq\delta_Q$ with inverses $Q(\gamma_0):=P(\gamma_0)^{-1}=Q_1$ that satisfy the uniform bound 
\begin{align*}
 \| Q(\gamma_0) (\fk,\omega) \|^{\fC\times\overline\Gamma}
&\leq \bigl\|  ({\rm Id}_{\overline W} - T)^{-1} \bigr\|^{\cL(\overline W,\overline W)}  
 \bigl\|  Q(0) (\fk,\omega)  \bigr\|^{\cL(\fK\times\Omega,\overline W)} \\
&\leq 
2 \,C^{\scriptscriptstyle \rm prelim}_Q \,
\bigl\| (\fk,\omega)  \bigr\|^{\fK\times\Omega} \;=\; C_Q \, \bigl\| (\fk,\omega)  \bigr\|^{\fK\times\Omega} . 
\end{align*}
This confirms \eqref{eq:Q-estimate} with the constant $C_Q= 2 \, C^{\scriptscriptstyle \rm prelim}_Q$.  

Next, we will show that these inverses vary continuously with $\gamma_0\in\cV_{\overline\Gamma}$. For that purpose we abbreviate $\ol V:=\fK \times \ol\Omega$ and consider $\gamma'_0,\gamma_0\in\cV_{\overline\Gamma}$ with $\|\gamma'_0\|^\Gamma, \|\gamma_0\|^\Gamma\leq\delta_Q$ to compare the classically constructed inverses of $P_1:=P(\gamma_0)$ and $P'_1:=P(\gamma'_0)$ with $\| P'_1 - P_1 \|^{\cL(\ol W,\ol V)}\leq \tilde c^1_\cF(\gamma'_0, \gamma_0)$ by \eqref{eq:P continuity classical}. 
These inverses are
$Q_1:=Q(\gamma_0) = ({\rm Id} - T )^{-1} Q_0$ with $T=Q_0 ( P_0 - P_1 )$ and 
$Q'_1:=Q(\gamma'_0) = ({\rm Id} - T' )^{-1} Q_0$ with $T'=Q_0 ( P_0 - P'_1 )$, where $P_0=P(0)$ and $Q_0=Q(0)$ as above.
So to estimate $\| Q'_1 - Q_1\|^{\cL(\ol V,\ol W)}$ we first keep track of 
\begin{align*}
\| T' - T \|^{\cL(\ol W,\ol W)} 
&= \| Q_0 ( P_0 - P'_1 ) - Q_0 ( P_0 - P_1 ) \|^{\cL(\ol W,\ol W)}  \\
&= \| Q_0 \|^{\cL(\ol V,\ol W)}   \| P_1 - P'_1  \|^{\cL(\ol W,\ol V)} 
\;\leq\; \tfrac 12  C_Q \,  \tilde c^1_\cF(\gamma'_0, \gamma_0) ,
\end{align*}
where we used \eqref{eq:P continuity classical} and $\| Q_0 \|^{\cL(\ol V,\ol W)}\leq \frac 12 C_Q$ from \eqref{eq:P-estimate classical}. 
Next we use the classical fact that inverting linear operators is a continuous map. In particular the following
map is uniformly continuous 
$$
\bigl\{ T\in \cL(\overline W,\overline W) \,\big|\, \|T\|<1 \bigr\} \to \cL(\overline W,\overline W) , \qquad 
T \mapsto ({\rm Id}_{\overline W} - T)^{-1} = \textstyle \sum_{n=0}^\infty T^n . 
$$
Indeed, we can use the fact that $T$ commutes with $({\rm Id} - T)^{-1}=\sum_{n=0}^\infty T^n$ to estimate
\begin{align}
 \bigl\| ({\rm Id} - T' )^{-1} - ({\rm Id} - T )^{-1}   \bigr\| 
&=  \bigl\| ({\rm Id} - T' )^{-1}({\rm Id} - T )^{-1} ({\rm Id} - T )  -   ({\rm Id} - T' ) ({\rm Id} - T' )^{-1}({\rm Id} - T )^{-1}   \bigr\|  \nonumber \\
&=  \bigl\| - ({\rm Id} - T' )^{-1}({\rm Id} - T )^{-1} T  + T' ({\rm Id} - T' )^{-1}({\rm Id} - T )^{-1}   \bigr\|  \nonumber \\
&=  \bigl\| ({\rm Id} - T' )^{-1} (T' - T) ({\rm Id} - T )^{-1} T   \bigr\|  \label{eq:inverse est} \\
&\leq  \bigl\| ({\rm Id} - T' )^{-1} \bigr\|   \| T' - T \|  \bigl\|  ({\rm Id} - T )^{-1}   \bigr\|   \nonumber\\
&\leq \tfrac 1{1-\|T'\|} \| T' - T \|   \tfrac 1{1-\|T\|}  . \nonumber
\end{align}
Here we bounded the norms of inverses by  $\| ({\rm Id} - T )^{-1} \| = \| \sum_{n=0}^\infty T^n \| \leq \sum_{n=0}^\infty \|T\|^n=\tfrac 1{1-\|T\|}$, and we can further use \eqref{eq: T small} to bound $\|T'\|,\|T\| \leq \frac 12$.
Putting this all together we obtain for any $\gamma'_0,\gamma_0\in \cV_{\overline\Gamma}$ with $\|\gamma'_0\|^\Gamma,\|\gamma_0\|^\Gamma\leq\delta_Q$
\begin{align*}
 \| Q(\gamma'_0) - Q(\gamma_0) \|^{\cL(\ol V,\ol W)} 
&= 
\bigl\|  ({\rm Id}_{\ol W} - T' )^{-1} Q_0 - ({\rm Id}_{\ol W} - T )^{-1} Q_0 \bigr\|^{\cL(\ol V,\ol W)} \\
&\leq 
\bigl\|  ({\rm Id}_{\ol W} - T' )^{-1} - ({\rm Id}_{\ol W} - T )^{-1} \bigr\|^{\cL(\ol W,\ol W)}   \| Q_0 \|^{\cL(\ol V,\ol W)} \\
&\leq 
\frac{\frac12 C_Q}{( 1-\|T'\|^{\cL(\ol W,\ol W)})( 1-\|T\|^{\cL(\ol W,\ol W)}  ) } \; \| T' - T \|^{\cL(\ol W,\ol W)}      \\
&\leq 
\tfrac{\frac12 C_Q}{( 1-\frac12)( 1-\frac12  ) } \;  \tfrac12 C_Q \; \tilde c^1_\cF(\gamma'_0 ,\gamma_0) 
\;=\; (C_Q)^2 \; \tilde c^1_\cF(\gamma'_0, \gamma_0)  .
\end{align*} 
This confirms \eqref{eq:Q cont}, and then continuity of $\gamma_0\mapsto  Q(\gamma_0)$ follows from $\tilde c^1_\cF(\gamma'_0 ,\gamma_0)\to 0$ for $\gamma'_0\to\gamma_0$. 
\end{proof}

When implementing this analysis for an adiabatic Fredholm family, it is instructive to recognize the Fredholm stabilization isomorphism \eqref{eq:P classical} for $\cF=\ol\cF_\eps$ as a differential of the map 
\begin{align*}
\ol g_{\eps,\fk}:  \fC \times \cV_{\ol\Gamma,\eps} &\rightarrow  \fK \times \ol\Omega_\eps, \quad
( \fc, \gamma) \mapsto (\ol\pi_\fK(\gamma)-\fk, \ol\cF_\eps(\gamma) - \fc ) 
\end{align*}
which arises from the last two components of the map 
$\ol\cG: \Delta \times \fK \times \fC \times \cV_{\ol\Gamma,\eps} \to \fC \times \fK \times \ol\Omega)_\eps$
from Lemma~\ref{lem:equivalent-fredholm}, that is we express
$\cG(\eps, \fk, \fc, \gamma)= ( \fc , \ol\pi_\fK(\gamma)-\fk, \ol\cF_\eps(\gamma) - \fc ) =: (\fc, \ol g_{\eps,\fk} (\fc,\gamma) )$. In terms of solving the equation $\cF_\eps(\gamma)=0 \;\Leftrightarrow\; \ol\cG(\eps, \fk, \fc, \gamma)=0$, this means we split off the equation $\fc=0$, which will in \S\ref{reduction} give rise to the finite dimensional reduction $f:\cV_\fK\to\fC$ -- after \S\ref{solution} solves the remaining equations 
\begin{equation} \label{eq:stabilized equation}
 \overline g_{\eps,\fk}(\fc,\gamma) = 0 \qquad \Longleftrightarrow\qquad  \overline\pi_\fK(\gamma)=\fk \quad \text{and}\quad \overline \cF_\eps(\gamma)=\fc .
\end{equation}
Alternatively, we can also understand $\overline  g_{\eps,\fk}$ as the continuous extension of the maps
\begin{align*}
 g_{\eps,\fk}:  \fC \times \cV_\Gamma &\rightarrow  \fK \times \Omega, \quad
( \fc, \gamma) \mapsto (\pi_\fK(\gamma)-\fk, \cF_\eps(\gamma) - \fc ) 
\end{align*}
given for each $\eps\in\Delta$ and $\fk\in\fK$ by the last two components of the map 
$\cG(\eps,\fk,\cdot,\cdot)$ from Lemma~\ref{lem:equivalent-fredholm}. 
That is the context for generalizing Lemma~\ref{lem:inverses classical} to adiabatic Fredholm families. 
Note here moreover that the estimate \eqref{eq:quadratic classical} no longer even makes sense as we have no universal topology on $\cV_{\ol\Gamma,\eps}\times\cV_{\ol\Gamma,\eps}$. 
This is why we need to enhance the $\cC^1$ regularity of each Fredholm section $\overline\cF_\eps$ by 
the [Quadratic-ish Estimate] and [Uniform Continuity of $\rD\cF_\eps$] in Definitions~\ref{def:fredholm} and \ref{def:adiabatic C-l}, which provide estimates that -- with respect the $\eps$-dependent norms -- are uniform in $\eps\in\Delta$. 

\begin{lemma} \label{lem:inverses}
Given an adiabatic Fredholm family $\bigl(  (\cF_\eps:\cV_\Gamma\to \Omega )_{\eps\in\Delta} , \ldots  \bigr)$ as in Definition~\ref{def:fredholm}, there exists a neighbourhood $\Delta_Q\subset\Delta$ of $0$ and a constant $\delta_Q>0$ such that for any $(\eps,\fk_0)\in  \Delta_Q\times\fK$ the linearizations of $\overline g_{\eps,\fk_0}$ at $(\fc_0,\gamma_0)\in  \fC \times \cV_{\overline\Gamma,\eps}$ with $\|\gamma_0\|^\Gamma_\eps\leq\delta_Q$ are isomorphisms
\begin{align}  \label{eq:P}
\overline P_\eps(\gamma_0) := \rD \overline g_{\eps,\fk_0} (\fc_0,\gamma_0) :  \; \fC \times \overline\Gamma_\eps &\rightarrow  \fK \times \overline\Omega_\eps, \\
( \fc, \gamma) &\mapsto (\overline \pi_\fK(\gamma), \rD\overline\cF_\eps(\gamma_0) \gamma - \fc )  
\nonumber
\end{align}
whose inverse operators $\overline Q_\eps(\gamma_0):=\overline P_\eps(\gamma_0)^{-1}:   \fK \times \overline\Omega_\eps\to \fC \times \overline\Gamma_\eps$ are uniformly bounded\footnote{
Recall from Lemma~\ref{lem:fredholm} that we equip $\fC\subset\Omega$ with the norm $\|\cdot\|^\fC:=\|\cdot\|^\Omega_0$ and $\fK\subset\Gamma$ with the norm $\|\cdot\|^\fK:=\|\cdot\|^\Gamma_0$.
} 
\begin{align}\label{eq:Q-estimate}
 \| \overline Q_\eps(\gamma_0) (\fk,\omega) \|^{\fC\times\Gamma}_\eps
&\leq
C_Q \| (\fk,\omega) \|^{\fK\times\Omega}_\eps   \qquad\forall  \eps\in\Delta_Q, \gamma_0\in\cV_{\overline\Gamma,\eps} , \|\gamma_0\|^\Gamma_\eps\leq\delta_Q,   (\fk,\omega)\in\fK\times\Omega 
\end{align}
by an $\eps$-independent constant $C_Q:=\max\{1,4 (C_1 + C_0 +C_0C_1 +C_0C_\fC)\}$.

If the adiabatic Fredholm family satisfies {\rm [Uniform Continuity of $\rD\cF_\eps$]} as in Definition~\ref{def:adiabatic C-l} then the inverse operators $ \overline Q_\eps(\gamma_0)$ vary uniformly continuously with $\gamma_0\in\cV_{\overline\Gamma,\eps}$, that is for all $\eps\in\Delta_Q$ and 
$\gamma_0,\gamma'_0\in\cV_{\overline\Gamma,\eps}$ with $\|\gamma_0\|^\Gamma_\eps, \|\gamma'_0\|^\Gamma_\eps\leq\delta_Q$ we have
\begin{equation}\label{eq:Q cont}
 \| \overline Q_\eps(\gamma'_0) -  \overline Q_\eps(\gamma_0) \|^{\cL(\fK\times\ol\Omega_\eps,\fC\times\overline\Gamma_\eps)}
\leq
(C_Q)^2 \; c^1_\cF(\|\gamma'_0- \gamma_0\|^\Gamma_\eps),
\end{equation}
where $c^1_\cF : [0,\infty) \to [0,\infty)$ is the monotone continuous function from Definition~\ref{def:adiabatic C-l} with $c^1_\cF(0)= 0$.

If the adiabatic Fredholm family is regularizing in the sense of Definition~\ref{def:regularizing}, then the inverse operators at base points $\gamma_0\in\cV_\Gamma$ restrict to bijections $Q_\eps(\gamma_0):=\overline Q_\eps(\gamma_0)|_{\fK\times\Omega}:   \fK \times \Omega \to \fC \times \Gamma$, which are the inverses of 
\begin{align*}
P_\eps(\gamma_0) := \rD g_{\eps,\fk_0} (\fc_0,\gamma_0) :  \; \fC \times \Gamma &\rightarrow  \fK \times \Omega, \\
( \fc, \gamma) &\mapsto (\pi_\fK(\gamma), \rD\cF_\eps(\gamma_0) \gamma - \fc )  .
\end{align*}
\end{lemma}

\begin{proof}
First note that the linearizations
$\overline P_\eps(\gamma_0) := \rD \overline g_{\eps,\fk_0} (\fc_0,\gamma_0) :  \fC \times \overline\Gamma_\eps \to  \fK \times \overline\Omega_\eps$ are independent of the base point data $\fk_0\in\fK$ and $\fc_0\in\fC$, as they are given by $(\fc, \gamma) \mapsto (\overline\pi_\fK(\gamma), \rD\overline\cF_\eps(\gamma_0) \gamma - \fc )$. 

For $\gamma_0=0$ and $\eps\in\Delta_Q$ to be determined we can prove injectivity of $\overline P_\eps(0)$ by starting from the [Uniform Fredholm-ish Estimate] in Lemma~\ref{lem:fredholm} to estimate for any $(\fc,\gamma)\in\fC\times\overline\Gamma_\eps$
\begin{align*}
 \| \fc \|^\fC + \|\gamma\|^\Gamma_\eps
&\leq
\| \fc \|^\fC 
+ C_1 \bigl( \| \rD\overline\cF_\eps(0) \gamma \|^\Omega_\eps + \| \gamma \|^\Gamma_0 \bigr)   \\
&\qquad\qquad
\text{\small using the triangle inequality for $\|\cdot\|^\Omega_\eps$} \\
&\leq
\| \fc \|^\fC
+ C_1 \bigl(  \| \rD\overline\cF_\eps(0) \gamma - \fc \|^\Omega_\eps + \| \fc \|^\Omega_\eps +  \| \gamma \|^\Gamma_0 \bigr)   \\
&\qquad\qquad
\text{\small using the [Uniform Cokernel Bound] in Lemma~\ref{lem:fredholm}} \\
&\leq
C_1 \| \rD\overline\cF_\eps(0) \gamma - \fc \|^\Omega_\eps + (1+C_1 + C_\fC) \bigl(  \| \gamma \|^\Gamma_0 + \| \fc \|^\fc  \bigr)  \\
&\qquad\qquad
\text{\small using the [$\epsilon=0$ Fredholm Estimate] in Lemma~\ref{lem:fredholm}} \\
&\leq
C_1 \| \rD\overline\cF_\eps(0) \gamma - \fc \|^\Omega_\eps  
+ (1+ C_1  + C_\fC) C_0 \bigl( \|\overline\pi_\fK(\gamma) \|^\fK + \| \rD\overline\cF_0(0) \gamma - \fc \|^\Omega_0  \bigr) 
\end{align*}
\begin{align*}
&\qquad\qquad
\text{\small using the triangle inequality for $\|\cdot\|^\Omega_0$ and the [Lower Bound on $\Omega$ Norms] in Lemma~\ref{lem:fredholm}} \\
&\leq
C_1 \| \rD\overline\cF_\eps(0) \gamma - \fc \|^\Omega_\eps   
+ (1+ C_1  + C_\fC) C_0  \|\overline\pi_\fK(\gamma) \|^\fK  \\
&\qquad 
+ (1+ C_1  + C_\fC) C_0  \bigl(  \| \rD\overline\cF_0(0) \gamma - \rD\overline\cF_\eps(0) \gamma \|^\Omega_0    + \| \rD\overline\cF_\eps(0) \gamma - \fc \|^\Omega_\eps  \bigr) \\ 
&\qquad\qquad
\text{\small using the [Continuity of Derivatives at $0$] in Lemma~\ref{lem:fredholm}} \\
&\leq
C_1 \| \rD\overline\cF_\eps(0) \gamma - \fc \|^\Omega_\eps    
+ (1+ C_1  + C_\fC) C_0  \|\overline\pi_\fK(\gamma) \|^\fK  \\
&\qquad 
+ (1+ C_1  + C_\fC) C_0 \bigl( c_\Delta(\eps)  \|  \gamma \|^\Gamma_\eps  + \| \rD\overline\cF_\eps(0) \gamma - \fc \|^\Omega_\eps  \bigr) \\ 
&\leq
(C_1 + (1+ C_1  + C_\fC) C_0 ) \bigl( \|\overline\pi_\fK(\gamma) \|^\fK +  \| \rD\overline\cF_\eps(0) \gamma - \fc \|^\Omega_\eps  \bigr) 
+ (1+ C_1  + C_\fC) C_0 c_\Delta(\eps)  \|  \gamma \|^\Gamma_\eps   \\ 
&\leq
 \tfrac 12 C^{\scriptscriptstyle prelim}_Q \| \rD \overline g_{\eps,\fk_0} (\fc_0,0) (\fc,\gamma) \|^{\fK\times\Omega}_\eps
+ \tfrac 12 \|  \gamma \|^\Gamma_\eps    
\end{align*}
with the preliminary constant $C^{\scriptscriptstyle prelim}_Q:= 2 (C_1 + C_0 +C_0C_1 +C_0C_\fC)$. 
Here the last step chooses $\Delta_Q\subset\Delta$ so that $c_\Delta(\eps)\leq  \frac 1{2(1+ C_1  + C_\fC) C_0}$ for all $\eps\in\Delta_Q$. 
Then we can rearrange this estimate to prove and quantify the injectivity of the linearization at base points $(\eps, \fk_0, \fc_0, \gamma_0=0)$ with $\eps\in\Delta_Q$, 
\begin{align} \label{eq:P-estimate}
 \| (\fc, \gamma) \|^{\fC\times\Gamma}_\eps
&\leq
C^{\scriptscriptstyle prelim}_Q \| \rD \overline g_{\eps,\fk_0} (\fc_0,0) (\fc,\gamma) \|^{\fK\times\Omega}_\eps  
\qquad \forall (\fc,\gamma)\in\fC\times\overline\Gamma_\eps . 
\end{align}
Moreover, the linearizations $\overline P_\eps(0) = \rD \overline g_{\eps,\fk_0} (\fc_0,0) :  \fC \times \overline\Gamma_\eps \to  \fK \times \overline\Omega_\eps$ at $\gamma_0=0$ are stabilizations of the Fredholm maps $\rD\overline\cF_\eps(0)$ in the sense that they result by Cartesian product with finite dimensional factors in the domain and codomain, and addition of a linear operator $(\fc,\gamma)\mapsto (\overline\pi_\fK(\gamma), -\fc)\in \fK\times\fC$ composed with the compact embedding $ \fK\times\fC\subset\fK\times\overline\Omega_\eps$. 
Here the embedding $\fC\hookrightarrow \overline\Omega_\eps$ is compact since $\fC$ is finite dimensional and the embedding is continuous by the [Uniform Cokernel Bound] property.
Thus $\overline P_\eps(0)$ are Fredholm operators, and we can use the [Index] property $\ind \rD\overline\cF_\eps(0)=\ind \rD\overline\cF_0(0)$ to show that their Fredholm index is $0$, 
\begin{align*}
\ind \overline P_\eps(0) &= \ind \bigl( (\fc,\gamma)\mapsto (0, \rD\overline\cF_\eps(0)\gamma) \bigr) \\
&=   \ind \rD\overline\cF_\eps(0)+ \dim\fC - \dim\fK \;=\; \ind \rD\overline\cF_\eps(0) - \ind \rD\overline\cF_0(0) \;=\; 0 . 
\end{align*}
Since we previously established that each $\overline P_\eps(0)$ is injective, the Fredholm index $0$ guarantees that each $\overline P_\eps(0)$ is also surjective, and hence a bijection. Moreover, the above injectivity estimate \eqref{eq:P-estimate} implies the uniform bound \eqref{eq:Q-estimate} for the inverse operators $\overline Q_\eps(0):=\overline P_\eps(0)^{-1} : \fK \times \overline\Omega_\eps\to \fC \times \overline\Gamma_\eps$ with  the preliminary constant $C^{\scriptscriptstyle prelim}_Q$. 

This invertibility extends to $\gamma_0\ne 0$ whenever $\overline P_\eps(\gamma_0)$ is sufficiently close to $\overline P_\eps(0)$ in the space of bounded linear operators. To achieve this for all $\|\gamma_0\|^\Gamma_\eps\leq\delta_Q$ with an $\eps$-independent constant $\delta_Q>0$ 
we use the [Quadratic-ish Estimate] in Lemma~\ref{lem:fredholm} to estimate for all $\eps\in\Delta_Q$ and $(\fc_0,\gamma_0)\in  \fC \times \cV_{\overline\Gamma,\eps}$ 
\begin{align*}
\bigl\| \overline P_\eps(\gamma_0) - \overline P_\eps(0) \bigr\|^{\cL(\fC \times \overline\Gamma_\eps ,\fK \times \overline\Omega_\eps)}
&= \sup_{\| \fc \|^\fC + \|\gamma\|^\Gamma_\eps\leq 1}     \bigl\| \bigl( \overline \pi_\fK(\gamma) , \rD\overline\cF_\eps(\gamma_0) \gamma - \fc  \bigr) - \bigl( \overline \pi_\fK(\gamma) , \rD\overline\cF_\eps(0) \gamma - \fc  \bigr) \bigr\|^{\fK \times \Omega}_\eps \\
&= \sup_{\| \fc \|^\fC + \|\gamma\|^\Gamma_\eps\leq 1}     \bigl\| \rD\overline\cF_\eps(\gamma_0) \gamma - \rD\overline\cF_\eps(0) \gamma  \bigr\|^\Omega_\eps  \leq 
 c(\|\gamma_0\|^\Gamma_\eps)  \;\leq\;  \tfrac 12(C^{\scriptscriptstyle prelim}_Q)^{-1} . 
\end{align*}
Here $c : [0,\infty) \to [0,\infty)$ is a monotone continuous function with $c(0)= 0$ that is a part of the data included in an adiabatic Fredholm family in Definition~\ref{def:fredholm}. So we can choose  $\delta_Q>0$ so that $c(x)\leq\frac 12(C^{\scriptscriptstyle prelim}_Q)^{-1}$ for all $0\leq x\leq \delta_Q$. 
This means that for each $\eps\in\Delta_Q$ the Fredholm map $\overline\cF_\eps:\cV_{\Gamma,\eps}\to\ol\Omega_\eps$ fits into Lemma~\ref{lem:inverses classical}, where the [Quadratic-ish Estimate] validates 
\eqref{eq:P continuity classical} for $\gamma'_0=0$ with the function 
$\tilde c^1_{\ol\cF_\eps}(0,\gamma_0)= c(\|\gamma_0\|^\Gamma_\eps)$. 
Thus the classical construction shows that
$\overline P_\eps(\gamma_0) = \rD \overline g_{\eps,\fk_0} (\fc_0,\gamma_0)$ is invertible for all $\eps\in\Delta_Q$ and $\|\gamma_0\|^\Gamma_\eps\leq\delta_Q$ with inverses $\overline Q_\eps(\gamma_0):=\overline P_\eps(\gamma_0)^{-1}$ that satisfy the uniform bound 
\eqref{eq:Q-estimate} with the constant $C_Q=2\, C^{\scriptscriptstyle prelim}_Q$. 
We then replace this by the final constant $C_Q:=\max\{1,2\, C^{\scriptscriptstyle prelim}_Q\}$ to simplify later estimates in \eqref{eq:DQ bound}.  

Next, we will assume [Uniform Continuity of $\rD\cF_\eps$] as in Definition~\ref{def:adiabatic C-l} and show that these inverses vary continuously with $\gamma_0\in\cV_{\overline\Gamma,\eps}$ -- and that this continuity is locally uniform. This again follows by applying the classical proof in Lemma~\ref{lem:inverses classical}
for each $\eps\in\Delta_Q$ to the Fredholm map $\overline\cF_\eps:\cV_{\Gamma,\eps}\to\ol\Omega_\eps$, which now satisfies \eqref{eq:P continuity classical} for $\gamma'_0,\gamma_0\in\cV_{\overline\Gamma,\eps}$ with the function $\tilde c^1_{\ol\cF_\eps}(\gamma'_0,\gamma_0)= c^1_\cF(\|\gamma'_0 - \gamma_0\|^\Gamma_\eps)$ from Definition~\ref{def:adiabatic C-l}. 
Thus \eqref{eq:Q cont} follows directly from Lemma~\ref{lem:inverses classical}.

Finally, suppose that the adiabatic Fredholm family is regularizing in the sense of Definition~\ref{def:regularizing}, 
which in particular means the implication
$\rD\overline\cF_\eps(\gamma_0)\gamma \in \Omega \Rightarrow \gamma\in\Gamma$ for $\gamma_0\in\cV_\Gamma$. 
Then consider the value $(\fc,\gamma)=\overline Q_\eps(\gamma_0)(\fk,\omega)$ of $\overline Q_\eps(\gamma_0)$ at some $(\fk,\omega)\in\fK\times\Omega$. This is the solution $(\fc,\gamma)\in\fC\times\overline\Gamma_\eps$ of 
$$
\overline\pi_\fK(\gamma)=\fk \quad\text{and}\quad \rD\overline\cF_\eps(\gamma_0) \gamma - \fc =\omega  . 
$$
In particular, $\gamma$ solves $\rD\overline\cF_\eps(\gamma_0) \gamma = \fc + \omega \in \Omega$, which implies $\gamma\in\Gamma$ by the regularizing property. 
This shows that the restriction $\overline Q_\eps(\gamma_0)|_{\fK \times \Omega}$ for $\gamma_0\in\cV_\Gamma$ takes values in $\fC \times \Gamma$. Then the injectivity and surjectivity of $\overline Q_\eps(\gamma_0)$ implies injectivity and surjectivity of the restriction, which establishes it as the claimed bijection $Q_\eps(\gamma_0):=\overline Q_\eps(\gamma_0)|_{\fK\times\Omega}:   \fK \times \Omega \to \fC \times \Gamma$. Moreover, $(\fc,\gamma)=Q_\eps(\gamma_0)(\fk,\omega)$ satisfies $\pi_\fK(\gamma)=\fk$ and $\rD\cF_\eps(\gamma_0) \gamma - \fc =\omega$, thus $Q_\eps$ is the inverse of $P_\eps(\gamma_0) := \rD g_{\eps,\fk_0} (\fc_0,\gamma_0)$.
\end{proof}

\subsection{Contractions}  \label{contraction}

The next step towards constructing finite dimensional reductions of adiabatic Fredholm families is to utilize the inverse operators from the previous section to rewrite the equations $\cF_\eps(\gamma)=0$ -- up to a finite dimensional factor -- as a fixed point problem for a family of contractions. 
As in \cite[Def.3.1.11]{hofer-wysocki-zehnder_polyfold} this will play the role of the Fredholm property in classical Fredholm descriptions of moduli spaces. 

In the classical setting of Theorem~\ref{thm:charts classical}, this is based on showing that any $\cC^1$ Fredholm map satisfies the [Quadratic-ish Estimate] of Definition~\ref{def:fredholm} -- after shrinking its domain. 

\begin{lemma} \label{lem:contraction classical}
Any Fredholm map $\cF:\cV_{\ol\Gamma}\to \Omega$ as in Theorem~\ref{thm:charts classical}
satisfies a {\rm [Quadratic-ish Estimate]}
\begin{equation} \label{eq:quadratic-ish classical}
 \bigl\| \rD\cF(\gamma_0) \gamma - \rD\cF(0) \gamma  \bigr\|^\Omega  \leq 
c(\|\gamma_0\|^\Gamma)  \| \gamma\|^{\ol\Gamma}
\qquad
\forall  \gamma_0\in \cV'_{\overline\Gamma}, \gamma\in\Gamma , 
\end{equation}
where $c:[0,\infty)\to [0,\infty)$ is a monotone continuous map with $c(0)=0$ 
and $\cV'_{\ol\Gamma}\subset\Gamma$ is a convex open neighbourhood of $0$ contained in $\cV_{\ol\Gamma}$.  
Further, let $Q(0): \fK \times \overline\Omega \to \fC \times \overline\Gamma$ be the isomorphism from Lemma~\ref{lem:inverses classical}, then the zero set $\cF^{-1}(0) \simeq \cG_Q^{-1}(0)$ is naturally identified with the zero set of the map 
$$
\cG_Q \,:\; \fK \times \fC \times \cV_{\overline\Gamma} \;\rightarrow\;  \fC \times \fC \times \overline\Gamma, 
\qquad
(\fk, \fc, \gamma) \;\mapsto\; \bigl(\fc ,  Q(0) (\pi_\fK(\gamma)-\fk, \cF(\gamma) - \fc ) \bigr) . 
$$
Moreover, this map is close to the identity map on
$\cV_{\overline W}:=\fC\times\cV'_{\overline\Gamma}\subset \fC\times \overline\Gamma=:\ol W$ 
up to finite dimensional factors and a contraction in the following sense: We can write
$$
\cG_Q  \,:\; \fK \times \cV_{\overline W} \;\rightarrow\;  \fC \times \overline W, 
\qquad
(\fk, w) \;\mapsto\; \bigl( A(\fk,w), w - B(\fk,w) \bigr) , 
$$
where $A : \fK \times\cV_{\overline W} \to \fC$ maps to the finite dimensional space $\fC$ by projection $A(\fk,\fc,\gamma)=\fc$,
and $B : \fK \times \cV_{\overline W} \to  \overline W$ given by $B(\fk,\fc,\gamma) = Q(0) \bigl( \fk   , \rD\cF(0) \gamma -  \cF(\gamma)   \bigr)$ is a family of contractions near $0\in \overline W$ parametrized by $\fk\in\fK$. 
More precisely, given any $\theta\in (0,1)$ there exists $\delta_\theta>0$ such that for all $\fk\in\fK, w,w'\in\cV_{\overline W}$ we have
\begin{align}\label{eq:B-contraction classical}
  \|w\|^W,  \|w'\|^W \leq \delta_\theta  \quad \Longrightarrow \quad
\| B (\fk, w' ) - B(\fk, w) \|^W \leq  \theta \| w' - w \|^W . 
\end{align}
In addition, given any $\delta>0$ we can find an open neighbourhood $\cV_{\fK,\delta}\subset\fK$ of $0$ with
\begin{align}\label{eq:B-small classical}
\fk\in\cV_{\fK,\delta}  \quad \Longrightarrow \quad
\| B (\fk, 0 )  \|^W <  \delta . 
\end{align}
If the map $\cF:\cV_{\ol\Gamma}\to\ol\Omega$ is $\cC^\ell$ for some $\ell\geq 1$, then the contraction $ B : \fK \times \cV_{\overline W} \to  \overline W$ is $\cC^\ell$ as well.
\end{lemma}

\begin{proof}
From Lemma~\ref{lem:inverses classical} we have 
$$ \bigl\| \rD\cF(\gamma_0') \gamma - \rD\cF(\gamma_0) \gamma  \bigr\|^\Omega  \leq 
\tilde c^1_\cF(\gamma'_0,\gamma_0)  \| \gamma\|^{\ol\Gamma}
\qquad
\forall \gamma'_0,  \gamma_0\in \cV_{\overline\Gamma}, \gamma\in\Gamma , 
$$
where $\tilde c^1_\cF:\cV_{\ol\Gamma}\times \cV_{\ol\Gamma}\to [0,\infty)$ is a continuous map that vanishes on the diagonal $\tilde c^1_\cF(\gamma_0,\gamma_0)=0$.
So to obtain \eqref{eq:quadratic-ish classical} we are tempted to define $c:[0,\infty)\to[0,\infty)$ by $c^{\scriptscriptstyle prelim}(x):=\sup_{\|\gamma_0\|^\Gamma\leq x} \tilde c^1_\cF(\gamma_0,0)$. While those supremums might not be finite for all $x>0$, the continuity $\tilde c^1_\cF(\gamma_0,0)\to 0$ as $\|\gamma_0\|^\Gamma\to 0$ guarantees $c^{\scriptscriptstyle prelim}(x)\to 0$ as $x\to 0$. Now pick $x_0>0$ sufficiently small so that $c^{\scriptscriptstyle prelim}(x_0)<\infty$ and $\cV'_{\ol\Gamma}:=\{ \gamma_0\in \Gamma \,|\, \|\gamma_0\|^\Gamma < x_0 \}\subset\cV_{\ol\Gamma}$. 
Then we obtain the desired estimate by setting $c(x):=\min\{ c^{\scriptscriptstyle prelim}(x), x_0\}$. This function is monotone by construction and inherits continuity from $\tilde c^1_\cF$. 

Next, Lemma~\ref{lem:equivalent-fredholm classical} identifies $\cF^{-1}(0)\simeq\cG^{-1}(0)$ for $\cG:\fK \times \fC \times \cV_\Gamma \rightarrow \fC \times \fK \times \Omega$ given by $\cG(\fk, \fc, \gamma) = ( \fc , \pi_\fK(\gamma)-\fk, \cF(\gamma) - \fc )$. 
The map considered here is obtained by composing $\cG$ with ${\rm Id}_\fC \times Q(0):\fC\times  \fK \times \Omega \to \fC\times  \fC \times \Gamma$. Here $Q(0):\fK\times\Omega\to\fC\times\Gamma$ is a bijection by Lemma~\ref{lem:inverses classical}, so this composition does not affect the zero set. 

To check the expression for $B : \fK \times \cV_{\overline W} \to  \overline W$ given by $B(\fk,w) := w- {\rm Pr}_{\ol W} ( \cG_Q(\fk,w) )$ we first use the fact $Q(0)$ is the inverse of the Fredholm stabilization isomorphism 
 $P(0) : \fC \times \overline\Gamma \to  \fK \times \overline\Omega$, given by 
$P(0) ( \fc, \gamma)= (\pi_\fK(\gamma), \rD\cF(0) \gamma - \fc )$ in \eqref{eq:P classical}, 
to rewrite
\begin{align*}
 B(\fk,\fc,\gamma)
&=
 (\fc,\gamma) - Q(0) (  \pi_\fK(\gamma)-\fk,  \cF(\gamma) - \fc )   \\
&= 
Q(0) \bigl(P(0) (\fc,\gamma) -  ( \pi_\fK(\gamma)-\fk   ,  \cF(\gamma) - \fc   )  \bigr)   \\
&= 
 Q(0) \bigl( ( \pi_\fK(\gamma)   , \rD \cF(0) \gamma - \fc   ) -  ( \pi_\fK(\gamma)-\fk   ,  \cF(\gamma) - \fc   )  \bigr)   
 \;=\;
Q(0) \bigl( \fk   , \rD \cF(0) \gamma -   \cF(\gamma)   \bigr)  .
\end{align*}
Now given any $w=(\fc,\gamma), w'=(\fc',\gamma')\in \cV_W=\fC\times\cV'_{\overline \Gamma}$, we estimate
\begin{align*}
\bigl\|   B(\fk,w') -   B(\fk,w) \bigr\|^W
&=
\bigl\| Q(0) (\fk, \rD \cF(0) \gamma' -   \cF(\gamma')  )   
- Q(0) ( \fk, \rD \cF(0) \gamma -   \cF(\gamma)  )   
\bigr\|^W \\
&=
\bigl\| Q(0) \bigl( 0 , \rD \cF(0) \gamma' -   \cF(\gamma') - \rD \cF(0) \gamma +   \cF(\gamma)  \bigr)   \bigr\|^W \\
&\qquad\qquad
\text{\small using \eqref{eq:Q-estimate classical}  in Lemma~\ref{lem:inverses classical}} \\
&\leq 
C_Q \bigl( \| 0 \|^\fK + \bigl\|   \cF(\gamma) -   \cF(\gamma') - \rD \cF(0) (\gamma -\gamma')    \bigr\|^\Omega  \bigr) \\
&\qquad\qquad
\text{\small using the convexity of $\cV'_{\overline \Gamma}$ to ensure $\gamma'+\lambda (\gamma-\gamma')\in\cV'_{\overline \Gamma}$ for $0\leq\lambda\leq 1$} \\
&= 
C_Q  \bigl\|  \textstyle\int_0^1 \bigl( \rD \cF(\gamma+\lambda (\gamma-\gamma') ) (\gamma-\gamma')  - \rD \cF(0) (\gamma -\gamma')   \bigr) \rd\lambda \bigr\|^\Omega  \\
&\leq 
C_Q   \textstyle\int_0^1 \bigl\|  \rD \cF(\gamma'+\lambda (\gamma-\gamma') ) (\gamma-\gamma') - \rD\cF(0) (\gamma -\gamma')   \bigr)  \bigr\|^\Omega  \rd\lambda \\
&\qquad\qquad
\text{\small using  \eqref{eq:quadratic-ish classical} } \\
&\leq 
C_Q   \textstyle\int_0^1 c\bigl( \gamma'+\lambda (\gamma-\gamma'), 0 \bigr)  \rd\lambda  \;  \| \gamma -\gamma'\|^\Gamma  , 
\end{align*}
where $c : [0,\infty) \to [0,\infty)$ is the monotone continuous function with $c(0)= 0$ from  \eqref{eq:quadratic-ish classical}. 
So, given any $\theta\in(0,1)$ we can find $\delta_\theta>0$ so that $c(\delta)\leq \frac1{C_Q}\theta$ for all $0\leq\delta\leq\delta_\theta$. Then $w=(\fc,\gamma),w'=(\fc',\gamma')\in\{v\in \cV_W | \|v\|^W\leq \delta_\theta \}$ implies $\|\gamma\|^\Gamma, \|\gamma'\|^\Gamma\leq \delta_\theta$, which guarantees $\| \gamma'+\lambda (\gamma-\gamma')\|^\Gamma \leq  \lambda \|\gamma\|^\Gamma + (1-\lambda) \|\gamma'\|^\Gamma \leq \delta_\theta$ for all $\lambda\in[0,1]$, and hence the desired contraction property  \eqref{eq:B-contraction classical}
\begin{align*}
\bigl\|   B(\fk,w') -   B(\fk,w) \bigr\|^W
&\leq 
C_Q   \textstyle\int_0^1 \frac1{C_Q}\theta  \rd\lambda    \| \gamma -\gamma'\|^\Gamma 
=
\theta \| \gamma -\gamma'\|^\Gamma .
\end{align*}
To establish smallness on $B(\fk,0)$ we again use \eqref{eq:Q-estimate classical} and the assumption $\cF(0)=0$ to obtain
$$
\bigl\| B(\fk,0) \bigr\|^W
\;=\; \bigl\| \overline Q(0) \bigl( \fk   , \rD\cF(0) 0 -  \cF(0)   \bigr)   \bigr\|^W
\;\leq\; C_Q \|\fk\|^\fK \;<\; \delta 
$$
for all $\fk\in \cV_{\fK,\delta}:=\{ \fk\in\fK \,|\, \|\fk\|^\fK < \frac \delta{C_Q} \}$, as claimed in \eqref{eq:B-small classical}. 
Finally, if the Fredholm map is $\cC^\ell$-regular, then the maps $B: (\fk,\gamma)\mapsto  Q(0) \bigl( \fk   , \rD\cF(0) \gamma -  \cF(\gamma) \bigr)$ are $\cC^\ell$ as well. Indeed, $\overline Q(0)$ and $\rD\cF(0)$ are linear operators, so that $B$ inherits its regularity from $\cF$. 
\end{proof}

To generalize this contraction formulation to an adiabatic Fredholm family, we first formulate the simpler version when the family is regularizing, then provide the general contraction formulation.  

\begin{lemma} \label{lem:contraction-regularizing}
Given a regularizing adiabatic Fredholm family $\bigl( (\cF_\eps:\cV_\Gamma\to \Omega )_{\eps\in\Delta} , \ldots  \bigr)$ as in Definitions~\ref{def:fredholm} and \ref{def:regularizing}, let $Q_\eps(0)=\overline Q_\eps(0)|_{\fK\times\Omega}:   \fK \times \Omega \to \fC \times \Gamma$ be the bijections from Lemma~\ref{lem:inverses} for $\eps\in\Delta_Q\subset\Delta$.
Then the union of zero sets $\bigcup_{\eps\in\Delta_Q} \{\eps\}\times \cF_\eps^{-1}(0) \subset \Delta_Q\times \cV_\Gamma$ is naturally identified with the zero set of the map
\begin{align*}
\cG_Q \,:\; 
\Delta_Q \times \fK \times \fC \times \cV_\Gamma &\rightarrow \fC \times \fC \times \Gamma, \\
(\eps, \fk, \fc, \gamma) &\mapsto \bigl( \fc , Q_\eps(0) ( \pi_\fK(\gamma)-\fk, \cF_\eps(\gamma) - \fc ) \bigr),
\end{align*} 
Moreover, each of the maps $\cG_Q(\eps,\fk,\cdot,\cdot)$ is close to the identity map on $\cV_W:=\fC\times \cV_\Gamma \subset W:=\fC \times \Gamma$ up to finite dimensional factors and a contraction in the following sense: 
We can write 
$$
\cG_Q \,:\; \Delta_Q\times \fK \times \cV_W \;\rightarrow\;  \fC \times W, 
\qquad
(\eps, \fk, w) \;\mapsto\; \bigl( A_\eps(\fk,w), w - B_\eps(\fk,w) \bigr) , 
$$
where $A_\eps : \fK \times\cV_W \to \fC$ maps to the finite dimensional space $\fC$ by projection $A_\eps(\fk,\fc,\gamma)=\fc$, and
\begin{equation} \label{eq:B}
B_\eps : \fK \times \cV_W \to W, \quad
B_\eps(\fk,\fc,\gamma) = Q_\eps(0) \bigl( \fk   , \rD\cF_\eps(0) \gamma -  \cF_\eps(\gamma)   \bigr)
\end{equation}
is a family of contractions near $0\in W$ parametrized by $(\eps,\fk)\in\Delta_Q\times\fK$. 
More precisely, if we equip $W$ with the norms $\| w=(\fc,\gamma) \|^W_\eps:= \|\fc\|^\fC + \|\gamma\|^\Gamma_\eps=\|\fc\|^\Omega_0 + \|\gamma\|^\Gamma_\eps$, then given any $\theta\in (0,1)$ there exists $\delta_\theta>0$ such that for all $\eps\in\Delta_Q, \fk\in\fK, w,w'\in\cV_W$ we have
\begin{align*}
  \|w\|^W_\eps,  \|w'\|^W_\eps \leq \delta_\theta  \quad \Longrightarrow \quad
\| B_\eps (\fk, w' ) - B_\eps(\fk, w) \|^W_\eps \leq  \theta \| w' - w \|^W_\eps . 
\end{align*}
In addition, given any $\delta>0$ we can find open neighbourhoods $\cV_{\fK,\delta}\subset\fK$ of $0$ and $\Delta_\delta\subset\Delta_Q$ of $0$ with
\begin{align*}
(\eps,\fk)\in\Delta_\delta\times\cV_{\fK,\delta}  \quad \Longrightarrow \quad
\| B_\eps (\fk, 0 )  \|^W_\eps <  \delta . 
\end{align*}
\end{lemma}

\begin{lemma} \label{lem:contraction}
Given an adiabatic Fredholm family $\bigl( (\cF_\eps:\cV_\Gamma\to \Omega )_{\eps\in\Delta} , \ldots  \bigr)$ as in Definitions~\ref{def:fredholm}, 
let $\overline Q_\eps(0):   \fK \times \overline\Omega_\eps \to \fC \times \overline\Gamma_\eps$ be the isomorphisms from Lemma~\ref{lem:inverses} for $\eps\in\Delta_Q\subset\Delta$.
Then the union of completed zero sets $\bigcup_{\eps\in\Delta_Q} \{\eps\}\times \overline\cF_\eps^{-1}(0) \simeq \overline\cG_Q^{-1}(0_{\Delta_Q})$ is naturally identified with the preimage of the "zero section" $0_{\Delta_Q} := \bigcup_{\eps\in\Delta_Q} \{\eps\}\times \{(0,0,0)\} \subset \bigcup_{\eps\in\Delta_Q} \{\eps\}\times \fC\times\fC\times \cV_{\overline\Gamma,\eps}$ under the fibrewise continuous extension of $\cG_Q$, 
\begin{align*}
\overline\cG_Q \,:\;
\textstyle \bigcup_{\eps\in\Delta_Q} \{\eps\} \times \fK \times \fC \times \cV_{\overline\Gamma,\eps} &\rightarrow \textstyle \bigcup_{\eps\in\Delta_Q} \{\eps\} \times \fC \times \fC \times \overline\Gamma_\eps, \\
(\eps, \fk, \fc, \gamma) &\mapsto \bigl( \eps , \fc , \overline Q_\eps(0) ( \overline\pi_\fK(\gamma)-\fk, \overline\cF_\eps(\gamma) - \fc ) \bigr) . 
\end{align*}
Moreover, each of the maps $\overline\cG_Q(\eps,\fk,\cdot,\cdot)$ is -- up to finite dimensional factors and a contraction -- close to the identity map on the interior $\cV_{\overline W,\eps}:={\rm int}(\overline{\cV_W}) =\fC\times\cV_{\overline\Gamma,\eps}$ of the closure of $\cV_W=\fC\times\cV_\Gamma$ inside the completed space $\overline W_\eps:= \overline W^{\|\cdot\|^W_\eps}=\fC\times\overline\Gamma_\eps$ in the following sense: We can write for $\eps\in\Delta_Q$ 
$$
\overline\cG_Q (\eps,\cdot,\cdot,\cdot)  \,:\;  \fK \times \cV_{\overline W,\eps} \;\rightarrow\;  \fC \times \overline W_\eps, 
\qquad
(\fk, w) \;\mapsto\; \bigl( \overline A_\eps(\fk,w), w - \overline B_\eps(\fk,w) \bigr) , 
$$
where $\overline A_\eps : \fK \times\cV_{\overline W,\eps} \to \fC$ maps to the finite dimensional space $\fC$ by projection $\overline A_\eps(\fk,\fc,\gamma)=\fc$,
and $\overline B_\eps : \fK \times \cV_{\overline W,\eps} \to  \overline W_\eps$ given by $\overline B_\eps(\fk,\fc,\gamma) = \overline Q_\eps(0) \bigl( \fk   , \rD\overline\cF_\eps(0) \gamma -  \overline\cF_\eps(\gamma)   \bigr)$ is a family of contractions near $0\in \overline W_\eps$ parametrized by $(\eps,\fk)\in\Delta_Q\times\fK$. 
More precisely,  if we equip $\overline W_\eps$ with the norm $\| w=(\fc,\gamma) \|^W_\eps:= \|\fc\|^\fC + \|\gamma\|^\Gamma_\eps = \|\fc\|^\Omega_0 + \|\gamma\|^\Gamma_\eps$, then given any $\theta\in (0,1)$ there exists $\delta_\theta>0$ such that for all $\eps\in\Delta_Q, \fk\in\fK, w,w'\in\cV_W$ we have
\begin{align}\label{eq:B-contraction}
  \|w\|^W_\eps,  \|w'\|^W_\eps \leq \delta_\theta  \quad \Longrightarrow \quad
\| \overline B_\eps (\fk, w' ) - \overline B_\eps(\fk, w) \|^W_\eps \leq  \theta \| w' - w \|^W_\eps . 
\end{align}
In addition, given any $\delta>0$ we can find open neighbourhoods $\cV_{\fK,\delta}\subset\fK$ of $0$ and $\Delta_\delta\subset\Delta_Q$ of $0$ with
\begin{align}\label{eq:B-small}
(\eps,\fk)\in\Delta_\delta\times\cV_{\fK,\delta}  \quad \Longrightarrow \quad
\| \overline B_\eps (\fk, 0 )  \|^W_\eps <  \delta . 
\end{align}

If the adiabatic Fredholm family is fibrewise $\cC^\ell$-regular as in Definition~\ref{def:fibrewise C-l} for some $\ell\geq 1$, then each contraction $\overline B_\eps : \fK \times \cV_{\overline W,\eps} \to  \overline W_\eps$ is $\cC^\ell$ as well.

Finally, if the adiabatic Fredholm family is regularizing, then the maps $\overline\cG_Q(\eps,\cdot)$, $\overline A_\eps$ and $\overline B_\eps$ are the continuous extensions of the maps $\cG_Q(\eps,\cdot)$, $A_\eps$ and $B_\eps$ in Lemma~\ref{lem:contraction-regularizing} w.r.t.\ the norms $\|\cdot\|^W_\eps$ on $W$, $\|\cdot\|^\fK=\|\cdot\|^\Gamma_0$ on $\fK$, and $\|\cdot\|^\fC=\|\cdot\|^\Omega_0$ on $\fC$. 
\end{lemma}

\begin{proof}[Proof of Lemmas~\ref{lem:contraction-regularizing} and \ref{lem:contraction}]
From Lemma~\ref{lem:equivalent-fredholm} we have $\bigcup_{\eps\in\Delta} \{\eps\}\times \cF_\eps^{-1}(0)\simeq\cG^{-1}(0)$ for $\cG:\Delta \times \fK \times \fC \times \cV_\Gamma \rightarrow \fC \times \fK \times \Omega$ given by $\cG(\eps, \fk, \fc, \gamma) = ( \fc , \pi_\fK(\gamma)-\fk, \cF_\eps(\gamma) - \fc )= (\fc, \overline g_{\eps,\fk}(\fc,\gamma))$. 
The map considered here is obtained by composing $\cG$ with ${\rm Id}_\fC \times Q_\eps(0):\fC\times  \fK \times \Omega \to \fC\times  \fC \times \Gamma$. Here $Q_\eps(0):\fK\times\Omega\to\fC\times\Gamma$ is a bijection by Lemma~\ref{lem:inverses}, so this composition does not affect the zero set. 

When we drop the regularizing condition, then Lemma~\ref{lem:inverses} only asserts bijectivity of the linearized operators $\rD \overline g_{\eps,\fk} (\fc_0,\gamma_0)$ after taking $\epsilon$-dependent completions, that is we need to work with the inverse operators $\overline Q_\eps(0):\fK\times\overline\Omega_\eps\to\fC\times\overline\Gamma_\eps$. 
These can be composed with the "section" $\overline\cG: \bigcup_{\eps\in\Delta} \{\eps\}\times \fK \times \fC \times \cV_{\overline\Gamma,\eps} \rightarrow \bigcup_{\eps\in\Delta} \{\eps\}\times \fC \times \fK \times \overline\Omega_\eps$  from Lemma~\ref{lem:equivalent-fredholm} to obtain the claimed identification
$$
\bigcup_{\eps\in\Delta_Q} \{\eps\}\times \overline\cF_\eps^{-1}(0) \simeq \overline\cG^{-1}(0_{\Delta_Q})  \simeq  \overline\cG_Q^{-1}(0_{\Delta_Q}) . 
$$
Here $\overline\cG_Q=({\rm Id}_\fC \times \overline Q_\eps(0))\circ\overline\cG : \bigcup_{\eps\in\Delta} \{\eps\}\times \fK \times \fC \times\cV_{\overline\Gamma,\eps} \rightarrow \bigcup_{\eps\in\Delta} \{\eps\}\times \fC \times \fC \times \overline\Gamma_\eps$ is given by 
$$
\overline\cG_Q(\eps, \fk, \fc, \gamma)= \bigl( \eps, \fc , \overline Q_\eps(0)( \overline\pi_\fK(\gamma)-\fk, \overline\cF_\eps(\gamma) - \fc ) \bigr) = \bigl( \eps, \overline A_\eps( \fk, \fc, \gamma), (\fc,\gamma) - \overline B_\eps( \fk, \fc, \gamma)  \bigr)  , 
$$
where 
$\overline B_\eps(\fk,\fc,\gamma) = \overline Q_\eps(0) \bigl( \fk   , \rD\overline\cF_\eps(0) \gamma -  \overline\cF_\eps(\gamma)   \bigr)$ follows as in Lemma~\ref{lem:contraction classical}. 
When the family is regularizing, then Lemma~\ref{lem:inverses} ensures that the inverse $ \overline Q_\eps(0)$ restricts to $Q_\eps(0)=\overline Q_\eps(0)|_{\fK\times\Omega}:   \fK \times \Omega \to \fC \times \Gamma$, and thus $\overline\cG_Q(\eps,\cdot)$, $\overline A_\eps$, resp.\ $\overline B_\eps$ restrict to the claimed expressions for $\cG_Q(\eps,\cdot)$, $A_\eps$ and $B_\eps$. 

Next, we will establish continuity of $\overline A_\eps$ and $\overline B_\eps$, and hence of $\overline\cG_Q(\eps,\cdot)=(\overline A_\eps, {\rm Id}_{\overline W_\eps} - \overline B_\eps)$, which in particular implies that these are in fact the continuous extensions of $A_\eps$, $B_\eps$, and  $\cG_Q(\eps,\cdot)$. 
Continuity of the linear map $\overline A_\eps(\fk,\fc,\gamma)=\fc$ is evident since $\|\fc\|^\fC \leq \|(\fc,\gamma)\|^W_\eps=\|\fc\|^\fC+\|\gamma\|^\Gamma_0$.  
To establish continuity of $\overline B_\eps$ for fixed $\eps\in\Delta_Q$ we use boundedness of the operators 
$\overline Q_\eps(0)$ in Lemma~\ref{lem:inverses} and the [Fibrewise $\mathcal{C}^1$ Regularity] of $\overline\cF_\eps$ to deduce that for $(\fk_i,\fc_i,\gamma_i)\to(\fk,\fc,\gamma)\in \fK\times \fC\times \cV_{\overline\Gamma,\eps}$ we have
\begin{align*}
\bigl\| \overline B_\eps(\fk_i,\fc_i,\gamma_i) - \overline B_\eps(\fk,\fc,\gamma) \bigr\|^W_\eps 
&= \bigl\|  \overline Q_\eps(0) \bigl( \fk_i-\fk   , \rD\overline \cF_\eps(0)( \gamma_i-\gamma) - \overline \cF_\eps(\gamma)+\overline \cF_\eps(\gamma_i)   \bigr)\bigr\|^W_\eps  \\
&\leq  C_Q \bigl( \| \fk_i-\fk \|^\fK +  \bigl\| \rD\overline \cF_\eps(0)( \gamma_i-\gamma)  \bigr\|^\Omega_\eps +     \bigl\|  \overline \cF_\eps(\gamma)- \overline \cF_\eps(\gamma_i)  \bigr\|^\Omega_\eps   \bigr)  \;\to\; 0. 
\end{align*}
It remains to check the estimates for $\overline B_\eps$ which then imply the same estimates for its restriction $B_\eps$ in the regularizing case. 
This is achieved by applying Lemma~\ref{lem:contraction classical} to the $\cC^1$ Fredholm map $\overline\cF_\eps:\cV_{\Gamma,\eps}\to\ol\Omega_\eps$ for each $\eps\in\Delta_Q$ and noting that the 
[Quadratic-ish Estimate] \eqref{eq:quadratic-ish classical} is provided by Lemma~\ref{lem:fredholm} with the $\eps$-independent function $c:[0,\infty)\to [0,\infty)$ that is a part of the data of the adiabatic Fredholm family. 

Finally, the identification of the completion $\overline W_\eps:= \overline W^{\|\cdot\|^W_\eps}=\fC\times\overline\Gamma_\eps$ results from the fact that $\fC$ is finite dimensional, so complete w.r.t.\ any norm, and $\overline\Gamma_\eps=\overline \Gamma^{\|\cdot\|^\Gamma_\eps}$ is the completion w.r.t.\ the norm on the infinite dimensional factor of $W=\fC\times\Gamma$. 
Similarly, the closure of $\cV_W=\fC\times\cV_\Gamma$ inside the completed space $\overline W_\eps=\fC\times\overline\Gamma_\eps$ is the product of closures $\overline{\cV_W}=\fC\times\overline{\cV_\Gamma}$, and then the interior is the product of interiors  $\cV_{\overline W,\eps}={\rm int}( \overline{\cV_W}) = {\rm int}(\fC\times\overline{\cV_\Gamma})= \fC\times\cV_{\overline\Gamma,\eps}$.

If the adiabatic Fredholm family is fibrewise $\cC^\ell$-regular, that is $\cF_\eps : (\cV_\Gamma,{\|\cdot\|^\Gamma_\eps}) \to (\Omega,\|\cdot\|^\Omega_\eps)$ is uniformly $\cC^\ell$ for each $\eps\in\Delta$, then the maps $\overline B_\eps: (\fk,\gamma)\mapsto \overline Q_\eps(0) \bigl( \fk   , \rD\overline\cF_\eps(0) \gamma -  \overline\cF_\eps(\gamma) \bigr)$ are uniformly $\cC^\ell$ as well for each fixed $\eps\in\Delta_Q$. Indeed, $\overline Q_\eps(0)$ and $\rD\overline\cF_\eps(0)$ are linear operators and $\overline\cF_\eps$ inherits its regularity from $\cF_\eps$ by Lemma~\ref{lem:extension}.  
\end{proof}

\subsection{Solution Maps} \label{solution}
This section constructs solution maps $\sigma_\eps:\cV_\fK \to W=\fC\times \Gamma$ that solve the fixed point equations $w=B_\eps(\fk,w)$ from \S\ref{contraction}. 
We also show how various regularity assumptions on the adiabatic Fredholm family $\cF_\eps$ transfer to the contractions $B_\eps$ and then the solution maps $\sigma_\eps$. 
The latter is the technically hardest -- and longest -- part of the proof. 
It can be understood as analogous to the properties of solution germs in \cite[Thm.3.3.3]{hofer-wysocki-zehnder_polyfold}, which arguably are the technical core that allows polyfold regularization theory to construct transverse perturbations over spaces of maps modulo reparametrization. 
Thus this section finishes the argument that adiabatic Fredholm theory is compatible with polyfold regularization theory. 

Starting again in the classical Fredholm setting of Theorem~\ref{thm:charts classical}, we work here just with the second component of the map $\cG_Q$ that was constructed in Lemma~\ref{lem:contraction classical}.

\begin{lemma} \label{lem:solution classical}
Consider a Banach space $(\overline W,\|\cdot\|)$, an open subset $\cV_{\overline W}\subset\overline W$ 
containing $0=0_{\overline W}\in\cV_{\overline W}$, a set $\cV_K$, and a map of the form
$$
\cV_K \times \cV_{\overline W} \;\rightarrow \; \overline W,  \qquad
(k,w) \;\mapsto\;   w - B(k,w) ,
$$ 
where
$B : \cV_K \times \cV_{\overline W} \to \overline W$ is a contraction near $0\in \cV_{\overline W}$ parametrized by $k\in \cV_K$. 
More precisely, assume that there exists $\theta\in (0,1)$ and $\delta>0$ such that for all $k\in\cV_K, w,w'\in  \cV_{\overline W}$ we have
\begin{align}\label{eq:contraction} 
\|w\|, \|w'\|\leq \delta \quad \Longrightarrow \quad
\| B (k, w' ) - B(k, w) \| \leq  \theta \| w' - w \| , 
\end{align}
and moreover $\{ w\in \overline W \,|\, \|w \|^{\overline W}\leq\delta \}\subset \cV_{\overline W}$ as well as
\begin{align}\label{eq:small}
\forall k \in \cV_K \quad 
\| B(k,0) \| < (1-\theta)\delta . 
\end{align}
Then there is a unique solution map $\sigma: \cV_K \to \cV_{\overline W}$ that parametrizes the zero set of the map (i.e.\ the fixed points of $B$) near $\cV_K\times 0\subset \cV_K\times\cV_{\overline W}$, that is
$$
\bigl\{ (k,w)\in \cV_K \times  \cV_{\overline W} \,\big|\, w - B(k,w) = 0 ,  \| w\| < \delta\bigr\} 
= 
\bigl\{  (k , \sigma (k) )  \,\big|\, k \in \cV_K \bigr\} . 
$$
In case $\cV_{\overline W}=\overline W$ we can drop the assumption \eqref{eq:small} as long as \eqref{eq:contraction} holds for all $w,w'\in\overline W$. In that case the conclusion is 
$\bigl\{ (k,w)\in \cV_K \times  \overline W \,\big|\, w - B(k,w) = 0 \bigr\} 
= \bigl\{  (k , \sigma(k) )  \,\big|\, k \in \cV_K \bigr\}$. 

If, moreover, $\cV_K$ is equipped with a topology such that $B : \cV_K \times \cV_{\overline W} \to \overline W$ is continuous, then the solution map $\sigma: \cV_K \to \cV_{\overline W}$ is continuous.

Furthermore, if $\cV_K\subset K$ is an open subset of a normed vector space
so that $B : \cV_K \times \cV_{\overline W} \to \overline W$ is $\cC^\ell$, then the solution map $\sigma: \cV_K \to \cV_{\overline W}$ is $\cC^\ell$ with $\rD\sigma(k)=\bigl(\rm Id -  \rD_W B (k,\sigma(k)) \bigr)^{-1} \rD_K B (k,\sigma(k))$.
\end{lemma}

\begin{proof}
To apply the Banach Fixed Point Theorem to one of the contractions $B(k,\cdot)$, we need to find a closed subset of $\cV_{\overline W}$ that is preserved by this map. This role will be played by the closed ball $\overline\cV_{\overline W,\delta}:=\{ w\in \overline W \,\|\, \|w \|^{\overline W}\leq\delta \}\subset \cV_{\overline W}$ of radius $\delta$ centered at $0$, which by assumption is contained in the domain where $B(k,\cdot)$ is defined. 
Moreover, this ball is invariant under $B(k,\cdot)$ if we further restrict ourselves to $k\in \cV_K$. Indeed, we have $B(k,\overline\cV_{\overline W,\delta})\subset \overline\cV_{\overline W,\delta}$ since for all $w\in\overline\cV_{\overline W,\delta}$
$$
\| B(k,w) \| 
\leq 
\| B(k,w)-B(k,0) \| +   \| B(k,0) \|
\leq 
\theta \| w - 0 \|  +  (1-\theta)\delta
\leq 
\theta \delta +  (1-\theta)\delta
=\delta . 
$$
Now the Banach Fixed Point Theorem applies to the contraction $B(k,\cdot): \overline\cV_{\overline W,\delta} \to \overline\cV_{\overline W,\delta}$ for each $k\in \cV_K$ to guarantee that each such map has a unique fixed point $w_\fk \in \overline\cV_{\overline W,\delta}$. This defines a map $\sigma: \cV_K \to \overline\cV_{\overline W,\delta}, k\mapsto w_k$, which by definition parametrizes the solutions: For $w\in\overline\cV_{\overline W,\delta}$ 
we have 
$$
w - B(k,w) = 0 , \|w\|\leq\delta
\quad \Leftrightarrow \quad
w =B(k,w) , w\in\overline\cV_{\overline W,\delta}
\quad \Leftrightarrow \quad
w=\sigma(k).
$$ 
Here we can replace the condition $\|w\|\leq\delta$ by  $\|w\|<\delta$ 
because the norm of the actual solutions is strictly less than $\delta$ due to the strict inequality in \eqref{eq:small}, 
\begin{align*}
\| \sigma(k) \| 
&\;=\;
\| B(k,\sigma(k)) \| 
\;\leq\;
\| B(k,\sigma(k))-B(k,0) \| +   \| B(k,0) \| \\
&\;\leq\; 
\theta \| \sigma(k) - 0 \|  +   \| B(k,0) \| 
\;\leq\; 
\theta \delta  +   \| B(k,0) \| 
\;<\; 
\theta \delta +  (1-\theta)\delta
\;=\; \delta . 
\end{align*}
In case $\cV_{\overline W}=\overline W$ the Banach Fixed Point Theorem applies for all $k\in\cV_K$ with $\overline\cV_{\overline W,\delta}$ replaced by $\overline W$ to prove the same conclusion with no need for assumption \eqref{eq:small} or a bound $\|w\|\leq\delta$.

If, moreover, $B : \cV_K \times \cV_{\overline W} \to \overline W$ is continuous, then we use the defining equation $\sigma(k) = B(k,\sigma(k))$ of the solution map and the contraction property \eqref{eq:contraction} to estimate for $k_0,k_1\in\cV_K$
\begin{align*}
\| \sigma(k_1) - \sigma(k_0) \| 
& = \| B(k_1,\sigma(k_1)) - B(k_0,\sigma(k_0)) \|   \\
& \leq \| B(k_1,\sigma(k_1)) - B(k_1,\sigma(k_0)) \|  +  \| B(k_1,\sigma(k_0)) - B(k_0,\sigma(k_0)) \|  \\
& \leq  \theta \| \sigma(k_1) - \sigma(k_0) \|  +  \| B(k_1,\sigma(k_0)) - B(k_0,\sigma(k_0)) \|  
\end{align*}
\begin{align} \label{eq:sigma-est}
\Rightarrow\qquad 
\| \sigma(k_1) - \sigma(k_0) \| 
& \leq  \tfrac 1{1-\theta}  \| B(k_1,\sigma(k_0)) - B(k_0,\sigma(k_0)) \| .
\end{align}
So continuity of $B$ at $(k_0,\sigma(k_0))$ implies continuity of $\sigma$ at $k_0$.  

If, furthermore, $B : \cV_K \times \cV_{\overline W} \to \overline W$ is $\cC^1$, then we will show that the solution map $\sigma: \cV_K \to \cV_{\overline W}$ is $\cC^1$ as well. Towards that, we write the differential $\rD B(k_0,w_0): K \times \overline W \to \overline W$ at $(k_0,w_0)$ as the sum $\rD B(k_0,w_0) (k,w) = \rD_K B(k_0,w_0) k  + \rD_W B(k_0,w_0) w$ of its partial differentials $\rD_K B(k_0,w_0): K \to \overline W$ and $\rD_W B(k_0,w_0): \overline W \to \overline W$. 
Each of these partial differentials is a bounded linear operator and varies continuously with the base point $(k_0,w_0)\in\cV_K\times\cV_{\overline W}$. Moreover, $\rD_W B(k_0,w_0)\in\cL(\overline W,\overline W)$ inherits a linear contraction property from \eqref{eq:contraction}: For all $w\in\overline W$ we have 
$$
\| \rD_W B (k_0, w_0 ) w \| = \lim_{h\to 0} h^{-1} \| B(k_0, w_0+hw ) -  B(k_0, w_0 )  \| 
 \leq  \lim_{h\to 0} h^{-1} \theta \| h w \| = \theta \|w\| .
$$
This can also be stated as $\| \rD_W B (k_0, w_0 ) \|^{\cL(\overline W,\overline W)} \leq \theta$, and since $\theta<1$ this implies that the operator ${\rm Id}_{\overline W} -  \rD_W B (k_0, w_0 )\in \cL(\overline W,\overline W)$ is invertible with $\| ({\rm Id}_{\overline W} -  \rD_W B (k_0, w_0 ))^{-1} \|^{\cL(\overline W,\overline W)} \leq  \tfrac 1{1-\theta}$. Indeed, this holds for any $T\in \cL(\overline W,\overline W)$ with $\|T\|<1$, since $({\rm Id}_{\overline W} - T)^{-1}=\sum_{n=0}^\infty T^n$ converges with norm bounded by $\sum_{n=0}^\infty \|T\|^n=\tfrac 1{1-\|T\|}$. 
Recall moreover that inverting these operators is a continuous map as in \eqref{eq:inverse est}.
Thus we have established that the inverse operators $({\rm Id}_{\overline W} -  \rD_W B (k_0, w_0 ))^{-1}$ exist and vary continuously with $(k_0,w_0)\in\cV_K\times\cV_{\overline W}$. 
This is useful for proving differentiability of the solution map since symbolic differentiation of the defining equation $\sigma(k) = B(k,\sigma(k))$ yields $\rD\sigma(k_0) k  = \rD_K B (k_0,\sigma(k_0)) k  + \rD_W B (k_0,\sigma(k_0)) \rD\sigma(k_0) k$, and hence $\bigl(\rm Id -  \rD_W B (k_0,\sigma(k_0)) \bigr) \rD\sigma(k_0) k  = \rD_K B (k_0,\sigma(k_0)) k$. While this computation isn't valid until differentiability of $\sigma$ is established, we use it to recognize $\Phi(k_0):=\bigl(\rm Id -  \rD_W B (k_0,\sigma(k_0)) \bigr)^{-1} \rD_K B (k_0,\sigma(k_0)) \in \cL(K,\overline W)$ as the well defined candidate for the differential of $\sigma$ at $(k_0,w_0)\in\cV_K\times\cV_{\overline W}$.
To establish that it is indeed the differential, we denote $\Delta\sigma(k_0,k):= \sigma(k_0+k) - \sigma(k_0)$ and aim to show that its best linear approximation for small $\|k\|^K$ is $\Phi(k_0) k$. 
For that purpose we first note that continuity of $\rD B$ at $(k_0,w_0:=\sigma(k_0))$ means that, given any $\varepsilon>0$ there is $\delta_\varepsilon>0$ such that 
\begin{equation}\label{eq:DB continuity}
\|(k_1, w_1) - (k_0,\sigma(k_0)) \|^{K\times\overline W}<\delta_\varepsilon \quad\Rightarrow\quad
 \bigl\|  \rD B(k_1,w_1) - \rD B (k_0,\sigma(k_0))   \bigr\|^{\cL(K\times\overline W, \overline W)} < \varepsilon . 
\end{equation}
We can use this fact for $(k_1,w_1)=(k_0+k,\sigma(k_0))$ to strengthen the continuity estimate \eqref{eq:sigma-est} for $\sigma$ by deducing that for $\|k\|^K < \delta_\varepsilon$ we have
\begin{align*}
\| \Delta\sigma(k_0,k) \| \;=\; \| \sigma(k_0+k) - \sigma(k) \| 
&\leq  \tfrac 1{1-\theta}  \| B(k_0+k,\sigma(k_0)) - B(k,\sigma(k_0)) \|  \\
&\leq  \tfrac 1{1-\theta} \bigl\| \textstyle \int_0^1 \rD_K B(k_0+\lambda k,\sigma(k_0)) \, k \, \rd\lambda  \bigr\|  \\
&\leq  \tfrac 1{1-\theta}  \bigl( \| \rD_K B(k_0,\sigma(k_0))\|^{\cL(K,\overline W)} + \varepsilon \bigr)  \| k\|^K . 
\end{align*}
Next, we use the defining equation $\sigma(k) = B(k,\sigma(k))$ again to rewrite 
\begin{align*}
& \bigl(\rm Id -  \rD_W B (k_0,\sigma(k_0)) \bigr)\bigl( \Delta\sigma(k_0,k) - \Phi(k_0) k \bigr)  \\
&=  B(k_0+k,\sigma(k_0+k)) - B(k_0,\sigma(k_0))  - \rD_W B (k_0,\sigma(k_0))\bigl( \sigma(k_0+k) - \sigma(k_0) \bigr) -  \rD_K B (k_0,\sigma(k_0)) k
\\
&= \textstyle \int_0^1 \rD B(k_0+\lambda k,\sigma(k_0) + \lambda \Delta\sigma(k_0,k))  \bigl( k , \Delta\sigma(k_0,k) \bigr) \rd\lambda  - \rD B (k_0,\sigma(k_0)) \bigl( k , \Delta\sigma(k_0,k) \bigr) \\
&= \textstyle \int_0^1 \bigl( \rD B(k_0+\lambda k,\sigma(k_0) + \lambda \Delta\sigma(k_0,k)) - \rD B (k_0,\sigma(k_0))\bigr) \bigl( k , \Delta\sigma(k_0,k) \bigr)   \rd\lambda  . 
\end{align*}
Thus for any $k_0\in\cV_K$ and sufficiently small $k\in K$ -- guaranteeing that $\{k_0+\lambda k\,|\, 0\leq\lambda\leq 1\} \subset\cV_K$ -- we can estimate 
\begin{align*}
& \bigl\| \sigma(k_0+k) - \sigma(k_0) - \bigl(\rm Id -  \rD_W B (k_0,\sigma(k_0)) \bigr)^{-1} \rD_K B (k_0,\sigma(k_0)) k \bigr\|  \\
&\leq \| \bigl(\rm Id -  \rD_W B (k_0,\sigma(k_0)) \bigr)^{-1} \|^{\cL(\overline W,\overline W)}
\bigl\|   \bigl(\rm Id -  \rD_W B (k_0,\sigma(k_0)) \bigr)\bigl( \Delta\sigma(k_0,k) - \Phi(k_0) k \bigr) \bigr\| 
\\
&\leq \tfrac 1{1-\theta}   \textstyle \int_0^1  \bigl\|  \rD B(k_0+\lambda k,\sigma(k_0) + \lambda \Delta\sigma(k_0,k)) - \rD B (k_0,\sigma(k_0))   \bigr\|^{\cL(K\times\overline W,\overline W)}  \rd\lambda   \bigl\|  \bigl( k , \Delta\sigma(k_0,k) \bigr)   \bigr\|^{K\times\overline W}. 
\end{align*}
To proceed, we wish to use \eqref{eq:DB continuity} at $(k_1,w_1)=(k_\lambda,w_\lambda):=(k_0+\lambda k,\sigma(k_0) + \lambda \Delta\sigma(k_0,k))$. To check the assumption we use continuity of $\sigma$ to note that given $\varepsilon>0$ and the above $\delta_\varepsilon>0$ there is $0<\delta'_\varepsilon\leq \frac 12 \min\{ \delta_\varepsilon,\varepsilon\}$ such that for all $k\in K$ with $\| k \|^K < \delta'_\varepsilon$ we have $ \| \sigma(k) - \sigma(k_0) \| < \frac 12  \min\{\delta_\varepsilon,\varepsilon\}$, and hence 
\begin{align*}
\forall\; 0\leq\lambda\leq 1 : \quad & \| (k_\lambda,w_\lambda)   - (k_0,\sigma(k_0)) \|^{K\times\overline W}  \\
& = \| (k_0+\lambda k,\sigma(k_0) + \lambda \Delta\sigma(k_0,k))   - (k_0,\sigma(k_0)) \|^{K\times\overline W} \\
& =  \lambda \| k \|^K +   \lambda \| \Delta\sigma(k_0,k) \| 
\;<\;  \min\{\delta_\varepsilon,\varepsilon\} . 
\end{align*}
Taking all this together, for every $\varepsilon>0$ there is $\delta'_\varepsilon>0$ such that for $k\in K$ with $\| k \|^K < \delta'_\varepsilon$ we have
\begin{align*}
& \bigl\| \sigma(k_0+k) - \sigma(k_0) - \bigl(\rm Id -  \rD_W B (k_0,\sigma(k_0)) \bigr)^{-1} \rD_K B (k_0,\sigma(k_0)) k \bigr\| \\
&\leq \tfrac 1{1-\theta}   \textstyle \int_0^1  \bigl\|  \rD B(k_\lambda,w_\lambda) - \rD B (k_0,\sigma(k_0))   \bigr\|^{\cL(K\times\overline W,\overline W)}  \rd\lambda   \bigl\|  \bigl( k , \Delta\sigma(k_0,k) \bigr)   \bigr\|^{K\times\overline W}  \\
&\leq \tfrac 1{1-\theta}  \, \varepsilon \, \bigl( 1 +   \tfrac 1{1-\theta}  \bigl( \| \rD_K B(k_0,\sigma(k_0))\|^{\cL(K,\overline W)} + \varepsilon \bigr) \bigr) \| k \|^K  ,
\end{align*}
which proves for $\|k\|^K\to 0$ the convergence 
$$
\bigl\| \sigma(k_0+k) - \sigma(k_0) - \bigl(\rm Id -  \rD_W B (k_0,\sigma(k_0)) \bigr)^{-1} \rD_K B (k_0,\sigma(k_0)) k \bigr\| / \|k\|^K \;\longrightarrow\; 0 . 
$$
Thus $\sigma$ is differentiable at all $k_0\in\cV_K$ with $\rD\sigma(k_0)=\bigl(\rm Id -  \rD_W B (k_0,\sigma(k_0)) \bigr)^{-1} \rD_K B (k_0,\sigma(k_0))$. To verify that this differential varies continuously with $k_0\in\cV_K$ in $\cL(K,\overline W)$, recall that $B$ was assumed to be $\cC^1$, thus $\rD_W B$ and $\rD_K B$ vary continuously in the operator norm with their base point in $\cV_K\times\cV_{\overline W}$. Then continuity of $k_0 \mapsto \rD_* B (k_0,\sigma(k_0))$ for $*=K$ and $*=W$ follows from the continuity of $\sigma$. Finally, we already established above that the inverses $(\rm Id - T )^{-1}$ exist and vary continuously with $T=\rD_W B (k_0,\sigma(k_0))$. Thus $\rD\sigma: \cV_K\to\cL(K,\overline W)$ is $\cC^0$ and $\sigma: \cV_K\to\cV_{\overline W}$ is $\cC^1$. This proves the Lemma for $\ell\leq 1$.  

Towards proving the Lemma for $\ell\geq 2$ note that the previous symbolic identity  
$\rD\sigma(k_0)  = \rD_K B (k_0,\sigma(k_0))  + \rD_W B (k_0,\sigma(k_0)) \rD\sigma(k_0)$ is now rigorous and can be understood to say that $\widetilde\sigma: \cV_K \to \widetilde W:=\cL(K,\overline W), k_0 \mapsto \rD\sigma(k_0)$ is the solution map for the equation $\widetilde w = \widetilde B(k,\widetilde w)$. Here
$$
\widetilde B : \cV_K \times \widetilde W \to  \widetilde W=\cL(K,\overline W), \quad
(k, \widetilde w ) \mapsto  \rD_K B (k,\sigma(k))   + \rD_W B (k,\sigma(k)) \circ \widetilde w 
$$
is a contraction that satisfies \eqref{eq:contraction}. Indeed, we have for all $k\in\cV_K$ and $\widetilde w,\widetilde w'\in\widetilde W$ 
\begin{align*}
\bigl\| \widetilde B (k, \widetilde w' ) - \widetilde B(k, \widetilde w) \bigr\| 
&= \bigl\|  \rD_K B (k,\sigma(k))   + \rD_W B (k,\sigma(k)) \circ \widetilde w'  
-   \rD_K B (k,\sigma(k))   -  \rD_W B (k,\sigma(k)) \circ \widetilde w \bigr\|  \\
&= \bigl\| \rD_W B (k,\sigma(k)) \circ ( \widetilde w' - \widetilde w) \bigr\|  \; \leq \;  \theta \| \widetilde w' - \widetilde w \| . 
\end{align*}
Thus $\widetilde B$ satisfies \eqref{eq:contraction} with the same $\theta$ as $B$, and on the entire Banach space $ \widetilde W=\cL(K,\overline W)$. And the above proof shows that $\cC^1$ regularity of $\widetilde B$ -- which would follow from $\cC^2$ regularity of $B$ -- implies $\cC^1$ regularity of $\widetilde\sigma=\rD\sigma$, and hence $\cC^2$ regularity of $\sigma$. We can iterate this argument to prove the Lemma for all $\ell$. 

Indeed, assume by induction that $\cC^\ell$ regularity of a contraction for any $\ell\leq L$ implies $\cC^\ell$ regularity of the solution map. Then consider a contraction $B$ of regularity $\cC^{L+1}$. As above, the differential of its solution map $\rD\sigma=\widetilde\sigma$ is the solution map of a contraction of the form $\widetilde B: \cV_K \times \widetilde W \to  \widetilde W$, which is of regularity $\cC^L$. Then the induction hypothesis ensures that its solution map $\widetilde\sigma=\rD\sigma$ is of regularity $\cC^L$, and hence $\sigma$ is of regularity $\cC^{L+1}$. This finishes the induction and thus the proof.
\end{proof}

Finally, we are prepared for the technical core of this paper, where we specify the meaning of adiabatic regularity for the solution maps -- and show how it follows from the adiabatic regularity of the adiabatic Fredholm family as formalized in Definition \ref{def:adiabatic C-l}.  
Note here that the crucial contribution of this paper is not just in proving this implication, but in doing so for a notion of adiabatic regularity that (a) is satisfied in Examples~\ref{ex:adiabatic} and (b) yields a reasonably regular finite dimensional reduction in \S\ref{reduction}. The latter in particular requires continuity of the global solution map $(\eps,\fk)\mapsto \sigma_\eps(\fk)$ -- and its derivatives in $\fk$ -- in some global topology.  This is what the following theorem establishes with the property of [Continuity w.r.t.\ $\|\cdot\|_0$]. 
To readers interested in strengthening the results or weakening the assumptions we recommend starting with this proof in \eqref{eq:continuity} and \eqref{eq:continuous derivatives} and reading backwards to analyze its ingredients.

\begin{theorem} \label{thm:family solution}
Given an adiabatic Fredholm family $\bigl( (\cF_\eps:\cV_\Gamma\to \Omega )_{\eps\in\Delta} , \ldots  \bigr)$ as in Definition~\ref{def:fredholm}, the contractions $\overline B_\eps: \fK\times\cV_{\overline W,\eps} \to \overline W_\eps$ from Lemma~\ref{lem:contraction} satisfy the assumptions of Lemma~\ref{lem:solution classical}, resulting in solution maps $\bigl( \sigma_\eps: \cV_\fK \to\cV_{ \overline W,\eps} \bigr)_{\eps\in\Delta_\sigma}$ 
defined on neighbourhoods $\Delta_\sigma\subset\Delta$ of $0$ and $\cV_\fK\subset\fK$ of $0$ 
such that for some $0<\delta_\sigma\leq\delta_Q$ we have
\begin{align*}
\bigl\{ (\fk,\fc,\gamma)\in \cV_\fK \times \fC\times \cV_{\overline\Gamma,\eps} \,\big|\,
 \overline\pi_\fK(\gamma)=\fk, 
\overline\cF_\eps(\gamma)  =  \fc , 
\|\fc\|^\fC + \|\gamma\|^\Gamma_\eps < \delta_\sigma \bigr\} & \\
=\; \bigl\{ (\fk,w)\in \cV_\fK \times \cV_{\overline W,\eps} \,\big|\, w- \overline B_\eps(\fk,w) =0 , \|w\|^W_\eps < \delta_\sigma   \bigr\}
& \\
=\; \bigl\{ (\fk,w)\in \cV_\fK \times \cV_{\overline W,\eps} \,\big|\, \overline g_{\eps,\fk}(w) =0 , \|w\|^W_\eps < \delta_\sigma   \bigr\}
& = \bigl\{  (\fk,  \sigma_\eps(\fk)  ) \,\big|\,  \fk\in\cV_\fK  \bigr\} .
\end{align*}
Moreover, each solution map $\sigma_\eps: \cV_\fK \to\cV_{ \overline W,\eps}$ is continuous and $\cC^1$, and the family of solution maps $(\sigma_\eps)_{\eps\in\Delta_\sigma}$ is uniformly bounded and uniformly continuous: 
$$
\bigl\| \sigma_\eps (\fk_0) \bigr\|^{W}_\eps  < \delta_\sigma
\qquad\text{and}\qquad
\bigl\| \sigma_\eps (\fk)  - \sigma_\eps (\fk_0) \bigr\|^{W}_\eps 
\leq    \tfrac {C_Q}{1-\theta}  \| \fk  -  \fk_0  \|^\fK
\qquad \forall \eps\in\Delta_\sigma, \fk,\fk_0\in\cV_\fK    
$$
with the contraction constant $\theta<1$ from Lemma~\ref{lem:contraction} and the uniform constant $C_Q$ from \eqref{eq:Q-estimate} .  

If the adiabatic Fredholm family is fibrewise $\cC^\ell$-regular as in Definition~\ref{def:fibrewise C-l} for some $\ell>1$, then each solution map $\sigma_\eps$ is $\cC^\ell$.

If the adiabatic Fredholm family is regularizing as in Definition~\ref{def:regularizing}, then each solution map takes values $\sigma_\eps: \cV_\fK \to\cV_W$ in the $\eps$-independent dense subset $\cV_W=\fC\times\cV_\Gamma\subset\cV_{\overline W,\eps}$. 

If the adiabatic Fredholm family is adiabatic $\cC^\ell$-regular as in Definition~\ref{def:adiabatic C-l}, then the family of solution maps $\Delta_\sigma\times\cV_\fK \to \cV_W, (\eps,\fk)\mapsto \sigma_\eps(\fk)$ is adiabatic $\cC^\ell$-regular in the following sense:\footnote{Note that the higher regularizing property is necessary to even make sense of the norm in the pointwise continuity.}
\begin{itemlist}
\item[{\bf[Higher Regularizing Property]}]
The $\ell$-fold fibrewise tangent map $\Delta_\sigma\times\rT^\ell\cV_\fK \to \rT^\ell\cV_W$, $(\eps,\ul\fk)\mapsto \rT^\ell\sigma_\eps(\ul\fk)$ takes values in the $\eps$-independent dense subspace $\rT^\ell\cV_W\subset \rT^\ell\cV_{\ol W,\eps}$. 

\item[\bf{[Pointwise Continuity in $\mathbf\Delta$]}]
$
\forall \eps_0\in\Delta_\sigma, \ul \fk_0 \in \rT^\ell\cV_\fK  \quad
\bigl\| \rT^\ell\sigma_\eps(\ul \fk_0) - \rT^\ell\sigma_{\eps_0}(\ul \fk_0) \bigr\|^{\rT^\ell W}_\eps \underset{\eps\to\eps_0}{\longrightarrow} 0 .
$
\item[\bf{[Uniform Continuity]}]
There are monotone continuous functions $c^\ell_\sigma : [0,\infty) \to [0,\infty)$ 
and $b^\ell_\sigma : [0,\infty) \to [1,\infty)$ with $c^\ell_\sigma(0)= 0$ 
so that for all $\eps\in\Delta_\sigma$ and $\ul\fk,\ul\fl\in\rT^\ell\cV_\fK$ we have\footnote{
Here the $b^\ell_\sigma$ factor accounts for the fact that tangent maps have mixed scaling properties with respect to the unbounded vector entries, as discussed in Remark~\ref{rmk:tangent map}.}
$$
\bigl\|  \rT^\ell\sigma_\eps (\ul\fl)  -   \rT^\ell\sigma_\eps (\ul\fk) \bigr\|^{\rT^\ell W}_\eps 
\leq c^\ell_\sigma(\|\ul\fl- \ul\fk\|^{\rT^\ell\fK}) \, b^\ell_\sigma(\max\{ \|\fl\|^{\rT_\bullet^\ell\fK} ,  \|\fk\|^{\rT_\bullet^\ell\fK} \}) . 
$$
\end{itemlist}
In particular, this guarantees
\begin{itemlist}
\item[\bf{[Continuity w.r.t.\ $\mathbf{\|\cdot\|_0}$]}]   $(\eps,\ul\fk)\mapsto \rT^\ell\sigma_\eps(\ul\fk)$ is a continuous map
$\Delta_\sigma\times\rT^\ell\cV_\fK \to \bigl( \rT^\ell\cV_W , \|\cdot\|^{\rT^\ell W}_0\bigr)$. 
\item[\bf{[Uniform Bound]}]
For all $\eps\in\Delta_\sigma$ and $\ul\fk \in\rT^\ell\cV_\fK$ we have
\begin{equation} \label{eq:uniform bound}
\bigl\|  \rT^\ell\sigma_\eps (\ul\fk)  \bigr\|^{\rT^\ell W}_\eps 
\leq \delta_\sigma +  c^\ell_\sigma(\| \ul\fk \|^{\rT_\bullet^\ell\fK}) \, b^\ell_\sigma(\|\fk\|^{\rT_\bullet^\ell\fK})    . 
\end{equation}
\end{itemlist}
\end{theorem}

\begin{proof}[Proof of Theorem~\ref{thm:charts}]
Here is a table of contents for the steps of this proof: 
\begin{itemlist}
\item[Contraction property:] page \pageref{page contraction} 
\item[Solution Map:]  page \pageref{solution contraction}
\item[Uniform Bound and Uniform Continuity for Solution Maps:]  page \pageref{page uniform}
\item[Fibrewise Regularity of Solution Maps:]  page \pageref{page fibrewise regularity}
\item[Restriction to the Regularizing Case:]  page \pageref{page regularizing}
\begin{iitemlist}
\item[Adiabatic $\mathcal{C}^0$ Regularity:]  page \pageref{page adiabatic 0}
\item[Overview of Adiabatic $\mathcal{C}^\ell$ Regularity:]  page \pageref{page adiabatic overview}
\begin{iiitemlist}
\item[Uniform Bound of $\rT^\ell\sigma_\eps$:] page \pageref{page uniform bound} 
\item[Continuity of $\rT^\ell\sigma_\eps$ w.r.t.\ $\|\cdot\|_0$:] page \pageref{page continuity of derivatives}
\end{iiitemlist}
\item[Induction Base Case -- Adiabatic ${\mathcal{C}^1}$ Regularity:] page \pageref{page base case}
\begin{iiitemlist}
\item[Pointwise Continuity of ${\rT\sigma_\eps}$ in $\Delta$:] page \pageref{page pointwise 1}
\item[Uniform Continuity of ${\rT\sigma_\eps}$:] page \pageref{page uniform 1}
\end{iiitemlist}
\item[Induction Step -- Adiabatic ${\mathcal{C}^{\ell+1}}$ Regularity:]
page \pageref{page step}
\begin{iiitemlist}
\item[Higher Regularizing Property:] page \pageref{page higher reg}
\item[Controlling derivatives of $\widetilde \sigma_\eps$ by derivatives of $\sigma_\eps$:] page \pageref{page sigma}
\item[Controlling derivatives of $\widetilde Q_\eps$ by derivatives of $\cF_\eps$:] page \pageref{page Q}
\item[Pointwise Continuity of $\rD\rT^\ell\sigma_\eps$ in $\Delta$:] page \pageref{page pointwise l}
\item[Uniform Continuity of $\rD\rT^\ell\sigma_\eps$:] page \pageref{page uniform l}
\end{iiitemlist}
\end{iitemlist}
\end{itemlist}

\smallskip
\noindent
{\bf Contraction Property:} \label{page contraction}
To check the contraction property \eqref{eq:contraction} with $B:=\overline B_\eps$ on $K:=\fK$ and the subset $\cV_{\overline W}=\cV_{\overline W,\eps}$ of $\overline W:= \overline W_\eps$ equipped with the norm $\|\cdot\|:=\|\cdot\|^W_\eps$ for any $\eps\in\Delta_Q$, we choose a fixed $\theta\in(0,1)$ for all $\eps\in\Delta_Q$ in \eqref{eq:B-contraction} to obtain a preliminary $\delta_\sigma:=\delta_\theta>0$ with  
 \begin{align*}
\fk\in\fK, w,w'\in \cV_{\overline W,\eps},  \|w\|^W_\eps,  \|w'\|^W_\eps \leq \delta_\sigma  \quad \Longrightarrow \quad
\| \overline B_\eps (\fk, w' ) - \overline B_\eps(\fk, w) \|^W_\eps \leq  \theta \| w' - w \|^W_\eps . 
\end{align*}

\smallskip
\noindent
{\bf Solution Maps:} \label{solution contraction}
Next, we choose a possibly smaller $0<\delta_\sigma\leq\min\{\delta_\theta,\delta_Q\}$ -- with $\delta_Q>0$ from Lemma~\ref{lem:inverses} for later purposes -- so that the above continues to hold along with the inclusion $\{ w\in \overline W \,|\, \|w \|^{\overline W}\leq \delta_\sigma \} \subset \cV_{\overline W}$. 
Now we use \eqref{eq:B-small} with $\delta=(1-\theta) \delta_\sigma$ to find neighbourhoods $\cV_\fK:=\cV_{\fK,\delta}\subset\fK$ of $0$ and $\Delta_\sigma:=\Delta_\delta\subset\Delta_Q$ of $0$ that guarantee \eqref{eq:small} for all $\eps\in\Delta_\sigma$, 
\begin{align*}
\fk\in\cV_\fK  \quad \Longrightarrow \quad
\| \overline B_\eps (\fk, 0 )  \|^W_\eps 
< (1-\theta) \delta_\sigma. 
\end{align*}
This confirms the assumptions of Lemma~\ref{lem:solution classical} for each $\eps\in\Delta_\sigma$, which then provides unique solution maps $\sigma_\eps: \cV_\fK \to \cV_{\overline W,\eps}$ such that
$$
\bigl\{ (\fk,w)\in \cV_\fK \times \cV_{\overline W,\eps} \,\big|\, w- \overline B_\eps(\fk,w) =0 ,  \|w\|^W_\eps < \delta_\sigma \bigr\}
= \bigl\{  (\fk,  \sigma_\eps(\fk)  ) \,\big|\,  \fk\in\cV_\fK  \bigr\} . 
$$
Equivalently, these solutions $\sigma_\eps(\fk)=w_{\eps,\fk}=(\fc_{\eps,\fk},\gamma_{\eps,\fk})\in \fC\times\cV_{\overline W,\eps}$ satisfy besides the smallness condition $\|w_{\eps,\fk}\|^W_\eps =\|\fc_{\eps,\fk}\|^\fC + \|\gamma_{\eps,\fk}\|^\Gamma_\eps \leq \delta_\sigma$ the equations  
\begin{align}
 (\fc_{\eps,\fk},\gamma_{\eps,\fk}) = \overline B_\eps(\fk,\fc_{\eps,\fk},\gamma_{\eps,\fk})
& \quad\Leftrightarrow\quad
(\fc_{\eps,\fk},\gamma_{\eps,\fk}) =
\overline Q_\eps(0) \bigl( \fk   , \rD\overline\cF_\eps(0) \gamma -  \overline\cF_\eps(\gamma) \bigr)   \nonumber\\
& \quad\Leftrightarrow\quad
\overline P_\eps(0)(\fc_{\eps,\fk},\gamma_{\eps,\fk}) =
 \bigl( \fk   , \rD\overline\cF_\eps(0) \gamma -  \overline\cF_\eps(\gamma) \bigr)   \nonumber\\
& \quad\Leftrightarrow\quad
\bigl( \overline\pi_\fK(\gamma_{\eps,\fk}), \rD\overline\cF_\eps(0) \gamma_{\eps,\fk} - \fc_{\eps,\fk} \bigr)
= \bigl( \fk   , \rD\overline\cF_\eps(0) \gamma_{\eps,\fk} -  \overline\cF_\eps(\gamma_{\eps,\fk}) \bigr)
\nonumber\\
& \quad\Leftrightarrow\quad
 \overline\pi_\fK(\gamma_{\eps,\fk})=\fk
\;\text{and}\;
\overline\cF_\eps(\gamma_{\eps,\fk})  =  \fc_{\eps,\fk}   \label{eq:defining}\\
& \quad\Leftrightarrow\quad
\overline g_{\eps,\fk}(\fc_{\eps,\fk},\gamma_{\eps,\fk}) = 0 .  \nonumber
\end{align}
Here we used \eqref{eq:B}, \eqref{eq:P}, and the last version of the equation results from \eqref{eq:stabilized equation} for the map $\overline g$ from Lemma~\ref{lem:inverses}. This confirms the first part of the Lemma. 

\smallskip
\noindent
{\bf Uniform Bound and Uniform Continuity for Solution Maps:} \label{page uniform}
The smallness condition for the solutions $\|\sigma_\eps(\fk)=(\fc_{\eps,\fk},\gamma_{\eps,\fk}) \|^W_\eps  \leq \delta_\sigma$ guarantees the claimed uniform bound on the solution maps with $c^0_\sigma\equiv 0$ and $b^0_\sigma\equiv 1$. 
Since $\delta_\sigma \leq \delta_\theta$ for an a priori fixed $\theta\in(0,1)$ this also guarantees the continuity estimate from \eqref{eq:sigma-est}, which further specifies with the help of \eqref{eq:Q-estimate} in Lemma~\ref{lem:inverses} to
confirm [Uniform Continuity] of the solution maps: For any $\eps\in\Delta_\sigma$ and $\fk,\fl\in\fK$ we have
\begin{align}
\| \sigma_\eps(\fl) - \sigma_\eps(\fk) \|^W_\eps 
& \leq  \tfrac 1{1-\theta}  \| \overline B_\eps(\fl,\sigma(\fk)) -  \overline B_\eps(\fk,\sigma(\fk)) \|^W_\eps \nonumber\\
& = \tfrac 1{1-\theta}  \bigl\| \overline Q_\eps(0) \bigl( \fl   , \rD\overline\cF_\eps(0) \gamma_{\eps,\fk} -  \overline\cF_\eps(\gamma_{\eps,\fk}) \bigr) - \overline Q_\eps(0) \bigl( \fk   , \rD\overline\cF_\eps(0)\gamma_{\eps,\fk} -  \overline\cF_\eps(\gamma_{\eps,\fk}) \bigr)  \bigr\|^W_\eps \nonumber \\
& = \tfrac 1{1-\theta}  \bigl\| \overline Q_\eps(0) \bigl( \fl  -  \fk   , 0 \bigr)  \bigr\|^W_\eps 
\;\leq \;  \tfrac 1{1-\theta} \, C_Q \| \fl  -  \fk  \|^\fK  \;=:\; c^0_\sigma(\|\fl-\fk\|^{\fK}) . 
\label{eq:sigma continuity}
\end{align}
Note that this uniform continuity of the solution maps did not require special regularity properties of the adiabatic Fredholm family. It is, however, also one of the two properties encoded in adiabatic $\cC^0$ regularity of the solution maps. Thus here is the explanation why the notion of adiabatic $\cC^0$ regularity of the Fredholm family in Definition~\ref{def:adiabatic C-l} requires just one property (pointwise continuity), whereas $\cC^\ell$ regularity for $\ell\geq 1$ requires two properties of the Fredholm family (pointwise and uniform continuity) -- which will ensure the two properties (pointwise and uniform continuity) encoded in adiabatic $\cC^\ell$ regularity of the solution maps. 

\smallskip
\noindent
{\bf Fibrewise Regularity of Solution Maps:} \label{page fibrewise regularity}
It remains to establish the regularity of the solution maps -- starting with the fibrewise regularity: 
If the adiabatic Fredholm family is fibrewise $\cC^\ell$-regular, that is $\cF_\eps : (\cV_\Gamma,{\|\cdot\|^\Gamma_\eps}) \to (\Omega,\|\cdot\|^\Omega_\eps)$ is uniformly $\cC^\ell$ for each $\eps\in\Delta$, then the maps $\overline B_\eps: (\fk,\gamma)\mapsto \overline Q_\eps(0) \bigl( \fk   , \rD\overline\cF_\eps(0) \gamma -  \overline\cF_\eps(\gamma) \bigr)$ are uniformly $\cC^\ell$ as well for each fixed $\eps\in\Delta_Q$. Indeed, $\overline Q_\eps(0)$ and $\rD\overline\cF_\eps(0)$ are linear operators and $\overline\cF_\eps$ inherits its regularity from $\cF_\eps$ by Lemma~\ref{lem:extension}.  
In particular, the notion of an adiabatic Fredholm family in Definition~\ref{def:fredholm} automatically includes fibrewise $\cC^1$ regularity, thus each contraction $\overline B_\eps$ is continuous and $\cC^1$ for fixed $\eps\in\Delta_Q$. When applying Lemma~\ref{lem:solution classical} to $B=\overline B_\eps: \cV_\fK \times \cV_{\overline W,\eps} \to \overline W_\eps$ this means, first, that $B$ is continuous when $\cV_K=\cV_\fK$ is equipped with the subspace topology of $\cV_\fK\subset\fK$. Then the Lemma asserts that each solution map $\sigma_\eps: \cV_\fK \to \cV_{\overline W,\eps}$ is continuous.  
Second, the $\cC^1$ regularity of each $\overline B_\eps$ means that in Lemma~\ref{lem:solution classical} the contraction $B$ is $\cC^1$, and hence each solution map $\sigma_\eps$ is $\cC^1$. 
Moreover, if the adiabatic Fredholm family is fibrewise $\cC^\ell$-regular for some $\ell>1$ in the sense of Definition~\ref{def:fibrewise C-l}, then the maps $\overline\cF_\eps$ are $\cC^\ell$ by Lemma~\ref{lem:fredholm}, and hence each contraction $B=\overline B_\eps: \cV_\fK \times \cV_{\overline W,\eps} \to \overline W_\eps$ is $\cC^\ell$. Then Lemma~\ref{lem:solution classical} guarantees that each solution map $\sigma_\eps: \cV_\fK \to \cV_{\overline W,\eps} \subset \overline W_\eps$ is $\cC^\ell$. 

\smallskip
\noindent{\bf Restriction to the Regularizing Case:} \label{page regularizing}
Next, if the adiabatic Fredholm family is regularizing, then $\overline\cF_\eps(\gamma_{\eps,\fk})  =  \fc_{\eps,\fk}\in\fC\subset\Omega$ implies $\gamma_{\eps,\fk}\in\cV_\Gamma\subset\overline\Gamma_\eps$, so that each solution map takes values $\sigma_\eps: \cV_\fK \to \cV_W, \fk\mapsto w_{\eps,\fk}=(\fc_{\eps,\fk},\gamma_{\eps,\fk})$ in the $\eps$-independent domain $\cV_W=\fC\times\cV_\Gamma$. This then allows us to rewrite the defining equation for the solution maps $\sigma_\eps(\fk)=\overline B_\eps(\fk,\sigma_\eps(\fk))=B_\eps(\fk,\sigma_\eps(\fk))$ in terms of the contractions $B_\eps : \fK \times \cV_W \to W$ with $\eps$-independent domain and target spaces from Lemma~\ref{lem:contraction-regularizing}. 

\smallskip
\noindent{\bf Adiabatic $\mathbf{\mathcal{C}^0}$ Regularity:} \label{page adiabatic 0}
Now assume in addition to the regularizing property that the adiabatic Fredholm family is adiabatic $\cC^0$-regular as in Definition~\ref{def:adiabatic C-l}, that is, given any $\eps_0\in\Delta$ and a solution $\gamma_0\in\cV_\Gamma$ of $\cF_{\eps_0}(\gamma_0)\in \fC$, we have 
$\bigl\|\cF_\eps (\gamma_0)  -  \cF_{\eps_0} (\gamma_0) \bigr\|^\Omega_\eps \rightarrow 0$ as $\eps\to \eps_0$.
We already established uniform continuity in \eqref{eq:sigma continuity} above -- which notably did not require any extra assumptions on the adiabatic Fredholm family. 
So it remains to prove pointwise continuity in $\eps\in\Delta_\sigma$ of the solution maps for fixed $\fk_0\in\cV_\fK$. This can be estimated from the defining equation $\sigma_\eps(\fk)=B_\eps(\fk,\sigma_\eps(\fk))$ and the contraction property \eqref{eq:B-contraction}
\begin{align}
\bigl\| \sigma_\eps(\fk_0) - \sigma_{\eps_0}(\fk_0) \bigr\|^W_\eps
&= \bigl\| B_\eps(\fk_0,\sigma_\eps(\fk_0)) - \sigma_{\eps_0}(\fk_0) \bigr\|^W_\eps
\nonumber\\ 
&\leq \bigl\| B_\eps(\fk_0,\sigma_\eps(\fk_0)) -  B_\eps(\fk_0,\sigma_{\eps_0}(\fk_0)) \bigr\|^W_\eps
+ \bigl\|B_\eps(\fk_0,\sigma_{\eps_0}(\fk_0)) - \sigma_{\eps_0}(\fk_0) \bigr\|^W_\eps
\label{needs regularizing}\\ 
&\leq \theta \bigl\| \sigma_\eps(\fk_0) - \sigma_{\eps_0}(\fk_0) \bigr\|^W_\eps
+ \bigl\| B_\eps(\fk_0,\sigma_{\eps_0}(\fk_0)) -  \sigma_{\eps_0}(\fk_0) \bigr\|^W_\eps . 
\nonumber
\end{align}
Here we used the regularizing property to ensure $\sigma_{\eps_0}(\fk_0)\in\cV_W=\fC\times\cV_\Gamma$, so
that we can apply the $\|\cdot\|^W_\eps$ norm to it. 
This is also crucial to make sense of the expression $B_\eps(\fk_0,\sigma_{\eps_0}(\fk_0))\in W$. Without the regularizing property, we would need to make sense of ``$\ol B_\eps(\fk_0,\sigma_{\eps_0}(\fk_0))$'' where $\ol B_\eps$ is defined on $\fK\times\cV_{\ol W,\eps}$ but $\sigma_{\eps_0}(\fk_0)\in \cV_{\ol W,\eps_0}$. This type of triangle inequality computation will be used repeatedly in the following, and each time crucially relies on the regularizing property although we won't keep pointing it out. 

Less crucially, the regularizing property also simplifies the defining equation \eqref{eq:defining} for the solution map, where we specify to $\eps=\eps_0$ and denote $(\fc_0,\gamma_0):=w_0:=\sigma_{\eps_0}(\fk_0)$,
\begin{align}
w_0
= B_0 (\fk,w_0) 
\qquad \Leftrightarrow\qquad
& 
\pi_\fK(\gamma_0) = \fk 
\quad\text{and}\quad 
\cF_{\eps_0}(\gamma_0)  =  \fc_0   
\label{eq:defining reg}
\end{align}
Now since $0<\theta<1$ in the above estimate, we can absorb the first summand into the left hand side to establish [Pointwise Continuity in $\Delta$] of the solution map,
\begin{align*}
\bigl\| \sigma_\eps(\fk_0) - \sigma_{\eps_0}(\fk_0) \bigr\|^W_\eps
&\leq \tfrac1{1-\theta} \bigl\| B_\eps(\fk_0,(\fc_0,\gamma_0)) -  \sigma_{\eps_0}(\fk_0) \bigr\|^W_\eps
\\ 
&\qquad\qquad
\text{\small using \eqref{eq:B} and the fact that $ Q_\eps(0)   P_\eps(0)={\rm Id}_{ \fC \times \Gamma}$ from Lemma~\ref{lem:inverses}} \\
&= \tfrac1{1-\theta} \bigl\|  Q_\eps(0) \bigl( \fk_0   , \rD\cF_\eps(0) \gamma_0 -  \cF_\eps(\gamma_0) \bigr) -   Q_\eps(0)   P_\eps(0) (\fc_0,\gamma_0)  \bigr\|^W_\eps
\\ 
&\leq \tfrac1{1-\theta} \bigl\|  Q_\eps(0) \bigr\|^{\cL(\fK\times\overline\Gamma_\eps,\overline W_\eps)}
\bigl\| \bigl( \fk_0   , \rD\cF_\eps(0) \gamma_0 -  \cF_\eps(\gamma_0) \bigr) 
- \bigl(  \pi_\fK(\gamma_0)  , \rD\cF_\eps(0) \gamma_0 -  \fc_0 \bigr)
\bigr\|^{\fK\times\overline\Gamma_\eps}
\\
&\qquad\qquad
\text{\small using \eqref{eq:Q-estimate} in Lemma~\ref{lem:inverses}} \\
&\leq \tfrac 1{1-\theta} \, C_Q 
\bigl(  \| \fk_0 -  \pi_\fK(\gamma_0) \|^{\fK} + 
\bigl\|  \cF_\eps(\gamma_0) - \fc_0 \bigr\|^\Gamma_\eps \bigr)
\\
&\qquad\qquad
\text{\small using \eqref{eq:defining reg} } \\
&\leq \tfrac 1{1-\theta} \, C_Q 
\bigl(  \| 0 \|^{\fK} + 
\bigl\|  \cF_\eps(\gamma_0) 
-  \cF_{\eps_0}(\gamma_0) \bigr\|^\Gamma_\eps \bigr)
\quad\underset{\eps\to\eps_0}{\longrightarrow} \quad 0 .
\end{align*}
Here the final convergence holds by adiabatic $\cC^0$ regularity of the adiabatic Fredholm family in Definition~\ref{def:adiabatic C-l} at the solution $\gamma_0$ of the $\eps_0$-equation modulo cokernel $\cF_{\eps_0}(\gamma_0)  =  \fc_0\in\fC$. 
Now continuity of the family of solution maps $\Delta_\sigma\times\cV_\fK \to \bigl( \cV_W , \|\cdot\|^W_0 \bigr), (\eps,\ul\fk)\mapsto\sigma_\eps(\ul\fk)$ can be deduced by combining [Lower Bound on Norms] with [Pointwise Continuity in $\Delta$]  and the previously established [Uniform Continuity] in \eqref{eq:sigma continuity}: 
For any $\eps_0\in\Delta_\sigma$ and $\fk_0\in\cV_\fK$ we obtain for $\Delta_\sigma\times\cV_\fK \ni (\eps,\fk)\to(\eps_0,\fk_0)$
\begin{align}
\bigl\| \sigma_\eps (\fk)  -   \sigma_{\eps_0} (\fk_0) \bigr\|^{W}_0 
&\leq 
\bigl\| \sigma_\eps (\fk)  - \sigma_{\eps_0} (\fk_0) \bigr\|^{W}_\eps 
\nonumber \\ 
& \leq 
\bigl\|  \sigma_\eps (\fk)  -   \sigma_\eps (\fk_0) \bigr\|^{W}_\eps 
+
\bigl\|  \sigma_\eps (\fk_0)  -   \sigma_{\eps_0} (\fk_0) \bigr\|^{W}_\eps 
\label{eq:continuity}\\
& \leq 
c^0_\sigma(\|\fk-\fk_0\|^{\fK})  +
\bigl\|  \sigma_\eps (\fk_0)  -   \sigma_{\eps_0} (\fk_0) \bigr\|^{W}_\eps 
\quad  \underset{(\eps,\fk)\to(\eps_0,\fk_0)}{\longrightarrow} \quad 0 .
\nonumber
\end{align}

\vfill
\pagebreak

\noindent
{\bf Overview of Adiabatic $\mathbf{\mathcal{C}^\ell}$ Regularity:} \label{page adiabatic overview}
Finally, we assume that the adiabatic Fredholm family is regularizing and adiabatic $\cC^\ell$-regular for some $\ell\geq 1$ as in Definition~\ref{def:adiabatic C-l}. Then we established above that each solution map $\sigma_\eps: \cV_\fK \to \cV_W$ takes values in the $\eps$-independent subspace $\cV_W\subset W$ and is $\cC^\ell$ with respect to the $\eps$-dependent norm $\|\cdot\|^W_\eps$ that gives rise to the ambient Banach space $\cV_W \subset \overline W_\eps$. 
The assumption of adiabatic $\cC^\ell$ regularity in Definition~\ref{def:adiabatic C-l} guarantees the higher order regularizing property 
$$
\ul\gamma\in \rT^\ell\cV_{\overline\Gamma,\eps}, \; \rT^\ell\overline\cF_\eps(\ul\gamma) \in \rT^\ell\Omega \quad \Longrightarrow \quad \ul\gamma\in \rT^\ell\cV_\Gamma 
$$
and it guarantees two types of regularity for varying $\eps\in\Delta$:
\begin{itemlist}
\item[{\bf [Pointwise Continuity of $\mathbf\rT^\ell\cF_\eps$ in $\Delta$ at solutions modulo $\mathbf{\mathfrak C}$]}] 
Given any $\eps_0\in\Delta$ and a solution $\ul\gamma_0\in\rT^\ell\cV_\Gamma$ of the linearized equation modulo cokernel $\rT^\ell\cF_{\eps_0}(\ul\gamma_0)\in \rT^\ell\fC$, we have 
\begin{equation} \label{eq:T-ell-F pointwise}
\bigl\|  \rT^\ell\cF_\eps (\ul\gamma_0)  -   \rT^\ell\cF_{\eps_0} (\ul\gamma_0) \bigr\|^{\rT^\ell\Omega}_\eps \underset{\eps\to \eps_0}{\longrightarrow} 0 .
\end{equation}
\item[{\bf[Uniform Continuity of $\rD\rT^{\ell-1}\cF_\eps$]}]
There is a monotone continuous function $c^\ell_{\rT\cF}:=c^{\ell,\delta_Q}_{\rT\cF} : [0,\infty) \to [0,\infty)$ with $c^\ell_{\rT\cF}(0)= 0$ so that for all $\eps\in\Delta$ and $\ul\gamma^\fl,\ul\gamma^\fk\in\rT^{\ell-1}\cV_\Gamma$ with $\|\gamma^\fl_0\|^\Gamma_\eps, \|\gamma^\fk_0\|^\Gamma_\eps \leq \delta_Q$
we have (via Remark~\ref{rmk:uniform DF} for $\delta=\delta_Q$)
\begin{align} \label{eq:T-ell-F uniform}
&\bigl\| \rD\rT^{\ell-1}\cF_\eps (\ul\gamma^\fl)  -   \rD\rT^{\ell-1}\cF_\eps (\ul\gamma^\fk) \bigr\|^{\cL(\rT^{\ell-1}\ol\Gamma_\eps,\rT^{\ell-1}\ol\Omega_\eps )}  \\
&
\qquad\qquad\qquad\qquad\qquad\qquad
\leq 
c^\ell_{\rT\cF}(\|\ul\gamma^\fl - \ul\gamma^\fk \|^{\rT^{\ell-1}\Gamma}_\eps) \,  \max\bigl\{ 1, \| \ul\gamma^\fl \|^{\rT_\bullet^{\ell-1} \Gamma}_\eps,  \| \ul\gamma^\fk \|^{\rT_\bullet^{\ell-1} \Gamma}_\eps \bigr\}^\ell    . 
\nonumber
\end{align}
\end{itemlist}
Now our goal is to show that the family of solution maps $\bigl( \sigma_\eps: \cV_\fK \to \cV_W \bigr)_{\eps\in\Delta_\sigma}$ is adiabatic $\cC^\ell$ in the sense that it satisfies the analogous types of regularity for varying $\eps\in\Delta$: 
\begin{itemlist}
\item[\bf{[Pointwise Continuity of $\mathbf\rT^\ell\sigma_\eps$ in $\mathbf\Delta$]}]
Given any $\eps_0\in\Delta_\sigma$ and $\ul \fk_0 \in \rT^\ell\cV_\fK$ we have 
\begin{equation}   \label{eq:T-ell-sigma pointwise}
\bigl\| \rT^\ell\sigma_\eps(\ul \fk_0) - \rT^\ell\sigma_{\eps_0}(\ul \fk_0) \bigr\|^{\rT^\ell W}_\eps \underset{\eps\to \eps_0}{\longrightarrow} 0 . 
\end{equation}
\item[\bf{[Uniform Continuity of $\mathbf\rT^\ell\sigma_\eps$]}]
There are monotone continuous functions $c^\ell_\sigma : [0,\infty) \to [0,\infty)$ 
and $b^\ell_\sigma : [0,\infty) \to [1,\infty)$ with $c^\ell_\sigma(0)= 0$ 
so that for all $\eps\in\Delta_\sigma$ and $\ul\fk,\ul\fl\in\rT^\ell\cV_\fK$ we have
\begin{equation}   \label{eq:T-ell-sigma uniform} 
\bigl\|  \rT^\ell\sigma_\eps (\ul\fl)  -   \rT^\ell\sigma_\eps (\ul\fk) \bigr\|^{\rT^\ell W}_\eps 
\leq c^\ell_\sigma(\|\ul\fl- \ul\fk\|^{\rT^\ell\fK}) \, b^\ell_\sigma(\max\{ \|\fl\|^{\rT_\bullet^\ell\fK} ,  \|\fk\|^{\rT_\bullet^\ell\fK} \}) . 
\end{equation}
\end{itemlist}
Before proving these two continuity properties we will show that they imply the remaining claims. 

\smallskip
\noindent
{\bf[Uniform Bound]} \label{page uniform bound} 
follows from combining the [Uniform Continuity of $\rT^\ell\sigma_\eps$] with the fact that 
$\rT^\ell\sigma_\eps (\fk^0, 0, \ldots , 0) = (\sigma_\eps(\fk^0), 0 , \ldots , 0 )$, and thus for any $\eps\in\Delta_\sigma$ and $\ul\fk=(\fk^0,\fk^1,\ldots,\fk^{N_\ell})\in\rT^\ell\cV_\fK$
\begin{align*}
\bigl\|  \rT^\ell\sigma_\eps (\ul\fk)  \bigr\|^{\rT^\ell W}_\eps 
&\leq \bigl\|  \rT^\ell\sigma_\eps (\fk^0, \fk^1 \ldots \fk^{N_\ell}) -  \rT^\ell\sigma_\eps (\fk^0, 0 \ldots  0)  \bigr\|^{\rT^\ell W}_\eps 
+ \bigl\|  \rT^\ell\sigma_\eps (\fk^0, 0 \ldots  0)  \bigr\|^{\rT^\ell W}_\eps  \\
&\leq
c^\ell_\sigma(\| (0, \fk^1 \ldots \fk^{N_\ell}) \|^{\rT^\ell\fK}) \, b^\ell_\sigma(\max\{ \|\fk\|^{\rT_\bullet^\ell\fK} ,  \|(\fk^0,0 \ldots 0)\|^{\rT_\bullet^\ell\fK} \}) 
 + \bigl\| (\sigma_\eps(\fk^0), 0  \ldots  0 ) \bigr\|^{\rT^\ell W}_\eps \\
&=
c^\ell_\sigma(\| \ul\fk \|^{\rT_\bullet^\ell\fK}) \, b^\ell_\sigma(\|\fk\|^{\rT_\bullet^\ell\fK}) 
 + \bigl\| \sigma_\eps(\fk^0)  \bigr\|^W_\eps \\
&\leq
c^\ell_\sigma(\| \ul\fk \|^{\rT_\bullet^\ell\fK}) \, b^\ell_\sigma(\|\fk\|^{\rT_\bullet^\ell\fK})  + \delta_\sigma . 
\end{align*}

\smallskip
\noindent
{\bf [Continuity of $\mathbf\rT^\ell\sigma_\eps$ w.r.t.\ $\mathbf{\|\cdot\|_0}$]} \label{page continuity of derivatives}
follows by combining 
[Pointwise Continuity] and [Uniform Continuity] with  [Lower Bound on Norms]: 
The map
$\Delta_\sigma\times\rT^\ell\cV_\fK \to \bigl( \rT^\ell\cV_W , \|\cdot\|^{\rT^\ell W}_0\bigr), 
(\eps,\ul\fk)\mapsto \rT^\ell\sigma_\eps(\ul\fk)$ is continuous since for any $\eps_0\in\Delta_\sigma$ and $\ul\fk_0\in\rT^\ell\cV_\fK$ we have for $\Delta_\sigma\times\rT^\ell\cV_\fK \ni (\eps,\ul\fk)\to(\eps_0,\ul\fk_0)$
\begin{align}
\bigl\| \rT^\ell\sigma_\eps (\ul\fk)  -   \rT^\ell\sigma_{\eps_0} (\ul\fk_0) \bigr\|^{\rT^\ell W}_0 
&\leq 
\bigl\| \rT^\ell\sigma_\eps (\ul\fk)  - \rT^\ell\sigma_{\eps_0} (\ul\fk_0) \bigr\|^{\rT^\ell W}_\eps 
\nonumber\\ 
& \leq 
\bigl\|  \rT^\ell\sigma_\eps (\ul\fk)  -   \rT^\ell\sigma_\eps (\ul\fk_0) \bigr\|^{\rT^\ell W}_\eps 
+
\bigl\|  \rT^\ell\sigma_\eps (\ul\fk_0)  -   \rT^\ell\sigma_{\eps_0} (\ul\fk_0) \bigr\|^{\rT^\ell W}_\eps 
\nonumber \\
& \leq 
c^\ell_\sigma(\|\ul\fk- \ul\fk_0\|^{\rT^\ell\fK} ) b^\ell_\sigma(\max\{ \|\fk\|^{\rT_\bullet^\ell\fK} ,  \|\fk_0\|^{\rT_\bullet^\ell\fK} \})  +
\bigl\|  \rT^\ell\sigma_\eps (\ul\fk_0)  -   \rT^\ell\sigma_{\eps_0} (\ul\fk_0) \bigr\|^{\rT^\ell W}_\eps 
\label{eq:continuous derivatives}\\
&\hspace{-6mm} \underset{(\eps,\ul\fk)\to(\eps_0,\ul\fk_0)}{\longrightarrow} \quad 0 .
\phantom{M^N}
\nonumber
\end{align}
Here we used the fact that $c^\ell_\sigma$ is continuous with $c^\ell_\sigma(0)=0$ and that
$\|\fk\|^{\rT_\bullet^\ell\fK}\to  \|\fk_0\|^{\rT_\bullet^\ell\fK}$ as $\ul\fk \to\ul\fk_0$, so that 
$b^\ell_\sigma(\max\{ \|\fk\|^{\rT_\bullet^\ell\fK} ,  \|\fk_0\|^{\rT_\bullet^\ell\fK} \})$ stays bounded by continuity of $b^\ell_\sigma$

So it remains to prove the two properties of adiabatic $\cC^\ell$ regularity for the solution maps -- which we will do by induction in $\ell\in\bN$. The induction step -- below after \eqref{eq:D sigma} -- can be interpreted to work with $\ell=0$ as the base case (which is established above), but to improve accessibility of the argument, we first go through the computations for the case $\ell=1$.  

\begin{center}
{\underline{\bf Induction Base Case -- Adiabatic $\mathbf{\mathcal{C}^1}$ Regularity}} \label{page base case}
\end{center}

To begin this proof we use the already established differentiability for fixed $\eps\in\Delta_Q$ to compute the derivative of the solution maps at $\fk_0\in\cV_\fK$ from the defining equation \eqref{eq:defining} in terms of the family of maps $\ol g_{\eps,\fk}: ( \fc, \gamma) \mapsto (\ol \pi_\fK(\gamma)-\fk, \ol \cF_\eps(\gamma) - \fc )$ from Lemma~\ref{lem:inverses}. Here we denote $\sigma_\eps(\fk_0)=:(\fc_\eps,\gamma_\eps)$ to compute for any $\fk_1\in\fK$ 
\begin{align}
& \ol g_{\eps,\fk_0+\lambda\fk_0}(\sigma_\eps(\fk_0+\lambda\fk_1)) = 0   \nonumber \\
& \quad\Rightarrow\quad
\tfrac{\rd}{\rd\lambda} \ol g_{\eps,\fk_0+\lambda\fk_0}(\sigma_\eps(\fk_0+\lambda\fk_1))\big|_{\lambda =0} 
\;=\; 
\rD \ol g_{\eps,\fk_0} (\sigma_\eps(\fk_0)) \rD \sigma_\eps (\fk_0) \fk_1 + \tfrac\partial{\partial\fk} \ol g_{\eps,\fk_0}(\sigma_\eps(\fk_0))  \fk_1
\;=\; 0 \nonumber \\
&\quad\Leftrightarrow\quad
\ol P_\eps(\gamma_\eps) \rD \sigma_\eps (\fk_0) \fk_1 + ( - \fk_1 ,  0) = (0,0) 
\nonumber \\
&\quad\Leftrightarrow\quad
\ol P_\eps(\gamma_\eps) \rD \sigma_\eps (\fk_0) \fk_1 = ( \fk_1 ,  0) \; , 
\qquad\text{where} \quad (\fc_\eps,\gamma_\eps)=\sigma_\eps(\fk_0)      \label{eq:D defining} \\
&\quad\Leftrightarrow\quad
 \rD \sigma_\eps (\fk_0) \fk_1 = \ol Q_\eps(\gamma_\eps) ( \fk_1 ,  0)  .   \nonumber
\end{align}
Thus $\rD \sigma_\eps (\fk_0): \rT_{\fk_0}\fK = \fK \to \rT_{\sigma_\eps (\fk_0)} W=W$ is the composition of the inclusion $\fK \hookrightarrow \fK\times \Gamma$ with the inverse $\ol Q_\eps(\gamma_\eps)=\rD\ol  g_{\eps,\fk_0} (\sigma_\eps(\fk_0))^{-1}:  \fK \times \Omega \to \fC \times \Gamma = W$ of $\rD\ol  g_{\eps,\fk_0} (\fc_\eps,\gamma_\eps)$. Here the inverse $\ol Q_\eps(\gamma_\eps)$ exists -- and varies continuously with $\gamma_\eps$ -- since we have taken care to construct the solution map so that
$\|\gamma_\eps\|^\Gamma_\eps \leq\| \sigma_\eps(\fk_0) = (\fc_\eps,\gamma_\eps)\|^W_\eps < \delta_\sigma\leq \delta_Q$ guarantees applicability of Lemma~\ref{lem:inverses}. 

\smallskip
\noindent
{\bf[Regularizing Property of $\mathbf{\rD\sigma_\eps}$]} \label{page regularizing 1}
Recall that $\cC^1$ adiabatic regularity includes two regularizing properties in Definition~\ref{def:regularizing}. 
The first was used to establish the solution maps as maps between the $\eps$-independent dense subspaces $\sigma_\eps: \cV_\fK \to \cV_W$. Now given $(\fk_0,\fk_1)\in\rT\cV_\fK$ and denoting 
$\sigma_\eps(\fk_0)=:(\fc_\eps,\gamma_\eps)$ and $(\fc,\zeta):=\rD \sigma_\eps (\fk_0) \fk_1 \in \ol W_\eps = \fC\times\ol\Gamma_\eps$, the above defining equation together with \eqref{eq:P} yields
\begin{align*} 
\ol P_\eps(\gamma_\eps) (\fc,\zeta) + ( - \fk_1 ,  0) = (0,0) 
&\quad\Leftrightarrow\quad
\ol\pi_\fK(\zeta) - \fk_1 = 0 , \quad\rD\ol\cF_\eps( \gamma_\eps ) \zeta + \fc = 0  \\
&\quad\Rightarrow\quad
\rD\ol\cF_\eps( \gamma_\eps ) \zeta = - \fc \in \fC 
\quad\Rightarrow\quad \zeta \in \Gamma , 
\end{align*}
where we used the second (linearized) regularizing property \eqref{eq:reg 1} in Definition~\ref{def:regularizing}. 
With Remark~\ref{rmk:regularizing} this establishes the tangent solution maps as maps between the $\eps$-independent dense subspaces $\rT \sigma_\eps: \rT\cV_\fK \to \rT\cV_W$. 

\smallskip
\noindent
{\bf[Pointwise Continuity of $\mathbf{\rT\sigma_\eps}$ in $\mathbf\Delta$]} \label{page pointwise 1}
will be proved  by fixing $\eps_0\in\Delta_\sigma$, $\fk_0\in\cV_\fK$, $\fk_1\in\fK$, and writing 
$\rT\sigma_\eps(\fk_0,\fk_1)=(\sigma_\eps(\fk_0),\rD \sigma_\eps (\fk_0)\fk_1)$
where
$\sigma_\eps(\fk_0)=:(\fc_\eps,\gamma_\eps)$ solves \eqref{eq:defining} for all $\eps\in\Delta_\sigma$. 
To establish pointwise continuity 
$\bigl\| \rT\sigma_\eps(\fk_0,\fk_1) - \rT\sigma_{\eps_0}(\fk_0,\fk_1) \bigr\|^{\rT W}_\eps \to 0$ as $\eps\to \eps_0$ as in \eqref{eq:T-ell-sigma pointwise} note that we already showed
$$
\| \sigma_\eps(\fk_0) - \sigma_{\eps_0}(\fk_0) \|^{W}_\eps 
= \| \gamma_\eps -\gamma_{\eps_0}\|^{\Gamma}_\eps 
+
\| \fc_\eps -\fc_{\eps_0}\|^{\fC}  
\underset{\eps\to \eps_0}{\longrightarrow} 0 ,
$$
so it remains to prove $\bigl\| \rD\sigma_\eps(\fk_0)\fk_1 - \rD\sigma_{\eps_0}(\fk_0)\fk_1 \bigr\|^{W}_\eps \to 0$ as $\eps\to \eps_0$. Note here that applying the $\|\cdot\|^W_\eps$ norm to this difference only makes sense because we already established that the differentials take values in the $\eps$-independent space $W$. 
Thus $\rD \sigma_{\eps_0} (\fk_0)\fk_1 =: (\fc_1,\gamma_1)\in \fC \times \Gamma$ 
is the solution of 
\begin{equation}\label{eq:D sigma equation}
\pi_\fK(\gamma_1) = \fk_1 \quad\text{and}\quad  \rD\cF_{\eps_0}(\gamma_{\eps_0}) \gamma_1 = \fc_1  . 
\end{equation}
Here we used the explicit form of $\rD g_{\eps,\fk_0} (\fc_\eps,\gamma_\eps): (\fk,\gamma)\mapsto(\pi_\fK, \rD\cF_\eps(\gamma_\eps) \gamma - \fc)$ from Lemma~\ref{lem:inverses}. 
Thus $(\gamma_{\eps_0}, \gamma_1)\in\rT\cV_\Gamma$ is a solution of the linearized equation modulo cokernel 
$$
\rT\cF_{\eps_0}(\gamma_{\eps_0}, \gamma_1)
\;=\; \bigl( \cF_{\eps_0}(\gamma_{\eps_0}) , \rD\cF_{\eps_0}(\gamma_{\eps_0}) \gamma_1 \bigr)
\;=\; \bigl( \fc_{\eps_0} , \fc_1 \bigr)
\;\in\; \rT\fC
$$ 
which will allow us to apply \eqref{eq:T-ell-F pointwise} below.
Now for $\eps\to\eps_0\in\Delta_\sigma$ we can use the fact that $Q_\eps(\gamma_\eps)\rD g_{\eps,\fk_0} (\fc_\eps,\gamma_\eps) ={\rm Id}_{ \fC \times \Gamma}$ from Lemma~\ref{lem:inverses} to estimate
\begin{align*}
\bigl\| \rD\sigma_\eps(\fk_0)\fk_1 - \rD\sigma_{\eps_0}(\fk_0)\fk_1 \bigr\|^W_\eps 
&= 
\bigl\| Q_\eps(\gamma_\eps) ( \fk_1 ,  0)  -  Q_\eps(\gamma_\eps)\rD g_{\eps,\fk_0} (\fc_\eps,\gamma_\eps)  (\fc_1,\gamma_1) \bigr\|^W_\eps 
\\
&\leq 
\bigl\| Q_\eps(\gamma_\eps) \bigr\|^{\cL( \fK \times \ol\Omega_\eps , \ol W_\eps)} 
\bigl( 
\bigl\|  \fk_1  - \pi_\fK(\gamma_1) \bigr\|^\fK
+ 
\bigl\|  - \rD\cF_{\eps} (\gamma_\eps) \gamma_1 + \fc_1  \bigr\|^\Omega_\eps 
\bigr) 
\\
&\qquad\qquad
\text{\small using \eqref{eq:Q-estimate} and \eqref{eq:D sigma equation} } \\
&\leq 
C_Q
\bigl\|  \rD \cF_{\eps} (\gamma_\eps) \gamma_1 -   \rD\cF_{\eps_0}(\gamma_{\eps_0}) \gamma_1 \bigr\|^\Omega_\eps   \\
&\leq 
C_Q \bigl( 
\bigl\|  \rD \cF_{\eps} (\gamma_\eps) \gamma_1 -   \rD\cF_\eps(\gamma_{\eps_0}) \gamma_1 \bigr\|^\Omega_\eps 
+ \bigl\|  \rD\cF_\eps(\gamma_{\eps_0}) \gamma_1 -   \rD\cF_{\eps_0}(\gamma_{\eps_0}) \gamma_1 \bigr\|^\Omega_\eps \bigr) \\
&\qquad\qquad
\text{\small using \eqref{eq:T-ell-F uniform} and \eqref{eq:T-ell-F pointwise} for $\rT\cF_\eps(\gamma_0,\gamma_1)=(\cF_\eps(\gamma_0),\rD\cF_\eps(\gamma_0)\gamma_1)$} \\
&\leq
C_Q \bigl( 
c^1_\cF(\|\gamma_\eps- \gamma_{\eps_0}\|^\Gamma_\eps)
+ \bigl\|  \rT\cF_\eps(\gamma_{\eps_0}, \gamma_1) -   \rT\cF_{\eps_0}(\gamma_{\eps_0}, \gamma_1) \bigr\|^\Omega_\eps \bigr) 
\quad\underset{\eps\to\eps_0}{\longrightarrow} \quad 0 .
\end{align*}

\smallskip
\noindent
{\bf[Uniform Continuity of $\mathbf{\rT\sigma_\eps}$]} \label{page uniform 1}
follows similarly from the previously established
uniform continuity of $\sigma_\eps$ together with uniform continuity of $\rD\sigma_\eps$. 
To establish the latter, we consider $\eps\in\Delta_\sigma$ and $\ul\fk=(\fk_0,\fk_1), \ul\fl=(\fl_0,\fl_1) \in \rT\cV_\fK$, 
and write 
$\sigma_\eps(\fk_0)=:(\fc_\eps,\gamma_\eps)$ resp.\ $\sigma_\eps(\fl_0)=: (\fc'_\eps,\gamma'_\eps)$ 
to estimate
\begin{align*}
 \bigl\| \rD\sigma_\eps (\fl_0) \,  \fl_1 - \rD\sigma_\eps(\fk_0) \,  \fk_1 \bigr\|^W_\eps  
&= \bigl\| Q_\eps(\gamma'_\eps) ( \fl_1 ,  0)  -  Q_\eps(\gamma_\eps) ( \fk_1 ,  0) \bigr\|^W_\eps 
\\
&\leq \bigl\| Q_\eps(\gamma'_{\eps}) (\fl_1,  0)  -  Q_\eps(\gamma'_{\eps}) ( \fk_1 ,  0) \bigr\|^W_\eps
+  \bigl\| Q_\eps(\gamma'_{\eps}) ( \fk_1,  0)  -  Q_\eps(\gamma_\eps) ( \fk_1 ,  0) \bigr\|^W_\eps 
\\
&\leq  \bigl\| Q_\eps(\gamma'_{\eps}) \bigr\|^{\cL( \fK \times \ol\Omega_\eps ,\ol W_\eps)} \| \fl_1 - \fk_1 \|^W_\eps
+  \bigl\| Q_\eps(\gamma'_{\eps})  -  Q_\eps(\gamma_\eps)\bigr\|^{\cL( \fK \times \ol\Omega_\eps ,\ol W_\eps)} \| \fk_1 \|^W_\eps \\
&\qquad\qquad
\text{\small using \eqref{eq:Q-estimate} and \eqref{eq:Q cont} in Lemma~\ref{lem:inverses} } \\
&\leq 
C_Q  \| \fl_1 - \fk_1  \|^\fK 
+ (C_Q )^2 \, c^1_\cF(\| \gamma'_{\eps} - \gamma_\eps \|^\Gamma_\eps)   \|  \fk_1\|^\fK .
\end{align*}
Combining this with the already established uniform continuity estimate \eqref{eq:sigma continuity} for $\sigma_\eps$ yields 
\begin{align*}
 \bigl\| \rT\sigma_\eps (\ul \fl) - \rT\sigma_\eps(\ul\fk) \bigr\|^{\rT W}_\eps  
&=
 \bigl\| \sigma_\eps (\fl_0) - \sigma_\eps(\fk_0)  \bigr\|^W_\eps  
+
 \bigl\| \rD\sigma_\eps (\fl_0) \,  \fl_1 - \rD\sigma_\eps(\fk_0) \,  \fk_1 \bigr\|^W_\eps   \\
&\leq 
\tfrac{C_Q}{1-\theta} \| \fl_0  -  \fk_0  \|^\fK  + C_Q  \| \fl_1 - \fk_1  \|^\fK
+ (C_Q )^2 \, c^1_\cF \bigl(  \tfrac{C_Q}{1-\theta} \,  \| \fl_0  -  \fk_0  \|^\fK \bigr) \| \fk_1\|^\fK \\
&\leq 
\tfrac{C_Q}{1-\theta}  \| (\fl_0,\fl_1) - (\fk_0,\fk_1) \|^{\rT\fK}
+ (C_Q )^2 \, c^1_\cF \bigl(  \tfrac{C_Q}{1-\theta} \,  \| \fl_0  -  \fk_0  \|^\fK \bigr)  \max\{1,  \|  \fl_1\|^\fK, \|  \fk_1\|^\fK \} \\
&\leq 
\tfrac{C_Q}{1-\theta}  \| \ul\fl - \ul\fk \|^{\rT\fK}
+ (C_Q )^2 \, c^1_\cF \bigl(  \tfrac{C_Q}{1-\theta} \,  \| \ul\fl  -  \ul\fk  \|^{\rT\fK} \bigr)  \max\{1,  \| \ul\fl\|^{\rT_\bullet\fK},  \| \ul\fk\|^{\rT_\bullet\fK} \} \\
&\leq 
c^1_\sigma( \| \ul\fl - \ul\fk \|^{\rT\fK} ) \, b^1_\sigma(\max\{ \|\fl\|^{\rT_\bullet\fK} ,  \|\fk\|^{\rT_\bullet\fK} \})
\end{align*}
with 
$$
c^1_\sigma(x) := \tfrac{C_Q}{1-\theta} x  + (C_Q )^2 \, c^1_\cF (  \tfrac{C_Q}{1-\theta}  x) 
\qquad\text{and}\qquad
b^1_\sigma(x) := \max\{1,x\} .
$$
This finishes the proof that adiabatic $\cC^1$ regularity of the adiabatic Fredholm family implies adiabatic $\cC^1$ regularity of the solution maps. 

\vfill
\pagebreak

\begin{center}
{\underline{\bf  Induction Step -- Adiabatic $\mathbf{\mathcal{C}^{\ell+1}}$ Regularity}}
\label{page step}
\end{center}

Now suppose we proved that adiabatic $\cC^\ell$ regularity of an adiabatic Fredholm family as in Definition~\ref{def:adiabatic C-l} implies adiabatic $\cC^\ell$ regularity of the solution maps and uniform bounds \eqref{eq:uniform bound} for some $\ell\geq 1$\footnote{To make sense of this proof with induction base case $\ell=0$, ignore intermediate arguments involving $\rT^{0-1}$ of any map $f$, and interpret $\rD\rT^{0-1}f(v)= f$.} 
and consider an adiabatic $\cC^{\ell+1}$-regular adiabatic Fredholm family $\bigl((\cF_\eps:\cV_\Gamma\to\Omega)_{\eps\in\Delta},\ldots\bigr)$ and its solution maps $\bigl( \sigma_\eps: \cV_\fK \to\cV_W \bigr)_{\eps\in\Delta_\sigma}$ as constructed above. 
Then the $(\ell+1)$-st tangent of the solution maps
\begin{align*}
\rT^{\ell+1}\sigma_\eps = \rT \, \rT^\ell\sigma_\eps : \quad \rT^{\ell+1} \cV_\fK = \rT^\ell \cV_\fK \times \rT^\ell \cV_\fK
&\quad \longrightarrow \quad 
\rT^{\ell+1} W =  \rT^\ell W \times \rT^\ell W \\
\bigl( \ul\fk_0 , \ul \fk_1 \bigr) 
&\quad\longmapsto\quad
\bigl(  \rT^\ell\sigma_\eps (\ul\fk_0 )  ,  \rD \rT^\ell\sigma_\eps (\ul\fk_0) \ul \fk_1 \bigr) 
\end{align*}
are pointwise continuous in $\Delta$ and uniformly continuous in the first component $\rT^\ell\sigma_\eps$ by the induction hypothesis. 
So to prove that $\sigma_\eps$ is adiabatic $\cC^{\ell+1}$-regular it remains to consider the second component $\rD \rT^\ell\sigma_\eps$ and prove its pointwise and uniform continuity. This is classically achieved by casting the linearizations of the fixed point problems with solutions $\sigma_\eps$ as higher order fixed point problems whose solutions are the derivatives under consideration. That leads to the same estimates as the following approach in which we generalize the explicit expression for the differential 
$\rD \sigma_\eps (\fk_0) \fk_1 = Q_\eps(\gamma_\eps) ( \fk_1 ,  0) $ when $\sigma_\eps (\fk_0)=(\fc_\eps,\gamma_\eps)$. 
To generalize this expression for the first differential of the solution maps we rewrite the derivative of the defining equation \eqref{eq:D defining} for the solution maps more concisely as the identity
\begin{equation} \label{eq:P sigma}
\widetilde P_\eps \circ  \rT \sigma_\eps =  \widetilde \sigma_\eps 
\end{equation}
for the maps with abbreviation $V:=\fK\times\Omega$
\begin{align*} 
 \widetilde \sigma_\eps &:& \rT\cV_\fK = \cV_\fK \times \fK &\to \cV_W \times  V , &
 \quad (\fk_0,\fk_1) &\mapsto \bigl(  \sigma_\eps(\fk_0), \fk_1 , 0 \bigr) , \\
 \rT \sigma_\eps &:& \rT\cV_\fK = \cV_\fK \times \fK &\to \rT \cV_W =  \cV_W \times W, &
 \quad (\fk_0,\fk_1) &\mapsto \bigl(  \sigma_\eps(\fk_0),  \rD\sigma_\eps(\fk_0) \fk_1 \bigr) , \\
 \widetilde P_\eps  &:& \rT\cV_W = \cV_W \times W  &\to   \cV_W \times   V, &
\quad
 \bigl(  w_0 = (\fc_0,\gamma_0) ,  w_1 \bigr) &\mapsto \bigl( w_0 , P_\eps(\gamma_0) \, w_1 \bigr) .
\end{align*}
Indeed, applying the left hand side of \eqref{eq:P sigma} to any $(\fk_0,\fk_1)\in\rT\cV_\fk$ yields as claimed
\begin{align*}
\widetilde P_\eps \bigl( \rT \sigma_\eps (\fk_0,\fk_1) \bigr)
& =  \widetilde P_\eps \bigl( \sigma_\eps (\fk_0) ,  \rD \sigma_\eps (\fk_0) \, \fk_1 \bigr)  
\;=\;   \bigl( \sigma_\eps (\fk_0) = (\fc_\eps,\gamma_\eps) ,  P_\eps(\gamma_\eps)  \, \rD \sigma_\eps (\fk_0) \, \fk_1 \bigr)  \\
& \overset{ \eqref{eq:D defining}}{=}  \bigl( \sigma_\eps (\fk_0)  ,  ( \fk_1 ,  0)  \bigr)  \;=\;  \widetilde \sigma_\eps (\fk_0,\fk_1) .
\end{align*}
Since $P_\eps(\gamma_0) ( \fc, \gamma) = (\pi_\fK(\gamma), \rD\cF_\eps(\gamma_0) \gamma - \fc )$ is given by \eqref{eq:P}, the identity \eqref{eq:P sigma} also encodes the fact that the tangent map of the solution map solves the tangent version of the stablized equation $\cF_\eps(\gamma)\in\fC$, 
$$
(\fc,\gamma,\fc',\xi)= \rT \sigma_\eps (\fk_0,\fk_1)
\qquad
\Longrightarrow
\qquad
\rT\cF_\eps(\gamma,\xi) = (\fc, \fc') \in \rT\fC , 
$$
where $\rT\cV_W=\rT(\fC\times\cV_\Gamma)= \fC \times \cV_\Gamma \times \fC\times \Gamma$. 
Identifying $\rT^\ell\rT\cV_W\simeq \rT^\ell\fC \times \rT^\ell\cV_\Gamma \times \rT^\ell\fC\times \rT^\ell\Gamma$
and taking tangent maps of the above implication yields the higher tangent stabilized equation
\begin{equation} \label{eq:T stabilized equation}
\bigl(\ul\fc  , \ul\gamma , \ul\fc' , \ul\xi \bigr) 
\;=\; \rT^{\ell+1} \sigma_{\eps_0}(\ul \fk_0, \ul \fk_1 )
\qquad
\Longrightarrow
\qquad
\rT^{\ell+1} \cF_\eps(\ul\gamma,\ul\xi) = (\ul\fc, \ul\fc') \in \rT^{\ell+1}\fC . 
\end{equation}
{\bf[Higher Regularizing Property]:} \label{page higher reg}
More precisely, the higher tangent solution maps $\rT^{\ell+1} \sigma_{\eps_0}:\rT^{\ell+1}\cV_\fK \to \rT^{\ell+1}\cV_{\ol W,\eps}$ inititially exist as maps to the completion by Lemma~\ref{lem:solution classical}, so that we obtain the above conclusion in terms of the extended Fredholm map 
$\rT^{\ell+1}  \overline\cF_\eps(\ul\gamma,\ul\xi) = (\ul\fc, \ul\fc') \in \rT^{\ell+1}\fC$. 
However, the higher regularizing property \eqref{def:regularizing T} in Definition~\ref{def:adiabatic C-l}
then implies $(\ul\gamma,\ul\xi)\in \rT^{\ell+1}\cV_\Gamma$  and hence 
$\rT^{\ell+1} \overline\cF_\eps(\ul\gamma,\ul\xi) =\rT^{\ell+1} \cF_\eps(\ul\gamma,\ul\xi) $. 
This also shows that the higher tangent solution maps 
$\rT^{\ell+1} \sigma_{\eps_0}:\rT^{\ell+1}\cV_\fK \to \rT^{\ell+1}\cV_W$
take values in the $\eps$-independent dense subspace. 

Furthermore, the map $\widetilde P_\eps$ is invertible by Lemma~\ref{lem:inverses}, with the inverse explicitly given by
\begin{align}\label{eq:tilde Q} 
\widetilde Q_\eps  &:&   \cV_W \times V &\to  \cV_W \times  W =   \rT\cV_W , &
\quad ( w_0 = (\fc_0,\gamma_0) , \fk , \omega ) &\mapsto \bigl( w_0 ,  Q_\eps(\gamma_0) (\fk,\omega) \bigr) .
\end{align}
With that notation we can rewrite \eqref{eq:P sigma} equivalently as
\begin{equation} \label{eq:D sigma}
\rT \sigma_\eps = \widetilde Q_\eps \circ \widetilde \sigma_\eps  . 
\end{equation}
Note here that the maps $\widetilde \sigma_\eps$ have the same regularity as $\sigma_\eps$, and the maps
$\widetilde P_\eps$ have the same regularity as $\rD\cF_\eps$ since 
$P_\eps(\gamma_0) (w=(\fc,\gamma)) 
=\bigl( \pi_\fK(\gamma), \rD\cF_\eps(\gamma_0) \gamma - \fc \bigr)$. 
Towards analyzing higher order derivatives, we apply the tangent map construction and chain rule \eqref{eq:chain} repeatedly to  \eqref{eq:P sigma} to obtain 
\begin{equation}\label{eq:Psigma}
\rT^{\ell-1}\widetilde P_\eps \circ  \rT^\ell \sigma_\eps =   \rT^{\ell-1}\widetilde \sigma_\eps . 
\end{equation}
Analogously, \eqref{eq:D sigma} yields
$\rT^\ell\sigma_\eps = \rT^{\ell-1} \widetilde Q_\eps \circ  \rT^{\ell-1} \widetilde \sigma_\eps $,  
and taking the differential of this identity yields 
\begin{equation}\label{eq:DTsigma}
\rD \rT^\ell\sigma_\eps(\ul \fk_0) = \rD \bigl(  \rT^{\ell-1}  \widetilde Q_\eps \circ  \rT^{\ell-1}  \widetilde \sigma_\eps \bigr) (\ul \fk_0)   
=  \rD  \rT^{\ell-1}  \widetilde Q_\eps ( \rT^{\ell-1} \widetilde \sigma_\eps (\ul \fk_0) )   \,\circ\,
\rD  \rT^{\ell-1}  \widetilde \sigma_\eps (\ul \fk_0) . 
\end{equation}
This is an inductively explicit formula for $\rD \rT^\ell\sigma_\eps$ since the $\ell$-th order derivatives of $\widetilde \sigma_\eps$ on the right hand side are determined by $\ell$-th order derivatives of $\sigma_\eps$. And this inductive formula will serve to prove adiabatic regularity of $\rD \rT^\ell\sigma_\eps$ since the $\ell$-th order derivative of $\sigma_\eps$ is adiabatic regular by induction hypothesis, and the derivatives of the inverse operators in $\rD  \rT^{\ell-1}  \widetilde Q_\eps$ can be estimated in terms of $(\ell+1)$-st order derivatives of the adiabatic $\cC^{\ell+1}$-regular adiabatic Fredholm family $\cF_\eps$. 
We will establish these estimates in two preliminary steps before proving the two adiabatic regularity properties (pointwise and uniform continuity) of $\rD \rT^\ell\sigma_\eps$. 

\smallskip
\noindent
{\bf Controlling derivatives of $\widetilde \sigma_\eps$ by derivatives of $\sigma_\eps$:} \label{page sigma}
The maps $\widetilde \sigma_\eps$ naturally split 
\begin{align*}
\widetilde \sigma_\eps \; \simeq \;  \sigma_\eps \times I_\fK \; : \quad 
\rT\cV_\fK = \cV_\fK \times \fK   & \quad \longrightarrow \quad  \cV_W \times V  &\quad =\quad & \cV_W \times  \fK \times \Omega \\
(\fk_0,\fk_1)  &\quad \longmapsto\quad  \bigl(  \sigma_\eps(\fk_0), I_\fK(\fk_1) \bigr)   &\quad =\quad & \bigl(  \sigma_\eps(\fk_0), \fk_1 , 0 \bigr)   
\end{align*}
into the original solution maps $\sigma_\eps$ and the $\eps$-independent inclusion $I_\fK: \fK \to V=\fK\times\Omega, \fk \mapsto (\fk,0)$. 
We will identify tangent maps analogously, by splitting higher tangent spaces according to 
$$
\rT \rT \cV_\fK
\;=\bigcup_{(\fv_0,\fk_0)\in\rT\cV_\fK} \rT_{(\fv_0,\fk_0)} \rT \cV_\fK 
\;= \bigcup_{(\fv_0,\fk_0)\in \cV\times \fK} \rT_{(\fv_0,\fk_0)} \bigl(\cV_\fK \times\fK\bigr) 
\;\simeq \bigcup_{\fv_0\in\cV_\fK} \rT_{\fv_0} \cV_\fK \times  \bigcup_{\fk_0\in\fK} \rT_{\fk_0} \fK 
$$
rather than the usual 
$$
\rT \rT \cV_\fK
\;=\bigcup_{(\fv_0,\fk_0)\in\rT\cV_\fK} \rT_{(\fv_0,\fk_0)} \rT \cV_\fK 
\;= \bigcup_{(\fv_0,\fk_0)\in \cV\times \fK} \rT_{(\fv_0,\fk_0)} \bigl(\cV_\fK \times\fK\bigr) 
\;= \bigcup_{(\fv_0,\fk_0)\in \cV\times \fK}  \rT_{\fv_0} \cV_\fK \times \rT_{\fk_0}\fK  . 
$$
Thus for all $k\geq 1$ we identify $\ul \fk = (\fk^0,\fk^1,\ldots,\fk^{N_k})\in\rT^k\cV_\fK$ for $N_k=2^k-1$ with 
\begin{equation}\label{eq:k ev od}
\ul\fk\; \simeq\;  
(\ul \fk^{\scriptscriptstyle ev},\ul \fk^{\scriptscriptstyle od}) \, := \; \bigl( (\fk^0,\fk^2,\ldots,\fk^{N_k-1}) , (\fk^1,\fk^3,\ldots,\fk^{N_k})  \bigr) \;\in\; \rT^{k-1}\cV_\fK \times \rT^{k-1}\fK
\end{equation}
when writing
\begin{align*}
\rT^{\ell}\widetilde \sigma_\eps  \;\simeq\;  \rT^{\ell}\sigma_\eps \times   \rT^\ell I_\fK \; : \quad
\rT^\ell\cV_\fK \times \rT^\ell\fK   & \quad \longrightarrow \quad   \rT^\ell\cV_W \times \rT^\ell V  \\
(\ul\fk^{\scriptscriptstyle ev},\ul\fk^{\scriptscriptstyle od})  \;&\quad \longmapsto\quad \bigl(  \rT^{\ell}\sigma_\eps(\ul\fk^{\scriptscriptstyle ev}),   \rT^\ell I_\fK ( \ul\fk^{\scriptscriptstyle od} ) \bigr)     .   
\end{align*}
Here the inclusion map $I_\fK$ is linear, hence $\rT^\ell I_\fK$ is given by $I_\fK$ in each component, which is naturally bounded $\|I_\fK(\fk)\|^V_\eps=\|(\fk,0)\|^{\fK\times\Omega}_\eps = \|\fk\|^\fK$. 
Now we can transfer the uniform bounds and continuity estimates from $\rT^\ell\sigma_\eps$ to $\rT^\ell\widetilde \sigma_\eps$ as follows: 
First, the inductive assumption \eqref{eq:uniform bound} implies the uniform bound for any $\eps\in\Delta_\sigma$ and $(\ul\fk_0,\ul\fk_1)=\ul\fk\simeq(\ul \fk^{\scriptscriptstyle ev},\ul \fk^{\scriptscriptstyle od}) \in \rT^\ell\cV_\fK \times \rT^\ell\fK\simeq \rT^{\ell+1}\cV_\fK$
\begin{align}
\bigl\| \rT^\ell  \widetilde \sigma_\eps (\ul\fk_0,\ul\fk_1) \bigr\|^{\rT^{\ell}(W \times V)}_\eps 
&= \bigl\|  \rT^{\ell-1}  \widetilde \sigma_\eps (\ul\fk_0) \bigr\|^{\rT^{\ell-1}(W \times V)}_\eps   
+ \bigl\| \rD  \rT^{\ell-1}  \widetilde \sigma_\eps (\ul\fk_0) \ul\fk_1  \bigr\|^{\rT^{\ell-1}(W \times V)}_\eps    
\nonumber\\
&=  \bigl\| \rT^\ell  \widetilde \sigma_\eps (\ul \fk) \bigr\|^{\rT^{\ell}(W \times V)}_\eps 
\;=\;
 \bigl\| \rT^\ell  \widetilde \sigma_\eps (\ul \fk^{\scriptscriptstyle ev},\ul \fk^{\scriptscriptstyle od}) \bigr\|^{\rT^{\ell}(W \times V)}_\eps 
\nonumber\\
&= 
 \bigl\| \rT^\ell  \sigma_\eps (\ul \fk^{\scriptscriptstyle ev})   \bigr\|^{\rT^\ell W}_\eps 
+
\bigl\| \rT^\ell I_\fK (\ul \fk^{\scriptscriptstyle od} ) \bigr\|^{\rT^\ell V }_\eps
\label{eq:DTsigma bound} \\
&\leq
\delta_\sigma  
+ c^\ell_\sigma(\| \ul\fk^{\scriptscriptstyle ev} \|^{\rT_\bullet^\ell\fK}) \, b^\ell_\sigma(\|\fk^{\scriptscriptstyle ev}\|^{\rT_\bullet^\ell\fK}) 
+ \| \ul \fk^{\scriptscriptstyle od} \|^{\rT^\ell \fK }   
\nonumber  \\
&\leq
\delta_\sigma  
+ \bigl( c^\ell_\sigma(\| \ul\fk \|^{\rT_\bullet^{\ell+1}\fK} ) 
+ \| \ul \fk \|^{\rT_\bullet^{\ell+1} \fK }   \bigr)
\, b^\ell_\sigma(\|\fk\|^{\rT_\bullet^{\ell+1}\fK}) 
\nonumber  \\
&\leq
\delta_\sigma  
+ \tilde c^\ell_{\sigma}(\| \ul\fk \|^{\rT_\bullet^{\ell+1}\fK} ) \, b^\ell_\sigma(\|\fk\|^{\rT_\bullet^{\ell+1}\fK})  
\nonumber 
\end{align}
with $\tilde c^\ell_{\sigma}(x):=c^\ell_\sigma(x) + x$ a new monotone, continuous function with $\tilde c^\ell_{\sigma}(0)=0$. 

Second, for fixed $(\ul\fk_0,\ul\fk_1)=\ul\fk\simeq(\ul \fk^{\scriptscriptstyle ev},\ul \fk^{\scriptscriptstyle od}) 
\in \rT^\ell\cV_\fK \times \rT^\ell\fK$ we write
$\rD  \rT^{\ell-1}  \widetilde \sigma_\eps (\ul \fk^{\scriptscriptstyle ev},\ul \fk^{\scriptscriptstyle od}):= 
 \rD  \rT^{\ell-1}  \widetilde \sigma_\eps (\ul \fk_0) \ul \fk_1$. 
Then for $\Delta_\sigma \ni \eps\to\eps_0$ the pointwise continuity of $\rT^\ell\sigma_\eps$ given by the inductive assumption in \eqref{eq:T-ell-sigma pointwise} implies the pointwise continuity 
\begin{align}
\bigl\|    \rT^{\ell}  \widetilde \sigma_\eps (\ul \fk) - 
 \rT^{\ell}  \widetilde \sigma_{\eps_0} (\ul \fk)
\bigr\|^{\rT^{\ell}(W \times V)}_\eps   
&=
 \bigl\| \bigl( \rT^{\ell}  \sigma_\eps (\ul \fk^{\scriptscriptstyle ev}) , \rT^\ell I_\fK (\ul \fk^{\scriptscriptstyle od} )  \bigr)  - 
 \bigl( \rT^{\ell}  \sigma_{\eps_0} (\ul \fk^{\scriptscriptstyle ev})   , \rT^\ell I_\fK (\ul \fk^{\scriptscriptstyle od} )\bigr)
\bigr\|^{\rT^{\ell}(W\times V)}_\eps 
\nonumber\\
&=
 \bigl\|  \rT^{\ell}  \sigma_\eps (\ul \fk^{\scriptscriptstyle ev})   - 
\rT^{\ell}  \sigma_{\eps_0} (\ul \fk^{\scriptscriptstyle ev})  
\bigr\|^{\rT^{\ell}W}_\eps 
\;\underset{\eps\to\eps_0}{\longrightarrow}\; 0.  
\label{eq:DTsigma pointwise}
\end{align}
Third, for any $\eps\in\Delta_\sigma$ and $\ul\fk\simeq(\ul \fk^{\scriptscriptstyle ev},\ul \fk^{\scriptscriptstyle od}), \ul\fl\simeq(\ul \fl^{\scriptscriptstyle ev},\ul \fl^{\scriptscriptstyle od}) \in \rT^\ell\cV_\fK \times \rT^\ell\fK$ the inductive assumption in \eqref{eq:T-ell-sigma uniform} implies the uniform continuity  
\begin{align}
 \bigl\|  \rT^{\ell}  \widetilde \sigma_\eps (\ul \fl )- 
 \rT^{\ell}  \widetilde \sigma_\eps (\ul \fk)
\bigr\|^{\rT^{\ell}(W \times V)}_\eps   
& =
 \bigl\| \bigl( \rT^{\ell}  \sigma_\eps (\ul \fl^{\scriptscriptstyle ev}) , \rT^\ell I_\fK ( \ul \fl^{\scriptscriptstyle od}) \bigr)  - 
 \bigl( \rD  \rT^{\ell}  \sigma_\eps (\ul \fk^{\scriptscriptstyle ev})   ,  \rT^\ell I_\fK ( \ul \fk^{\scriptscriptstyle od}) \bigr)
\bigr\|^{\rT^{\ell}(W\times V)}_\eps 
\nonumber\\
&=
 \bigl\|  \rT^{\ell}  \sigma_\eps (\ul \fl^{\scriptscriptstyle ev})   - 
 \rT^{\ell}  \sigma_\eps (\ul \fk^{\scriptscriptstyle ev})   \bigr\|^{\rT^{\ell} W }_\eps 
+
 \bigl\| \rT^\ell I_\fK (\ul \fl^{\scriptscriptstyle od} ) -  \rT^\ell I_\fK ( \ul \fk^{\scriptscriptstyle od} ) \bigr\|^{\rT^{\ell} V}_\eps  
 \nonumber\\
&\leq 
c^\ell_\sigma(  \| \ul \fl^{\scriptscriptstyle ev} - \ul \fk^{\scriptscriptstyle ev} \|^{\rT^\ell \fK }   )  \, b^\ell_\sigma(\max\{ \|\fl^{\scriptscriptstyle ev} \|^{\rT_\bullet^\ell\fK} ,  \|\fk^{\scriptscriptstyle ev} \|^{\rT_\bullet^\ell\fK} \}) 
+
\| \ul \fl^{\scriptscriptstyle od} - \ul \fk^{\scriptscriptstyle od} \|^{\rT^\ell \fK } 
\label{eq:DTsigma continuity} \\
&\leq 
\bigl( c^\ell_\sigma(  \| \ul \fl - \ul \fk \|^{\rT^{\ell+1} \fK }   )  
+ \| \ul \fl - \ul \fk \|^{\rT_\bullet^{\ell+1} \fK }  \bigr)\, b^\ell_\sigma(\max\{ \|\fl \|^{\rT_\bullet^{\ell+1}\fK} ,  \|\fk \|^{\rT_\bullet^{\ell+1}\fK} \}) 
\nonumber\\
&\leq 
\tilde c^\ell_{\sigma}(  \| \ul \fl - \ul \fk \|^{\rT^{\ell+1} \fK }   )  \, b^\ell_\sigma(\max\{ \|\fl \|^{\rT_\bullet^{\ell+1}\fK} ,  \|\fk \|^{\rT_\bullet^{\ell+1}\fK} \}) . 
\nonumber
\end{align}
These three implications of the inductive assumption will be crucial for the adiabatic $\cC^{\ell+1}$ regularity. Before proving it, we will establish one more set of estimates in preparation. 

\smallskip
\noindent
{\bf Controlling derivatives of $\widetilde Q_\eps$ by derivatives of $\cF_\eps$:} \label{page Q}
The final preparation step is to establish uniform continuity and bounds for the derivatives of the map $\widetilde Q_\eps$ given by \eqref{eq:tilde Q}. For that purpose we will go back to its expression in components\footnote{
Alternatively, the chain rule applied to the identity $\widetilde Q_\eps\circ \widetilde P_\eps={\rm Id}_{W\times W}$ allows one to express $\rD\rT^{\ell-1}\widetilde Q_\eps$ as an inverse of $\rD\rT^{\ell-1}\widetilde P_\eps$. However, while the latter is uniformly continuous and invertible, we couldn't find an explicit formula for this inverse that was suitable for transferring the uniform continuity estimates to $\rD\rT^{\ell-1}\widetilde Q_\eps$.}
\begin{align*}
\widetilde Q_\eps  &:&   \cV_W \times V &\to  \cV_W \times  W  , &
\quad ( w = (\fc ,\gamma ) , v ) &\mapsto \bigl( w ,  Q_\eps(\gamma) v \bigr)
\end{align*}
and use the product rule in the identity $Q_\eps(\gamma) \, P_\eps(\gamma)={\rm Id}_W$ from Lemma~\ref{lem:inverses}. 
More precisely, our goal is to prove uniform bounds and continuity of $\rD \rT^{\ell-1} \widetilde Q_\eps$.  
So we begin by analyzing the higher tangent maps $\rT^{\ell-1}\widetilde Q_\eps$ defined in Definition~\ref{def:tangent map notation}, 
\begin{align*}
\rT^{\ell-1}\widetilde Q_\eps  &:&  \rT^{\ell-1}( \cV_W \times V) \simeq \rT^{\ell-1}\fC \times \rT^{\ell-1}\cV_\Gamma \times  \rT^{\ell-1} V &\;\to\;  \rT^{\ell-1}\fC \times \rT^{\ell-1}\cV_\Gamma \times \rT^{\ell-1} W  \simeq \rT^{\ell-1}( \cV_W \times W), \\
&&
\quad \bigl( \ul\fc , \ul\gamma , \ul v ) &\;\mapsto\; \bigl( \ul\fc , \ul\gamma ,  \ul w:= {\Pi}_{\rT^{\ell-1} W} \rT^{\ell-1}\widetilde Q_\eps (\ul\fc , \ul\gamma , \ul v) 
\bigr) , 
\end{align*}
where the $N_{\ell-1}=2^{\ell-1}-1$ components of $\ul w=(w_0,\ldots,w_{N_\ell})$ are functions of $ \ul\gamma=(\gamma_0,\ldots,\gamma_{N_{\ell-1}})$ and $\ul v=(v_0,\ldots,v_{N_{\ell-1}})$ starting with 
\begin{align*}
w_0 &= Q_\eps(\gamma_0) v_0  \\
w_1 &=  \tfrac\rd{\rd t}|_{t=0} \bigl(  Q_\eps(\gamma_0+t \gamma_1) (v_0 +t v_1)\bigr) 
\;=\; Q_\eps(\gamma_0) v_1 + [\rD Q_\eps(\gamma_0) \gamma_1 ] v_0 .
\end{align*}
Next, $(w_2,w_3)$ arise from the differential $\rD(w_0,w_1)$ at base point $(\gamma_0,v_0)$ applied to $(\gamma_2,\gamma_3, v_2,v_3)$, 
\begin{align*}
w_2 &=  \tfrac\rd{\rd t}|_{t=0} w_0\bigl( (\gamma_0,v_0) +t (\gamma_2, v_2)\bigr)  
\;=\; Q_\eps(\gamma_0) v_2 + [\rD Q_\eps(\gamma_0) \gamma_2 ] v_0 \\
w_3 &= \tfrac\rd{\rd t}|_{t=0} w_1\bigl( (\gamma_0,\gamma_1,v_0,v_1) +t (\gamma_2,\gamma_3, v_2,v_3)\bigr)   \\
&=
Q_\eps(\gamma_0) v_3 + [\rD Q_\eps(\gamma_0) \gamma_2 ] v_1 
+   [\rD Q_\eps(\gamma_0) \gamma_1 ] v_2   +   [\rD Q_\eps(\gamma_0) \gamma_3 ] v_0  + [\rD^2 Q_\eps(\gamma_0) (\gamma_1,\gamma_2) ] v_0 . 
\end{align*}
Then $(w_4,\ldots,w_7)$ arise from the differential $\rD (w_0,\ldots,w_3)$ at base point $(\gamma_0,\ldots,\gamma_3,v_0,\ldots,v_3)$ applied to $(\gamma_4,\ldots,\gamma_7, v_2,\ldots,v_7)$, hence 
\begin{align*}
w_4 &=  \tfrac\rd{\rd t}|_{t=0} w_0\bigl( (\gamma_0,v_0) +t (\gamma_4, v_4) \bigr) 
\;=\; Q_\eps(\gamma_0) v_4 + [\rD Q_\eps(\gamma_0) \gamma_4 ] v_0 \\
\ldots \\
w_7 &=  \tfrac\rd{\rd t}|_{t=0} w_3\bigl( (\gamma_0,v_0,\ldots,\gamma_3,v_3) + t  (\gamma_4,v_4,\ldots,\gamma_7,v_7) \bigr)  \\
&= Q_\eps(\gamma_0) v_7  + [\rD Q_\eps(\gamma_0) \gamma_4] v_3 
+ [\rD Q_\eps(\gamma_0) \gamma_2 ] v_5  + [\rD Q_\eps(\gamma_0) \gamma_6 ] v_1 + [\rD^2 Q_\eps(\gamma_0) (\gamma_2,\gamma_4) ] v_1  \\
&\quad
+   [\rD Q_\eps(\gamma_0) \gamma_1 ] v_6   +   [\rD Q_\eps(\gamma_0) \gamma_5 ] v_2   +   [\rD^2 Q_\eps(\gamma_0) (\gamma_1,\gamma_4) ] v_2   \\
&\quad
+   [\rD Q_\eps(\gamma_0) \gamma_3 ] v_4 +   [\rD Q_\eps(\gamma_0) \gamma_7 ] v_0 +   [\rD^2 Q_\eps(\gamma_0) (\gamma_3,\gamma_4) ] v_0 \\
&\quad
+ [\rD^2 Q_\eps(\gamma_0) (\gamma_1,\gamma_2) ] v_4
+ [\rD^2 Q_\eps(\gamma_0) (\gamma_5,\gamma_2) ] v_0
+ [\rD^2 Q_\eps(\gamma_0) (\gamma_1,\gamma_6) ] v_0
+ [\rD^3 Q_\eps(\gamma_0) (\gamma_1,\gamma_2,\gamma_4) ] v_0 
\end{align*}
Continuing inductively, each component of ${\Pi}_{\rT^{\ell-1} W} \rT^{\ell-1}\widetilde Q_\eps (\ul\fc , \ul\gamma , \ul v)$ is a sum of expressions $[\rD^kQ_\eps(\gamma_0)(\gamma_*,\ldots,\gamma_*)] v_*$ with $k\leq \ell-1$, whose argument $\gamma_*,\ldots,\gamma_*$ is a permutation of a subset of $(\gamma_1,\ldots,\gamma_{N_{\ell-1}})$. 
Furthermore, differentials of these expressions at a base point $(\ul\gamma = (\gamma_0,\ldots) , \ul v)$ are sums of operators 
$O^k_{i,*} (\ul\gamma, \ul v) : \rT^{\ell-1} \Gamma \times  \rT^{\ell-1} V \to \rT^{\ell-1} W$
of the forms
\begin{align} 
O^k_{1,*}(\ul\gamma, \ul v) \,:\; 
\bigl( \ul\xi=(\xi_0,\ldots), \ul y=(y_0,\ldots) \bigr) &\mapsto 
[\rD^{k+1}Q_\eps(\gamma_0)(\gamma_*,\ldots,\gamma_*,\xi_0)] v_*   
\nonumber\\
O^k_{2,*}(\ul\gamma, \ul v) \,:\; 
\bigl( \ul\xi=(\xi_0,\ldots), \ul y=(y_0,\ldots) \bigr) &\mapsto 
[\rD^kQ_\eps(\gamma_0)(\gamma_*,\ldots, \xi_*, \ldots, \gamma_*)] v_*  
\label{eq:parts of TQ} \\
O^k_{3,*}(\ul\gamma, \ul v) \,:\; 
\bigl( \ul\xi=(\xi_0,\ldots), \ul y=(y_0,\ldots) \bigr) &\mapsto 
[\rD^kQ_\eps(\gamma_0)(\gamma_*,\ldots,\gamma_*)] y_* . 
\nonumber
\end{align}
Consequently, the differential of $\rT^{\ell-1}\widetilde Q_\eps$ at a base point $(\ul\fc , \ul\gamma, \ul v)\in \rT^{\ell-1}\fC \times \rT^{\ell-1}\cV_\Gamma \times  \rT^{\ell-1} V$ is 
\begin{align*}
\rD\rT^{\ell-1}\widetilde Q_\eps (\ul\fc , \ul\gamma, \ul v)  &:&  \rT^{\ell-1}\fC \times \rT^{\ell-1}\cV_\Gamma \times  \rT^{\ell-1} V &\;\to\;  \rT^{\ell-1}\fC \times \rT^{\ell-1}\cV_\Gamma \times \rT^{\ell-1} W , \\
&&
\quad ( \ul\fc' , \ul\xi , \ul y) &\;\mapsto\; \bigl( \ul\fc' , \ul\xi,  \rD\ul w (\ul\gamma,\ul v)  (\ul\xi,\ul y) , 
\bigr)
\end{align*}
where each component $\rD w_i (\ul\gamma,\ul v)$ is a sum of operators of the three types above -- all involving higher differentials\footnote{
See Remark~\ref{rmk:multilinear} for the precise definition and notation used for higher differentials.}
 of the inverses $Q_\eps:\cV_{\Gamma,Q} \to\cL(V,W)$  on $\cV_{\Gamma,Q}:=\{\gamma\in\cV_\Gamma \,|\,   \|\gamma_0\|^\Gamma_\eps<\delta_Q \}$ from Lemma~\ref{lem:inverses}. 
To compute these, we use the identity $Q_\eps(\gamma) \, P_\eps(\gamma)={\rm Id}_W$ and the product rule:
\begin{align*}
0 = \rD \bigl[ Q_\eps \, P_\eps \bigr] (\gamma_0) \gamma_1 
&=
\rD Q_\eps (\gamma_0)(\gamma_1) \, P_\eps(\gamma_0) 
+ Q_\eps (\gamma_0) \, \rD P_\eps(\gamma_0)(\gamma_1) \\
\Rightarrow\quad 
- \rD Q_\eps (\gamma_0)(\gamma_1) 
&= 
Q_\eps (\gamma_0) \, \rD P_\eps(\gamma_0)(\gamma_1) \, Q_\eps(\gamma_0)  
\qquad\forall (\gamma_0,\gamma_1)\in \cV_{\Gamma,Q}\times\Gamma  . 
\end{align*}
Taking another differential of this identity yields for $(\gamma_0,\gamma_1,\gamma_2)\in \cV_{\Gamma,Q}\times\Gamma\times\Gamma$ 
\begin{align*}
- \rD^2 Q_\eps (\gamma_0)(\gamma_1,\gamma_2) 
&= 
\rD Q_\eps (\gamma_0)(\gamma_2) \, \rD P_\eps(\gamma_0)(\gamma_1) \, Q_\eps(\gamma_0)  
+ Q_\eps (\gamma_0) \, \rD^2 P_\eps(\gamma_0)(\gamma_1,\gamma_2) \, Q_\eps(\gamma_0) \\
&\qquad
+ Q_\eps (\gamma_0) \, \rD P_\eps(\gamma_0)(\gamma_1) \, \rD Q_\eps(\gamma_0) (\gamma_2) \\
&= 
- Q_\eps (\gamma_0) \, \rD P_\eps(\gamma_0)(\gamma_2) \, Q_\eps(\gamma_0) \, \rD P_\eps(\gamma_0)(\gamma_1) \, Q_\eps(\gamma_0)   \\
&\qquad
+ Q_\eps (\gamma_0) \, \rD^2 P_\eps(\gamma_0)(\gamma_1,\gamma_2) \, Q_\eps(\gamma_0) \\
&\qquad
- Q_\eps (\gamma_0) \, \rD P_\eps(\gamma_0)(\gamma_1) \, Q_\eps (\gamma_0) \, \rD P_\eps(\gamma_0)(\gamma_2) \, Q_\eps(\gamma_0) . 
\end{align*}
Continuing these computations by induction identifies $\rD^k Q_\eps (\gamma_0)(\gamma_1,\ldots,\gamma_k)\in\cL(V,W)$ with a sum of products of operators
\begin{equation} \label{eq:DQ parts}
\pm Q_\eps (\gamma_0) \, \rD^{k_1} P_\eps(\gamma_0) (\gamma_* {\scriptstyle\ldots}) \, Q_\eps (\gamma_0) \, \ldots  \, Q_\eps (\gamma_0) \, \rD^{k_n} P_\eps(\gamma_0) (\gamma_* {\scriptstyle\ldots} ) \, Q_\eps(\gamma_0)
\end{equation}
with $k_1+\ldots+k_n = k$ and the arguments $\gamma_* {\scriptstyle\ldots} \, ,\ldots , \gamma_* {\scriptstyle\ldots}$ given as a permutation of $\gamma_1,\ldots,\gamma_k$. 
Now recall from Lemma~\ref{lem:inverses} that $P_\eps(\gamma_0) \in \cL(W,V)$ is given by $w=(\fc,\xi)\mapsto \bigl( \pi_\fK(\gamma) , \rD\cF_\eps(\gamma_0)\xi- \fc  \bigr)$, so that the differential
$\rD P_\eps (\gamma_0) \gamma_1 \in \cL(W,V)$ is given by 
\begin{align*}
(\fc,\xi) \mapsto  \tfrac{\rd}{\rd t}|_{t=0} P_\eps(\gamma_0+t\gamma_1) (\fc,\xi)
\;=\;
 \tfrac{\rd}{\rd t}|_{t=0}  \bigl( \pi_\fK(\xi) , \rD\cF_\eps(\gamma_0+t\gamma_1)\xi- \fc  \bigr)
 \;=\;
 \bigl( 0 , \rD^2\cF_\eps(\gamma_0) (\xi,\gamma_1)  \bigr) . 
\end{align*}
By induction, all differentials $\rD^k P_\eps (\gamma_0) (\gamma_1,\ldots,\gamma_k) \in \cL(W,V)$ for $k\geq 1$ are given by 
\begin{align*}
W = \fC \times \Gamma \;\ni \; 
(\fc,\xi) \mapsto  \bigl( 0 , \rD^{k+1}\cF_\eps(\gamma_0) (\xi,\gamma_1,\ldots,\gamma_k)  \bigr)  \;\in\; \fK\times\Omega = V,
\end{align*}
so that the [Uniform Bound on $\rD^k\cF_\eps(0)$ for $1\leq k\leq \ell+1\,$] in Definition~\ref{def:adiabatic C-l} 
together with \eqref{rmk:uniform DF} in Remark~\ref{rmk:uniform DF}
implies uniform bounds 
%
%
for all $\eps\in\Delta$, $\gamma_0\in\cV_\Gamma$ with $\|\gamma_0\|^\Gamma_\eps, \leq \delta_Q$, and $\ul\gamma=(\gamma_1,\ldots,\gamma_k) \in\Gamma^k$ with $1\leq k\leq \ell$, 
\begin{align} 
\bigl\| \rD^k P_\eps(\gamma_0) (\gamma_1,\ldots,\gamma_k) \bigr\|^{\cL(\overline W_\eps, \overline V_\eps)} 
&=\sup_{\|w\| \leq 1} \bigl\| [ \rD^k P_\eps(\gamma_0)  (\gamma_1,\ldots,\gamma_k) ] w \bigr\|^V_\eps   
\nonumber \\
&=\sup_{\|(\fc,\xi)\| \leq 1} \bigl\| \bigl( 0 , \rD^{k+1}\cF_\eps(\gamma_0) (\xi,\gamma_1,\ldots,\gamma_k) \bigr) \bigr\|^{\fK\times\overline\Omega_\eps}   
\nonumber \\
&= \sup_{\|\xi\| \leq 1} \bigl\| \rD^{k+1}\cF_\eps(\gamma_0) (\xi,\gamma_1,\ldots,\gamma_k)  \bigr\|^\Omega_\eps \label{eq:DP bound}\\
&\leq \sup_{\|\xi\| \leq 1} \bigl\| \rD^{k+1}\cF_\eps(\gamma_0) \bigr\|^{\cL^k(\overline \Gamma_\eps, \overline \Omega_\eps)} \|\gamma_1\|^\Gamma_\eps \ldots  \|\gamma_k\|^\Gamma_\eps 
\nonumber\\
&\leq \sup_{\|\xi\| \leq 1} \ti C^{k+1}_\cF \|\xi\|^\Gamma_\eps \|\gamma_1\|^\Gamma_\eps \ldots  \|\gamma_k\|^\Gamma_\eps 
\nonumber\\
&\leq
\ti C^{k+1}_\cF  \|\gamma_1\|^\Gamma_\eps \ldots  \|\gamma_k\|^\Gamma_\eps
\;\leq\; \ti C^{k+1}_\cF \bigl( \| \ul\gamma \|^{\Gamma^k}_{\eps,\infty} \bigr)^k  
\nonumber
\end{align}
with 
\begin{equation} \label{eq:Gamma infty}
\ti C^{k}_\cF:=C^k_\cF + c^k_\cF(\delta_Q) 
\qquad\text{and}\qquad
\| \ul\gamma \|^{\Gamma^k}_{\eps,\infty} :=  \max_{1\leq i \leq k} \| \gamma_i \|^{\Gamma}_\eps .
\end{equation} 
Similarly, the 
[Uniform Bound and Uniform Continuity of $\rD^k\cF_\eps$ for $1\leq k\leq \ell+1\,$] in Definition~\ref{def:adiabatic C-l} and Remark~\ref{rmk:uniform DF} imply uniform continuity for all $1\leq k\leq \ell$ and $\eps\in\Delta$, 
$\gamma^\fl_0,\gamma^\fk_0\in\cV_\Gamma$
with $\|\gamma^\fl_0\|^\Gamma_\eps, \|\gamma^\fk_0\|^\Gamma_\eps \leq \delta_Q$, and $\ul\gamma^\fl=(\gamma^\fl_1,\ldots,\gamma^\fl_k), \ul\gamma^\fk=(\gamma^\fk_1,\ldots,\gamma^\fk_k )\in\Gamma^k$ 
\begin{align}
& \bigl\|  \rD^k P_\eps(\gamma^\fl_0) (\gamma^\fl_1,\ldots,\gamma^\fl_k)  -  \rD^k P_\eps(\gamma^\fk_0) (\gamma^\fk_1,\ldots,\gamma^\fk_k)  \bigr\|^{\cL(\overline W_\eps, \overline V_\eps)}  
\nonumber\\
&=
\bigl\|\rD^{k+1}\cF_\eps(\gamma^\fl_0) (\,\cdot\,,\gamma^\fl_1,\ldots,\gamma^\fl_k)  -\rD^{k+1}\cF_\eps(\gamma^\fk_0) (\,\cdot\,,\gamma^\fk_1,\ldots,\gamma^\fk_k)  \bigr\|^{\cL(\overline\Gamma_\eps,\overline\Omega_\eps)}  
\nonumber\\
&=
\sup_{\|\xi\|\leq 1}
\bigl\|\rD^{k+1}\cF_\eps(\gamma^\fl_0) (\xi,\gamma^\fl_1,\ldots,\gamma^\fl_k)  -\rD^{k+1}\cF_\eps(\gamma^\fk_0) (\xi,\gamma^\fk_1,\ldots,\gamma^\fk_k)  \bigr\|^\Omega_\eps  
\nonumber
\end{align}
\begin{align}
&\leq 
\sup_{\|\xi\|\leq 1} \bigl( 
\bigl\|\rD^{k+1}\cF_\eps(\gamma^\fl_0) (\xi,\gamma^\fl_1,\ldots,\gamma^\fl_k)  -\rD^{k+1}\cF_\eps(\gamma^\fk_0) (\xi,\gamma^\fl_1,\ldots,\gamma^\fl_k)  \bigr\|^\Omega_\eps
\nonumber\\
&\qquad\qquad + 
\bigl\|\rD^{k+1}\cF_\eps(\gamma^\fk_0) (\xi,\gamma^\fl_1,\ldots,\gamma^\fl_k)  -\rD^{k+1}\cF_\eps(\gamma^\fk_0) (\xi,\gamma^\fk_1,\gamma^\fl_2,\ldots,\gamma^\fl_k)  \bigr\|^\Omega_\eps
+ \ldots
\nonumber\\
&\qquad\qquad + 
\bigl\|\rD^{k+1}\cF_\eps(\gamma^\fk_0) (\xi,\gamma^\fk_1,\ldots,\gamma^\fk_{k-1},\gamma^\fl_k)  -\rD^{k+1}\cF_\eps(\gamma^\fk_0) (\xi,\gamma^\fk_1,\ldots,\gamma^\fk_k)  \bigr\|^\Omega_\eps
\bigr) 
\nonumber\\
&= 
\sup_{\|\xi\|\leq 1} \bigl( 
\bigl\| \bigl( \rD^{k+1}\cF_\eps(\gamma^\fl_0) -\rD^{k+1}\cF_\eps(\gamma^\fk_0) \bigr) (\xi,\gamma^\fl_1,\ldots,\gamma^\fl_k)  \bigr\|^\Omega_\eps
\nonumber\\
&\qquad\qquad + 
\bigl\|\rD^{k+1}\cF_\eps(\gamma^\fk_0) (\xi,\gamma^\fl_1 - \gamma^\fk_1,\gamma^\fl_2,\ldots,\gamma^\fl_k)  \bigr\|^\Omega_\eps +\ldots 
\nonumber\\
&\qquad\qquad + 
\bigl\|\rD^{k+1}\cF_\eps(\gamma^\fk_0) (\xi,\gamma^\fk_1,\ldots,\gamma^\fk_{k-1},\gamma^\fl_k - \gamma^\fk_k)  \bigr\|^\Omega_\eps
\bigr) 
\nonumber \\
&\leq 
\sup_{\|\xi\|\leq 1} \bigl( 
\bigl\|\rD^{k+1}\cF_\eps(\gamma^\fl_0) -\rD^{k+1}\cF_\eps(\gamma^\fk_0) \bigr\|^{\cL^{k+1}(\overline\Gamma_\eps^{k+1}, \overline\Omega_\eps)}   \|  \xi \|^\Gamma_\eps   \| \gamma^\fl_1 \|^\Gamma_\eps \ldots  \| \gamma^\fl_k \|^\Gamma_\eps
\label{eq:DP continuity}\\
&\qquad\qquad + 
\bigl\|\rD^{k+1}\cF_\eps(\gamma^\fk_0) \bigr\|^{\cL^{k+1}(\overline\Gamma_\eps^{k+1}, \overline\Omega_\eps)} 
\|  \xi \|^\Gamma_\eps   \| \gamma^\fl_1 - \gamma^\fk_1 \|^\Gamma_\eps   \| \gamma^\fl_2 \|^\Gamma_\eps \ldots  \| \gamma^\fl_k \|^\Gamma_\eps 
+\ldots
\nonumber\\
&\qquad\qquad + 
\bigl\|\rD^{k+1}\cF_\eps(\gamma^\fk_0) \bigr\|^{\cL^{k+1}(\overline\Gamma_\eps^{k+1}, \overline\Omega_\eps)} 
\|  \xi \|^\Gamma_\eps   \| \gamma^\fk_1 \|^\Gamma_\eps \ldots  \| \gamma^\fk_{k-1} \|^\Gamma_\eps \| \gamma^\fl_k - \gamma^\fk_k \|^\Gamma_\eps
\bigr) 
\nonumber \\
&\leq 
\sup_{\|\xi\|\leq 1} \bigl( 
c^{k+1}_\cF( \|\gamma^\fl_0 - \gamma^\fk_0 \|^\Gamma_\eps )   \|  \xi \|^\Gamma_\eps   \| \gamma^\fl_1 \|^\Gamma_\eps  \ldots  \| \gamma^\fl_k \|^\Gamma_\eps
\nonumber\\
&\qquad\qquad + 
C^{k+1,\delta_Q}_\cF  \|  \xi \|^\Gamma_\eps   \| \gamma^\fl_1 - \gamma^\fk_1 \|^\Gamma_\eps  \| \gamma^\fl_2 \|^\Gamma_\eps \ldots  \| \gamma^\fl_k \|^\Gamma_\eps 
+\ldots 
\nonumber\\
&\qquad\qquad  + 
C^{k+1,\delta_Q}_\cF  \|  \xi \|^\Gamma_\eps  \| \gamma^\fk_1 \|^\Gamma_\eps \ldots  \| \gamma^\fk_{k-1} \|^\Gamma_\eps   \| \gamma^\fl_k - \gamma^\fk_k \|^\Gamma_\eps 
\bigr) 
\nonumber\\
&= 
c^{k+1}_\cF( \|\gamma^\fl_0 - \gamma^\fk_0 \|^\Gamma_\eps )   \| \gamma^\fl_1 \|^\Gamma_\eps  \ldots  \| \gamma^\fl_k \|^\Gamma_\eps
+ 
C^{k+1,\delta_Q}_\cF  \| \gamma^\fl_1 - \gamma^\fk_1 \|^\Gamma_\eps  \| \gamma^\fl_2 \|^\Gamma_\eps \ldots  \| \gamma^\fl_k \|^\Gamma_\eps +\ldots  
\nonumber\\
&\qquad\qquad\qquad\qquad\qquad\qquad\qquad\qquad\qquad
 + 
C^{k+1,\delta_Q}_\cF  \| \gamma^\fk_1 \|^\Gamma_\eps \ldots  \| \gamma^\fk_{k-1} \|^\Gamma_\eps   \| \gamma^\fl_k - \gamma^\fk_k \|^\Gamma_\eps 
\nonumber\\
&\leq
c^{k+1}_\cF( \|\gamma^\fl_0 - \gamma^\fk_0 \|^\Gamma_\eps )  \bigl( \| \ul\gamma^\fl \|^{\Gamma^k}_{\eps,\infty} \bigr)^k
+ 
k C^{k+1,\delta_Q}_\cF  \max_{1\leq i \leq k} \| \gamma^\fl_i - \gamma^\fk_i \|^{\Gamma}_\eps 
\bigl(\max_{1\leq i \leq k} \max\{ \| \gamma^\fl_i \|^{\Gamma}_\eps, \| \gamma^\fk_i \|^{\Gamma}_\eps\}  \bigr)^{k-1} 
\nonumber\\
&\leq
\bigl( c^{k+1}_\cF( \|\gamma^\fl_0 - \gamma^\fk_0 \|^\Gamma_\eps ) 
+ 
k C^{k+1,\delta_Q}_\cF \| \ul\gamma^\fl - \ul\gamma^\fk \|^{\Gamma^k}_{\eps,\infty} \bigr) \cdot
M_\eps(\ul\gamma^\fl ,\ul\gamma^\fk )^k  
\nonumber
\end{align}
with 
\begin{equation}\label{eq:M}
M_\eps(\ul\gamma^\fl,\ul\gamma^\fk):= \max\{ 1, \| \ul\gamma^\fl \|^{\Gamma^k}_{\eps,\infty} ,  \| \ul\gamma^\fk \|^{\Gamma^k}_{\eps,\infty}\} . 
\end{equation}
We now combine these two estimates with Lemma~\ref{lem:inverses} to bound the higher differentials of the inverse map $\rD^k Q_\eps$ for $1\leq k \leq \ell$ to obtain for all $\eps\in\Delta$, $\gamma_0\in\cV_\Gamma$ with $\|\gamma_0\|^\Gamma_\eps\leq\delta_Q$, and $\ul\gamma=(\gamma_1,\ldots,\gamma_k) \in\Gamma^k$
\begin{align} 
& \bigl\| \rD^k Q_\eps (\gamma_0)(\gamma_1,\ldots,\gamma_k) \bigr\|^{\cL(\ol V_\eps, \ol W_\eps)}
\nonumber\\
&\qquad \text{using the decomposition \eqref{eq:DQ parts} with $\gamma_* {\scriptstyle\ldots} \, , \ldots, \gamma_* {\scriptstyle\ldots}$ a permutation of $\gamma_1,\ldots,\gamma_k$}
\nonumber \\
&\leq 
\textstyle\sum_{*}
\bigl\|  Q_\eps (\gamma_0) \, \rD^{k_*} P_\eps(\gamma_0) (\gamma_* {\scriptstyle\ldots}) \, Q_\eps (\gamma_0) \, \ldots \, \rD^{k_*} P_\eps(\gamma_0) (\gamma_* {\scriptstyle\ldots}) \, Q_\eps(\gamma_0) 
 \bigr\|^{\cL(\ol V_\eps, \ol W_\eps)}
 \nonumber\\
&\qquad \text{where $k_*\geq 1$ add to $k_1+\ldots+k_{n_*} = k$ }
\nonumber \\
&\leq 
\textstyle\sum_{*}
\bigl\|  Q_\eps (\gamma_0)  \bigr\|^{\scriptscriptstyle\cL}
\bigl\|  \rD^{k_1} P_\eps(\gamma_0) (\gamma_* {\scriptstyle\ldots}) \bigr\|^{\scriptscriptstyle\cL}
\bigl\|  Q_\eps (\gamma_0)  \bigr\|^{\scriptscriptstyle\cL} {\scriptstyle\ldots}  
\bigl\|   \rD^{k_{n_*}} P_\eps(\gamma_0) (\gamma_* {\scriptstyle\ldots})  \bigr\|^{\scriptscriptstyle\cL(\ol W_\eps, \ol V_\eps)}
\bigl\|  Q_\eps(\gamma_0)  \bigr\|^{\scriptscriptstyle\cL(\ol V_\eps, \ol W_\eps)}
\nonumber\\
&\qquad \text{using \eqref{eq:Q-estimate} with $C_Q\geq 1$}
\nonumber \\
&\leq 
\textstyle\sum_{*} (C_Q)^{n_*+1}
\bigl\|  \rD^{k_1} P_\eps(\gamma_0) (\gamma_* {\scriptstyle\ldots})  \bigr\|^{\cL(\ol V_\eps, \ol W_\eps)}  \ldots \bigl\| \rD^{k_{n_*}} P_\eps(\gamma_0) (\gamma_* {\scriptstyle\ldots}) \bigr\|^{\cL(\ol V_\eps, \ol W_\eps)}
\label{eq:DQ bound}
\\
&\qquad \text{using \eqref{eq:DP bound}} \nonumber \\
&\leq 
\textstyle\sum_{*} (C_Q)^{n_*+1} \, 
\ti C^{k_1+1}_\cF  \|\gamma_*\|^\Gamma_\eps \ldots  \|\gamma_*\|^\Gamma_\eps  \ldots \, \ti C^{k_{n_*}+1}_\cF  \|\gamma_*\|^\Gamma_\eps \ldots  \|\gamma_*\|^\Gamma_\eps
 \nonumber\\
&\leq 
C^k_{\rD Q} \|\gamma_1\|^\Gamma_\eps \ldots  \|\gamma_k\|^\Gamma_\eps 
\;\leq\; 
C^k_{\rD Q}  \bigl( \| \ul\gamma \|^{\Gamma^k}_{\eps,\infty} \bigr)^k
\text{with}\quad C^k_{\rD Q} := (C_Q)^{k+1} \bigl( \textstyle\sum_{*}  
\ti C^{k_1+1}_\cF  \, \ldots \, \ti C^{k_{n_*}+1}_\cF \bigr). 
\nonumber
\end{align}
Similarly, we obtain uniform continuity of $\rD^k Q_\eps$ for $1\leq k \leq \ell$ in the sense that for all 
$\eps\in\Delta$, $\gamma^\fl_0,\gamma^\fk_0\in\cV_\Gamma$ with $\|\gamma^\fl_0\|^\Gamma_\eps,\|\gamma^\fk_0\|^\Gamma_\eps\leq\delta_Q$, and $\ul\gamma^\fl=(\gamma^\fl_1,\ldots,\gamma^\fl_k), \ul\gamma^\fk=(\gamma^\fk_1,\ldots,\gamma^\fk_k )\in\Gamma^k$
\begin{align*} 
& \bigl\| \rD^k Q_\eps (\gamma^\fl_0)(\gamma^\fl_1,\ldots,\gamma^\fl_k) 
	- \rD^k Q_\eps (\gamma^\fk_0)(\gamma^\fk_1,\ldots,\gamma^\fk_k)  \bigr\|^{\scriptscriptstyle\cL(\ol V_\eps, \ol W_\eps)}
\nonumber\\
&\qquad \text{using the decomposition \eqref{eq:DQ parts} with $\gamma^{\fl/\fk}_* {\scriptstyle\ldots}\, ,\ldots ,\gamma^{\fl/\fk}_*{\scriptstyle\ldots}$ a permutation of $\gamma^{\fl/\fk}_1,\ldots,\gamma^{\fl/\fk}_k$}
\nonumber \\
&\leq 
\textstyle\sum_{*}
\bigl\|  Q_\eps (\gamma^\fl_0) \, \rD^{k_*} P_\eps(\gamma^\fl_0) (\gamma^\fl_* {\scriptstyle\ldots}) \, Q_\eps (\gamma^\fl_0) \, \ldots  \, \rD^{k_*} P_\eps(\gamma^\fl_0) (\gamma^\fl_* {\scriptstyle\ldots} ) \, Q_\eps(\gamma^\fl_0) 
\nonumber\\
&\qquad\quad
- 
Q_\eps (\gamma^\fk_0) \, \rD^{k_*} P_\eps(\gamma^\fk_0) (\gamma^\fk_* {\scriptstyle\ldots}) \, Q_\eps (\gamma^\fk_0) \, \ldots  \, \rD^{k_*} P_\eps(\gamma^\fk_0) (\gamma^\fk_* {\scriptstyle\ldots} ) \, Q_\eps(\gamma^\fk_0) 
 \bigr\|^{\scriptscriptstyle\cL(\ol V_\eps, \ol W_\eps)}
 \nonumber\\
&\qquad \text{where $k_*\geq 1$ add to $k_1+\ldots+k_{n_*} = k$ }
\nonumber \\
&\leq 
\textstyle\sum_{*} \Bigl( 
\bigl\|  \bigl( Q_\eps (\gamma^\fl_0) - Q_\eps (\gamma^\fk_0) \bigr)\, \rD^{k_1} P_\eps(\gamma^\fl_0) (\gamma^\fl_* {\scriptstyle\ldots}) \, Q_\eps (\gamma^\fl_0) \, \ldots \, \rD^{k_{n_*}} P_\eps(\gamma^\fl_0) (\gamma^\fl_* {\scriptstyle\ldots} ) \, Q_\eps(\gamma^\fl_0)  \bigr\|^{\scriptscriptstyle\cL(\ol V_\eps, \ol W_\eps)} 
\nonumber\\
&\qquad
+
 \bigl\| 
Q_\eps (\gamma^\fk_0) \, \bigl( \rD^{k_1} P_\eps(\gamma^\fl_0) (\gamma^\fl_* {\scriptstyle\ldots}) - \rD^{k_1} P_\eps(\gamma^\fk_0) (\gamma^\fk_* {\scriptstyle\ldots}) \bigr) \, Q_\eps (\gamma^\fl_0) \, \ldots   \, \rD^{k_{n_*}} P_\eps(\gamma^\fl_0) (\gamma^\fl_* {\scriptstyle\ldots} ) \, Q_\eps(\gamma^\fl_0) 
 \bigr\|^{\scriptscriptstyle\cL}  
\nonumber \\
&\qquad + \ldots
\nonumber \\
&\qquad
+
 \bigl\| 
Q_\eps (\gamma^\fk_0) \, \rD^{k_1} P_\eps(\gamma^\fk_0) (\gamma^\fk_* {\scriptstyle\ldots}) \, Q_\eps (\gamma^\fk_0) \, \ldots \, \bigl( \rD^{k_{n_*}} P_\eps(\gamma^\fl_0) (\gamma^\fl_* {\scriptstyle\ldots} ) -  \rD^{k_{n_*}} P_\eps(\gamma^\fk_0) (\gamma^\fk_* {\scriptstyle\ldots} ) \bigr) \, Q_\eps(\gamma^\fl_0) 
 \bigr\|^{\scriptscriptstyle\cL}  
\nonumber \\
&\qquad
+
 \bigl\| 
Q_\eps (\gamma^\fk_0) \, \rD^{k_1} P_\eps(\gamma^\fk_0) (\gamma^\fk_* {\scriptstyle\ldots}) \, Q_\eps (\gamma^\fk_0) \, \ldots  \, \rD^{k_{n_*}} P_\eps(\gamma^\fk_0) (\gamma^\fk_* {\scriptstyle\ldots} ) \, \bigl( Q_\eps(\gamma^\fl_0)  - Q_\eps(\gamma^\fk_0)  \bigr)
 \bigr\|^{\scriptscriptstyle\cL(\ol V_\eps, \ol W_\eps)}  
\Bigl)
\nonumber \\
&\qquad \text{using \eqref{eq:Q-estimate} with $C_Q\geq 1$ and \eqref{eq:DP bound} with 
$\| \ul\gamma \|^{\Gamma^k}_{\eps,\infty} =  \max_{1\leq i \leq k} \| \gamma_i \|^{\Gamma}_\eps$}
\\
&\leq 
\textstyle\sum_{*} \Bigl( 
\bigl\|  Q_\eps (\gamma^\fl_0) - Q_\eps (\gamma^\fk_0)   \bigr\|^{\scriptscriptstyle\cL(\ol V_\eps, \ol W_\eps)} 
\ti C^{k_1+1}_\cF  \bigl(\|\ul\gamma^\fl\|^{\Gamma^k}_{\eps,\infty}\bigr)^{k_1} \,  
C_Q \; \ldots\;  \ti C^{k_{n_*}+1}_\cF  \bigl(\|\ul\gamma^\fl\|^{\Gamma^k}_{\eps,\infty}\bigr)^{k_{n_*}}  \, C_Q
\nonumber\\
&\qquad
+
C_Q  \bigl\|  \rD^{k_1} P_\eps(\gamma^\fl_0) (\gamma^\fl_* {\scriptstyle\ldots}) - \rD^{k_1} P_\eps(\gamma^\fk_0) (\gamma^\fk_* {\scriptstyle\ldots})  \bigr\|^{\scriptscriptstyle\cL(\ol W_\eps, \ol V_\eps)} 
C_Q 
\; \ldots\;  \ti C^{k_{n_*}+1}_\cF  \bigl(\|\ul\gamma^\fl\|^{\Gamma^k}_{\eps,\infty}\bigr)^{k_{n_*}}  \,  C_Q
\nonumber \\
&\qquad + \ldots
\nonumber \\
&\qquad
+
C_Q \, \ti C^{k_1+1}_\cF  \bigl(\|\ul\gamma^\fk\|^{\Gamma^k}_{\eps,\infty}\bigr)^{k_1}   \, C_Q 
 \; \ldots\; 
\bigl\| \rD^{k_{n_*}} P_\eps(\gamma^\fl_0) (\gamma^\fl_* {\scriptstyle\ldots} ) -  \rD^{k_{n_*}} P_\eps(\gamma^\fk_0) (\gamma^\fk_* {\scriptstyle\ldots} ) \bigr\|^{\scriptscriptstyle\cL(\ol W_\eps, \ol V_\eps)}    C_Q
\nonumber \\
&\qquad
+
C_Q \, \ti C^{k_1+1}_\cF   \bigl(\|\ul\gamma^\fk\|^{\Gamma^k}_{\eps,\infty}\bigr)^{k_1}   \, C_Q 
 \; \ldots\; 
 \ti  C^{k_{n_*}+1}_\cF  \bigl(\|\ul\gamma^\fk\|^{\Gamma^k}_{\eps,\infty}\bigr)^{k_{n_*}} 
 \bigl\| Q_\eps(\gamma^\fl_0)  - Q_\eps(\gamma^\fk_0)   \bigr\|^{\scriptscriptstyle\cL(\ol V_\eps, \ol W_\eps)}  
\Bigl)
\nonumber \\
&\qquad \text{using \eqref{eq:Q cont} and \eqref{eq:DP continuity} with
$M_\eps(\ul\gamma^\fl,\ul\gamma^\fk)= \max\{ 1, \| \ul\gamma^\fl \|^{\Gamma^k}_{\eps,\infty} ,  \| \ul\gamma^\fk \|^{\Gamma^k}_{\eps,\infty}\}$ }
\nonumber \\
&\leq 
\textstyle\sum_{*}  \Bigl( (n_*+1) (C_Q)^{n_*} \ti C^{k_1+1}_\cF \; \ldots\;  \ti C^{k_{n_*}+1}_\cF
 \, (C_Q)^2 \, c^1_\cF(\|\gamma^\fl_0- \gamma^\fk_0\|^\Gamma_\eps) \,
M_\eps(\ul\gamma^\fl , \ul\gamma^\fk)^{k_1+\ldots+k_{n_*}} 
\nonumber\\
&\qquad
+ (C_Q)^{n_*+1}
\bigl( c^{k_1+1}_\cF( \|\gamma^\fl_0 - \gamma^\fk_0 \|^\Gamma_\eps )  + 
k_1\, \ti C^{k_1+1}_\cF  \| \ul\gamma^\fl - \ul\gamma^\fk \|^{\Gamma^k}_{\eps,\infty} \bigr)
M_\eps(\ul\gamma^\fl , \ul\gamma^\fk)^{k_1} 
\nonumber \\
&\qquad\qquad\qquad\qquad\qquad\qquad\qquad\qquad\qquad\qquad\cdot
\ti C^{k_2+1}_\cF M_\eps(\ul\gamma^\fl , \ul\gamma^\fk)^{k_2}  \; \ldots\;  \ti C^{k_{n_*}+1}_\cF M_\eps(\ul\gamma^\fl , \ul\gamma^\fk)^{k_{n_*}} 
\nonumber \\
&\qquad + \ldots
\nonumber \\
&\qquad
+ (C_Q)^{n_*+1} \ti C^{k_1+1}_\cF M_\eps(\ul\gamma^\fl , \ul\gamma^\fk)^{k_1}  \; \ldots\;  \ti C^{k_{n_*-1}+1}_\cF M_\eps(\ul\gamma^\fl , \ul\gamma^\fk)^{k_{n_*-1}} 
\nonumber \\
&\qquad\qquad\qquad\quad\cdot
\bigl( c^{k_{n_*}+1}_\cF( \|\gamma^\fl_0 - \gamma^\fk_0 \|^\Gamma_\eps ) 
+k_{n_*}\, \ti C^{k_{n_*}+1}_\cF  \|\ul\gamma^\fl - \ul\gamma^\fk \|^{\Gamma^k}_\eps \bigr) M_\eps(\ul\gamma^\fl , \ul\gamma^\fk)^{k_{n_*}} 
\Bigr)
\nonumber\\
&\qquad \text{using $C_Q\geq 1$ and $\ti C^{k_*+1}_\cF\geq C^{k_*+1}_\cF \geq 1$} \nonumber\\
&\leq 
\textstyle\sum_{*}  (C_Q)^{n_*+1}  \ti C^{k_1+1}_\cF {\scriptstyle\ldots}\; \ti C^{k_{n_*}+1}_\cF   M_\eps(\ul\gamma^\fl , \ul\gamma^\fk)^k \Bigl( (n_*+1)  \, c^1_\cF(\|\gamma^\fl_0- \gamma^\fk_0\|^\Gamma_\eps) 
+ k \| \ul\gamma^\fl - \ul\gamma^\fk \|^{\Gamma^k}_\eps
\nonumber\\
&\qquad\qquad\qquad\qquad 
 +  c^{k_1+1}_\cF( \|\gamma^\fl_0 - \gamma^\fk_0 \|^\Gamma_\eps ) 
 + \ldots
+ c^{k_{n_*}+1}_\cF( \|\gamma^\fl_0 - \gamma^\fk_0 \|^\Gamma_\eps )  
 \Bigr) \nonumber
\\
&\qquad \text{using the constant $C^k_{\rD Q}= (C_Q)^{k+1} \bigl( \textstyle\sum_{*}  
\ti C^{k_1+1}_\cF  \, \ldots \, \ti C^{k_{n_*}+1}_\cF \bigr)$ from \eqref{eq:DQ bound}} \nonumber\\
&\leq 
C^k_{\rD Q} \, M_\eps(\ul\gamma^\fl , \ul\gamma^\fk)^k
 \Bigl( (k+1)  \, c^1_\cF(\|\gamma^\fl_0- \gamma^\fk_0\|^\Gamma_\eps)  +  k \| \ul\gamma^\fl - \ul\gamma^\fk \|^{\Gamma^k}_\eps 
\nonumber\\
&\qquad\qquad\qquad\qquad
  + k \max\{c^2_\cF( \|\gamma^\fl_0 - \gamma^\fk_0 \|^\Gamma_\eps ), \ldots, c^{k+1}_\cF( \|\gamma^\fl_0 - \gamma^\fk_0 \|^\Gamma_\eps ) \}  \Bigr) \nonumber 
\end{align*}
and hence 
\begin{align} \label{eq:DQ continuity}
& \bigl\| \rD^k Q_\eps (\gamma^\fl_0)(\gamma^\fl_1,\ldots,\gamma^\fl_k) 
	- \rD^k Q_\eps (\gamma^\fk_0)(\gamma^\fk_1,\ldots,\gamma^\fk_k)  \bigr\|^{\scriptscriptstyle\cL(\ol V_\eps, \ol W_\eps)} \\
&\qquad\qquad\qquad\qquad\qquad\qquad\qquad\qquad
\leq 
C^k_{\rD Q} \, \bigl( c^k_{\rD Q}\bigl( \|\gamma^\fl_0 - \gamma^\fk_0 \|^\Gamma_\eps  \bigr) +   \| \ul\gamma^\fl - \ul\gamma^\fk \|^{\Gamma^k}_\eps \bigr)   \, M_\eps(\ul\gamma^\fl , \ul\gamma^\fk)^k  
\nonumber
\end{align}
with the function $c^k_{\rD Q}:[0,\infty)\to[0,\infty)$, 
$c^k_{\rD Q}(x) :=  (k+1) \, c^1_\cF(x)  + k \max\{c^2_\cF(x), \ldots, c^{k+1}_\cF(x) \}$ which inherits monotonicity, continuity, and the value $c^k_{\rD Q}(0)=0$ from its constituents. 

The notation $\| \ul\gamma \|^{\Gamma^k}_{\eps,\infty}$ in \eqref{eq:Gamma infty} and $M_\eps(\ul\gamma^\fl,\ul\gamma^\fk)$ for tuples in $\Gamma^k$ in \eqref{eq:M} will in the following be applied to tuples $(\gamma_*,\ldots,\gamma_*)$ arising as subsets of vector entries of higher tangent vectors, when it is bounded by the fiber norm in \eqref{eq:fiber norm}. That is, for $\ul\gamma,\ul\gamma^\fl,\ul\gamma^\fk\in\rT^{\ell-1}\cV_\Gamma$ we have
$$
\| (\gamma_*,\ldots,\gamma_*) \|^{\Gamma^k}_{\eps,\infty} 
\;=\;   \max_{*} \| \gamma_* \|^{\Gamma}_\eps 
\;\leq\; \max_{1\leq i \leq N_{\ell-1}} \| \gamma_i \|^{\Gamma}_\eps 
\;\leq\; \| \ul\gamma \|^{\rT_\bullet^{\ell-1}\Gamma}_\eps
$$
and thus
\begin{align*}
M_\eps( (\gamma^\fl_*\ldots \gamma^\fl_*) , (\gamma^\fk_*\ldots\gamma^\fk_*) )
&=
\max\{ 1, \| (\gamma^\fl_*\ldots \gamma^\fl_*)\|^{\Gamma^k}_{\eps,\infty} ,  \| (\gamma^\fk_*\ldots\gamma^\fk_*) \|^{\Gamma^k}_{\eps,\infty}\} \\
&\leq 
\max\{ 1,  \| \ul\gamma^\fl \|^{\rT_\bullet^{\ell-1}\Gamma}_\eps  ,  \| \ul\gamma^\fk \|^{\rT_\bullet^{\ell-1}\Gamma}_\eps \}
\;=\; M_\eps(\ul\gamma^\fl,\ul\gamma^\fk) , 
\end{align*}
where we extended the notation $M_\eps(\ldots)$ to higher tangent vectors by 
$$
M_\eps(\ul\gamma^\fl,\ul\gamma^\fk) := \max\{ 1,  \| \ul\gamma^\fl \|^{\rT_\bullet^{\ell-1}\Gamma}_\eps  ,  \| \ul\gamma^\fk \|^{\rT_\bullet^{\ell-1}\Gamma}_\eps \bigr\} . 
$$
Finally, recall that the differential $\rD \rT^{\ell-1}\widetilde Q_\eps$ at a base point $(\ul\fc , \ul\gamma, \ul v)\in \rT^{\ell-1}\fC \times \rT^{\ell-1}\cV_\Gamma \times  \rT^{\ell-1} V$ is a Cartesian product of the identity on $\rT^{\ell-1}\fC \times \rT^{\ell-1}\cV_\Gamma$ and maps 
$ \rD\ul w (\ul\gamma,\ul v) : \rT^{\ell-1}\cV_\Gamma \times  \rT^{\ell-1} V \to \rT^{\ell-1} W$ whose components are all sums of operators of the three types in \eqref{eq:parts of TQ}.  
We will now deduce uniform bounds and continuity estimates for these operators from \eqref{eq:DQ bound} and \eqref{eq:DQ continuity}.

\smallskip
\noindent
{\bf Uniform Bounds on $O^k_{3,*}$:}
For any $1\leq k \leq \ell-1$, $\eps\in\Delta$, $\ul\gamma=(\gamma_0,\ldots)\in \rT^{\ell-1}\cV_\Gamma$ with $\|\gamma_0\|^\Gamma_\eps\leq\delta_Q$, and $\ul v \in \rT^{\ell-1} V$ we obtain from \eqref{eq:DQ bound}  
\begin{align*}
 \bigl\| O^k_{3,*}(\ul\gamma , \ul v) \bigr\|^{\cL(\rT^{\ell-1}\overline\Gamma_\eps\times\rT^{\ell-1}\overline V_\eps, \overline W_\eps)} 
&= 
\sup_{\|(\ul\xi,\ul y)\|\leq 1} 
\bigl\|
\rD^kQ_\eps(\gamma_0)(\gamma_*\ldots \gamma_*)] y_* \bigr\|^W_\eps
\\
&\leq
\sup_{\|y_*\|\leq 1}  C^k_{\rD Q} \,  \bigl( \| (\gamma_*\ldots \gamma_* ) \|^{\Gamma^k}_{\eps,\infty} \bigr)^k
 \|y_* \|^V_\eps 
\;\leq\;
 C^k_{\rD Q} \, \bigl( \| \ul\gamma \|^{\rT_\bullet^{\ell-1}\Gamma}_\eps \bigr)^k  .
\end{align*}

\smallskip
\noindent
{\bf Uniform Continuity of $O^k_{3,*}$:} For any $1\leq k \leq \ell-1$, $\eps\in\Delta$, $\ul\gamma^\fl=(\gamma^\fl_0,\ldots), \ul\gamma^\fk=(\gamma^\fk_0,\ldots )\in \rT^{\ell-1}\cV_\Gamma$ with $\|\gamma^\fl_0\|^\Gamma_\eps,\|\gamma^\fk_0\|^\Gamma_\eps\leq\delta_Q$, and
$\ul v^\fl, \ul v^\fk\in \rT^{\ell-1} V$ we obtain from \eqref{eq:DQ continuity} 
\begin{align*}
& \bigl\| O^k_{3,*}(\ul\gamma^\fl, \ul v^\fl)  - O^k_{3,*} (\ul\gamma^\fk , \ul v^\fk)  \bigr\|^{\cL(\rT^{\ell-1}\overline\Gamma_\eps\times\rT^{\ell-1}\overline V_\eps, \overline W_\eps)}  \\
&= 
\sup_{\|(\ul\xi,\ul y)\|\leq 1} 
\bigl\|
\rD^kQ_\eps(\gamma^\fl_0)(\gamma^\fl_*\ldots \gamma^\fl_*)] y_*
- \rD^kQ_\eps(\gamma^\fk_0)(\gamma^\fk_*\ldots \gamma^\fk_*)] y_* \bigr\|^W_\eps
\\
&\leq
\sup_{\|y_*\|\leq 1}
\bigl\|
\bigl( \rD^kQ_\eps(\gamma^\fl_0)(\gamma^\fl_*\ldots \gamma^\fl_*)] 
- \rD^kQ_\eps(\gamma^\fk_0)(\gamma^\fk_*\ldots \gamma^\fk_*)] \bigr) y_* \bigr\|^W_\eps
\\
&\leq
\sup_{\|y_*\|\leq 1}  C^k_{\rD Q} \bigl( c^k_{\rD Q}\bigl( \|\gamma^\fl_0 - \gamma^\fk_0 \|^\Gamma_\eps  \bigr) +   \| (\gamma^\fl_*\ldots \gamma^\fl_*) - (\gamma^\fk_*\ldots \gamma^\fk_*) \|^{\Gamma^k}_\eps \bigr) \, M_\eps((\gamma^\fl_*\ldots \gamma^\fl_* ), (\gamma^\fk_*\ldots \gamma^\fk_*))^k \|y_* \|^V_\eps \Bigr)  \\
&\leq
 C^k_{\rD Q} \bigl( c^k_{\rD Q}\bigl( \|\gamma^\fl_0 - \gamma^\fk_0 \|^\Gamma_\eps  \bigr) +   \| \ul\gamma^\fl - \ul\gamma^\fk \|^{\rT_\bullet^{\ell-1}\Gamma}_\eps \bigr) \, M_\eps(\ul\gamma^\fl ,\ul\gamma^\fk )^k  .
\end{align*}

\smallskip
\noindent
{\bf Uniform Bounds on $O^k_{2,*}$:} For any $1\leq k \leq \ell-1$, $\eps\in\Delta$, $\ul\gamma=(\gamma_0,\ldots)\in \rT^{\ell-1}\cV_\Gamma$ with $\|\gamma_0\|^\Gamma_\eps\leq\delta_Q$, and $\ul v \in \rT^{\ell-1} V$ we obtain from \eqref{eq:DQ bound}  
\begin{align*}
\bigl\| O^k_{2,*} (\ul\gamma , \ul v)  \bigr\|^{\cL(\rT^{\ell-1}\overline\Gamma_\eps\times\rT^{\ell-1}\overline V_\eps, \overline W_\eps)} 
&= 
\sup_{\|(\ul\xi,\ul y)\|\leq 1} \bigl\|
 \rD^kQ_\eps(\gamma_0)(\gamma_*\ldots \xi_* \ldots \gamma_*)] v_*
\bigr\|^W_\eps
\\
&\leq
\sup_{\|\xi_*\|\leq 1}  C^k_{\rD Q} \, \|\gamma_*\|^\Gamma_\eps \ldots  \|\xi_*\|^\Gamma_\eps \ldots 
 \|\gamma_*\|^\Gamma_\eps  \|v_* \|^V_\eps \\
 &\leq
 C^k_{\rD Q} \, 
 \bigl( \| (\gamma_*\ldots \gamma_*) \|^{\Gamma^{k-1}}_{\eps,\infty} \bigr)^{k-1}  \|v_* \|^V_\eps \\
&\leq
 C^k_{\rD Q} \, \bigl( \| \ul\gamma \|^{\rT_\bullet^{\ell-1}\Gamma}_\eps \bigr)^{k-1}  \| \ul v \|^{\rT_\bullet^{\ell-1} V}_\eps  .
\end{align*}

\smallskip
\noindent
{\bf Uniform Continuity of $O^k_{2,*}$:} For any $1\leq k \leq \ell-1$, $\eps\in\Delta$, $\ul\gamma^\fl=(\gamma^\fl_0,\ldots), \ul\gamma^\fk=(\gamma^\fk_0,\ldots)\in \rT^{\ell-1}\cV_\Gamma$  with $\|\gamma^\fl_0\|^\Gamma_\eps,\|\gamma^\fk_0\|^\Gamma_\eps\leq\delta_Q$, and
$\ul v^\fl, \ul v^\fk\in \rT^{\ell-1} V$ we obtain from \eqref{eq:DQ bound} and \eqref{eq:DQ continuity}
\begin{align*}
& \bigl\| O^k_{2,*}(\ul\gamma^\fl, \ul v^\fl)  - O^k_{2,*} (\ul\gamma^\fk , \ul v^\fk)  \bigr\|^{\cL(\rT^{\ell-1}\overline\Gamma_\eps\times\rT^{\ell-1}\overline V_\eps, \overline W_\eps)}  \\
&= 
\sup_{\|(\ul\xi,\ul y)\|\leq 1} \bigl\|
\rD^kQ_\eps(\gamma^\fl_0)(\gamma^\fl_*\ldots \xi_* \ldots \gamma^\fl_*)] v^\fl_*
- \rD^kQ_\eps(\gamma^\fk_0)(\gamma^\fk_*\ldots \xi_* \ldots \gamma^\fk_*)] v^\fk_*
\bigr\|^W_\eps
\\
&\leq
\sup_{\|\xi_*\|\leq 1} \Bigl( 
\bigl\| \rD^kQ_\eps(\gamma^\fl_0)(\gamma^\fl_*\ldots \xi_* \ldots \gamma^\fl_*)] (v^\fl_* -  v^\fk_*) \bigr\|^W_\eps
\\
&\qquad\qquad\qquad
+ \bigl\| \rD^kQ_\eps(\gamma^\fl_0)(\gamma^\fl_*\ldots \xi_* \ldots \gamma^\fl_*)] v^\fk_*
- \rD^kQ_\eps(\gamma^\fk_0)(\gamma^\fk_*\ldots \xi_* \ldots \gamma^\fk_*)] v^\fk_* \bigr\|^W_\eps
\Bigr)
\\
&\leq
\sup_{\|\xi_*\|\leq 1}  C^k_{\rD Q} \, \Bigl( 
\|\gamma^\fl_*\|^\Gamma_\eps \ldots  \|\xi_*\|^\Gamma_\eps \ldots 
 \|\gamma^\fl_*\|^\Gamma_\eps   \, \|v^\fl_* - v^\fk_* \|^V_\eps \\
&\qquad\qquad\qquad
+\bigl( c^k_{\rD Q}\bigl( \|\gamma^\fl_0 - \gamma^\fk_0 \|^\Gamma_\eps  \bigr) +   \| (\gamma^\fl_*\ldots \xi_* \ldots \gamma^\fl_*) - (\gamma^\fk_*\ldots \xi_* \ldots \gamma^\fk_*) \|^{\Gamma^k}_\eps \bigr) \\
&\qquad\qquad\qquad\qquad\qquad\qquad\qquad\qquad\qquad
\cdot M_\eps((\gamma^\fl_*\ldots \xi_* \ldots \gamma^\fl_* ), (\gamma^\fk_*\ldots \xi_* \ldots \gamma^\fk_*))^k \|v^\fk_* \|^V_\eps \Bigr) \\
&\leq
C^k_{\rD Q} \, \Bigl( 
\|\gamma^\fl_*\|^\Gamma_\eps \ldots  \|\gamma^\fl_*\|^\Gamma_\eps   \, \|v^\fl_* - v^\fk_* \|^V_\eps \\
&\qquad\qquad
+\bigl( c^k_{\rD Q}\bigl( \|\gamma^\fl_0 - \gamma^\fk_0 \|^\Gamma_\eps  \bigr) +  \| \ul\gamma^\fl - \ul\gamma^\fk \|^{\rT^{\ell-1}_\bullet\Gamma}_\eps \bigr) \, M_\eps((\gamma^\fl_* \ldots \gamma^\fl_* ), (\gamma^\fk \ldots \gamma^\fk_*))^k \|v^\fk_* \|^V_\eps \Bigr) \\
&\leq
 C^k_{\rD Q} \, \bigl( c^k_{\rD Q}\bigl( \|\gamma^\fl_0 - \gamma^\fk_0 \|^\Gamma_\eps  \bigr) +   \| \ul\gamma^\fl - \ul\gamma^\fk \|^{\rT_\bullet^{\ell-1}\Gamma}_\eps +  \| \ul v^\fl - \ul v^\fk \|^{\rT_\bullet^{\ell-1}V}_\eps \bigr)   
\, M_\eps(\ul\gamma^\fl , \ul\gamma^\fk)^k M_\eps(\ul v^\fl , \ul v^\fk)   , 
\end{align*}

\smallskip
\noindent
{\bf Uniform Bounds on $O^k_{1,*}$:} 
For any $1\leq k \leq \ell-1$, $\eps\in\Delta$, $\ul\gamma=(\gamma_0,\ldots)\in \rT^{\ell-1}\cV_\Gamma$ with $\|\gamma_0\|^\Gamma_\eps\leq\delta_Q$, and $\ul v \in \rT^{\ell-1} V$ we obtain from \eqref{eq:DQ bound}  
\begin{align*}
\bigl\| O^k_{1,*}(\ul\gamma, \ul v)   \bigr\|^{\cL(\rT^{\ell-1}\overline\Gamma_\eps\times\rT^{\ell-1}\overline V_\eps, \overline W_\eps)} 
&= 
\sup_{\|(\ul\xi,\ul y)\|\leq 1} \bigl\|
[\rD^{k+1}Q_\eps(\gamma_0)(\gamma_*\ldots\gamma_*,\xi_0)] v_*  \bigr\|^W_\eps
\\
&\leq
\sup_{\|\xi_0\|\leq 1}  C^{k+1}_{\rD Q} \, \|\gamma_*\|^\Gamma_\eps \ldots  \|\gamma_*\|^\Gamma_\eps
\|\xi_0\|^\Gamma_\eps  \| v_* \|^V_\eps \\
&\leq
C^{k+1}_{\rD Q} \,  \bigl( \| (\gamma_*\ldots \gamma_*) \|^{\Gamma^{k}}_{\eps,\infty} \bigr)^k   \| v_* \|^V_\eps \\
&\leq
 C^{k+1}_{\rD Q} \,  \bigl( \| \ul\gamma \|^{\rT_\bullet^{\ell-1}\Gamma}_\eps \bigr)^k   \| \ul v \|^{\rT_\bullet^{\ell-1}V}_\eps. 
\end{align*}

\smallskip
\noindent
{\bf Uniform Continuity of $O^k_{1,*}$:} For any $1\leq k \leq \ell-1$, $\eps\in\Delta$, $\ul\gamma^\fl=(\gamma^\fl_0,\ldots), \ul\gamma^\fk=(\gamma^\fk_0,\ldots)\in \rT^{\ell-1}\cV_\Gamma$  with $\|\gamma^\fl_0\|^\Gamma_\eps,\|\gamma^\fk_0\|^\Gamma_\eps\leq\delta_Q$, and
$\ul v^\fl, \ul v^\fk\in \rT^{\ell-1} V$ we obtain from \eqref{eq:DQ bound} and \eqref{eq:DQ continuity}
\begin{align*}
& \bigl\| O^k_{1,*}(\ul\gamma^\fl, \ul v^\fl)  - O^k_{1,*} (\ul\gamma^\fk , \ul v^\fk)  \bigr\|^{\cL(\rT^{\ell-1}\overline\Gamma_\eps\times\rT^{\ell-1}\overline V_\eps, \overline W_\eps)} \\
&= 
\sup_{\|(\ul\xi,\ul y)\|\leq 1} \bigl\|
[\rD^{k+1}Q_\eps(\gamma^\fl_0)(\gamma^\fl_*\ldots\gamma^\fl_*,\xi_0)] v^\fl_* 
-
[\rD^{k+1}Q_\eps(\gamma^\fk_0)(\gamma^\fk_*\ldots\gamma^\fk_*,\xi_0)] v^\fk_* 
\bigr\|^W_\eps
\\
&\leq 
\sup_{\|(\ul\xi,\ul y)\|\leq 1} \Bigl( \bigl\|
[\rD^{k+1}Q_\eps(\gamma^\fl_0)(\gamma^\fl_*\ldots\gamma^\fl_*,\xi_0)] (v^\fl_* 
- v^\fk_* ) \bigr\|^W_\eps
\\
&\qquad\qquad\qquad
+ \bigl\| [\rD^{k+1}Q_\eps(\gamma^\fl_0)(\gamma^\fl_*\ldots\gamma^\fl_*,\xi_0)] v^\fk_* 
- [\rD^{k+1}Q_\eps(\gamma^\fk_0)(\gamma^\fk_*\ldots\gamma^\fk_*,\xi_0)] v^\fk_* 
\bigr\|^W_\eps
\Bigr)
\\
&\leq
\sup_{\|\xi_0\|\leq 1}  C^{k+1}_{\rD Q} \, \Bigl( 
\|\gamma^\fl_*\|^\Gamma_\eps \ldots  \|\gamma^\fl_*\|^\Gamma_\eps \|\xi_0\|^\Gamma_\eps  \|v^\fl_* - v^\fk_* \|^V_\eps \\
&\qquad
+\bigl( c^{k+1}_{\rD Q}\bigl( \|\gamma^\fl_0 - \gamma^\fk_0 \|^\Gamma_\eps  \bigr) +   \| (\gamma^\fl_*\ldots \gamma^\fl_*,\xi_0) - (\gamma^\fk_*\ldots\gamma^\fk_*,\xi_0) \|^{\Gamma^{k+1}}_\eps \bigr) \, 
M_\eps((\gamma^\fl_*\ldots\gamma^\fl_*,\xi_0) , (\gamma^\fk_*\ldots\gamma^\fk_*,\xi_0))^{k+1} \|v^\fk_* \|^V_\eps \Bigr) \\
&\leq
 C^{k+1}_{\rD Q} \, \bigl( c^{k+1}_{\rD Q}\bigl( \|\gamma^\fl_0 - \gamma^\fk_0 \|^\Gamma_\eps  \bigr) +   \| \ul\gamma^\fl - \ul\gamma^\fk \|^{\rT_\bullet^{\ell-1}\Gamma}_\eps +  \| \ul v^\fl - \ul v^\fk \|^{\rT_\bullet^{\ell-1}V}_\eps \bigr)   
\, M_\eps(\ul\gamma^\fl , \ul\gamma^\fk)^{k+1}  M_\eps(\ul v^\fl , \ul v^\fk) . 
\end{align*}
Summarizing, the three types of operators $O^k_{*,*}$ for $1\leq k \leq \ell-1$ satisfy the following:  

\smallskip
\noindent
{\bf Uniform Bounds on $O^k_{*,*}$:} 
For any $\eps\in\Delta$, $(\ul\gamma,\ul v)=(\gamma_0,\ldots)\in \rT^{\ell-1}\cV_\Gamma\times \rT^{\ell-1} V \simeq \rT^{\ell-1}(\cV_\Gamma\times V)$ with $\|\gamma_0\|^\Gamma_\eps\leq\delta_Q$ we have 
\begin{align*}
\bigl\| O^k_{*,*}(\ul\gamma, \ul v)   \bigr\|^{\cL(\rT^{\ell-1}\overline\Gamma_\eps\times\rT^{\ell-1}\overline V_\eps, \overline W_\eps)} 
&\leq
 C^\ell_{\rD Q} \,  \max \bigl\{1, \| (\ul\gamma,\ul v) \|^{\rT_\bullet^{\ell-1}(\Gamma\times V)}_\eps \bigr\}^{k+1} 
\end{align*}
with $C^{\leq\ell}_{\rD Q}:=\max_{1\leq k \leq \ell}  C^k_{\rD Q}$. 

\smallskip
\noindent
{\bf Uniform Continuity of $O^k_{*,*}$:} For any $\eps\in\Delta$, $(\ul\gamma^\fl,\ul v^\fl)=(\gamma^\fl_0,\ldots), (\ul\gamma^\fk,\ul v^\fk)=(\gamma^\fk_0,\ldots)\in \rT^{\ell-1}\cV_\Gamma\times \rT^{\ell-1} V \simeq \rT^{\ell-1}(\cV_\Gamma\times V)$  with $\|\gamma^\fl_0\|^\Gamma_\eps,\|\gamma^\fk_0\|^\Gamma_\eps\leq\delta_Q$ we have
\begin{align*}
& \bigl\| O^k_{*,*}(\ul\gamma^\fl, \ul v^\fl)  - O^k_{1,*} (\ul\gamma^\fk , \ul v^\fk)  \bigr\|^{\cL(\rT^{\ell-1}\overline\Gamma_\eps\times\rT^{\ell-1}\overline V_\eps, \overline W_\eps)} \\
&\qquad\qquad \leq
 C^{\leq\ell}_{\rD Q} \, \bigl( c^{\leq\ell}_{\rD Q}\bigl( \|\gamma^\fl_0 - \gamma^\fk_0 \|^\Gamma_\eps  \bigr) +   \| (\ul\gamma^\fl,\ul v^\fl) - (\ul\gamma^\fk,\ul v^\fk) \|^{\rT_\bullet^{\ell-1}(\Gamma\times V)}_\eps  \bigr)   
\, M_\eps( (\ul\gamma^\fl,\ul v^\fl) , (\ul\gamma^\fk,\ul v^\fk))^{k+2} 
\end{align*}
with the new function $c^{\leq\ell}_{\rD Q}:[0,\infty)\to [0,\infty)$ given by
$c^{\leq\ell}_{\rD Q}(x):=\max_{1\leq k \leq \ell}  c^k_{\rD Q}(x)$, which inherits monotonicity, continuity, and the value $c^{\leq\ell}_{\rD Q}(0)=0$ from its constituents. . 
 
\smallskip

Since these operators sum to the nontrivial components of $\rD \rT^{\ell-1}\widetilde Q_\eps$ in \eqref{eq:parts of TQ}, we now obtain:

\smallskip
\noindent
{\bf Uniform Bounds on $\rD \rT^{\ell-1}\widetilde Q_\eps$:} 
For any $\eps\in\Delta$ and
$\ul z=(\ul\fc , \ul\gamma, \ul v)\in \rT^{\ell-1}\fC \times \rT^{\ell-1}\cV_\Gamma \times  \rT^{\ell-1} V\simeq \rT^{\ell-1} (\cV_W\times V)$ with $\|\gamma_0\|^\Gamma_\eps\leq\delta_Q$ we have
 \begin{align} 
& \bigl\|  \rD  \rT^{\ell-1}  \widetilde Q_\eps ( \ul z )  \bigr\|^{\cL(\rT^{\ell-1}(\overline W_\eps \times \overline V_\eps) , \rT^{\ell-1}(\overline W_\eps \times \overline W_\eps) )}  
\nonumber\\
&\qquad\qquad\qquad\qquad=
\sup_{\|( \ul\fc' , \ul\xi , \ul y)\|\leq 1} \bigl\| \bigl( \ul\fc' , \ul\xi,  \rD\ul w (\ul\gamma,\ul v)  (\ul\xi,\ul y) \bigr) \bigr\|^{\rT^{\ell-1}(W\times W)}_\eps 
\nonumber\\
&\qquad\qquad\qquad\qquad=
\sup_{\|( \ul\fc' , \ul\xi , \ul y)\|\leq 1} \Bigl( 
\bigl\| \bigl( \ul\fc' , \ul\xi \bigr) \bigr\|^{\rT^{\ell-1} W}_\eps 
+ \textstyle\sum_{i=0}^{ N_{\ell-1}} 
 \bigl\|  \rD w_i (\ul\gamma, \ul v)  ( \ul\xi , \ul y) \bigr\|^W_\eps   \Bigr)
\label{eq:TQ bound} \\
&\qquad\qquad\qquad\qquad\leq 
1 + 
\textstyle\sum_{i=0}^{ N_{\ell-1}} 
\textstyle\sum_{*,*}
 \bigl\| O^{k_*}_{*,*}(\ul\gamma, \ul v)  \bigr\|^{\cL(\rT^{\ell-1}\overline\Gamma_\eps\times\rT^{\ell-1}\overline V_\eps, \overline W_\eps)} 
 \nonumber \\
&\qquad\qquad\qquad\qquad\leq
1 + \textstyle\sum_{i=0}^{ N_{\ell-1}} 
\textstyle\sum_{*,*}    C^{\leq\ell}_{\rD Q} \, 
  \max \bigl\{1, \| (\ul\gamma,\ul v) \|^{\rT_\bullet^{\ell-1}(\Gamma\times V)}_\eps \bigr\}^{k_*+1}  
 \nonumber\\
 &\qquad\qquad\qquad\qquad\leq
 1 + \bigl(\textstyle\sum_{i=0}^{ N_{\ell-1}} 
\textstyle\sum_{*,*}    C^{\leq\ell}_{\rD Q} \bigr) \ \max \bigl\{1, \| (\ul\gamma,\ul v) \|^{\rT_\bullet^{\ell-1}(\Gamma\times V)}_\eps \bigr\}^{k_*+1}  
\nonumber \\
&\qquad\qquad\qquad\qquad\leq
 C^\ell_{\rT Q} \,   \max \bigl\{1, \| \ul z \|^{\rT_\bullet^{\ell-1}(W \times V) }_\eps \bigr\}^\ell 
\nonumber
\end{align}
with a new constant $C^\ell_{\rT Q}:= 1+ \textstyle\sum_{i=0}^{ N_{\ell-1}} \textstyle\sum_{*,*}    C^{\leq\ell}_{\rD Q}$.

\smallskip
\noindent
{\bf Uniform Continuity of $\rD \rT^{\ell-1}\widetilde Q_\eps$:} For any $\eps\in\Delta$ and
$\ul z^\fl=(\ul\fc^\fl , \ul\gamma^\fl, \ul v^\fl), \ul z^\fk=(\ul\fc^\fk , \ul\gamma^\fk, \ul v^\fk)\in \rT^{\ell-1}\fC \times \rT^{\ell-1}\cV_\Gamma \times  \rT^{\ell-1} V\simeq \rT^{\ell-1} (\cV_W\times V)$  with $\|\gamma^\fl_0\|^\Gamma_\eps,\|\gamma^\fk_0\|^\Gamma_\eps\leq\delta_Q$ we have

\begin{align} 
& \bigl\|  \rD  \rT^{\ell-1}  \widetilde Q_\eps (\ul z^\fl ) - 
 \rD  \rT^{\ell-1}  \widetilde Q_\eps (\ul z^\fk) \bigr\|^{\cL(\rT^{\ell-1}(\overline W_\eps \times \overline V_\eps) , \rT^{\ell-1}(\overline W_\eps \times \overline W_\eps))}  
 \nonumber\\
&=
 \bigl\|  \bigl( 0  ,  \rD \ul w (\ul\gamma^\fl, \ul v^\fl)  -  \rD \ul w (\ul\gamma^\fk , \ul v^\fk)  \bigr)  \bigr\|^{\cL(\rT^{\ell-1}(\overline W_\eps \times \overline V_\eps) , \rT^{\ell-1}(\overline W_\eps \times \overline W_\eps))} 
 \nonumber \\
 &=
\textstyle\sum_{i=0}^{ N_{\ell-1}} 
 \bigl\|  \rD w_i (\ul\gamma^\fl, \ul v^\fl)  -  \rD w_i (\ul\gamma^\fk , \ul v^\fk)  \bigr\|^{\cL(\rT^{\ell-1}\overline\Gamma_\eps\times\rT^{\ell-1}\overline V_\eps , \overline W_\eps)} 
 \label{eq:TQ continuity} \\
&\leq
\textstyle\sum_{i=0}^{ N_{\ell-1}} 
\textstyle\sum_{*,*}
 \bigl\| O^{k_*}_{*,*}(\ul\gamma^\fl, \ul v^\fl)  - O^{k_*}_{*,*} (\ul\gamma^\fk , \ul v^\fk)  \bigr\|^{\cL(\rT^{\ell-1}\overline\Gamma_\eps\times\rT^{\ell-1}\overline V_\eps, \overline W_\eps)} 
 \nonumber \\
&\leq 
\textstyle\sum_{i=0}^{ N_{\ell-1}} 
\textstyle\sum_{*,*}
C^{\leq\ell}_{\rD Q} \, \bigl( c^{\leq\ell}_{\rD Q}\bigl( \|\gamma^\fl_0 - \gamma^\fk_0 \|^\Gamma_\eps  \bigr) +   \| (\ul\gamma^\fl,\ul v^\fl) - (\ul\gamma^\fk,\ul v^\fk) \|^{\rT_\bullet^{\ell-1}(\Gamma\times V)} _\eps  \bigr)  
\nonumber \\
&\qquad\qquad\qquad\qquad\qquad\qquad\qquad\qquad\qquad\qquad\qquad\qquad
\cdot M_\eps((\ul\gamma^\fl,\ul v^\fl) ,(\ul\gamma^\fk,\ul v^\fk))^{k_*+2}  
\nonumber \\
&\leq 
 C^\ell_{\rT Q} \bigl( c^{\leq\ell}_{\rD Q}\bigl(  \|\gamma^\fl_0 - \gamma^\fk_0 \|^\Gamma_\eps \bigr) 
+  \| \ul z^\fl - \ul z^\fk \|^{\rT_\bullet^{\ell-1} (W \times V) }_\eps \bigr) \, M( \ul z^\fl , \ul z^\fk )^{\ell+1} . 
\nonumber \\
&=
 C^\ell_{\rT Q} \bigl( c^{\leq\ell}_{\rD Q}\bigl(  \|\gamma^\fl_0 - \gamma^\fk_0 \|^\Gamma_\eps \bigr) 
+  \| \ul z^\fl - \ul z^\fk \|^{\rT_\bullet^{\ell-1} (W \times V) }_\eps \bigr) \,
 \max\{ 1,  \| \ul z^\fl \|^{\rT_\bullet^{\ell-1}(W \times V)}_\eps  ,  \| \ul z^\fk\|^{\rT_\bullet^{\ell-1}(W \times V)}_\eps \bigr\}^{\ell+1} . 
\nonumber
\end{align}
Note here that $\gamma^\fl_0,\gamma^\fk_0\in\cV_\Gamma$ are the second entries of the base points of 
$\ul z^\fl,\ul z^\fk$, that is 
$$
\ul z^\fl=(\fc^\fl_0 , \gamma^\fl_0, v^\fl_0, \ldots )
\quad \text{and}\quad 
\ul z^\fk=(\fc^\fk_0 , \gamma^\fk_0, v^\fk_0, \ldots ).
$$

Finally, we have all preparations in place to prove adiabatic $\cC^{\ell+1}$ regularity in two claims.

\smallskip
\noindent
{\bf Claim: [Pointwise Continuity of $\mathbf\rD\rT^\ell\sigma_\eps$ in $\mathbf\Delta$]}
\label{page pointwise l}
{\it 
For any $\eps_0\in\Delta_\sigma$ and $(\ul \fk_0,\ul \fk_1) \in \rT^{\ell+1}\cV_\fK$ 
$$
\bigl\| \rD\rT^\ell\sigma_\eps(\ul \fk_0)\ul\fk_1 - \rD\rT^\ell\sigma_{\eps_0}(\ul \fk_0)\ul\fk_1 \bigr\|^{\rT^\ell W}_\eps \underset{\eps\to \eps_0}{\longrightarrow} 0 . 
$$}
To prove this claim first recall from Lemma~\ref{lem:inverses} that for each $\eps\in\Delta_\sigma$ -- in the above notation -- we have the identity $\widetilde Q_\eps \circ \widetilde P_\eps = {\rm Id}_{\rT\cV_W}$. This yields
$\rT^{\ell-1}\widetilde Q_\eps \circ  \rT^{\ell-1}\widetilde P_\eps =   \rT^{\ell-1}{\rm Id}_{\rT\cV_W}$ where $\rT^{\ell-1}{\rm Id}_{\rT\cV_W}$ is the identity map on $\rT^{\ell-1}\rT\cV_W = \rT^\ell \cV_W$, whose differential is again the identity map on $\rT^\ell  \cV_W$. The resulting identity 
$$
{\rm Id}_{\rT^\ell \cV_W}
=
\rD\rT^{\ell-1}\bigl( \widetilde Q_\eps \circ \widetilde P_\eps \bigr) ( \ul w_\eps) 
=
\rD\rT^{\ell-1} \widetilde Q_\eps \bigl(\rT^{\ell-1} \widetilde P_\eps ( \ul w_\eps)  \bigr)  \circ \rD\rT^{\ell-1} \widetilde P_\eps( \ul w_\eps) 
$$ 
holds for each base point $\ul w_\eps\in\rT^\ell \cV_W$ on the tangent space $\rT_{ \ul w_\eps}\rT^\ell \cV_W$, which is canonically identified with $\rT^\ell W$. 
Hence for fixed $\eps_0\in\Delta_\sigma$ and  $(\ul\fk_0,\ul\fk_1)\in\rT^{\ell+1}\cV_\fK$ we can apply the identity with varying base points $\ul w_\eps:=\rT^\ell \sigma_\eps (\ul \fk_0)$ to the fixed vector 
$\rD \rT^\ell \sigma_{\eps_0} (\ul \fk_0) \, \ul \fk_1\in \rT_{ \rT^\ell \sigma_{\eps_0}(\fk_0)}\rT^\ell \cV_W=\rT^\ell W$ 
to obtain
\begin{align}
\rD \rT^\ell \sigma_{\eps_0} (\ul \fk_0) \, \ul \fk_1
&=
\rD\rT^{\ell-1}\widetilde Q_\eps \bigl( \rT^{\ell-1}\widetilde P_\eps (\rT^\ell \sigma_\eps (\ul \fk_0)) \bigr) \,  \rD \rT^{\ell-1}\widetilde P_\eps (\rT^\ell \sigma_\eps (\ul \fk_0)) \, \rD \rT^\ell \sigma_{\eps_0} (\ul \fk_0) \, \ul \fk_1 \nonumber \\
&\overset{\eqref{eq:Psigma}}{=}
\rD\rT^{\ell-1}\widetilde Q_\eps \bigl( \rT^{\ell-1}\widetilde \sigma_\eps  (\ul \fk_0) \bigr) \,  \rD \rT^{\ell-1}\widetilde P_\eps (\rT^\ell \sigma_\eps (\ul \fk_0)) \, \rD \rT^\ell \sigma_{\eps_0} (\ul \fk_0) \, \ul \fk_1  . 
\label{eq:DTsigma0}
\end{align}
After these preparations we can now estimate for $\eps_0\in\Delta_\sigma$, $(\ul\fk_0,\ul\fk_1)\in\rT^{\ell+1}\cV_\fK$, and $\Delta_\sigma\ni\eps\to\eps_0$
\begin{align*}
& \bigl\| \rD \rT^\ell\sigma_\eps (\ul \fk_0) \,  \ul \fk_1- \rD\rT^\ell\sigma_{\eps_0}(\ul \fk_0) \,  \ul \fk_1 \bigr\|^{\rT^\ell W}_\eps  \\
&\qquad\qquad  \text{\small using \eqref{eq:DTsigma} and \eqref{eq:DTsigma0}} \\
&= 
\bigl\| \rD  \rT^{\ell-1}  \widetilde Q_\eps( \rT^{\ell-1} \widetilde \sigma_\eps (\ul \fk_0) )   \bigl( 
\rD  \rT^{\ell-1}  \widetilde \sigma_\eps (\ul \fk_0)   \,  \ul \fk_1   
- \rD \rT^{\ell-1}\widetilde P_\eps (\rT^\ell \sigma_\eps (\ul \fk_0)) \, \rD \rT^\ell \sigma_{\eps_0} (\ul \fk_0) \, \ul \fk_1
\bigr)  \bigr\|^{\rT^\ell W}_\eps 
\\
&\leq 
\bigl\|  \rD  \rT^{\ell-1}  \widetilde Q_\eps( \rT^{\ell-1} \widetilde \sigma_\eps (\ul \fk_0) )   \bigr\|^{\cL(\rT^{\ell-1}(\overline W_\eps \times \overline V_\eps) , \rT^\ell\overline W_\eps)} 
\Bigl( \bigl\|    \rD  \rT^{\ell-1}  \widetilde \sigma_\eps (\ul \fk_0)   \,  \ul \fk_1 
-  \rD \rT^{\ell-1}\widetilde P_\eps (\rT^\ell \sigma_\eps (\ul \fk_0))  \bigr) \rD \rT^\ell \sigma_{\eps_0} (\ul \fk_0) \, \ul \fk_1 \bigr\|^{\rT^{\ell-1}(W \times V)}_\eps
\Bigr) \\
&\qquad\qquad  \text{\small using  \eqref{eq:TQ bound}, the triangle inequality, and \eqref{eq:Psigma} at $\eps=\eps_0$ } \\
&\leq 
C^\ell_{\rT Q} \, \max\{1, \| \rT^{\ell-1} \widetilde \sigma_\eps (\ul \fk_0) \|^{\rT^{\ell-1} (W \times V) }_\eps \}^\ell
\Bigl( \bigl\|    \rD  \rT^{\ell-1}  \widetilde \sigma_\eps (\ul \fk_0)   \,  \ul \fk_1 - 
 \rD  \rT^{\ell-1}  \widetilde \sigma_{\eps_0} (\ul \fk_0)   \,  \ul \fk_1
\bigr\|^{\rT^{\ell-1}(W \times V)}_\eps \\
&\qquad\qquad
+ \bigl\| \bigl( \rD \rT^{\ell-1}\widetilde P_{\eps_0} (\rT^\ell \sigma_{\eps_0}(\ul \fk_0)) 
-  \rD \rT^{\ell-1}\widetilde P_\eps (\rT^\ell \sigma_\eps (\ul \fk_0))  \bigr) \rD \rT^\ell \sigma_{\eps_0} (\ul \fk_0) \, \ul \fk_1 \bigr\|^{\rT^{\ell-1}(W \times V)}_\eps
\Bigr) \\
&\qquad\qquad  \text{\small using \eqref{eq:DTsigma bound} and  \eqref{eq:DTsigma pointwise} for $(\ul\fk_0,\ul\fk_1)\simeq(\ul \fk^{\scriptscriptstyle ev},\ul \fk^{\scriptscriptstyle od}) \in \rT^\ell\cV_\fK \times \rT^\ell\fK$ 
and \eqref{eq:DTP pointwise} with notation \eqref{eq:base point}, \eqref{eq:vector}} \\
&\leq
C^\ell_{\rT Q} \,   \max\{1 ,\delta_\sigma  
+ c^{\ell-1}_\sigma(\| \ul\fk_0 \|^{\rT_\bullet^{\ell}\fK} ) \, b^{\ell-1}_\sigma(\|\fk_0\|^{\rT_\bullet^{\ell}\fK}) 
+ \| \ul \fk_0 \|^{\rT_\bullet^{\ell} \fK } \}^{\ell}
\Bigl( \bigl\|    \rD  \rT^{\ell-1}  \sigma_\eps (\fk^{\scriptscriptstyle ev})   - 
 \rD  \rT^{\ell-1}  \sigma_{\eps_0} (\fk^{\scriptscriptstyle ev})  \bigr\|^{\rT^{\ell-1}W}_\eps  \\
&\qquad\qquad
+  c^{\ell+1}_{\rT\cF}(\|\ul\gamma_\eps - \ul\gamma_{\eps_0} \|^{\rT^\ell\Gamma}_\eps ) \, \| \ul\xi_0 \|^{\rT^\ell\Gamma}_\eps
+ \bigl\|  \rT^{\ell+1}\cF_\eps (\ul\gamma_{\eps_0}, \ul\xi_0) - \rT^{\ell+1} \cF_{\eps_0} (\ul\gamma_{\eps_0},\ul\xi_0) \bigr\|^{\rT^{\ell+1}\Omega}_\eps 
\Bigr) \\
&\qquad\qquad  \text{\small using  the convergence in \eqref{eq:DTsigma pointwise} and \eqref{eq:DTP pointwise} } \\
& \quad\underset{\eps\to\eps_0}{\longrightarrow} \quad 0 .
\end{align*}
Here the last step requires an estimate for the maps
$\widetilde P_\eps  :  \fC\times\cV_\Gamma \times \fC\times\Gamma  \to   \fC\times\cV_\Gamma  \times   \fK \times \Omega$
given by 
$$
\widetilde P_\eps (  \fc, \gamma ,  \fc' ,\xi) 
= \bigl( \fc ,\gamma , P_\eps(\gamma) (\fc',\xi)  \bigr)
= \bigl( \fc ,\gamma ,  \pi_\fK(\xi) , \rD\cF_\eps(\gamma) \xi - \fc'    \bigr) . 
$$ 
These are independent from $\eps$ in the first three factors, and in the last factor amount to the maps
$\wh P_\eps: \cV_\Gamma\times\Gamma \times \fC \times \fC  \to \Omega, ( \gamma , \xi , \fc, \fc')  \mapsto \rD\cF_\eps(\gamma) \xi - \fc'$, which are the difference between
$\rD\cF_\eps:\rT\cV_\Gamma\to\Omega$ and the projection $\Pi:\rT\fC \to \Omega, (\fc,\fc') \mapsto \fc'$.
Their higher tangent maps 
$\rT^{\ell-1} \wh P_\eps : \rT^{\ell-1} ( \rT\cV_\Gamma\times\rT\fC ) \to \rT^{\ell-1}\Omega$ are the difference between $\rT^{\ell-1}\rD\cF_\eps:\rT^\ell\cV_\Gamma\to\rT^{\ell-1}\Omega$ and the projection $\rT^{\ell-1}\Pi : \rT^{\ell-1}\rT\fC\simeq \rT^{\ell-1}\fC \times \rT^{\ell-1}\fC \to \rT^{\ell-1}\Omega, (\ul\fc,\ul\fc') \mapsto \ul\fc'$.
Taking the differential of this identification then identifies $\rD\rT^{\ell-1} \wh P_\eps$   
with the sum of $\rD\rT^{\ell-1}\rD\cF_\eps:\rT^{\ell+1}\cV_\Gamma\to\rT^{\ell-1}\Omega$ and the $\eps$-independent map $\rT^{\ell-1}\Pi$.  
Now the term to be estimated is a difference between the maps $\rD\rT^{\ell-1}\widetilde P$ at $\eps_0$ and $\eps$ with base points 
\begin{equation} \label{eq:base point}
\ul w_{\eps_0}:=(\ul\gamma_{\eps_0},\ul\fc_{\eps_0}):= \rT^\ell \sigma_{\eps_0}(\ul \fk_0)
\quad\text{and}\quad
\ul w_\eps:=(\ul\gamma_\eps,\ul\fc_\eps):= \rT^\ell \sigma_\eps(\ul \fk_0)
\end{equation}
applied to the fixed vector 
\begin{equation} \label{eq:vector}
(\ul\xi_0, \ul\fc_0)\,:=\; \rD \rT^\ell \sigma_{\eps_0} (\ul \fk_0) \, \ul \fk_1 \;\in\; \rT_{\ul w_{\eps_0}}\rT^\ell\cV_W \;\simeq\; \rT_{\ul w_\eps}\rT^\ell\cV_W \;\simeq\; \rT^\ell\cV_\Gamma \times \rT^{\ell-1}\fC\times \rT^{\ell-1}\fC . 
\end{equation}
Using the above identifications we can bound this difference by 
\begin{align} 
& \bigl\| \bigl( \rD \rT^{\ell-1}\widetilde P_{\eps_0} (\ul w_{\eps_0}) 
-  \rD \rT^{\ell-1}\widetilde P_\eps (\ul w_\eps)  \bigr) (\ul\xi_0, \ul\fc_0) \bigr\|^{\rT^{\ell-1}(W \times V)}_\eps \nonumber\\
&=
\bigl\| \bigl(  \rD \rT^{\ell-1}\rD \cF_{\eps_0} (\ul\gamma_{\eps_0}) \, \ul\xi_0  -  \rT^{\ell-1}\Pi \, \ul\fc_0 \bigr)
-  \bigl(  \rD \rT^{\ell-1}\rD \cF_\eps (\ul\gamma_\eps) \, \ul\xi_0 -  \rT^{\ell-1}\Pi \, \ul\fc_0 \bigr) \bigr\|^{\rT^{\ell-1}\Omega}_\eps \nonumber\\
&=
\bigl\|  \rD \rT^{\ell-1}\rD \cF_{\eps_0} (\ul\gamma_{\eps_0}) \, \ul\xi_0  - \rD \rT^{\ell-1}\rD \cF_\eps (\ul\gamma_\eps) \, \ul\xi_0 \bigr\|^{\rT^{\ell-1}\Omega}_\eps \nonumber\\
&\leq
\bigl\|  \rD\rT^\ell \cF_\eps (\ul\gamma_\eps) \, \ul\xi_0  - \rD\rT^\ell \cF_{\eps_0} (\ul\gamma_{\eps_0}) \, \ul\xi_0  \bigr\|^{\rT^\ell\Omega}_\eps \nonumber\\
&\leq
\bigl\| \bigl( \rD\rT^\ell \cF_\eps (\ul\gamma_\eps) - \rD\rT^\ell \cF_\eps (\ul\gamma_{\eps_0}) \bigr) \, \ul\xi_0  \bigr\|^{\rT^\ell\Omega}_\eps
+
\bigl\|   \rD\rT^\ell \cF_\eps (\ul\gamma_{\eps_0}) \, \ul\xi_0 - \rD\rT^\ell \cF_{\eps_0} (\ul\gamma_{\eps_0}) \, \ul\xi_0 \bigr\|^{\rT^\ell\Omega}_\eps  \nonumber\\
&\qquad\qquad  \text{\small using \eqref{eq:T-ell-F uniform} with $\ell$ replaced by $\ell+1$ } \nonumber\\
&\leq 
c^{\ell+1}_{\rT\cF}(\|\ul\gamma_\eps - \ul\gamma_{\eps_0} \|^{\rT^\ell\Gamma}_\eps ) \, \| \ul\xi_0 \|^{\rT^\ell\Gamma}_\eps
+ \bigl\|  \rT^{\ell+1}\cF_\eps (\ul\gamma_{\eps_0}, \ul\xi_0) - \rT^{\ell+1} \cF_{\eps_0} (\ul\gamma_{\eps_0},\ul\xi_0) \bigr\|^{\rT^{\ell+1}\Omega}_\eps 
 \label{eq:DTP pointwise} \\
& \underset{\eps\to\eps_0}{\longrightarrow} 0.  \nonumber
\end{align}
To deduce the claimed convergence note that $\ul\xi_0$ is a fixed vector, $c^{\ell+1}_{\rT\cF}$ is continuous with $c^{\ell+1}_{\rT\cF}(0)=0$, and by the induction hypothesis 
$$
\bigl\| \ul\gamma_\eps - \ul\gamma_{\eps_0} \bigr\|^{\rT^\ell \Gamma}_\eps
\;\leq\;
\bigl\| (\ul\gamma_\eps,\ul\fc_\eps) - (\ul\gamma_{\eps_0},\ul\fc_{\eps_0}) \bigr\|^{\rT^\ell W}_\eps
\;=\;
\bigl\| \rT^\ell\sigma_\eps(\ul \fk_0) - \rT^\ell\sigma_{\eps_0}(\ul \fk_0) \bigr\|^{\rT^\ell W}_\eps
\;\underset{\eps\to\eps_0}{\longrightarrow}\; 0 .
$$ 
For the second term in \eqref{eq:DTP pointwise}, convergence is guaranteed by 
[Continuity of $\rT^{\ell+1}\cF_\eps$ in $\Delta$ rel.\ $\mathbf{\mathfrak C}$]
since $(\ul\gamma_{\eps_0},\ul\xi_0)\in\rT^{\ell+1}\cV_\Gamma$ is a part of the vector
$\bigl( \ul\gamma_{\eps_0} ,\ul\fc_{\eps_0}  , \ul\xi_0, \ul\fc_0  \bigr) = \rT^{\ell+1} \sigma_{\eps_0}(\ul \fk_0, \ul \fk_1 )$ which solves the stabilized equation 
$\rT^{\ell+1} \cF_\eps(\ul\gamma_{\eps_0},\ul\xi_0) = (\ul\fc_{\eps_0}, \ul\fc_0) \in \rT^{\ell+1}\fC$
by \eqref{eq:T stabilized equation}. 
This establishes pointwise continuity of the $(\ell+1)$-st tangent of the solution maps 
$\rT^{\ell+1}\sigma_\eps = (\rT^\ell\sigma_\eps ,  \rD \rT^\ell\sigma_\eps)$. To establish adiabatic $\cC^{\ell+1}$ regularity of the solution maps as in Definition~\ref{def:adiabatic C-l}, it remains to prove the following.

\smallskip
\noindent
{\bf Claim: [Uniform Continuity of $\mathbf\rD\rT^\ell\sigma_\eps$]}
\label{page uniform l}
{\it 
There are monotone continuous functions $\hat c^{\ell+1}_\sigma : [0,\infty) \to [0,\infty)$ 
and $\hat b^{\ell+1}_\sigma : [0,\infty) \to [1,\infty)$ 
with $\hat c^{\ell+1}_\sigma(0)= 0$ so that for all $\eps\in\Delta_\sigma$ 
and $\ul\fk=(\ul\fk_0,\ul\fk_1), \ul\fl=(\ul\fl_0,\ul\fl_1)\in\rT^{\ell+1}\cV_\fK = \rT^\ell\cV_\fK \times \rT^\ell \fK$ 
we have
$$
\bigl\|  \rD\rT^\ell\sigma_\eps (\ul\fl_0)\ul\fl_1  -   \rD\rT^\ell\sigma_\eps (\ul\fk_0)\ul\fk_1 \bigr\|^{\rT^\ell W}_\eps 
\leq \hat c^{\ell+1}_\sigma(\| \ul\fl - \ul\fk \|^{\rT^{\ell+1}\fK}) \, \hat b^{\ell+1}_\sigma(\max\{ \|\fl\|^{\rT_\bullet^{\ell+1}\fK} ,  \|\fk\|^{\rT_\bullet^{\ell+1}\fK} \})  . 
$$}
Once established, this claim combines with the induction hypothesis  \eqref{eq:T-ell-sigma uniform} to prove [Uniform Continuity of $\rT^{\ell+1}\sigma_\eps$]: For all $\eps\in\Delta_\sigma$ 
and $\ul\fk=(\ul\fk_0,\ul\fk_1), \ul\fl=(\ul\fl_0,\ul\fl_1)\in\rT^{\ell+1}\cV_\fK$ we have
\begin{align*}
& \bigl\|  \rT^{\ell+1}\sigma_\eps (\ul\fl_0,\ul\fl_1)  -   \rT^{\ell+1}\sigma_\eps (\ul\fk_0,\ul\fk_1) \bigr\|^{\rT^{\ell+1} W}_\eps  \\
&= 
\bigl\|  \rT^\ell\sigma_\eps (\ul\fl_0)  -   \rT^\ell\sigma_\eps (\ul\fk_0) \bigr\|^{\rT^\ell W}_\eps 
+ \bigl\|  \rD\rT^\ell\sigma_\eps (\ul\fl_0)\ul\fl_1  -   \rD\rT^\ell\sigma_\eps (\ul\fk_0)\ul\fk_1 \bigr\|^{\rT^\ell W}_\eps 
\\
&\leq
 c^\ell_\sigma(\|\ul\fl_0- \ul\fk_0\|^{\rT^\ell\fK}) \, b^\ell_\sigma(\max\{ \|\fl_0\|^{\rT_\bullet^\ell\fK} ,  \|\fk_0\|^{\rT_\bullet^\ell\fK} \}) 
+ \hat c^{\ell+1}_\sigma(\| \ul\fl - \ul\fk \|^{\rT^{\ell+1}\fK}) \, \hat b^{\ell+1}_\sigma(\max\{ \|\fl\|^{\rT_\bullet^{\ell+1}\fK} ,  \|\fk\|^{\rT_\bullet^{\ell+1}\fK} \})  
\\
&\leq
\bigl(  c^\ell_\sigma(\|\ul\fl- \ul\fk\|^{\rT^{\ell+1}\fK}) 
+ \hat c^{\ell+1}_\sigma(\| \ul\fl - \ul\fk \|^{\rT^{\ell+1}\fK}) \bigr) \, \max\{ 
b^\ell_\sigma(\max\{ \|\fl\|^{\rT_\bullet^{\ell+1}\fK} ,  \|\fk\|^{\rT_\bullet^{\ell+1}\fK} \}) , 
\hat b^{\ell+1}_\sigma(\max\{ \|\fl\|^{\rT_\bullet^{\ell+1}\fK} ,  \|\fk\|^{\rT_\bullet^{\ell+1}\fK} \})  
\\
&\leq  c^{\ell+1}_\sigma(\| \ul\fl - \ul\fk \|^{\rT^{\ell+1}\fK}) \, b^{\ell+1}_\sigma(\max\{ \|\fl\|^{\rT_\bullet^{\ell+1}\fK} ,  \|\fk\|^{\rT_\bullet^{\ell+1}\fK} \}) 
\end{align*}
where $c^{\ell+1}_\sigma(x):=c^\ell_\sigma(x) + \hat c^{\ell+1}_\sigma(x)$
and $b^{\ell+1}_\sigma(x):=\max\{b^\ell_\sigma(x), \hat b^{\ell+1}_\sigma(x)\}$  
defines the required functions $c^{\ell+1}_\sigma : [0,\infty) \to [0,\infty)$ 
and $b^{\ell+1}_\sigma : [0,\infty) \to [1,\infty)$, 
which inherit monotonicity, continuity, and the value $c^{\ell+1}_\sigma(0)= 0$ from their constituents. 

To prove the claim we consider $\eps\in\Delta_\sigma$ and $\ul\fk=(\ul\fk_0,\ul\fk_1), \ul\fl=(\ul\fl_0,\ul\fl_1)\in\rT^\ell\cV_\fK \times\rT^\ell\fK$ and denote
$$
\ul z_\eps^\fk\,:=\; \rT^{\ell-1} \ti\sigma_\eps(\ul \fk_0) \;=\; (\fc_\eps^\fk, \gamma_\eps^\fk, \ldots ) 
\qquad\text{and}\qquad
\ul z_\eps^\fl\,:=\; \rT^{\ell-1} \ti\sigma_\eps(\ul \fl_0) ;=\; (\fc_\eps^\fl, \gamma_\eps^\fl, \ldots ) , 
$$
where $(\fc_\eps^\fk, \gamma_\eps^\fk)=\sigma_\eps(\fk_0^0)$ and $(\fc_\eps^\fl, \gamma_\eps^\fl)=\sigma_\eps(\fl_0^0)$ are the solution maps applied to the base points of $\ul\fk_0=(\fk_0^0,\ldots)$ and $\ul\fl_0=(\fl_0^0,\ldots)$. 
Then we can use \eqref{eq:DTsigma bound} to estimate the scaling factor as a function of 
 $x= \max\{ \|\fl \|^{\rT_\bullet^{\ell+1}\fK} ,  \|\fk \|^{\rT_\bullet^{\ell+1}\fK} \}$ by
\begin{align}
M( \ul z_\eps^\fl , \ul z_\eps^\fk )
&= \max\{ 1,  \|   \ul z_\eps^\fl \|^{\rT_\bullet^{\ell-1}(W \times V)}_\eps  ,  \| \ul z_\eps^\fk \|^{\rT_\bullet^{\ell-1}(W \times V)}_\eps \bigr\} 
\nonumber \\
&= \max\{ 1,  \|  \rT^{\ell-1} \ti\sigma_\eps(\ul \fl_0) \|^{\rT_\bullet^{\ell-1}(W \times V)}_\eps  ,  \|  \rT^{\ell-1} \ti\sigma_\eps(\ul \fk_0) \|^{\rT_\bullet^{\ell-1}(W \times V)}_\eps \bigr\} 
\nonumber \\
&\leq 
 \max\{ 1, 
 \delta_\sigma  
+ \tilde c^{\ell-1}_{\sigma}(\| \ul\fl_0 \|^{\rT_\bullet^{\ell-1}\fK} ) \, b^{\ell-1}_\sigma(\|\fl_0\|^{\rT_\bullet^{\ell-1}\fK})  ,  
 \delta_\sigma  
+ \tilde c^{\ell-1}_{\sigma}(\| \ul\fk_0 \|^{\rT_\bullet^{\ell-1}\fK} ) \, b^{\ell-1}_\sigma(\|\fk_0\|^{\rT_\bullet^{\ell-1}\fK}) \bigr\}  
\label{eq:M est} \\
&\leq 
 \max\{ 1,  \delta_\sigma  
+ \tilde c^{\ell-1}_{\sigma}(x ) \, b^{\ell-1}_\sigma(x) \bigr\}
\nonumber  . 
\end{align}
With that we can finally estimate the continuity of $\rD\rT^\ell\sigma_\eps$ by 
\begin{align*}
& \bigl\|  \rD\rT^\ell\sigma_\eps (\ul\fl_0)\ul\fl_1  -  \rD \rT^\ell\sigma_\eps (\ul\fk_0)\ul\fk_1 \bigr\|^{\rT^\ell W}_\eps \\
&=
\bigl\| 
 \rD  \rT^{\ell-1}  \widetilde Q_\eps ( \ul z_\eps^\fl )   \,
\rD  \rT^{\ell-1}  \widetilde \sigma_\eps (\ul \fl_0) \, \ul\fl_1
- 
 \rD  \rT^{\ell-1}  \widetilde Q_\eps (\ul z_\eps^\fk )   \,
\rD  \rT^{\ell-1}  \widetilde \sigma_\eps (\ul \fk_0) \, \ul\fk_1
\bigr\|^{\rT^\ell W}_\eps 
 \\ 
&\leq
\bigl\| 
 \rD  \rT^{\ell-1}  \widetilde Q_\eps ( \ul z_\eps^\fl )   \,
\rD  \rT^{\ell-1}  \widetilde \sigma_\eps (\ul \fl_0) \, \ul\fl_1
- 
\rD  \rT^{\ell-1}  \widetilde Q_\eps (\ul z_\eps^\fl )   \,
\rD  \rT^{\ell-1}  \widetilde \sigma_\eps (\ul \fk_0) \, \ul\fk_1
\bigr\|^{\rT^\ell W}_\eps 
 \\ 
&\quad +
\bigl\| 
 \rD  \rT^{\ell-1}  \widetilde Q_\eps (\ul z_\eps^\fl )   \,
\rD  \rT^{\ell-1}  \widetilde \sigma_\eps (\ul \fk_0) \, \ul\fk_1
- 
 \rD  \rT^{\ell-1}  \widetilde Q_\eps (\ul z_\eps^\fk )   \,
\rD  \rT^{\ell-1}  \widetilde \sigma_\eps (\ul \fk_0) \, \ul\fk_1
\bigr\|^{\rT^\ell W}_\eps 
 \\ 
&\leq
\bigl\|  \rD  \rT^{\ell-1}  \widetilde Q_\eps ( \ul z_\eps^\fl )  \bigr\|^{\cL(\rT^{\ell-1}(\overline W_\eps \times \overline V_\eps) , \rT^{\ell-1}(\overline W_\eps\times \overline W_\eps))} 
\bigl\|\rD  \rT^{\ell-1}  \widetilde \sigma_\eps (\ul \fl_0) \ul\fl_1 -
\rD  \rT^{\ell-1}  \widetilde \sigma_\eps (\ul \fk_0) \ul\fk_1 \|^{\rT^{\ell-1}(W \times V)}_\eps
\\ 
&\quad +
\bigl\|  \rD  \rT^{\ell-1}  \widetilde Q_\eps ( \ul z_\eps^\fl ) - 
 \rD  \rT^{\ell-1}  \widetilde Q_\eps (\ul z_\eps^\fk ) \bigr\|^{\cL(\rT^{\ell-1}(\overline W_\eps \times \overline V_\eps) , \rT^{\ell-1}(\overline W_\eps\times \overline W_\eps))} 
\bigl\| \rD  \rT^{\ell-1}  \widetilde \sigma_\eps (\ul \fk_0) \, \ul\fk_1 \|^{\rT^{\ell-1}(W \times V)}_\eps
 \\ 
&\qquad\qquad
\text{\small using \eqref{eq:TQ bound} and \eqref{eq:TQ continuity} } \\
&\leq
C^\ell_{\rT Q} \, \max\{1, \| \ul z_\eps^\fl \|^{\rT_\bullet^{\ell-1} (W \times V) }_\eps \}^\ell 
\cdot
\bigl\| \rT^{\ell} \widetilde \sigma_\eps (\ul \fl) - \rT^{\ell}  \widetilde \sigma_\eps (\ul \fk) \|^{\rT^{\ell}(W \times V)}_\eps
\\ 
&\quad +
 C^\ell_{\rT Q} \bigl( c^{\leq\ell}_{\rD Q}\bigl(  \|\gamma^\fl_\eps - \gamma^\fk_\eps \|^\Gamma_\eps \bigr) 
+ \| \ul z_\eps^\fl - \ul z_\eps^\fk \|^{\rT_\bullet^{\ell-1} (W \times V) }_\eps \bigr) 
\,  M( \ul z_\eps^\fl , \ul z_\eps^\fk )^{\ell+1} 
\cdot \bigl\| \rT^{\ell}  \widetilde \sigma_\eps (\ul \fk) \|^{\rT^{\ell}(W \times V)}_\eps
\\
&\qquad\qquad
\text{\small using \eqref{eq:DTsigma continuity} and  \eqref{eq:DTsigma bound}   } \\
&\leq
C^\ell_{\rT Q} \, \max\{1, \| \ul z_\eps^\fl \|^{\rT^{\ell-1} (W \times V) }_\eps \}^\ell
\cdot \tilde c^\ell_{\sigma}(  \| \ul \fl - \ul \fk \|^{\rT^{\ell+1} \fK }   )  \, b^\ell_\sigma(\max\{ \|\fl \|^{\rT_\bullet^{\ell+1}\fK} ,  \|\fk \|^{\rT_\bullet^{\ell+1}\fK} \})
\\
&\qquad 
+   C^\ell_{\rT Q}  \,  M( \ul z_\eps^\fl , \ul z_\eps^\fk )^\ell 
 \bigl( c^{\leq\ell}_{\rD Q}\bigl(  \|\gamma^\fl_\eps - \gamma^\fk_\eps \|^\Gamma_\eps \bigr) 
+ \| \ul z_\eps^\fl - \ul z_\eps^\fk \|^{\rT_\bullet^{\ell-1} (W \times V) }_\eps \bigr)  
\cdot \bigl(\delta_\sigma  
+ \tilde c^\ell_{\sigma}(\| \ul\fk \|^{\rT_\bullet^{\ell+1}\fK} ) \, b^\ell_\sigma(\|\fk\|^{\rT_\bullet^{\ell+1}\fK})  \bigr)
\\
&\qquad\qquad
\text{\small using \eqref{eq:M est} with $x:= \max\{ \|\fl \|^{\rT_\bullet^{\ell+1}\fK} ,  \|\fk \|^{\rT_\bullet^{\ell+1}\fK} \}$}
\\
&\leq
C^\ell_{\rT Q} \,   \max\{ 1,  \delta_\sigma + \tilde c^{\ell-1}_{\sigma}(x ) \, b^{\ell-1}_\sigma(x) \bigr\}^\ell \Bigl( 
b^\ell_\sigma(x) \, 
 \tilde c^\ell_{\sigma}(  \| \ul \fl - \ul \fk \|^{\rT^{\ell+1} \fK }   ) 
\\
&\qquad\qquad \qquad \qquad \qquad \qquad \qquad \qquad
+ \bigl(\delta_\sigma  + \tilde c^\ell_{\sigma}(x) \, b^\ell_\sigma(x)  \bigr) \, 
 \bigl( c^{\leq\ell}_{\rD Q}\bigl(  \|\gamma^\fl_\eps - \gamma^\fk_\eps \|^\Gamma_\eps \bigr) 
+ \| \ul z_\eps^\fl - \ul z_\eps^\fk \|^{\rT_\bullet^{\ell-1} (W \times V) }_\eps \bigr)  \Bigr)
\\
&\qquad\qquad
\text{using \eqref{eq:DTsigma continuity} for  
$\ul z_\eps^\fl= \rT^{\ell-1} \ti\sigma_\eps(\ul \fl_0), \ul z_\eps^\fk= \rT^{\ell-1} \ti\sigma_\eps(\ul \fk_0)$ and
\eqref{eq:sigma continuity} for $\gamma^\fl_0=\sigma_\eps(\fk_0^0) , \gamma^\fk_0=\sigma_\eps(\fl_0^0)$
 } 
\\
&\leq
C^\ell_{\rT Q}  \max\bigl\{ 1,  \delta_\sigma + \tilde c^{\ell-1}_{\sigma}(x ) b^{\ell-1}_\sigma(x) \bigr\}^\ell 
\, \max\bigl\{ b^\ell_\sigma(x) ,  \delta_\sigma  + \tilde c^\ell_{\sigma}(x ) \, b^\ell_\sigma(x)  \bigr\}
 \\
&\qquad 
\cdot  \Bigl(   \tilde c^\ell_{\sigma}(  \| \ul \fl - \ul \fk \|^{\rT^{\ell+1} \fK }   )  
+  c^{\leq\ell}_{\rD Q}\bigl( c^0_\sigma( \|\fl^0_0-\fk^0_0 \|^\fK ) \bigr) 
+ \tilde c^{\ell-1}_{\sigma}\bigl(  \| \ul \fl_0 - \ul \fk_0 \|^{\rT^\ell \fK }  \bigr) \, b^{\ell-1}_\sigma(\max\{ \|\fl_0 \|^{\rT_\bullet^{\ell}\fK} , \|\fk_0 \|^{\rT_\bullet^{\ell}\fK} \} )  \Bigr)
\\
&\leq 
C^\ell_{\rT Q}  \max\bigl\{ 1,  \delta_\sigma + \tilde c^{\ell-1}_{\sigma}(x ) b^{\ell-1}_\sigma(x) \bigr\}^\ell 
\, \max\bigl\{ b^\ell_\sigma(x) ,  b^{\ell-1}_\sigma(x )  ( \delta_\sigma  + \tilde c^\ell_{\sigma}(x ) \, b^\ell_\sigma(x)  )\bigr\}
 \\
&\qquad 
\cdot  \Bigl(   \tilde c^\ell_{\sigma}(  \| \ul \fl - \ul \fk \|^{\rT^{\ell+1} \fK }   )  
+  c^{\leq\ell}_{\rD Q}\bigl( c^0_\sigma( \| \ul \fl - \ul \fk \|^{\rT^{\ell+1} \fK }  ) \bigr) 
+ \tilde c^{\ell-1}_{\sigma}\bigl( \| \ul \fl - \ul \fk \|^{\rT^{\ell+1} \fK }  \bigr) \Bigr)
\\
&\leq 
\hat b^{\ell+1}_\sigma\bigl( \max\{ \|\fl \|^{\rT_\bullet^{\ell+1}\fK} ,  \|\fk \|^{\rT_\bullet^{\ell+1}\fK} \} \bigr)  \, \hat c^{\ell+1}_\sigma(\| \ul\fl - \ul\fk \|^{\rT^{\ell+1}\fK})  , 
\end{align*}
where the functions $\hat b^{\ell+1}_\sigma : [0,\infty) \to [1,\infty)$ and $\hat c^{\ell+1}_\sigma : [0,\infty) \to [0,\infty)$ are given by 
\begin{align*}
\hat b^{\ell+1}_\sigma(x) &:=   C^\ell_{\rT Q}  \max\bigl\{ 1,  \delta_\sigma + \tilde c^{\ell-1}_{\sigma}(x ) b^{\ell-1}_\sigma(x) \bigr\}^\ell 
\, \max\bigl\{ b^\ell_\sigma(x) ,  b^{\ell-1}_\sigma(x )  ( \delta_\sigma  + \tilde c^\ell_{\sigma}(x ) \, b^\ell_\sigma(x)  )\bigr\} \\
\text{and} \qquad 
\hat c^{\ell+1}_\sigma(x) &:=    \tilde c^\ell_{\sigma}( x )  +  c^{\leq\ell}_{\rD Q}\bigl( c^0_\sigma( x  ) \bigr) + \tilde c^{\ell-1}_{\sigma}(x) , 
\end{align*}
and inherit monotonicity, continuity, and the value $\hat c^{\ell+1}_\sigma(0)= 0$ from their constituents. 
In particular, $\hat b^{\ell+1}_\sigma(x)\geq 1$ is guaranteed by $C^\ell_{\rT Q}\geq 1$ in \eqref{eq:TQ bound}. 
This finishes the last claim in the inductive step and thus proves for any $\ell\geq 1$ that adiabatic $\cC^\ell$ regularity of the adiabatic Fredholm family implies adiabatic $\cC^\ell$ regularity of the solution map. 
\end{proof}

\subsection{Finite Dimensional Reduction} \label{reduction}

This section finishes the proof of Theorem~\ref{thm:charts} in Corollary~\ref{cor:family finite dimensional reduction} by inserting the solution maps from \S\ref{solution} back into the remaining finite dimensional equations identified in \S\ref{contraction}. More precisely -- in the regularizing case -- the two components of the solution map $(\eps,\fk)\mapsto\sigma_\eps(\fk)\in\fC\times\Gamma$ play different roles: 
The $\fC$ component gives rise to the finite dimensional equation $(\eps,\fk)\mapsto {\rm Pr}_\fC(\sigma_\eps(\fk))$, whereas the $\Gamma$ component yields the map $(\eps,\fk)\mapsto \bigl(\eps,{\rm Pr}_\Gamma(\sigma_\eps(\fk))\bigr)$ to the solution spaces $\textstyle\bigcup_{\eps\in\Delta_\sigma} \{\eps\}\times \cF_\eps^{-1}(0)$. 

Starting again in the classical Fredholm setting of Theorem~\ref{thm:charts classical}, we work here directly with the form of the map $\cG_Q$ that was constructed in Lemma~\ref{lem:contraction classical}. Then this Lemma finishes the proof of Theorem~\ref{thm:charts classical}.


\begin{lemma} \label{lem:finite dimensional reduction classical}
Consider a Banach space $(\overline W,\|\cdot\|)$, an open subset $\cV_{\overline W}\subset\overline W$ 
containing $0=0_{\overline W}\in\cV_{\overline W}$, 
finite dimensional normed vector spaces $\fC,K$, 
an open subset $\cV_K\subset K$, 
and a continuous map 
\begin{align*}
\cG \,:\quad \cV_K \times \cV_{ \overline W} &\rightarrow \fC \times \overline W, \\
(k,w) &\mapsto \bigl( A(k,w) , w - B(k,w) \bigr) . 
\end{align*} 
Suppose it is close to the identity map on $\cV_{ \overline W}$ up to the finite dimensional factors $K,\fC$ and a contraction in the following sense: 
$A: \cV_K \times\cV_{ \overline W} \to \fC$ maps to the finite dimensional space $\fC$, 
and $B : \cV_K \times \cV_{ \overline W}  \to \overline W$ is a contraction near $0\in\overline W$ parametrized by $k\in\cV_K$ that satisfies \eqref{eq:contraction}, \eqref{eq:small}. 

Then there is a finite dimensional reduction describing the zero set $\cG^{-1}(0)$, locally, as the zero set of a map between finite dimensional spaces. That is, we have continuous maps 
$$
f: \cV_K\to\fC
\qquad\text{and}\qquad 
\phi: f^{-1}(0) \to \cG^{-1}(0)
$$ 
such that $\phi$ is a homeomorphism to an open subset of $\cG^{-1}(0)$. 
More precisely, Lemma~\ref{lem:solution classical} constructs a solution map $\sigma: \cV_K \to \cV_{\overline W}$, which induces the maps $f(k) := A \bigl(k , \sigma(k) \bigr)$ and $\phi(k) := (k , \sigma (k))$.
And with $\delta>0$ from \eqref{eq:small}, we have 
$$
\cG^{-1}(0)  \cap
\bigl\{ (k,w)\in\cV_K\times \cV_{\overline W} \,\big|\, \| w\|^W < \delta \bigr\}
\;=\;
\bigl\{   \phi(k)  \,\big|\, k \in f^{-1}(0) \bigr\} . 
$$
If, moreover, the map $\cG$ is $\cC^\ell$, then the finite dimensional reduction $f: \cV_K\to\fC$ is $\cC^\ell$ as well.
\end{lemma}

\begin{proof}
By Lemma~\ref{lem:solution classical}, the solutions of $w - B(k,w) = 0$ with $\|w\|^W<\delta$ are parametrized by a uniquely determined map $\sigma:\cV_K\to\overline\cV_W$. Thus we can rewrite the zero set of the original map: For $(k,w)\in\cV_K\times\cV_{\overline W}$ we have
\begin{align*}
\cG(k,w) = (0,0) , \quad \|w\|^W<\delta
 &\quad \Leftrightarrow \quad 
 \bigl( A(k,w) , w - B(k,w) \bigr) = (0,0) , \quad \|w\|^W<\delta  \\
 & \quad \Leftrightarrow \quad 
  A(k,w)=0, \quad  w= B(k,w)  , \quad \|w\|^W<\delta\\
& \quad \Leftrightarrow \quad 
  A(k,w)=0, \quad w=\sigma(k) \\
 &\quad \Leftrightarrow \quad 
  A(k,\sigma(k))=0 ,   \quad  (k,w)=(k,\sigma(k)). 
\end{align*}
The resulting map $f(k):= A(k,\sigma(k))$ has finite dimensional domain $\cV_K\subset K$ and target $\fC$. 
The map $\phi: f^{-1}(0) \to \cV_K\times\cV_{\overline W}$ given by $\phi(k):=(k,\sigma(k))$ takes values in $\cG^{-1}(0)$ by construction, is locally surjective as specified above. To check that $\phi$ is a homeomorphism to its image, note that ${\rm Pr}_K(\phi(k))=k$. This shows that $\phi$ is injective with continuous inverse ${\rm Pr}_K$. 

If, moreover, $\cG$ is $\cC^\ell$-regular, then both maps $A: \cV_K \times \cV_{ \overline W} \to \fC$ and
$B: \cV_K \times \cV_{ \overline W} \to \overline W$ are $\cC^\ell$-regular. The latter, combined with Lemma~\ref{lem:solution classical}, implies $\cC^\ell$ regularity of $\sigma$. Then $\cC^\ell$ regularity of 
 $f: \cV_K\to\fC$ $k\mapsto A \bigl(k , \sigma(k) \bigr)$ follows from the chain rule.
\end{proof}

This finishes the proof of Theorem~\ref{thm:charts classical}, so it remains to finish the proof of Theorem~\ref{thm:charts}.

\begin{corollary} \label{cor:family finite dimensional reduction}
Every adiabatic Fredholm family $\bigl( (\cF_\eps:\cV_\Gamma\to \Omega )_{\eps\in\Delta} , \ldots  \bigr)$ as in Definition~\ref{def:fredholm} has a finite dimensional reduction describing the union of completed zero sets, locally, as the zero set of a map between finite dimensional spaces. That is, the family induces maps 
$$
f:\Delta_\sigma\times\cV_\fK\to\fC
\qquad\text{and}\qquad 
\phi: f^{-1}(0) \to \textstyle\bigcup_{\eps\in\Delta_\sigma} \{\eps\}\times \overline\cF_\eps^{-1}(0)
$$ 
defined on neighbourhoods $\Delta_\sigma\subset\Delta$ of $0$ and $\cV_\fK\subset\fK$ of $0$ such that $\phi$ is injective and locally surjective in the sense that 
for some $\delta_\sigma>0$ we have 
$$
\bigl( \textstyle\bigcup_{\eps\in\Delta_\sigma} \{\eps\}\times \overline\cF_\eps^{-1}(0)  \bigr) \cap
\bigl\{ (\eps,\gamma) \,\big|\, \eps\in\Delta_\sigma, \gamma\in\cV_{\overline \Gamma,\eps}, \|\gamma\|^\Gamma_\eps < \delta_\sigma \bigr\}
\;=\;
\bigr\{ \phi(\eps,\fk)  \,\big|\,  (\eps,\fk)\in f^{-1}(0) \bigr\} . 
$$
More precisely, $\Delta_\sigma$, $\cV_\fK$, and $\delta_\sigma$ are given by Theorem~\ref{thm:family solution} in the process of constructing
solution maps $\bigl( \sigma_\eps: \cV_\fK \to\cV_{\overline W,\eps} = \fC\times \cV_{\overline\Gamma,\eps} \bigr)_{\eps\in\Delta_\sigma}$, whose components define 
$f(\eps,\fk) := f_\eps(\fk):= {\rm Pr}_\fC \bigl( \sigma_\eps (\fk) \bigr) $ and
$\phi(\eps,\fk) :=  (\eps, \phi_\eps(\fk)) := (\eps,  {\rm Pr}_{\ol\Gamma_\eps} \bigl( \sigma_\eps(\fk)\bigr) \bigr)$. 
This construction also yields an explicit inverse $\phi^{-1}: \im\phi \to f^{-1}(0)$  given by 
$(\eps,\gamma) \mapsto (\eps, \ol\pi_\fK(\gamma))$. 
Finally, it has the following regularity properties: 

\begin{itemlist}
\item
If the adiabatic Fredholm family is fibrewise $\cC^\ell$-regular for some for some $\ell\geq1$ (which is automatic for $\ell=0$ and $\ell=1$) as in Definition~\ref{def:fibrewise C-l}, then the finite dimensional reduction is fibrewise $\cC^\ell$-regular in the sense that the maps
$f_\eps :  \cV_\fK \to \fC$ are $\cC^\ell$-regular and $\phi_\eps: f_\eps^{-1}(0)\to \ol \cF_\eps^{-1}(0)$ are homeomorphisms to their images for each $\eps\in\Delta_\sigma$.


\item
If the adiabatic Fredholm family is regularizing as in Definition~\ref{def:regularizing}, then its zero sets $\cF_\eps^{-1}(0)=\ol\cF_\eps^{-1}(0)$ agree with the completions and each map $\phi_\eps: f_\eps^{-1}(0)\to\cF_\eps^{-1}(0)\subset \Gamma$ takes values in the $\eps$-independent space $\Gamma$. 

\item
If the adiabatic Fredholm family is regularizing and adiabatic $\cC^0$-regular as in Definition~\ref{def:adiabatic C-l}, then the finite dimensional reduction is adiabatic $\cC^0$ in the sense that 
$f: \Delta_\sigma\times \cV_\fK \to \fC$ is continuous
and $\phi: f^{-1}(0) \to \textstyle\bigcup_{\eps\in\Delta_\sigma} \{\eps\}\times \cF_\eps^{-1}(0) \subset \Delta_\sigma\times \Gamma$ is a homeomorphism onto its image with respect to the relative topology induced by $\Delta_\sigma$ and $(\Gamma,\|\cdot\|^\Gamma_0)$. 

\item
If, furthermore, the adiabatic Fredholm family is adiabatic $\cC^\ell$-regular as in Definition~\ref{def:adiabatic C-l} for some $\ell\geq 1$, then the finite dimensional reduction is adiabatic $\cC^\ell$ in the sense that 
\begin{itemize}
\item 
the $\ell$-th tangent maps from differentiation in $\fK$ (see Definition~\ref{def:tangent map notation}) form a continuous map 
$$
\Delta_\sigma\times\rT^\ell\cV_\fK \to \rT^\ell\fC , \quad 
(\eps,\ul\fk)\mapsto \rT^\ell f_\eps(\ul\fk) ,
$$ 
\item
 the $\fK$-differentials of order $0\leq k \leq \ell$ (see Remark~\ref{rmk:multilinear}) form continuous maps 
$$
\Delta_\sigma\times\cV_\fK \to \cL^k(\fK^k,\fC) , \quad 
(\eps,\fk_0)\mapsto \rD^k f_\eps(\fk_0) . 
$$ 
\end{itemize}
\end{itemlist}
\end{corollary}

\begin{proof}
By Lemma~\ref{lem:contraction}, restricted to the subset $\Delta_\sigma\subset\Delta_Q$ from Theorem~\ref{thm:family solution}, the completed zero sets $\bigcup_{\eps\in\Delta_\sigma} \{\eps\}\times \overline\cF_\eps^{-1}(0) \simeq \overline\cG_Q^{-1}(0_{\Delta_\sigma})$ are naturally identified with the preimage of the "zero section" $0_{\Delta_\sigma} := \bigcup_{\eps\in\Delta_\sigma} \{\eps\}\times \{(0,0,0)\}$ under 
\begin{align*}
\overline\cG_Q : 
\textstyle \bigcup_{\eps\in\Delta_Q} \{\eps\} \times \fK \times \bigl( \cV_{\overline W,\eps} = \fC \times \cV_{\overline\Gamma,\eps} \bigr) &\rightarrow \textstyle \bigcup_{\eps\in\Delta_Q} \{\eps\} \times \fC \times \bigl( \overline W_\eps = \fC \times \overline\Gamma_\eps \bigr), \\
\bigl(\eps, \fk,  w=(\fc, \gamma) \bigr) &\mapsto  \bigl( \overline A_\eps(\fk,w), w - \overline B_\eps(\fk,w) \bigr) .
\end{align*}
Here $\overline A_\eps : \fK \times\cV_{\overline W,\eps} \to \fC$ maps to the finite dimensional space $\fC$ by projection $\overline A_\eps(\fk,\fc,\gamma)=\fc$. 
On the infinite dimensional factor, Theorem~\ref{thm:family solution} establishes $\overline B_\eps(\fk,\cdot)$ as contractions whose fixed points define solution maps $\bigl( \sigma_\eps: \cV_\fK \to\cV_{ \overline W,\eps} \bigr)_{\eps\in\Delta_\sigma}$ defined on neighbourhoods $\Delta_\sigma\subset\Delta$ of $0$ and $\cV_\fK\subset\fK$ of $0$ 
such that for some $\delta_\sigma>0$ we have 
$$
\bigl\{ (\fk,\fc,\gamma)\in \cV_\fK \times \fC\times \cV_{\overline\Gamma,\eps} \,\big|\,
\|\fc\|^\fC + \|\gamma\|^\Gamma_\eps < \delta_\sigma , 
 \overline\pi_\fK(\gamma)=\fk, 
\overline\cF_\eps(\gamma)  =  \fc \bigr\} 
 \;=\; \bigl\{  (\fk,  \sigma_\eps(\fk)  ) \,\big|\,  \fk\in\cV_\fK  \bigr\} .
$$
We denote the components of the solution maps
$\sigma_\eps : \cV_\fK \to\cV_{ \overline W,\eps}= \fC \times \cV_{ \overline\Gamma,\eps}$
by
$$
f(\eps,\cdot):= {\rm Pr}_\fC \circ \sigma_\eps : \cV_\fK \to \fC
\quad\text{and}\quad 
g(\eps,\cdot):= {\rm Pr}_{\ol\Gamma_\eps} \circ \sigma_\eps : \cV_\fK \to\cV_{ \overline\Gamma,\eps}
$$ 
and 
note that $g(\eps,\cdot):\cV_\fK \to \cV_{ \overline W,\eps}$ by construction takes values in $\overline\cF_\eps^{-1}(\fC)\cap\{ \gamma\in  \cV_{\overline\Gamma,\eps} \,\big|\,\|\gamma\|^\Gamma_\eps < \delta_\sigma \bigr\}$ and
solves $\overline\pi_\fK\circ g(\eps,\cdot) = {\rm Id}_\fK$. 
Then intersecting the above identity of sets with $\bigl\{  (\fk,\fc,\gamma)\in \cV_\fK \times \fC\times \cV_{\overline\Gamma,\eps}\,\big|\, \fk=\ol \pi_\fK(\gamma), \fc=0 \bigr\}$
yields
$$
\bigl\{ (\ol\pi_\fK(\gamma),0,\gamma) \,\big|\, \gamma\in \cV_{\overline\Gamma,\eps} , 
\|\gamma\|^\Gamma_\eps < \delta_\sigma , \overline\cF_\eps(\gamma)  =  0 \bigr\} 
 \;=\; \bigl\{  (\fk, 0 , g(\eps,\fk)  ) \,\big|\,  \fk\in\cV_\fK ,  f(\eps,\fk)=0 \bigr\}  . 
$$
Next, projecting the sets by 
$\cV_\fK \times \{0\} \times \cV_{\overline\Gamma,\eps}\to \cV_{\overline\Gamma,\eps}$ yields
$$
\bigl\{\gamma \in \cV_{\overline\Gamma,\eps} \,\big|\,
\|\gamma\|^\Gamma_\eps < \delta_\sigma , \overline\cF_\eps(\gamma)  =  0  \bigr\} 
 \;=\; \bigl\{  g(\eps,\fk) \,\big|\,   \fk\in\cV_\fK ,  f(\eps,\fk)=0 \bigr\}  ,
$$
where the $\fk\in\cV_\fK$ on the right hand side is uniquely determined by $\gamma= g(\eps,\fk)$. 
We keep track of this fact by observing that $g(\eps,\cdot): \{ \fk\in\cV_\fK \,|\, f(\eps,\fk)=0 \} \to \cV_{\overline\Gamma,\eps}$ is injective since $\gamma=g(\eps,\fk) \;\Rightarrow\; \ol\pi_\fK(\gamma)= \ol\pi_\fK(g(\eps,\fk))=\fk$. 
Now taking the union over $\eps\in\Delta_\sigma$ yields the claimed identity
\begin{align*}
\bigl( \textstyle\bigcup_{\eps\in\Delta_\sigma} \{\eps\}\times \overline\cF_\eps^{-1}(0)  \bigr) \cap
\bigl\{ (\eps,\gamma) \,\big|\, \eps\in\Delta_\sigma, \gamma\in\cV_{\overline \Gamma,\eps}, \|\gamma\|^\Gamma_\eps < \delta_\sigma \bigr\}  
& 
\\
=\; \bigl\{ (\eps,\gamma) \,\big|\,
\eps\in\Delta_\sigma, \gamma\in\cV_{\overline \Gamma,\eps}, \|\gamma\|^\Gamma_\eps < \delta_\sigma, 
 \overline\cF_\eps(\gamma)  =  0  \bigr\}
&=\;
\bigr\{ \phi(\eps,\fk)  \,\big|\,  (\eps,\fk)\in f^{-1}(0) \bigr\} 
\end{align*}
in terms of the maps
\begin{align*}
f \,:\quad \Delta_\sigma\times\cV_\fK &\;\to\; \fC  , & (\eps,\fk) &\;\mapsto\; {\rm Pr}_\fC \bigl( \sigma_\eps (\fk) \bigr) , \\
\phi \,:\;\;\quad f^{-1}(0) \; &\;\to\; \textstyle\bigcup_{\eps\in\Delta_\sigma} \{\eps\}\times \overline\cF_\eps^{-1}(0)  , & (\eps,\fk) &\;\mapsto\; (\eps,  {\rm Pr}_{\ol\Gamma_\eps} \bigl( \sigma_\eps(\fk)\bigr) \bigr) . 
\end{align*}
Here $\phi$ takes values in the union of completed zero sets since each
$g(\eps,\cdot)={\rm Pr}_{\ol\Gamma_\eps} \circ \sigma_\eps:\cV_\fK \to \cV_{ \overline W,\eps}$ takes values in $\overline\cF_\eps^{-1}$.
It is locally surjective in the sense established above, and it is injective with explicit inverse
$\phi^{-1}: \im\phi \to f^{-1}(0)$  given by 
$(\eps,\gamma) \mapsto (\eps, \ol\pi_\fK(\gamma))$.

This finite dimensional reduction can also be seen as the result of applying Lemma~\ref{lem:finite dimensional reduction classical} for each fixed $\eps\in\Delta_\sigma$ to the map 
$\overline\cG_Q (\eps,\cdot,\cdot): \cV_\fK \times \cV_{\overline W,\eps} \to \fC \times  \overline W_\eps$, 
$(\fk,  w) \mapsto  \bigl( \overline A_\eps(\fk,w), w - \overline B_\eps(\fk,w) \bigr)$. 
The resulting finite dimensional reduction is given by $\fk \mapsto A_\eps \bigl(\fk , \sigma_\eps(\fk) \bigr)$ and $\fk \mapsto (\fk , \sigma_\eps (\fk)$. 
Since $A_\eps(\fk,w=(\fc,\gamma))={\rm Pr}_\fC(w)=\fc$, this gives rise to $f_\eps:\fk \mapsto A_\eps \bigl(\fk , \sigma_\eps(\fk) \bigr)= {\rm Pr}_\fC(\sigma_\eps(\fk))$. To identify the zero sets, we need to include the identification 
$\overline\cG_Q(\eps,\cdot)^{-1}(0)=\overline\cG(\eps,\cdot)^{-1}(0)\simeq \overline\cF_\eps^{-1}(0)$ from 
Lemmas~\ref{lem:equivalent-fredholm} and \ref{lem:contraction}, given by the projection ${\rm Pr}: (\overline\pi_\fK(\gamma),0, \gamma) \mapsto \gamma$. This results, as claimed, in $\phi_\eps(\fk)={\rm Pr}( \fk , \sigma_\eps (\fk))= {\rm Pr}_{\ol\Gamma_\eps} \bigl( \sigma_\eps(\fk)\bigr)$. The benefit of this identification is that we can now deduce fibrewise regularity from Lemma~\ref{lem:finite dimensional reduction classical}. 

If the adiabatic Fredholm family is fibrewise $\cC^\ell$-regular for some $\ell>1$, then each contraction $\overline B_\eps$ is $\cC^\ell$ by Lemma~\ref{lem:contraction} -- as is the linear map $\overline A_\eps$. 
Thus $f_\eps$ is $\cC^\ell$ by Lemma~\ref{lem:finite dimensional reduction classical}.

If the adiabatic Fredholm family is regularizing, then its zero sets $\cF_\eps^{-1}(0)=\ol\cF_\eps^{-1}(0)$ agree with the completions by Lemma~\ref{lem:equivalent-fredholm}.
This in particular guarantees that each map $\phi_\eps: f_\eps^{-1}(0)\to\cF_\eps^{-1}(0)\subset \Gamma$ takes values in the $\eps$-independent space $\Gamma$. 

If the adiabatic Fredholm family is regularizing and adiabatic $\cC^\ell$-regular, then Theorem~\ref{thm:family solution} shows that $(\eps,\ul\fk)\mapsto \rT^\ell\sigma_\eps(\ul\fk)$ is a continuous map
$\Delta_\sigma\times\rT^\ell\cV_\fK \to \bigl( \rT^\ell\cV_W , \|\cdot\|^{\rT^\ell W}_0\bigr)$. 
Its target space can be identified with the product $\rT^\ell\cV_W=\rT^\ell(\fC\times\cV_\Gamma)\simeq \rT^\ell\fC\times\rT^\ell\cV_\Gamma$. In this identification the first factor is the finite dimensional reduction map $\rT^\ell f_\eps={\rm Pr}_{\rT^\ell\fC}\circ\rT^\ell\sigma_\eps$ and the norm is 
$\|\ul w=(\ul\fc,\ul\gamma)\|^{\rT^\ell W}_0 = \|\ul\fc\|^{\rT^\ell \fC} + \|\ul\gamma\|^{\rT^\ell\Gamma}_0$.  
Thus continuity of 
$\Delta_\sigma\times\rT^\ell\cV_\fK \to \rT^\ell\fC$, $(\eps,\ul\fk)\mapsto \rT^\ell f_\eps(\ul\fk)$ follows directly from the adiabatic $\cC^\ell$ regularity of the solution maps.  

To deduce continuity of the higher differentials, recall from Remark~\ref{rmk:tangent map} that for $0\leq k\leq \ell$ the $k$-th tangent map $\rT^k f_\eps$ is contained in some components of the $\ell$-th tangent map $\rT^\ell f_\eps$. Moreover, we can recover the $k$-th differential $\rD^k f_\eps (\fk_0) (\fk'_1,\ldots, \fk'_{2^{\ell-1}})$ as the last entry of $\rT^k f (\fk_0,\fk'_1,\fk'_2,0,\fk'_4,0,\ldots,\fk'_{2^{\ell-1}},0,\ldots,0)$.
Thus -- renumbering components $\fk_i:=\fk'_{2^{i-1}}$ -- the above continuity of $\rT^\ell f_\eps$ also implies continuity of 
$\Delta_\sigma\times\cV_\fK \times \fK^k \to \fC$, $(\eps,\fk_0,\fk_1,\ldots,\fk_k)\mapsto \rD^\ell f_\eps(\fk_0)$ for all $0\leq k\leq \ell$. 
Since this map is linear in the last $k$ factors, and $\fK$ is finite dimensional, this implies the continuity of  
$\Delta_\sigma\times\cV_\fK \to \cL^k(\fK^k,\fC)$ ,
$(\eps,\fk_0)\mapsto \rD^k f_\eps(\fk_0)$
for all $0\leq k\leq \ell$.

Finally, the adiabatic $\cC^0$ regularity in particular implies that 
$\phi: f^{-1}(0) \to \Delta_\sigma\times \Gamma, (\eps,\fk) \mapsto \bigl(\eps,{\rm Pr}_{\Gamma}\bigl(\sigma_\eps(\fk)\bigr)$ is continuous with respect to the relative topology induced by $\Delta_\sigma$ and $(\Gamma,\|\cdot\|^\Gamma_0)$. It is a homeomorphism onto its image due to continuity of the explicit inverse map 
$\phi^{-1}: \im\phi \to f^{-1}(0)$, $(\eps,\gamma) \mapsto (\eps, \ol\pi_\fK(\gamma))$.
\end{proof}

\bibliographystyle{alpha}
\small
\bibliography{biblio}

\end{document}